\documentclass[reqno]{amsart}
\usepackage[top=25mm, bottom=25mm, footskip=8mm, left=60mm, right=25mm]{geometry}
\usepackage{times}
\usepackage{fancyhdr}

\numberwithin{equation}{section}

\clearpage{}%

\usepackage{harvard_acta}
\usepackage{amsmath,amsfonts,mathrsfs,amssymb}
\usepackage{bm}
\usepackage{url}

\usepackage{xcolor}

\usepackage{tikz}

\usepackage[noend]{algpseudocode}
\usepackage{algorithm,algorithmicx}

\algrenewcommand\alglinenumber[1]{\sf\scriptsize\color{blue}{#1}}
\algrenewcommand\algorithmicrequire{\textbf{Input:}}
\algrenewcommand\algorithmicensure{\textbf{Output:}}

\newcommand{\ignore}[1]{}

\usepackage[normalem]{ulem}

\newtheorem{theorem}{Theorem}[section]
\newtheorem{lemma}[theorem]{Lemma}
\newtheorem{proposition}[theorem]{Proposition}
\newtheorem{corollary}[theorem]{Corollary}

\newtheorem{example}[theorem]{Example}
\newtheorem{remark}[theorem]{Remark}
\newtheorem{warning}[theorem]{Warning}

\newcommand{\eps}{\varepsilon}

\newcommand{\econst}{\mathrm{e}}
\newcommand{\iunit}{\mathrm{i}}

\newcommand{\N}{\mathbb{N}}

\newcommand{\RR}{\mathbb{R}}
\newcommand{\CC}{\mathbb{C}}
\newcommand{\F}{\mathbb{F}}

\newcommand{\Sym}{\mathbb{H}}

\newcommand{\Id}{\mtx{I}}

\newcommand{\vct}[1]{\bm{#1}}
\newcommand{\mtx}[1]{\bm{#1}}

\newcommand{\lspan}{\operatorname{span}}
\newcommand{\trace}{\operatorname{trace}}
\newcommand{\range}{\operatorname{range}}
\newcommand{\rank}{\operatorname{rank}}
\newcommand{\diag}{\operatorname{diag}}

\newcommand{\intdim}{\operatorname{intdim}}
\newcommand{\srank}{\operatorname{srank}}

\newcommand{\psdle}{\preccurlyeq}

\newcommand{\pinv}{\dagger}

\newcommand{\nys}[1]{\langle #1 \rangle}

\newcommand{\real}{\operatorname{real}}

\newcommand{\abs}[1]{\vert #1 \vert}

\newcommand{\norm}[1]{\Vert #1 \Vert}
\newcommand{\lnorm}[1]{\left\Vert #1 \right\Vert}

\newcommand{\fnorm}[1]{\norm{#1}_{\mathrm{F}}}

\newcommand{\ip}[2]{\langle #1, \, #2 \rangle}

\newcommand{\triplenorm}[1]{\vert\!\vert\!\vert #1 \vert\!\vert\!\vert}

\newcommand{\diff}[1]{\mathrm{d}{#1}}

\newcommand{\argmin}{\arg\min}

\newcommand{\Expect}{\operatorname{\mathbb{E}}}
\newcommand{\Prob}[1]{\mathbb{P}\left\{ #1 \right\}}

\newcommand{\Var}{\operatorname{Var}}

\newcommand{\normal}{\textsc{normal}}
\newcommand{\unif}{\textsc{unif}}

\usepackage{kbordermatrix}
\usepackage{bm}
\usepackage{graphicx}
\usepackage{color}
\renewcommand{\vct}[1]{\bm{\mathsf{#1}}}
\renewcommand{\mtx}[1]{\bm{\mathsf{#1}}}

\newcommand{\bigO}{O}
\clearpage{}%

\title[Randomized Numerical Linear Algebra]{Randomized Numerical Linear Algebra: \\ Foundations \& Algorithms}
\author[Martinsson and Tropp]{%
Per-Gunnar Martinsson \\
\texttt{pgm@oden.utexas.edu} \\
\and
Joel A.~Tropp\\
\texttt{jtropp@cms.caltech.edu}
}

\usepackage{xpatch}
\makeatletter
\xapptocmd{\@sect}{\csname #1mark\endcsname{#7}}{}{}
\makeatother

\renewcommand{\subsectionmark}[1]{}%

\begin{document}
\thispagestyle{empty}
\newgeometry{top=25mm,bottom=25mm,left=45mm,right=45mm}

\begin{center}

\phantom{.}\vspace{6mm}

{\huge\textbf{Randomized Numerical Linear Algebra: \\[4pt] Foundations \& Algorithms}}

\vspace{6mm}

Per-Gunnar Martinsson, University of Texas at Austin

\vspace{1mm}

Joel A.~Tropp, California Institute of Technology

\vspace{10mm}

\begin{minipage}{\textwidth}
\textbf{Abstract:}
This survey describes probabilistic algorithms
for linear algebra computations,
such as factorizing matrices and solving linear systems.
It focuses on techniques that have a proven track record for real-world problems.
The paper treats both the theoretical foundations
of the subject and practical computational issues.

\vspace{2mm}

Topics include
norm estimation;
matrix approximation by sampling;
structured and unstructured random embeddings;
linear regression problems;
low-rank approximation;
subspace iteration and Krylov methods;
error estimation and adaptivity;
interpolatory and CUR factorizations;
Nystr\"om approximation of positive semidefinite matrices;
single-view (``streaming'') algorithms;
full rank-revealing factorizations;
solvers for linear systems; and
approximation of kernel matrices that arise
in machine learning and in scientific computing.
\end{minipage}
\end{center}

\vfill

\tableofcontents

\clearpage

\restoregeometry

\section{Introduction}

Numerical linear algebra (NLA) is one of the great achievements
of scientific computing.
On most computational platforms, we can now routinely and automatically
solve small- and medium-scale linear algebra problems to high precision.
The purpose of this survey is to describe a set of probabilistic
techniques that have joined the mainstream of NLA over the last decade.
These new techniques have accelerated everyday computations
for small- and medium-size problems,
and they have enabled large-scale computations
that were beyond the reach of classical methods.

\subsection{Classical numerical linear algebra}

NLA definitively treats several major classes of problems,
including

\begin{itemize}  \setlength{\itemsep}{1mm}
\item	solution of dense and sparse linear systems;
\item	orthogonalization, least-squares, and Tikhonov regularization;
\item	determination of eigenvalues, eigenvectors, and invariant subspaces;
\item	singular value decomposition (SVD) and total least-squares.
\end{itemize}

In spite of this catalog of successes, important challenges remain.
The sheer scale of certain datasets (terabytes and beyond) makes
them impervious to classical NLA algorithms.
Modern computing architectures (GPUs, multi-core CPUs, massively distributed systems)
are powerful, but this power can only be unleashed by algorithms
that minimize data movement and that are designed \textit{ab initio} with parallel computation in mind.
New ways to organize and present data (out-of-core, distributed, streaming) also demand
alternative techniques.

Randomization offers novel tools for addressing all of these challenges.
This paper surveys these new ideas, provides detailed descriptions of algorithms
with a proven track record, and outlines the mathematical techniques
used to analyze these methods. %

\subsection{Randomized algorithms emerge}

Probabilistic algorithms have held a central place in scientific computing
ever since Ulam and von Neumann's groundbreaking work on Monte Carlo methods in the 1940s.
For instance, Monte Carlo algorithms are essential for high-dimensional integration
and for solving PDEs set in high-dimensional spaces.
They also play a major role in modern machine learning and uncertainty quantification.

For many decades, however, numerical analysts regarded randomized algorithms as a method of
last resort---to be invoked only in the absence of an effective deterministic alternative.
Indeed, probabilistic techniques have several undesirable features.
First, Monte Carlo methods often produce output with low accuracy.
This is a consequence of the central limit theorem, and in many situations it cannot be avoided.
Second, many computational scientists have a strong attachment to the engineering principle that
two successive runs of the same algorithm should produce identical results.
This requirement aids with debugging, and it can be critical for applications where
safety is paramount, for example simulation of infrastructure or control of aircraft.
Randomized methods do not generally offer this guarantee. (Controlling the seed of
the random number generator can provide a partial work-around, but this necessarily
involves additional complexity.)

Nevertheless, in the 1980s, randomized algorithms started to make inroads into NLA.
Some of the early work concerns spectral
computations, where it was already traditional
to use random initialization. %
\citeasnoun{Dix83:Estimating-Extremal} recognized that
(a variant of) the power method with a random start
\emph{provably} approximates the largest eigenvalue
of a positive semidefinite (PSD) matrix,
even without a gap between
the first and second eigenvalue. %
\citeasnoun{KW92:Estimating-Largest}
provided a sharp analysis of this
phenomenon for both the power method and the
Lanczos algorithm.
Around the same time, \citeasnoun{Gir89:Fast-Monte-Carlo}
and \citeasnoun{Hut90:Stochastic-Estimator} proposed
Monte Carlo methods for estimating the trace
of a large psd matrix.  Soon after,
\citeasnoun{1995_parker_randombutterfly}
demonstrated that randomized transformations
can be used to avoid pivoting steps in
Gaussian elimination.

Starting in the late 1990s, researchers in theoretical
computer science
identified other ways to apply
probabilistic algorithms in NLA.
\citeasnoun{AMS99:Space-Complexity} and \citeasnoun{AGMS02:Tracking-Join}
showed that randomized embeddings allow for
computations on streaming data with limited storage.
\cite{PRTV00:Latent-Semantic} and \cite{FKV04:Fast-Monte-Carlo}
proposed Monte Carlo methods for low-rank matrix approximation.
\cite{DKM06:Fast-Monte-Carlo-I},
\cite{DKM06:Fast-Monte-Carlo-II}, and
\cite{DKM06:Fast-Monte-Carlo-III}
wrote the first statement of
theoretical principles for randomized NLA.
\citeasnoun{Sar06:Improved-Approximation} showed
how subspace embeddings support linear
algebra computations.

In the mid-2000s, numerical analysts
introduced practical randomized
algorithms for low-rank matrix approximation
and least-squares problems.
This work includes the first computational evidence
that randomized algorithms outperform classical NLA algorithms
for particular classes of problems.
Early contributions include \cite{2006_martinsson_random1_orig,2007_martinsson_PNAS,RT08:Fast-Randomized,WLRT08:Fast-Randomized}.
These papers inspired later work, such as
\cite{2010_avron_BLENDENPIK,HMT11:Finding-Structure,HMST11:Algorithm-Principal},
that has made a direct impact in applications

Parallel with the advances in numerical analysis,
a tide of enthusiasm for randomized algorithms
has flooded into cognate fields.
In particular, stochastic gradient
descent~\cite{Bot10:Large-Scale-Machine}
has become a standard algorithm for solving
large optimization problems in machine learning.

At the time of writing, in late 2019,
randomized algorithms have
joined the mainstream of NLA.
They now appear in major reference works and textbooks
\cite{GVL13:Matrix-Computations-4ed,2019_strang_LA_and_data}.
Key methods are being incorporated into standard software
libraries \cite{2019_NAG_library27,2017_gu_langou,2017_ghysels_robust}.

\subsection{What does randomness accomplish?}
\label{sec:accomplishments}

Over the course of this survey we will explore a number
of different ways that randomization can be used to design
effective NLA algorithms.
For the moment, let us just summarize the most
important benefits.

Randomized methods can handle certain NLA problems
faster than any classical algorithm.  In Section~\ref{sec:overdet-ls},
we describe a randomized algorithm that can solve
a dense $m \times n$ least-squares problem with $m \gg n$
using about $O(mn + n^3)$ arithmetic operations
\cite{RT08:Fast-Randomized}.
Meanwhile, classical methods require $O(mn^2)$ operations.
In Section~\ref{sec:sparse-cholesky}, we present
an algorithm called \textsc{SparseCholesky} that can
solve the Poisson problem on a dense undirected graph
in time that is roughly \emph{quadratic} in the number
of vertices~\cite{KS16:Approximate-Gaussian}.
Standard methods have cost that is \emph{cubic}
in the number of vertices.
The improvements can be even larger for sparse graphs.

Randomization allows us to tackle problems
that otherwise seem impossible.
Section~\ref{sec:singlepass} contains an algorithm
that can compute a rank-$r$ truncated
SVD of an $m \times n$ matrix in a single pass
over the data using working storage $O(r (m + n))$.
The first reference for this kind of algorithm
is~\citeasnoun{WLRT08:Fast-Randomized}.
We know of no classical method with this
computational profile.

From an engineering point of view, randomization
has another crucial advantage: it allows us to
restructure NLA computations in a fundamentally different way. %
In Section~\ref{sec:random-rangefinder}, we will introduce
the randomized SVD algorithm \cite{2006_martinsson_random1_orig,HMT11:Finding-Structure}.
Essentially all the arithmetic in this procedure takes place
in a short sequence of matrix--matrix multiplications.
Matrix multiplication is a highly optimized primitive
on most computer systems;
it parallelizes easily; and it performs particularly well on
modern hardware such as GPUs.  In contrast, classical
SVD algorithms require either random access to the
data or sequential matrix--vector multiplications.
As a consequence, the randomized SVD can process
matrices that are beyond the reach of classical
SVD algorithms.

\subsection{Algorithm design considerations}

Before we decide what algorithm to use for a linear algebra computation,
we must ask how we are permitted to interact with the data.
A recurring theme of this survey is that randomization %
allows us to reorganize algorithms so that they control
whichever computational resource is the most scarce (flops, communication,
matrix entry evaluation, etc.).
Let us  illustrate with some representative examples:

\begin{itemize} \setlength{\itemsep}{1mm}

\item \textit{Streaming computations (``single-view''):}
There is rising demand for algorithms that can treat matrices
that are so large that they cannot be stored at all;
other applications involve matrices that are presented dynamically.
In the streaming setting,
the input matrix $\mtx{A}$ is given by a sequence of simple
linear updates that can viewed only once:
\begin{equation} \label{eqn:stream-intro}
\mtx{A} = \mtx{H}_1 + \mtx{H}_2 + \mtx{H}_3 + \dots.
\end{equation}
We must discard each innovation $\mtx{H}_i$ after it has been processed.
As it happens, the \emph{only} type of algorithm that can handle the
model~\eqref{eqn:stream-intro} is one based on randomized linear dimension reduction
\cite{LNW14:Turnstile-Streaming}.  Our survey describes a number of
algorithms that can operate in the streaming setting;
see Sections~\ref{sec:trace-est}, \ref{sec:schatten-p}, \ref{sec:nystrom}, and~\ref{sec:singlepass}.

\item \textit{Dense matrices stored in RAM:}
One traditional computational model for NLA assumes that the input matrix is stored in fast memory,
so that any entry can quickly be read and/or overwritten as needed.  The ability of
CPUs to perform arithmetic operations keeps growing rapidly, but memory latency has not kept up.
Thus, it has become essential to formulate blocked algorithms that operate on submatrices.
Section \ref{sec:full} shows how randomization can help.

\item \textit{Large sparse matrices:}
For sparse matrices, it is natural to search for techniques that interact with a matrix
only through its application to vectors, such as Krylov methods or subspace iteration.
Randomization expands the design space for these methods.
When the iteration is initialized with a random matrix, we can reach provably
correct and highly accurate results after a few iterations;
see Sections \ref{sec:rrf-subspace} and \ref{sec:rrf-krylov}.

Another idea is to apply randomized sampling to control sparsity levels.
This technique arises in Section \ref{sec:sparse-cholesky},
which contains a randomized algorithm
that accelerates incomplete Cholesky preconditioning for sparse graph Laplacians.

\item \textit{Matrices for which entry evaluation is expensive:}
In machine learning and computational physics, it is often desirable to solve linear
systems where it is too expensive to evaluate the full coefficient matrix.
Randomization offers a systematic way to extract data from the matrix
and to compute approximations that serve for the downstream applications.
See Sections~\ref{sec:kernel} and \ref{sec:rankstructured}.
\end{itemize}

\subsection{Overview}

This paper covers fundamental mathematical ideas,
as well as algorithms that have proved to be effective
in practice.  The balance shifts from theory at the
beginning toward computational practice at the end.
With the practitioner in mind, we have attempted to make
the algorithmic sections self-contained, so that they
can be read with a minimum of references to other
parts of the paper.

After introducing notation and covering preliminaries from linear algebra
and probability in Sections~\ref{sec:lin-alg}--\ref{sec:prob}, the
survey covers the following topics.

\begin{itemize} \setlength{\itemsep}{1mm}

\item	Sections~\ref{sec:trace-est}--\ref{sec:schatten-p}
discuss algorithms for trace estimation and Schatten $p$-norm estimation based on
randomized sampling (i.e., Monte Carlo methods).
Section~\ref{sec:max-eig} shows how iteration can improve the quality of estimates for maximum
eigenvalues, maximum singular values, and trace functions.

\item	Section~\ref{sec:matrix-mc} develops randomized sampling methods for
approximating matrices, including applications to matrix multiplication
and approximation of combinatorial graphs.

\item	Sections~\ref{sec:gauss}--\ref{sec:dimension-reduction} introduce
the notion of a randomized linear embedding.  These maps are frequently
used to reduce the dimension of a set of vectors, while preserving their
geometry.
In Section~\ref{sec:overdet-ls}, we explore several
ways to use randomized embeddings in the context of an overdetermined
least-squares problem.

\item	Sections~\ref{sec:random-rangefinder}--\ref{sec:error-est}
demonstrate how randomized methods can be used to find a subspace that
is aligned with the range of a matrix and to assess the quality of
this subspace.  Sections~\ref{sec:natural},~\ref{sec:nystrom}, and~\ref{sec:singlepass}
show how to use this subspace to compute a variety of low-rank
matrix approximations.

\item	Section~\ref{sec:full} develops randomized algorithms for computing a
factorization of a full-rank matrix, such as a pivoted QR decomposition
or a URV decomposition.

\item	Section~\ref{sec:linear-solve} describes some general approaches
to solving linear systems using randomized techniques.  Section~\ref{sec:sparse-cholesky}
presents the \textsc{SparseCholesky} algorithm for solving the Poisson
problem on an undirected graph (i.e., a linear system in a graph Laplacian).

\item	Last, Sections~\ref{sec:kernel} and~\ref{sec:rankstructured} show
how to use randomized methods to approximate kernel matrices that arise in machine
learning, computational physics, and scientific computing.

\end{itemize}

\subsection{Omissions}

While randomized NLA was a niche topic 15 years ago,
we have seen an explosion of research over the last decade.
This survey can only cover a small subset
of the many important and interesting ideas that have emerged.

Among many other omissions, we do not discuss spectral computations in detail.
There have been interesting and very recent developments,
especially for the challenging problem
of computing a spectral decomposition of a nonnormal
matrix~\cite{BGKS19:Pseudospectral-Shattering}.
We also had to leave out a treatment of
tensors and the rapidly developing field of randomized
multilinear algebra.

There is no better way to demonstrate the value of a numerical
method than careful numerical experiments that measure its speed
and accuracy against state-of-the-art implementations of competing
methods. Our selection of topics to cover is heavily influenced
by such comparisons; for reasons of space, we have often had to settle
for citations to the literature, instead of including the numerical evidence.

The intersection between optimization and linear algebra is
of crucial importance in applications, and it remains a fertile ground
for theoretical work. Randomized algorithms are invaluable in this context,
but we realized early on that the paper would double in length
if we included just the essential facts about randomized
optimization algorithms.

There is a complementary perspective on randomized
numerical analysis algorithms, called \emph{probabilistic numerics}.
See the website~\cite{HO:Probabilistic-Numerics} for a comprehensive bibliography.

For the topics that we do cover, we have made every effort to include
all essential citations. Nevertheless, the literature is vast, and we
are sure to have overlooked important work; we apologize in advance
for these oversights.

\subsection{Other surveys}

There are a number of other survey papers on randomized
algorithms for randomized NLA and related topics.

\begin{itemize}  \setlength{\itemsep}{1mm}
\item	\citeasnoun{HMT11:Finding-Structure} develop
and analyse computational methods for low-rank matrix approximation
(including the ``randomized SVD'' algorithm)
from the point of view of a numerical analyst.
The main idea is that randomization can furnish a subspace that captures the
action of a matrix, and this subspace can be used to build
structured low-rank matrix approximations.

\item	\citeasnoun{2011_mahoney_survey}
treats randomized methods for least-squares computations
and for low-rank matrix approximation.
He emphasizes
the useful principle that we can often decouple the linear algebra
and the probability when analyzing randomized NLA algorithms.

\item	\citeasnoun{2014_woodruff_sketching} %
describes how to use subspace embeddings as a primitive for developing randomized linear algebra algorithms.
A distinctive feature is the development of lower bounds.

\item	\citeasnoun{Tro15:Introduction-Matrix} %
gives an introduction to matrix concentration inequalities,
and includes some applications to randomized NLA algorithms.

\item The survey of \citeasnoun{2017_kannan_vempala_acta} appeared
in a previous volume of \emph{Acta Numerica}.
A unique aspect is the discussion of randomized tensor computations.

\item	\citeasnoun{2018_PCMI_mahoney_drineas} have
updated the presentation in~\citeasnoun{2011_mahoney_survey},
and include an introduction to linear algebra and probability
that is directed toward NLA applications.

\item	\citeasnoun{2018_PCMI_martinsson} %
focuses on computational aspects of randomized NLA.  A distinctive
feature is the discussion of efficient algorithms for factorizing
matrices of full, or nearly full, rank.

\item	\citeasnoun{Tro19:Matrix-Concentration-LN} %
gives a mathematical treatment of how matrix concentration
supports a few randomized NLA algorithms, and includes a
complete proof of correctness for the \textsc{SparseCholesky}
algorithm described in Section \ref{sec:sparse-cholesky}.
\end{itemize}

\subsection{Acknowledgments}

We are grateful to Arieh Iserles for proposing that we write this survey.
Both authors have benefited greatly from our collaborations with Vladimir Rokhlin
and Mark Tygert.
Most of all, we would like to thank Richard Kueng for his critical reading
of the entire manuscript, which has improved the presentation in many places.
Madeleine Udell, Riley Murray, James Levitt, and Abinand Gopal also gave us useful feedback on parts of the paper.
Lorenzo Rosasco offered invaluable assistance with the section on kernel methods
for machine learning.  Navid Azizan, Babak Hassibi, and Peter Richt{\'a}rik
helped with citations to the literature on SGD.  Finally, we would like to thank our
ONR programme managers, Reza Malek-Madani and John Tague, for supporting
research on randomized numerical linear algebra.

JAT acknowledges support from the Office of Naval Research (awards N-00014-17-1-2146 and N-00014-18-1-2363).
PGM acknowledges support from the Office of Naval Research (award  N00014-18-1-2354),
from the National Science Foundation (award DMS-1620472), and from Nvidia Corp.

\section{Linear algebra preliminaries}
\label{sec:lin-alg}

This section contains an overview of the linear algebra
tools that arise in this survey.
It collects the basic notation, along with
some standard and not-so-standard definitions.
It also contains a discussion about the role of the spectral norm.

Background references for linear algebra
and matrix analysis include
\citeasnoun{Bha97:Matrix-Analysis}
and
\citeasnoun{HJ13:Matrix-Analysis-2ed}.
For a comprehensive treatment of
matrix computations, we refer to
\cite{GVL13:Matrix-Computations-4ed,1997_trefethen_bau,1998_stewart_volume1,1998_stewart_volume2}.

\subsection{Basics}

We will work in the real field ($\RR$) or the complex field ($\CC$).
The symbol $\F$ refers to either the real or complex field,
in cases where the precise choice is unimportant.
As usual, scalars are denoted by lowercase italic
Roman ($a, b, c$) or Greek ($\alpha, \beta$) letters.

Vectors are elements of $\F^n$, where $n$ is a natural number.
We always denote vectors with lowercase bold
Roman ($\vct{a}, \vct{b}, \vct{u}, \vct{v}$)
or Greek ($\vct{\alpha}, \vct{\beta}$) letters.
We write $\vct{0}$ for the zero vector
and $\vct{1}$ for the vector of ones.
The standard basis vectors are denoted as
$\vct{\delta}_1, \dots, \vct{\delta}_n$.
The dimensions of these special vectors are
determined by context.

A general matrix is an element of $\F^{m \times n}$,
where $m, n$ are natural numbers.
We always denote matrices with
uppercase bold Roman ($\mtx{A}, \mtx{B}, \mtx{C}$)
or Greek ($\mtx{\Delta}, \mtx{\Lambda}$) letters.
We write $\mtx{0}$ for the zero matrix and $\Id$ for the identity matrix;
their dimensions are determined by a subscript or by context.

The parenthesis notation is used for indexing into vectors
and matrices: $(\vct{a})_i$ is the $i$th coordinate of vector $\vct{a}$,
while $(\mtx{A})_{ij}$ is the $(i, j)$th coordinate of matrix $\mtx{A}$.
In some cases it is more convenient to invoke the functional
form of indexing.  For example, $\mtx{A}(i, j)$ also refers
to the $(i, j)$th coordinate of the matrix $\mtx{A}$.

The colon notation is used to specify ranges of coordinates.
For example, $(\vct{a})_{1:i}$ and $\vct{a}(1:i)$ refer
to the vector comprising the first $i$ coordinates of $\vct{a}$.
The colon by itself refers to the entire range of coordinates.
For instance, $(\mtx{A})_{i:}$ denotes the $i$th row of $\mtx{A}$,
while $(\mtx{A})_{:j}$ denotes the $j$th column.

The symbol ${}^*$ is reserved for the (conjugate) transpose
of a matrix of vector.  A matrix that satisfies $\mtx{A} = \mtx{A}^*$
is said to be self-adjoint.
It is convenient to distinguish the space $\Sym_n$
of self-adjoint $n \times n$ matrices over
the scalar field.  We may write $\Sym_n(\F)$ if it
is necessary to specify the field.

The operator ${}^\pinv$ extracts the Moore--Penrose
pseudoinverse of a matrix.  More precisely,
for $\mtx{A} \in \F^{m \times n}$,
the pseudoinverse $\mtx{A}^\pinv \in \F^{n \times m}$
is the unique matrix that satisfies the following:

\begin{enumerate}
\item	$\mtx{AA}^\pinv$ is self-adjoint.
\item	$\mtx{A}^\pinv \mtx{A}$ is self-adjoint.
\item	$\mtx{AA}^\pinv \mtx{A} = \mtx{A}$.
\item	$\mtx{A}^\pinv \mtx{A} \mtx{A}^\pinv = \mtx{A}^\pinv$.
\end{enumerate}

\noindent
If $\mtx{A}$ has full column rank, then $\mtx{A}^\pinv = (\mtx{A}^* \mtx{A})^{-1} \mtx{A}^*$,
where $(\cdot)^{-1}$ denotes the ordinary matrix inverse.

\subsection{Eigenvalues and singular values}

A positive semidefinite matrix is a self-adjoint
matrix with nonnegative eigenvalues.  We will generally
abbreviate positive semidefinite to PSD.  Likewise,
a positive definite (PD) matrix is a self-adjoint
matrix with positive eigenvalues.

The symbol $\psdle$ denotes the semidefinite order on
self-adjoint matrices.  The relation $\mtx{A} \psdle \mtx{B}$
means that $\mtx{B} - \mtx{A}$ is psd.

We write $\lambda_1 \geq \lambda_2 \geq \dots$ for the
eigenvalues of a self-adjoint matrix.
We write $\sigma_1 \geq \sigma_2 \geq \dots$ for the singular
values of a general matrix.  If the matrix is not clear
from context, we may include it in the notation so that
$\sigma_{j}(\mtx{A})$ is the $j$th singular value of $\mtx{A}$.

Let $f : \RR \to \RR$ be a function on the real line.
We can extend $f$ to a spectral function  $f : \Sym_n \to \Sym_n$ on (conjugate)
symmetric matrices.
Indeed, for a matrix $\mtx{A} \in \Sym_n$ with eigenvalue decomposition
$$
\mtx{A} = \sum_{i=1}^n \lambda_i \, \vct{u}_i \vct{u}_i^*,
\quad\text{we define}\quad
f(\mtx{A}) := \sum_{i=1}^n f(\lambda_i) \, \vct{u}_i \vct{u}_i^*.
$$
The Pascal notation for definitions ($:=$ and $=:$) is used sparingly,
when we need to emphasize that the definition is taking place.

\subsection{Inner product geometry}

We equip $\F^n$ with the standard inner product
and the associated $\ell_2$ norm.
For all vectors $\vct{a}, \vct{b} \in \F^{n}$,
$$
\ip{ \vct{a} }{ \vct{b} } := \vct{a} \cdot \vct{b} := \sum_{i=1}^n (\vct{a})_i^* (\vct{b})_i
\quad\text{and}\quad
\norm{ \vct{a} }^2 := \ip{ \vct{a} }{ \vct{a} }.
$$
We write $\mathbb{S}^{n-1}$ for the set of vectors
in $\F^n$ with unit $\ell_2$ norm.  If needed, we
may specify the field: $\mathbb{S}^{n-1}(\F)$.

The trace of a square matrix is the sum of its diagonal entries:
$$
\trace(\mtx{A}) := \sum_{i=1}^n (\mtx{A})_{ii}
\quad\text{for $\mtx{A} \in \F^{n \times n}$.}
$$
Nonlinear functions bind before the trace.
We equip $\F^{m \times n}$ with the standard
trace inner product and the Frobenius norm.
For all matrices $\mtx{A}, \mtx{B} \in \F^{m \times n}$,
$$
\ip{ \mtx{A} }{ \mtx{B} } := %
	\trace(\mtx{A}^* \mtx{B})
\quad\text{and}\quad
\fnorm{ \mtx{A} }^2 := \ip{ \mtx{A} }{ \mtx{A} }.
$$
For vectors, these definitions coincide with the
ones in the last paragraph.

We say that a matrix $\mtx{U}$ is \emph{orthonormal} when
its columns are orthonormal with respect to the standard
inner product.  That is, $\mtx{U}^* \mtx{U} = \Id$.
If $\mtx{U}$ is also square, we say instead that
$\mtx{U}$ is \emph{orthogonal} ($\mathbb{F} = \mathbb{R}$) or
\emph{unitary} ($\mathbb{F} = \mathbb{C}$).

\subsection{Norms on matrices}

Several different norms on matrices arise during this survey.
We use consistent notation for these norms.

\begin{itemize}
\item	The unadorned norm $\norm{ \cdot }$ refers to the
spectral norm of a matrix, also known as the $\ell_2$ operator norm.
It reports the maximum singular value of its argument.
For vectors, it coincides with the $\ell_2$ norm.

\item	The norm $\norm{ \cdot }_{*}$ is the nuclear norm
of a matrix, which is the dual of the spectral norm.
It reports the sum of the singular values of its argument.

\item	The symbol $\fnorm{\cdot}$ refers to the Frobenius norm,
defined in the last subsection.  The Frobenius norm coincides
with the $\ell_2$ norm of the singular values of its argument.

\item	The notation $\norm{ \cdot }_{p}$
denotes the Schatten $p$-norm for each $p \in [1, \infty]$.
The Schatten $p$-norm is the $\ell_p$ norm of the
singular values of its argument.
Special cases with their own notation include the nuclear norm (Schatten $1$),
the Frobenius norm (Schatten $2$), and the spectral norm (Schatten ${\infty}$).
\end{itemize}

Occasionally, other norms may arise, and we will define them
explicitly when they do.

\subsection{Approximation in the spectral norm}

Throughout this survey, we will almost exclusively
use the spectral norm to measure the error in
matrix computations.
Let us recall some of the implications that follow from spectral norm bounds.

Suppose that $\mtx{A} \in \F^{m \times n}$ is a matrix,
and $\widehat{\mtx{A}} \in \F^{m \times n}$ is an approximation.
If the approximation satisfies the spectral norm error bound
$$
\norm{ \mtx{A} - \widehat{\mtx{A}} } \leq \eps,
$$
then we can transfer the following information:

\begin{itemize} \setlength{\itemsep}{1mm}
\item	\textbf{Linear functionals:}
$\abs{ \ip{\mtx{F}}{\mtx{A}} - \ip{\mtx{F}}{\widehat{\mtx{A}}} } \leq \eps \norm{\mtx{F}}_{*}$
for every matrix $\mtx{F} \in \F^{m \times n}$

\item	\textbf{Singular values:} $\abs{ \sigma_j(\mtx{A}) - \sigma_j(\widehat{\mtx{A}}) } \leq \eps$
for each index $j$.

\item	\textbf{Singular vectors:}
if the $j$th singular value $\sigma_j(\mtx{A})$ is well separated
from the other singular values, then the $j$th right singular vector of $\mtx{A}$ is well approximated by the $j$th
right singular vector of $\widehat{\mtx{A}}$;
a similar statement holds for the left singular vectors.

\end{itemize}

\noindent
Detailed statements about the singular vectors are complicated,
so we refer the reader to~\citeasnoun[Chaps.~VII, X]{Bha97:Matrix-Analysis}
for his treatment of perturbation of spectral subspaces.

For one-pass and streaming data models, it may not be possible
to obtain good error bounds in the spectral norm.  In this case,
we may retrench to Frobenius norm or nuclear norm error bounds.
These estimates give weaker information about linear functionals,
singular values, and singular vectors.

\begin{remark}[Frobenius norm approximation] \label{rem:frob}
In the literature on randomized NLA, some authors prefer
to bound errors with respect to the Frobenius norm because the
arguments are technically simpler.  In many instances, these
bounds are less valuable because
the error can have the same scale
as the matrix that we wish to approximate.

For example, let us consider a variant of the spiked covariance
model that is common in statistics applications~\cite{Joh01:Distribution-Largest}.
Suppose we need to approximate a rank-one matrix contaminated with additive noise:
$\mtx{A} = \vct{uu}^* + \eps \mtx{G} \in \RR^{n \times n}$,
where $\norm{\mtx{u}} = 1$ and $\mtx{G} \in \RR^{n \times n}$ has independent $\textsc{normal}(0, n^{-1})$ entries.
It is well known that $\norm{\mtx{G}} \approx 2$, while  $\fnorm{\mtx{G}} \approx \sqrt{n}$.
With respect to the Frobenius norm, the zero matrix is almost as good an approximation
of $\mtx{A}$ as the rank-one matrix $\vct{uu}^*$:
$$
\Expect \fnorm{ \mtx{A} - \vct{uu}^* }^2 = \eps^2 n
\quad\text{and}\quad
\Expect \fnorm{ \mtx{A} - \mtx{0} }^2 = \eps^2 n + 1.
$$
The difference is visible only when the size of the perturbation $\eps \approx n^{-1/2}$.
In contrast, the spectral norm error can easily distinguish between the
good approximation $\vct{uu}^*$ and the vacuous approximation $\mtx{0}$,
even when $\eps = O(1)$.

For additional discussion, see \citeasnoun[Sec.~6.2.3]{Tro15:Introduction-Matrix}
and \citeasnoun[App.]{LLS+17:Algorithm-971}.
\end{remark}

\subsection{Intrinsic dimension and stable rank}

Let $\mtx{A} \in \Sym_n$ be a psd matrix.
We define its \emph{intrinsic dimension}:
\begin{equation} \label{eqn:intdim}
\intdim(\mtx{A}) := \frac{\trace(\mtx{A})}{\norm{\mtx{A}}}.
\end{equation}
The intrinsic dimension of a nonzero matrix
satisfies the inequalities
$1 \leq \intdim(\mtx{A}) \leq \rank(\mtx{A})$;
the upper bound is saturated when $\mtx{A}$ is
an orthogonal projector.
We can interpret the intrinsic dimension
as a continuous measure of the rank,
or the number of energetic dimensions
in the matrix.

Let $\mtx{B} \in \F^{m \times n}$ be a rectangular matrix.
Its \emph{stable rank} is
\begin{equation} \label{eqn:stable-rank}
\srank(\mtx{B}) := \intdim(\mtx{B}^*\mtx{B}) = \frac{\fnorm{\mtx{B}}^2}{\norm{\mtx{B}}^2}.
\end{equation}
Similar to the intrinsic dimension, the stable rank provides
a continuous measure of the rank of $\mtx{B}$.

\subsection{Schur complements}

Schur complements arise from partial Gaussian elimination
and partial least-squares.  They also play a key role in
several parts of randomized NLA.  We give the
basic definitions here, referring to~\citeasnoun{Zha05:Schur-Complement}
for a more complete treatment.

Let $\mtx{A} \in \F^{n \times n}$ be a psd matrix,
and let $\mtx{X} \in \F^{n \times k}$ be a fixed matrix.
First, define the psd matrix
\begin{equation} \label{eqn:nys-def}
\mtx{A}\nys{\mtx{X}}
	:= (\mtx{AX})(\mtx{X}^* \mtx{A} \mtx{X})^\pinv (\mtx{AX})^*.
\end{equation}
The \emph{Schur complement} of $\mtx{A}$ with respect to
$\mtx{X}$ is the psd matrix
\begin{equation} \label{eqn:schur-complement}
\mtx{A} / \mtx{X} := \mtx{A} - \mtx{A}\langle \mtx{X} \rangle.
\end{equation}
The matrices $\mtx{A}\nys{\mtx{X}}$ and $\mtx{A}/\mtx{X}$
depend on $\mtx{X}$ only through its range.  They also enjoy
geometric interpretations in terms of orthogonal projections
with respect to the $\mtx{A}$ semi-inner-product.

\subsection{Miscellaneous}

We use big-$O$ notation following standard computer science convention.
For instance, we say that
a method has (arithmetic)
complexity $O(n^{\omega})$ if there is a finite $C$ for which
the number of floating point operations (flops) expended is bounded by
$Cn^{\omega}$ as the problem size $n\rightarrow \infty$.

We use MATLAB-inspired syntax in summarizing algorithms. For instance, the task of
computing an SVD $\mtx{A} = \mtx{U\Sigma V}^{*}$ of a given matrix $\mtx{A}$
is written as $[\mtx{U},\mtx{\Sigma},\mtx{V}]=\texttt{svd}(\mtx{A})$.
We have taken
the liberty to modify the syntax when we believe that this improves clarity. For
instance, we write $[\mtx{Q},\mtx{R}] = \texttt{qr\_econ}(\mtx{A})$ to denote the
\textit{economy-size} QR factorization where the matrix $\mtx{Q}$ has size
$m\times \min(m,n)$ for an input matrix $\mtx{A} \in \F^{m\times n}$.
Arguments that are not needed are replaced by ``$\sim$'', so that, for example,
$[\mtx{Q},\sim] = \texttt{qr\_econ}(\mtx{A})$ returns only the matrix $\mtx{Q}$
whose columns form an ON basis for the range of $\mtx{A}$.

\section{Probability preliminaries}
\label{sec:prob}

This section summarizes the key definitions and background
from probability and high-dimensional probability.
Later in the survey, we will present more complete
statements of foundational results, as they are needed.

\citeasnoun{GS01:Probability-Random} provide
an accessible overview of applied probability.
\citeasnoun{Ver18:High-Dimensional-Probability}
introduces the field of
high-dimensional probability.
For more mathematical presentations,
see the classic book of \citeasnoun{LT91:Probability-Banach}
or the lecture notes
of \citeasnoun{VH16:Probability-High}.

\subsection{Basics}

We work in a master probability space that is rich enough
to support all of the random variables that are defined.
We will not comment further about the underlying model.

In this paper, the unqualified term \emph{random variable}
encompasses random scalars, vectors, and matrices.
Scalar-valued random variables are usually (but not always)
denoted by uppercase italic Roman letters ($X, Y, Z$).
A random vector is denoted by a lowercase bold letter
($\vct{x}, \vct{\omega}$).
A random matrix is denoted by an uppercase bold letter ($\mtx{X}, \mtx{Y}, \mtx{\Gamma}, \mtx{\Omega}$).
This notation works in concert with the notation for
deterministic vectors and matrices.

The map $\mathbb{P}(E)$ returns the probability of an
event $E$.  We usually specify the event using the compact
set builder notation that is standard in probability.
For example, $\Prob{ X > t }$ is the probability that
the scalar random variable $X$ exceeds a level $t$.

The operator $\Expect$ returns the expectation of a random
variable.  For vectors and matrices, the expectation can
be computed coordinate by coordinate.  The expectation
is linear, which justifies relations like
$$
\Expect[ \mtx{AX} ] = \mtx{A} \Expect[ \mtx{X} ]
\quad\text{when $\mtx{A}$ is deterministic and $\mtx{X}$ is random.}
$$
We use the convention that nonlinear functions bind
before the expectation; for instance, $\Expect X^2 = \Expect[ X^2 ]$.
The operator $\Var[\cdot]$ returns the variance of a scalar random
variable.

We say that a random variable is \emph{centered} when its
expectation equals zero.  A random vector
$\vct{x}$ is \emph{isotropic} when $\Expect[ \vct{xx}^* ] = \Id$.
A random vector is \emph{standardized} when
it is both centered and isotropic.  In particular,
a scalar random variable is standardized when
it has expectation zero and variance one.

When referring to independent random variables, we often
include the qualification ``statistically independent''
to make a distinction with ``linearly independent.''
We abbreviate the term (statistically)
``independent and identically distributed'' as i.i.d.

\subsection{Distributions}

To refer to a named distribution, we use small capitals.
In this context, the symbol $\sim$ means ``has the same distribution as.''

We write $\textsc{unif}$ for the uniform distribution
over a finite set (with counting measure).
In particular, a scalar Rademacher random variable
has the distribution $\textsc{unif}\{ \pm 1 \}$.
A Rademacher random vector has iid coordinates,
each distributed as a scalar Rademacher random
variable.
We sometimes require the uniform distribution
over a Borel subset of $\F^n$, equipped with
Lebesgue measure.

We write $\textsc{normal}(\vct{\mu}, \mtx{C})$
for the normal distribution on $\F^n$ with
expectation $\vct{\mu} \in \F^n$ and
psd covariance matrix $\mtx{C} \in \Sym_n(\F)$.
A \emph{standard normal} random variable or random
vector has expectation zero and covariance matrix
equal to the identity.  We often use the term
\emph{Gaussian} to refer to normal distributions.

\subsection{Concentration inequalities}

Concentration inequalities provide bounds on the probability
that a random variable is close to its expectation.
A good reference for the scalar case is
the book by \citeasnoun{BLM13:Concentration-Inequalities}.
In the matrix setting, closeness is measured in the spectral norm.
For an introduction to matrix concentration,
see \cite{Tro15:Introduction-Matrix,Tro19:Matrix-Concentration-LN}.
These results play an important role in randomized linear algebra.

\subsection{Gaussian random matrix theory}

On several occasions, we use comparison principles to
study the action of a random matrix with iid Gaussian entries.
In particular, these methods can be used
to control the largest and smallest singular values.
The main classical comparison theorems
are associated with the names
Slepian, Chevet, and Gordon.
For accounts of this,
see \citeasnoun[Sec.~3.3]{LT91:Probability-Banach}
or \citeasnoun[Sec.~2.3]{DS01:Local-Operator}
or \citeasnoun[Secs.~7.2--7.3]{Ver18:High-Dimensional-Probability}.
More recently, it has been observed that Gordon's
inequality can be reversed in certain settings~\cite{TOH14:Gaussian-Minmax}.

In several instances, we require more detailed
information about Gaussian random matrices.
General resources include
\citeasnoun{Mui82:Aspects-Multivariate}
and \citeasnoun{BS10:Spectral-Analysis}.
Most of the specific results we need are presented
in~\citeasnoun{HMT11:Finding-Structure}.

\section{Trace estimation by sampling}
\label{sec:trace-est}

We commence with a treatment of matrix trace estimation problems.
These questions stand among the simplest
linear algebra problems because the desired result
is just a scalar.  Even so, the algorithms have a vast sweep
of applications, ranging from computational statistics
to quantum chemistry.  They also serve as building
blocks for more complicated randomized NLA algorithms.
We have chosen to begin our presentation here because
many of the techniques that drive algorithms
for more difficult problems already appear in
a nascent---and especially pellucid---form in this section.

Randomized methods for trace
estimation depend on a natural technical idea:
One may construct an unbiased estimator for the trace
and then average independent copies to reduce the
variance of the estimate.  Algorithms of this
type are often called \emph{Monte Carlo methods}.
We describe how to use standard methods
from probability and statistics to develop
\emph{a priori} and \emph{a posteriori}
guarantees for Monte Carlo trace estimators.
We show how to use structured random distributions
to improve the computational profile of the estimators.
Last, we demonstrate that trace estimators also yield
approximations for the Frobenius norm and
the Schatten 4-norm of a general matrix.

In Section~\ref{sec:schatten-p}, we present
more involved Monte Carlo methods that
are required to estimate Schatten $p$-norms for larger
values of $p$, which give better approximations for the spectral norm.
Section~\ref{sec:max-eig} describes iterative
algorithms that lead to much higher accuracy than Monte Carlo methods.
In Section~\ref{sec:slq}, we touch on related
probabilistic techniques for evaluating trace functions.

\subsection{Overview}

We will focus on the problem of estimating the trace
of a nonzero psd matrix $\mtx{A} \in \Sym_n$. Our goal is to
to produce an approximation of $\trace(\mtx{A})$, along with a measure of quality.

Trace estimation is easy in the case where we have
inexpensive access to the entries of the matrix $\mtx{A}$
because we can simply read off the $n$ diagonal entries.
But there are many environments where the primitive
operation is the matrix--vector product
$\vct{u} \mapsto \mtx{A}\vct{u}$.
For example, $\vct{u} \mapsto \mtx{A}\vct{u}$ might be the solution
of a (discretized) linear differential equation with initial condition $\vct{u}$,
implemented by some computer program.
In this case, we would really prefer to avoid $n$ applications
of the primitive.  (We can obviously compute the
trace by applying the primitive to each standard
basis vector $\vct{\delta}_i$.)

The methods in this section all use linear information
about the matrix $\mtx{A}$.  In other words, we will extract data
from the input matrix by computing the product $\mtx{Y} = \mtx{A\Omega}$,
where $\mtx{\Omega} \in \F^{n \times k}$ is a (random) test matrix.
All subsequent operations involve only the sample matrix $\mtx{Y}$
and the test matrix $\mtx{\Omega}$.  Since the data
collection process is linear, we can apply randomized
trace estimators in the one-pass or streaming environments.
Moreover, parts of these algorithms are trivially parallelizable.

The original application of randomized trace estimation was
to perform \emph{a posteriori} error estimation for large
least-squares computations. More specifically, it was used
to accelerate cross-validation procedures for estimating
the optimal regularization parameter in a smoothing
spline~\cite{Gir89:Fast-Monte-Carlo,Hut90:Stochastic-Estimator}.
See~\citeasnoun{FORF18:Improved-Stochastic} for a list of
contemporary applications in machine learning,
uncertainty quantification, and other fields.

\subsection{Trace estimation by randomized sampling}
\label{sec:trace-est-basic}

Randomized trace estimation is based on the insight
that it is easy to construct a random variable
whose expectation equals the trace of the input matrix.

Consider a random {test vector} $\vct{\omega} \in \F^n$
that is isotropic: %
$\Expect[ \vct{\omega} \vct{\omega}^* ] = \Id$.
By the cyclicity of the trace and by linearity,
\begin{equation} \label{eqn:trace-est}
X = \vct{\omega}^* (\mtx{A} \vct{\omega})
\quad\text{satisfies}\quad
\Expect X = \trace(\mtx{A}).
\end{equation}
In other words, the random variable $X$ is an unbiased
estimator of the trace.  Note that the distribution of
$X$ depends on the unknown matrix $\mtx{A}$.

A single sample of $X$ is rarely adequate
because its variance, $\Var[X]$, will be large.
The most common mechanism for reducing the variance is to average
$k$ independent copies of $X$.  For $k \in \N$, define
\begin{equation} \label{eqn:trace-est-avg}
\bar{X}_k = \frac{1}{k} \sum\nolimits_{i=1}^k X_i
\quad\text{where $X_i \sim X$ are iid.}
\end{equation}
By linearity, $\bar{X}_k$ is also an unbiased estimator
of the trace.  The individual samples are statistically
independent, so the variance decreases.  Indeed,
$$
\Expect[ \bar{X}_k ] = \trace(\mtx{A})
\quad\text{and}\quad
\Var[ \bar{X}_k ] = \frac{1}{k} \Var[X].
$$
The estimator~\eqref{eqn:trace-est-avg} can be regarded as the most elementary method
in randomized linear algebra.  See Algorithm~\ref{alg:trace-est}.

To compute $\bar{X}_k$, we must simulate
$k$ independent copies of the random vector $\vct{\omega} \in \F^n$
and perform $k$ matrix--vector products with $\mtx{A}$,
plus $O(kn)$ additional arithmetic.

\begin{example}[\citeasnoun{Gir89:Fast-Monte-Carlo}]
Consider a standard normal random vector $\vct{\omega} \sim \normal(\vct{0}, \Id)$.
The variance of the resulting trace estimator~\eqref{eqn:trace-est}--\eqref{eqn:trace-est-avg} satisfies
\begin{equation} \label{eqn:trace-est-girard-var}
\Var[ \bar{X}_k ] = \frac{2}{k} \sum_{i,j = 1}^n \abs{ (\mtx{A})_{ij} }^2 = \frac{2}{k} \fnorm{\mtx{A}}^2
	\leq \frac{2}{k} \norm{\mtx{A}} \trace(\mtx{A}).
\end{equation}
The rotational invariance of the standard normal distribution
allows us to characterize the behavior of this estimator in full detail.
\end{example}

\begin{example}[\citeasnoun{Hut90:Stochastic-Estimator}]
Consider a Rademacher random vector $\vct{\omega} \sim \unif\{\pm 1\}^n$.
The variance of the resulting trace estimator~\eqref{eqn:trace-est}--\eqref{eqn:trace-est-avg}
satisfies
$$
\Var[ \bar{X}_k ] =
\frac{4}{k} \sum_{1 \leq i < j \leq n} \abs{ (\mtx{A})_{ij} }^2 <
\frac{2}{k} \fnorm{\mtx{A}}^2 \leq
\frac{2}{k}  \norm{\mtx{A}} \trace(\mtx{A}).
$$
This is the minimum variance trace estimator generated by
an isotropic random vector $\vct{\omega}$ with statistically independent coordinates.
It also avoids the simulation of normal variables.
\end{example}

\citeasnoun{Gir89:Fast-Monte-Carlo} also studied the estimator
obtained by drawing $\vct{\omega} \in \F^n$ uniformly at random from
the sphere $\sqrt{n} \, \mathbb{S}^{n-1}(\F)$ for $\F=\mathbb{R}$. %
When $\F = \CC$, this approach has the minimax variance among all
trace estimators of the form~\eqref{eqn:trace-est}--\eqref{eqn:trace-est-avg}.
We return to this example in Section~\ref{sec:near-isotropic}.

\begin{algorithm}[t]
\begin{algorithmic}[1]
\caption{\textit{Trace estimation by random sampling.} \newline
See Section~\ref{sec:trace-est-basic}.}
\label{alg:trace-est}

\Require	Psd input matrix $\mtx{A} \in \Sym_n$, number $k$ of samples
\Ensure		Trace estimate $\bar{X}_k$ and sample variance $S_k$
\Statex

\Function{TraceEstimate}{$\mtx{A}$, $k$}

\For{$i = 1, \dots, k$}
	\Comment Compute trace samples
\State	Draw isotropic test vector $\vct{\omega}_i \in \F^n$
\State	Compute $X_i = \vct{\omega}_i^* (\mtx{A} \vct{\omega}_i)$
\EndFor

\State	Form trace estimator: $\bar{X}_k = k^{-1} \sum_{i=1}^k X_i$

\State	Form sample variance: $S_k = (k-1)^{-1} \sum_{i=1}^k (X_i - \bar{X}_k)^2$

\Statex

\Comment	Use compensated summation techniques for large $k$!

\EndFunction
\end{algorithmic}
\end{algorithm}

\begin{remark}[General matrices]
The assumption that $\mtx{A}$ is psd allows us
to conclude that the standard deviation
of the randomized trace estimate is smaller
than the trace of the matrix.
The same methods allow us to estimate the
trace of a general square matrix, but the
variance of the estimator may no longer
be comparable with the trace.
\end{remark}

\subsection{{A priori} error estimates}

We can use theoretical analysis to obtain prior guarantees
on the performance of the trace estimator.  These
results illuminate what features of the input matrix
affect the quality of the trace estimate, and they
tell us how many samples $k$ suffice to achieve a
given error tolerance.  Note, however, that these
bounds depend on properties of the input matrix
that are often unknown to the user of the trace estimator.

Regardless of the distribution of the isotropic test vector
$\vct{\omega}$, Chebyshev's inequality delivers a simple
probability bound for the trace estimator:
\begin{equation} \label{eqn:trace-est-chebyshev}
\Prob{ \abs{ \bar{X}_k - \trace(\mtx{A}) } \geq t }
	\leq \frac{ \Var[X] }{k t^2}
	\quad\text{for $t > 0$.}
\end{equation}
We can specialize this result to specific trace estimators
by inserting the variance.

\begin{example}[Girard trace estimator]
If the test vector $\vct{\omega}$ is standard normal,
the trace estimator $\bar{X}_k$ satisfies
$$
\Prob{ \abs{ \bar{X}_k - \trace(\mtx{A}) } \geq t \cdot \trace(\mtx{A}) }
	\leq \frac{ 2 }{k \intdim(\mtx{A}) \, t^2}.
$$
The bound follows from~\eqref{eqn:trace-est-girard-var},~\eqref{eqn:trace-est-chebyshev},
and~\eqref{eqn:intdim}.
In words, the trace estimator achieves a relative error bound
that is sharpest when the intrinsic dimension~\eqref{eqn:intdim}
of $\mtx{A}$ is large.
\end{example}

For specific distributions of the random test vector $\vct{\omega}$,
we can obtain much stronger probability bounds for the
resulting trace estimator using exponential concentration inequalities.
Here is a recent analysis for Girard's estimator based on
fine properties of the standard normal distribution.

\begin{theorem}[\citeasnoun{GT18:Improved-Bounds}] \label{thm:trace-est-exp}
Let $\mtx{A} \in \Sym_n(\RR)$ be a nonzero psd matrix.  Consider the trace
estimator~\eqref{eqn:trace-est}--\eqref{eqn:trace-est-avg}
obtained from a standard normal test vector $\vct{\omega} \in \RR^n$.
For $\tau > 1$ and $k \leq n$,
$$
\begin{aligned}
\Prob{ \bar{X}_k \geq \tau \trace(\mtx{A}) } &\leq \exp\left( -\tfrac{1}{2} k \intdim(\mtx{A}) \big(\sqrt{\tau} - 1 \big)^2 \right); \\
\Prob{ \bar{X}_k \leq \tau^{-1} \trace(\mtx{A}) } &\leq \exp\left( -\tfrac{1}{4} k \intdim(\mtx{A}) \big(\tau^{-1} - 1 \big)^2 \right). \\
\end{aligned}
$$
When $\mtx{A} \in \Sym_n(\CC)$ is psd and $\vct{\omega} \in \CC^n$ is complex standard normal,
the same bounds hold with an extra factor two in the exponent.  (So the estimator works better
in the complex setting.)
\end{theorem}

\begin{proof}(Sketch)
Carefully estimate the moment generating function of the random variable $X$,
and use the Cram{\'e}r--Chernoff method to obtain the probability inequalities.
\end{proof}

\subsection{Universality}

Empirically, for a large sample, the performance of the trace estimator $\bar{X}_k$
only depends on the distribution of the test vector $\vct{\omega}$
through the variance of the resulting sample $X$.
In a word, the estimator exhibits \emph{universality}.
As a consequence, we can select the distribution
that is most convenient for computational purposes.

Classical probability theory furnishes justification for these claims.
The strong law of large numbers tells us that
$$
\bar{X}_k \to \trace(\mtx{A})
\quad\text{almost surely as $k \to \infty$.}
$$
Concentration inequalities~\cite{BLM13:Concentration-Inequalities}
allow us to derive rates of convergence akin to Theorem~\ref{thm:trace-est-exp}.

To understand the sampling distribution of the estimator
$\bar{X}_k$, we can invoke the central limit theorem:
$$
\sqrt{k} ( \bar{X}_k - \trace(\mtx{A}) )
	\to \normal(0, \Var[X])
	\quad\text{in distribution as $k \to \infty$.}
$$
We can obtain estimates for the rate of convergence to normality using
the Berry--Ess{\'e}en theorem and its
variants~\cite{Ros11:Fundamentals-Steins,CGS11:Normal-Approximation}.

Owing to the universality phenomenon, we can formally use the
normal limit to obtain heuristic error estimates and insights
for trace estimators constructed with test vectors
from any distribution.
This %
strategy becomes even more valuable
when the linear algebra problem is more complicated.

\begin{warning}[CLT]
Estimators based on averaging independent samples
cannot overcome the central limit theorem.
Their accuracy will always be limited
by fluctuations on the scale of $\sqrt{\Var[X]}$.
In other words, we must extract $\eps^{-2}$
samples to reduce the error to $\eps \sqrt{\Var[X]}$
for small $\eps > 0$.
This is the curse of Monte Carlo.
\end{warning}

\subsection{{A posteriori} error estimates}
\label{sec:trace-est-post}

In practice, we rarely have access to all the information
required to activate \emph{a priori} error bounds.
It is wiser to assess the quality of the estimate
from the information that we actually collect.
Since we have full knowledge of the
random process that generates the trace estimate, we
can confidently use approaches from classical statistics.

At the most basic level,
the sample variance is an unbiased estimator
for the variance of the individual samples:
$$
S_k = \frac{1}{k-1} \sum_{i=1}^k (X_i - \bar{X}_k)^2
\quad\text{satisfies}\quad
\Expect[ S_k ] = \Var[X].
$$
The variance of $S_k$ depends on the fourth moment of
the random variable $X$.  A standard estimate is
$$
\Var[S_k] \leq \frac{1}{k} \Expect[ (X - \Expect X)^4 ].
$$
Bounds and empirical estimates for the variance of $S_k$ can also be obtained using
the Efron--Stein inequality~\cite[Sec.~3.1]{BLM13:Concentration-Inequalities}.

For $\alpha \in (0, 1/2)$, we can construct the (symmetric, Student's $t$)
confidence interval at level $1 - 2\alpha$:
$$
\trace(\mtx{A}) \in \bar{X}_k \pm t_{\alpha, k-1} \sqrt{S_k}
$$
where $t_{\alpha, k-1}$ is the $\alpha$ quantile of the
Student's $t$-distribution with $k-1$ degrees of freedom.
We interpret this result as saying that $\trace(\mtx{A})$ lies in
the specified interval with probability roughly $1 - 2\alpha$
(over the randomness in the trace estimator).
The usual rule of thumb is that the sample size should be moderate (say, $k \geq 30$),
while $\alpha$ cannot be too small (say, $\alpha \geq 0.025$).

\subsection{Bootstrapping the sampling distribution}
\label{sec:trace-est-boot}

Miles Lopes has proposed a sweeping program
that uses the bootstrap to construct
data-driven confidence sets for randomized
NLA algorithms~\cite{Lop19:Estimating-Algorithmic}.
For trace estimation, this approach is straightforward to describe
and implement; see Algorithm~\ref{alg:trace-est-boot}.

Let $\mathcal{X} = ( X_1, \dots, X_k )$ be the empirical sample from~\eqref{eqn:trace-est-avg}.
The bootstrap draws further random samples from $\mathcal{X}$
to elicit more information about the sampling distribution of the trace estimator $X$,
such as confidence sets.

\begin{enumerate}
\item	For each $b = 1, \dots, B$,
\begin{enumerate}
\item	Draw a bootstrap replicate $( X_1^*, \dots, X_k^* )$ uniformly from $\mathcal{X}$ with replacement.
\item	Compute the error estimate $e_b^* = \bar{X}_k^* - \bar{X}_k$, where $\bar{X}_k^*$
is the sample average of the bootstrap replicate.
\end{enumerate}
\item	Compute quantiles $q_\alpha$ and $q_{1-\alpha}$ of the error distribution $( e_1^*, \dots, e_B^* )$.
\item	Report the $1-2\alpha$ confidence set $[\bar{X}_k + q_{\alpha}, \bar{X}_k + q_{1 - \alpha}]$.
\end{enumerate}

Typical values are $k \geq 30$ samples and $B \geq 1000$ bootstrap replicates
when $\alpha \geq 0.025$. %
This method is effective for a wide range of distributions on the test vector,
and it extends to other problems.  See~\citeasnoun{Efr82:Jackknife-Bootstrap}
for an introduction to resampling methods.

\begin{algorithm}[t]
\begin{algorithmic}[1]
\caption{\textit{Bootstrap confidence interval for trace estimation.} \newline
See Section~\ref{sec:trace-est-boot}.}
\label{alg:trace-est-boot}

\Require	Psd input matrix $\mtx{A} \in \Sym_n$, number $k$ of trace samples, number $B$ of bootstrap replicates, parameter $\alpha$ for level of confidence
\Ensure		Confidence interval $[\bar{X}_k + q_{\alpha}, \bar{X}_k + q_{1-\alpha}]$ at level $1-2\alpha$
\Statex

\Function{BootstrapTraceEstimate}{$\mtx{A}$, $k$, $b$, $\alpha$}

\For{$i = 1, \dots, k$}
	\Comment Compute trace estimators
\State	Draw isotropic test vector $\vct{\omega}_i \in \F^n$
\State	Form $X_i = \vct{\omega}_i^* (\mtx{A} \vct{\omega}_i)$
\EndFor

\State	$\mathcal{X} = ( X_1, \dots, X_k )$
	\Comment	Collate sample
\State	$\bar{X}_k = k^{-1} \sum_{i=1}^k X_i$
	\Comment	Trace estimate
	
\For{$b = 1, \dots, B$}
	\Comment Bootstrap replicates
\State	Draw $(X_1^*, \dots, X_k^*)$ from $\mathcal{X}$ with replacement
\State	Compute $e_b^* = (k^{-1} \sum_{i=1}^k X_i^*) - \bar{X}_k$
\EndFor

\State	Find $q_\alpha$ and $q_{1-\alpha}$ quantiles of errors $(e_1^*, \dots, e_b^*)$

\EndFunction
\end{algorithmic}
\end{algorithm}

\subsection{Structured distributions for test vectors}

As discussed, there is a lot of flexibility in designing the
distribution of the test vector.  We can exploit this freedom
to achieve additional computational goals. For example,
we might:

\begin{itemize} \setlength{\itemsep}{1mm}
\item	minimize the variance $\Var[X]$ of each sample;
\item	minimize the number of random bits required to construct $\vct{\omega}$;
\item	design a test vector $\vct{\omega}$ that is ``compatible''
with the input matrix $\mtx{A}$ to facilitate the matrix--vector product.
For example, if $\mtx{A}$ has a tensor product structure,
we might require $\vct{\omega}$ to share the tensor structure.
\end{itemize}

Let us describe a general construction for test vectors
that can help achieve these desiderata.
The ideas come from frame theory and quantum information theory.
The approach here extends the work in
\citeasnoun{FORF18:Improved-Stochastic}.

\subsubsection{Optimal measurement systems}
\label{sec:near-isotropic}

In this section, we work in the complex field.
Consider a discrete set $\mathcal{U} := \{ \vct{u}_1, \dots, \vct{u}_m \} \subset \CC^n$
of vectors, each with unit $\ell_2$ norm.  We say that
the $\mathcal{U}$ is an \emph{optimal measurement system} when
\begin{equation} \label{eqn:near-isotropic}
\frac{1}{m} \sum\nolimits_{i=1}^m (\vct{u}_i^*\mtx{M} \vct{u}_i) \, \vct{u}_i \vct{u}_i^*
	= \frac{1}{(n+1) n} \left[ \mtx{M} + \trace(\mtx{M}) \, \Id \right]
\end{equation}
for all $\mtx{M} \in \Sym_n(\CC)$.  The reproducing property~\eqref{eqn:near-isotropic}
shows that the system of vectors acquires enough information
to reconstruct an arbitrary self-adjoint matrix.
A similar definition is valid for an infinite system of unit vectors,
provided we replace the sum in~\eqref{eqn:near-isotropic} with an integral.

Now, suppose that we draw a random test vector $\vct{\omega} = \sqrt{n} \, \vct{u}$,
where $\vct{u}$ is drawn uniformly at random from an optimal measurement system $\mathcal{U}$.
Then the resulting trace estimator is unbiased:
\begin{equation}
\label{eq:unbi1}
X = \vct{\omega}^* (\mtx{A} \vct{\omega})
	\quad\text{satisfies}\quad
	\Expect X = \trace(\mtx{A}).
\end{equation}
The variance of this trace estimator satisfies
\begin{equation}
\label{eq:unbi2}
\Var[X] = \frac{n}{n+1} \left[ \fnorm{\mtx{A}}^2 %
	- \frac{1}{n} \trace(\mtx{A})^2 \right].
\end{equation}
The identities (\ref{eq:unbi1}) and (\ref{eq:unbi2}) follow quickly from~\eqref{eqn:near-isotropic}.
As it happens, this is the minimax variance achievable
for a best isotropic distribution on test vectors~\cite{Kue19:2-Designs-Minimize}

\subsubsection{Examples}

Optimal measurement systems arise in quantum information theory
(as near-isotropic measurement systems),
in approximation theory (as projective 2-designs), and in frame
theory (as tight fusion frames).  The core examples are as follows:

\begin{enumerate} \setlength{\itemsep}{1mm}
\item	A set of $n^2$ equiangular lines in $\CC^n$, each spanned by a unit vector
in $\{\vct{u}_1, \dots, \vct{u}_{n^2}\}$, gives an optimal
measurement system.  In this case, equiangularity means that $\abs{\ip{\vct{u}_i}{\vct{u}_j}}^2 = (d+1)^{-1}$
whenever $i \neq j$.  It is conjectured that these systems exist for every
natural number $n$.

\item	The columns of a family of $(n + 1)$ mutually unbiased bases
in $\F^n$ compose an optimal measurement system with $n(n+1)$ unit vectors.
For reference, a pair $(\mtx{U}, \mtx{V})$ of $n \times n$ unitary matrices
is called \emph{mutually unbiased} if $\vct{\delta}_i^* (\mtx{U}^* \mtx{V}) \vct{\delta}_j = n^{-1}$
for all $i, j$.  (For instance, consider the identity matrix and the discrete Fourier transform matrix.)
These systems exist whenever $n$ is a power of a prime number
\cite{WF89:Optimal-State-Determination}.

\item	The $\ell_2$ unit sphere $\mathbb{S}^{n-1}(\CC)$ in $\CC^n$, equipped with
the uniform measure, is a continuous optimal measurement system.
The real case was studied by \citeasnoun{Gir89:Fast-Monte-Carlo},
but the complex case is actually more natural.
\end{enumerate}

See~\citeasnoun[Lec.~3]{Tro19:Matrix-Concentration-LN} for more discussion
and an application to quantum state tomography.
\citeasnoun{Wal18:Introduction-Finite}
provides a good survey of what is currently known about finite optimal
measurement systems.

\subsection{Extension: The Frobenius norm and the Schatten 4-norm}
\label{sec:schatten-2-4}

The randomized trace estimators developed in this section
can also be deployed to estimate a couple of matrix norms.

Consider a rectangular matrix $\mtx{B} \in \F^{m \times n}$,
accessed via the matrix--vector product $\vct{u} \mapsto \mtx{B} \vct{u}$.
Let us demonstrate how to estimate the Frobenius norm (i.e., Schatten 2-norm)
and the Schatten 4-norm of $\mtx{B}$.

For concreteness, suppose that we extract test vectors from the standard normal distribution.
Draw a standard normal matrix $\mtx{\Omega} \in \F^{n \times k}$
with columns $\vct{\omega}_i$.  Construct the random variable
$$
\bar{X}_k := \frac{1}{k} \fnorm{ \mtx{B} \mtx{\Omega} }^2
	= \frac{1}{k} \sum_{i=1}^k \vct{\omega}_i^* (\mtx{B}^* \mtx{B}) \vct{\omega}_i
	=: \frac{1}{k} \sum_{i=1}^k X_i.
$$
We compute $\bar{X}_k$ by simulating $nk$ standard normal variables,
taking $k$ matrix--vector products with $\mtx{B}$,
and performing $\bigO(km)$ additional arithmetic.

To analyze $\bar{X}_k$, note that it is an instance of the randomized trace estimator,
where $\mtx{A} = \mtx{B}^* \mtx{B}$.
In particular, its statistics are
$$
\Expect[ \bar{X}_k ] = \fnorm{\mtx{B}}^2
\quad\text{and}\quad
\Var[ \bar{X}_k ] = \frac{2}{k} \norm{\mtx{B}}_{4}^4.
$$
We see that $\bar{X}_k$ provides an unbiased estimate
for the squared Frobenius norm of the matrix $\mtx{B}$.
Meanwhile, by rescaling the sample variance $S_k^2$ of the data
$( X_1, \dots, X_k)$, we obtain an unbiased estimate for the
fourth power of the Schatten 4-norm of $\mtx{B}$.

Our discussion shows how we can obtain \emph{a priori} guarantees and
\emph{a posteriori} error estimates for these norm computations;
the results can be expressed in terms of the stable
rank~\eqref{eqn:stable-rank} of $\mtx{B}$.  We can also
obtain norm estimators that are more computationally efficient
using structured distributions for the test vectors,
such as elements from an optimal measurement system.

\section{Schatten $p$-norm estimation by sampling}
\label{sec:schatten-p}

As we saw in Section~\ref{sec:schatten-2-4}, we can easily construct
unbiased estimators for the Schatten $2$-norm and the Schatten $4$-norm
of a matrix by randomized sampling.  What about the other Schatten norms?
This type of computation can be used to obtain better approximations of the
spectral norm, or it can be combined with the method of moments to
approximate the spectral density of the input matrix \cite{KV17:Spectrum-Estimation}.

In this section, we show that it is possible to
use randomized sampling to construct unbiased
estimators for the Schatten $2p$-norm
for each natural number $p \in \mathbb{N}$.
In contrast with the case $p \in \{1, 2\}$,
estimators for $p \geq 3$ are combinatorial.
They may also require a large number of samples
to ensure that the variance of the estimator is
controlled.

In the next section, we explain how to use iterative methods
to approximate the spectral norm (i.e., the Schatten $\infty$-norm).
Iterative algorithms also lead to much more reliable estimators
for general Schatten norms.

\subsection{Overview}

Consider a general matrix $\mtx{B} \in \F^{m \times n}$,
accessed via the matrix--vector product
$\vct{u} \mapsto \mtx{B}\vct{u}$.
For a sample size $k$, let $\mtx{\Omega} \in \F^{n \times k}$
be a (random) test matrix that does not depend on $\mtx{B}$.
For a natural number $p \geq 3$,
we consider the problem of estimating the Schatten $2p$-norm
$\norm{\mtx{B}}_{2p}$ from the sample matrix
$\mtx{Y} = \mtx{B}\mtx{\Omega}$.  The key
idea is to cleverly process the sample
matrix $\mtx{Y}$ to form an unbiased
estimator for the $2p$-th power of the norm.
Since the methods in this section use linear information,
they can be parallelized, and
they are applicable in the one-pass
and streaming environments.

Of course,
if we are given a singular value decomposition (SVD) of $\mtx{B}$,
it is straightforward to extract the Schatten $2p$-norm.
We are interested in methods that are much less expensive than the
$O(\min\{mn^2, nm^2\})$ cost of computing an SVD
with a classical direct algorithm.  Randomized SVD or URV
algorithms (Sections~\ref{sec:rsvd}, \ref{sec:singlepass}, and \ref{sec:full})
can also be used for Schatten norm estimation,
but this approach is typically overkill.

\subsection{Interlude: Lower bounds}

How large a sample size $k$ is required to estimate
$\norm{\mtx{B}}_{2p}$ up to a fixed constant factor
with 75\% probability over the randomness in $\mtx{\Omega}$?
The answer, unfortunately, turns out to be $k \gtrsim \min\{m,n\}^{1 - 2/p}$.
In other words, for a general matrix,
we cannot approximate the Schatten $2p$-norm
for any $p > 2$ unless the sample size $k$
grows polynomially with the dimension.

A more detailed version of this statement appears
in \citeasnoun[Thm.~3.2]{LNW14:Sketching-Matrix}.
The authors exhibit a particularly difficult
type of input matrix (a standard normal matrix with
a rank-one spike, as in Remark \ref{rem:frob}) to arrive at the negative conclusion.

If you have a more optimistic nature, you can
also take the inspiration that other types of
input matrices might be easier to handle.
For example, it is possible to use a small
random sample to compute the Schatten norm
of a matrix that enjoys some decay in the
singular value spectrum.

\subsection{Estimating Schatten norms the hard way}

First, let us describe a technique from classical statistics
that leads to an unbiased estimator of $\norm{\mtx{B}}_{2p}^{2p}$.
This estimator is both highly variable and computationally
expensive, so we must proceed with caution.

For the rest of this section,
we assume that the random test matrix $\mtx{\Omega} \in \F^{n \times k}$
has isotropic columns $\vct{\omega}_i$ that are iid.
Form the sample matrix $\mtx{Y} = \mtx{B}\mtx{\Omega}$.
Abbreviate $\mtx{A} = \mtx{B}^*\mtx{B}$ and $\mtx{X} = \mtx{Y}*\mtx{Y}$.
Observe that
$$
(\mtx{X})_{ij} = (\mtx{Y}^* \mtx{Y})_{ij} = \vct{\omega}_i^* \mtx{A} \vct{\omega}_j.
$$
Therefore, for any natural numbers that satisfy $1 \leq i_1, \dots, i_p \leq k$,
$$
(\mtx{X})_{i_1 i_2} (\mtx{X})_{i_2 i_3} \dots (\mtx{X})_{i_p i_1}
	= \trace\big[ \vct{\omega}_{i_1} \vct{\omega}_{i_1}^* \mtx{A} \dots \vct{\omega}_{i_p} \vct{\omega}_{i_p}^* \mtx{A} \big].
$$
If we assume that $i_1, \dots, i_p$ are distinct, we can use independence and isotropy
to compute the expectation:
$$
\Expect[ (\mtx{X})_{i_1 i_2} (\mtx{X})_{i_2 i_3} \dots (\mtx{X})_{i_p i_1} ]
	= \trace[ \mtx{A}^p ]
	= \norm{ \mtx{B} }_{2p}^{2p}
$$
By averaging over all sequences of distinct indices, we obtain an unbiased estimator:
$$
U_p = \frac{(k-p)!}{k!} \sum_{1 \leq i_1, \dots, i_p \leq k}^{\circ} (\mtx{X})_{i_1 i_2} (\mtx{X})_{i_2 i_3} \dots (\mtx{X})_{i_p i_1}.
$$
The circle over the sum indicates that the indices must be distinct.
Since $\Expect U_p = \norm{\mtx{B}}_{2p}^{2p}$, our hope is that
$$
U_p^{1/(2p)} \approx \norm{\mtx{B}}_{2p}. %
$$
To ensure that the approximation is precise,
the standard deviation of $U_p$ should be somewhat
smaller than the mean of $U_p$.

To understand the statistic $U_p$,
we can use tools from the theory of $U$-statistics~\cite{KB94:Theory-U-Statistics}.
For instance, when $p$ is fixed, we have the limit %
$$
k \Var[ U_p ] \to p^2 \Var[ \vct{\omega}^* \mtx{A}^p \vct{\omega} ]
	\quad\text{as $k \to \infty$.}
$$
In particular, if the test vector $\vct{\omega} \sim \normal(\vct{0}, \Id)$, then
$$
k \Var[ U_p ] \to 2p^2 \norm{ \mtx{B} }_{4p}^{4p}
	\quad\text{as $k \to \infty$.}
$$
We can reduce the variance further
by using test vectors from an optimal measurement
system.

Unfortunately, it is quite expensive to compute the statistic $U_p$
because it involves almost $k^p$ summands.  When $p$ is a small
constant (say, $p = 4$ or $p = 5$), it is not too onerous to form
$U_p$.
On the other hand, in the worst case,
we cannot beat the lower bound
$k \gtrsim \min\{m,n\}^{1-2/p}$,
where computation of $U_p$ requires $O(\min\{m,n\}^p)$ operations.
\citeasnoun{Goo77:New-Formula} proposes
some small economies in this computation.

\subsection{Estimating Schatten norms the easy way}
\label{sec:valiant-kong}

Next, we describe a more recent approach that was proposed
by \citeasnoun{KV17:Spectrum-Estimation}.  This method suffers
from even higher variance, but it is computationally efficient.
As a consequence, we can target larger values of $p$
and exploit a larger sample size $k$.

Superficially, the Kong \& Valiant estimator appears similar to the statistic $U_p$.
With the same notation, they restrict their attention to
\emph{increasing} sequences $i_1 < i_2 < \dots < i_p$ of indices.
In other words,
$$
V_p = \binom{k}{p}^{-1} \sum_{1 \leq i_1 < \dots < i_p \leq k}
	(\mtx{X})_{i_1 i_2} (\mtx{X})_{i_2 i_3} \dots (\mtx{X})_{i_p i_1}.
$$
Much as before, $V_p$ gives an unbiased estimator for $\norm{\mtx{B}}_{2p}^{2p}$.

Although $V_p$ still appears to be combinatorial, the restriction
to increasing sequences allows for a linear algebraic reformulation
of the statistic.
Let $\mathcal{T} : \Sym_k \to \F^{k \times k}$ be the linear
map that reports the strict upper triangle of a (conjugate) symmetric matrix.
Then
$$
V_p = \binom{k}{p}^{-1} \trace[ \mathcal{T}(\mtx{X})^{p-1} \mtx{X} ].
$$
The cost of computing $V_p$ is usually dominated by the $\bigO(k^2 n)$
arithmetic required to form $\mtx{X}$ given $\mtx{Y}$.
See Algorithm~\ref{alg:valiant-kong} for the procedure.

\citeasnoun{KV17:Spectrum-Estimation} obtain bounds for the variance
of the estimator $V_p$ to justify its employment when the number
$k$ of samples satisfies $k \gtrsim \min\{m,n\}^{1 - 2/p}$.  This bound is
probably substantially pessimistic for matrices that exhibit
spectral decay, but these theoretical and computational questions remain open.

\begin{algorithm}[t]
\begin{algorithmic}[1]
\caption{\textit{Schatten $2p$-norm estimation by random sampling.} \newline
See Section~\ref{sec:valiant-kong}.}
\label{alg:valiant-kong}

\Require	Input matrix $\mtx{B} \in \F^{m \times n}$, order $p$ of norm, number $k$ of samples
\Ensure		Schatten $2p$-norm estimate $V_p$
\Statex

\Function{SchattenEstimate}{$\mtx{B}$, $p$, $k$}

\State	Draw test matrix $\mtx{\Omega} \in \F^{n \times k}$ with iid isotropic columns
\State	Compute the sample matrix $\mtx{Y} = \mtx{B\Omega}$

\State	Form the Gram matrix $\mtx{X} = \mtx{Y}^* \mtx{Y} \in \F^{k \times k}$
\State	Extract the strict upper triangle $\mtx{T} = \mathcal{T}(\mtx{X})$
\State	Compute $\mtx{T}^{p-1}$ by repeated squaring
\State	Return $V_p = \trace( \mtx{T}^{p-1} \mtx{X} )$

\EndFunction
\end{algorithmic}
\end{algorithm}

\subsection{Bootstrapping the sampling distribution}

Given the lack of prior guarantees for the estimators $U_p$
and $V_p$ described in this section, we recommend that users
apply resampling methods to obtain empirical information about
the sampling distribution.  See \citeasnoun{AG92:Bootstrap-U-V}
for reliable bootstrap procedures for $U$-statistics.

\subsection{Extension: Estimating the spectral norm by random sampling}

Recall that the spectral norm of $\mtx{B}$
is comparable with the Schatten $2p$-norm of $\mtx{B}$
for an appropriate choice of $p$.  Indeed,
$$
\norm{ \mtx{B} }_{2p} \geq \norm{ \mtx{B} }
\geq \min\{m,n\}^{-1/(2p)} \norm{\mtx{B}}_{2p}
\quad\text{for $p \geq 1/2$.}
$$
Thus, the Schatten $2p$-norm is equivalent
to the spectral norm when $p \gtrsim \log \min\{m,n\}$.
In fact, when the matrix $\mtx{B}$ has a decaying
singular value spectrum, the Schatten $2p$-norm
may already be comparable with the spectral norm
for much smaller values of $p$.

Thus, we can try to approximate the spectral norm by estimating
the Schatten $2p$-norm for a sufficiently large value of $p$.
Resampling techniques can help ensure that the estimate
is reliable.  Nevertheless, this method should be used with caution.

\section{Maximum eigenvalues and trace functions}
\label{sec:max-eig}

Our discussion of estimating the spectral norm from a random sample
indicates that there is no straightforward way to construct
an unbiased estimator for the maximum eigenvalue of a PSD matrix.
Instead, we turn to Krylov methods, which repeatedly apply the matrix
to appropriate vectors to extract more information.

The power method and the Lanczos method are two
classic algorithms of this species.  Historically,
these algorithms have been initialized with a random
vector to ensure that the starting vector has
a component in the direction of a maximum eigenvector.
Later, researchers recognized that the randomness has
ancillary benefits.  In particular, randomized algorithms
can produce reliable estimates of the maximum eigenvalue,
even when it is not separated from the rest of the
spectrum \cite{Dix83:Estimating-Extremal,KW92:Estimating-Largest}.

In this section, we summarize theoretical results
on the randomized power method and randomized Krylov
methods for computing the maximum eigenvalue of a psd matrix.
These results also have implications for estimating
the minimum eigenvalue of a psd matrix
and the spectral norm of a general matrix.
Last, we explain how the Lanczos method leads
to accurate estimates for trace functions.

\subsection{Overview}

Consider a psd matrix $\mtx{A} \in \Sym_n$ %
with decreasingly ordered eigenvalues $\lambda_1 \geq \lambda_2 \geq \dots \geq \lambda_n \geq 0$.
We are interested in the problem of estimating the maximum eigenvalue $\lambda_1$.
As in the last two sections, we assume access to $\mtx{A}$ via the matrix--vector product
$\vct{u} \mapsto \mtx{A} \vct{u}$.

In contrast to the methods in Sections~\ref{sec:trace-est} and~\ref{sec:schatten-p},
the algorithms in this section require sequential applications
of the matrix--vector product.  In other words, we now demand
nonlinear information about the input matrix $\mtx{A}$.
As a consequence, these algorithms resist parallelization,
and they cannot be used in the one-pass or streaming environments.

The theoretical treatment in this section also covers the case of estimating
the spectral norm $\norm{\mtx{B}}$ of a rectangular
matrix $\mtx{B} \in \F^{m \times n}$.  Indeed, we can
simply pass to the psd matrix $\mtx{A} = \mtx{B}^* \mtx{B}$.
From an applied point of view, however, it is
important to develop separate algorithms
that avoid squaring the matrix $\mtx{B}$.
For brevity,
we omit a discussion about estimating spectral norms;
see \citeasnoun[Sec.~10.4]{GVL13:Matrix-Computations-4ed}.

\subsection{The randomized power method}
\label{sec:rand-power}

The randomized power method is a simple iterative algorithm
for estimating the maximum eigenvalue.

\subsubsection{Procedure}

First, draw a random test vector $\vct{\omega} \in \F^n$.
Since the maximum eigenvalue is unitarily invariant, it is
most natural to draw the test vector $\vct{\omega}$ from
a rotationally invariant distribution, such as
$\vct{\omega} \sim \normal(\vct{0}, \Id)$.
The power method iteratively constructs the sequence
$$
\vct{y}_0 = \frac{\vct{\omega}}{\norm{\vct{\omega}}}
\quad\text{and}\quad
\vct{y}_q = \frac{\mtx{A} \vct{y}_{q-1}}{\norm{\mtx{A} \vct{y}_{q-1}}}
\quad\text{for $q \geq 1$.}
$$
At each step, we obtain an eigenvalue estimate
$$
\xi_q = \vct{y}_q^* \mtx{A} \vct{y}_q
	= \frac{\vct{\omega}^* \mtx{A}^{2q+1} \vct{\omega}}{\vct{\omega}^* \mtx{A}^{2q} \vct{\omega}}
\quad\text{for $q \geq 0$.}
$$
The randomized power method requires simulation of a single
random test vector $\vct{\omega}$.
To perform $q$ iterations, it takes
$q$ sequential matrix--vector products with $\mtx{A}$
and lower-order arithmetic.
It operates with storage $O(n)$.
See Algorithm~\ref{alg:rand-power} for pseudocode.

\begin{algorithm}[t]
\begin{algorithmic}[1]
\caption{\textit{Randomized power method.} \newline
See Section~\ref{sec:rand-power}.}
\label{alg:rand-power}

\Require	Psd input matrix $\mtx{A} \in \Sym_n$, maximum number $q$ of iterations,
stopping tolerance $\eps$
\Ensure		Estimate $\xi$ of maximum eigenvalue of $\mtx{A}$
\Statex

\Function{RandomizedPower}{$\mtx{A}$, $q$, $\eps$}

\State	$\vct{\omega} = \texttt{randn}(n,1)$
	\Comment	Starting vector is random
\State	$\vct{y}_0 = \vct{\omega} / \norm{\vct{\omega}}$

\For{$i = 1, \dots, q$}
\State	$\vct{y}_{i} = \mtx{A} \vct{y}_{i-1}$
\State	$\xi_{i-1} = \vct{y}_{i-1}^* \vct{y}_i$
\State	$\vct{y}_{i} = \vct{y}_{i} / \norm{\vct{y}_i}$
\If		{$\abs{\xi_{i-1} - \xi_{i-2}} \leq \eps \xi_{i-1}$} \textbf{break}
	\Comment	[opt] Stopping rule
\EndIf
\EndFor

\State	\Return $\xi_{i-1}$

\EndFunction
\end{algorithmic}
\end{algorithm}

\subsubsection{Analysis}

The question is how many iterations $q$
suffice to make $\xi_q$ close to the
maximum eigenvalue $\lambda_1$.
More precisely, we aim to control the relative error $e_q$
in the eigenvalue estimate $\xi_q$:
$$
e_q = \frac{\lambda_1 - \xi_q}{\lambda_1}.
$$
The error $e_q$ is always nonnegative because of the Rayleigh theorem.
Note that the computed vector $\vct{y}_q$ always has a substantial
component in the invariant subspace associated
with eigenvalues larger than $\xi_q$, but it
may not be close to any maximum eigenvector,
even when $\xi_q \approx \lambda_1$.

\citeasnoun{KW92:Estimating-Largest} have established
several remarkable results about the evolution of the error.

\begin{theorem}[Randomized power method] \label{thm:rand-power}
Let $\mtx{A} \in \Sym_n(\RR)$ be a real psd matrix.  Draw a real test
vector $\vct{\omega} \sim \normal(\vct{0}, \Id)$.  After
$q$ iterations of the randomized power method,
the error $e_q$ in the maximum eigenvalue estimate $\xi_q$ satisfies
\begin{equation} \label{eqn:rand-power-nogap}
\Expect e_q \leq 0.871 \cdot \frac{\log n}{q - 1}
\quad\text{for $q \geq 2$.}
\end{equation}
Furthermore, if $\gamma = (\lambda_1 - \lambda_2) / \lambda_1$
is the relative spectral gap, then
\begin{equation} \label{eqn:rand-power-gap}
\Expect e_q \leq 1.254 \cdot \sqrt{n} \, \gamma \econst^{-q \gamma}
\quad\text{for $q \geq 1$.}
\end{equation}
\end{theorem}

The second result~\eqref{eqn:rand-power-gap}
does not appear explicitly in \citeasnoun{KW92:Estimating-Largest},
but it follows from related ideas~\cite{Tro18:Analysis-Randomized}.

\subsubsection{Discussion}

Theorem~\ref{thm:rand-power} makes two different claims.  First,
the power method can exhibit a burn-in period of $q \approx \log n$
iterations before it produces a nontrivial estimate of the maximum
eigenvalue; after this point, it always decreases the error in
proportion to the number $q$ of iterations.  The second claim
concerns the situation where the
matrix has a spectral gap $\gamma$ bounded away from zero.
In the latter case, after the burn-in period,
the error decreases at each iteration by a constant factor
that depends on the spectral gap.
The burn-in period of $q \approx \log n$ iterations is necessary
for \emph{any} algorithm that estimates the maximum eigenvalue
of $\mtx{A}$ from $q$ matrix--vector products
with the matrix~\cite{SER17:Gap-Strict-Saddles}.

Whereas classical analyses of the power method depend
on the spectral gap $\gamma$,
Theorem~\ref{thm:rand-power} comprehends
that we can estimate the maximum eigenvalue
even when $\gamma \approx 0$.
On the other hand, it is generally not possible to obtain a reliable
estimate of the maximum eigenvector in this extreme~\cite{LW98:Estimating-Largest}.

Theorem~\ref{thm:rand-power} can be improved in several respects.
First, we can develop variants where the dimension $n$ of the
matrix $\mtx{A}$ is replaced by its intrinsic dimension~\eqref{eqn:intdim}, %
or by smaller quantities that reflect spectral decay.
Second, when the maximum eigenvalue has multiplicity greater
than one, the power method estimates the maximum eigenvalue faster.
Third, the result can be extended to the complex setting.
See~\citeasnoun{Tro18:Analysis-Randomized}
for further discussion.

Although the power method is often deprecated
because it converges slowly, it is numerically
stable, and it enjoys the (minimal) storage cost
of $\bigO(n)$.

\subsection{Randomized Krylov methods}

The power method uses the $q$th power of the matrix to estimate
the maximum eigenvalue.  A more sophisticated approach %
allows \emph{any} polynomial with degree $q$.  Algorithms
based on this general technique are often referred to
as \emph{Krylov subspace methods}.  The most famous instantiation
is the \emph{Lanczos method}, which is an efficient implementation
of a Krylov subspace method for estimating the eigenvalues of
a self-adjoint matrix.

\subsubsection{Abstract procedure}

Draw a random test vector $\vct{\omega} \in \F^n$.  It is natural
to use a rotationally invariant distribution, such as
$\vct{\omega} \sim \normal(\vct{0}, \Id)$.  For a depth parameter $q \in \mathbb{N}$,
a randomized Krylov subspace method
implicitly constructs the subspace
$$
K_{q+1} %
	:= \lspan \{ \vct{\omega}, \mtx{A}\vct{\omega}, \dots, \mtx{A}^q \vct{\omega} \}.
$$
We can estimate the maximum eigenvalue of $\mtx{A}$ as
$$
\xi_q = \max_{\vct{u} \in K_{q+1}} \frac{\vct{u}^* \mtx{A} \vct{u}}{\vct{u}^*\vct{u}}
	= \max_{\deg p \leq q} \frac{\vct{\omega}^* \mtx{A} p^2(\mtx{A}) \vct{\omega}}{\norm{ p(\mtx{A}) \vct{\omega} }^2}.
$$
The maximum occurs over polynomials $p$ with coefficients in $\F$
and with degree at most $q$.
The notation $p(\mtx{A})$ refers to the spectral function induced by the polynomial $p$.
We will discuss implementations of Krylov methods below
in Section~\ref{sec:rand-lanczos}.

\subsubsection{Analysis}

Since the Krylov subspace is invariant to shifts in the spectrum
of $\mtx{A}$, it is more natural to compute
the error relative to the spectral range of the matrix:
$$
f_q = \frac{\lambda_1 - \xi_q}{\lambda_1 - \lambda_n}.
$$
The error $f_q$ is always nonnegative because it is a
Rayleigh quotient.

\citeasnoun{KW92:Estimating-Largest} have established striking
results for the maximum eigenvalue estimate obtained via a
randomized Krylov subspace method.

\begin{theorem}[Randomized Krylov method] \label{thm:rand-krylov}
Let $\mtx{A} \in \Sym_n(\RR)$ be a real psd matrix.  Draw a real test
vector $\vct{\omega} \sim \normal(\vct{0}, \Id)$.  After
$q$ iterations of the randomized Krylov method,
the error $f_q$ in the maximum eigenvalue estimate $\xi_q$ satisfies
\begin{equation} \label{eqn:rand-krylov-nogap}
\Expect f_q \leq 2.575 \cdot \left(\frac{\log n}{q - 1} \right)^2
\quad\text{for $q \geq 4$.}
\end{equation}
Furthermore, if $\gamma = (\lambda_1 - \lambda_2) / \lambda_1$
is the relative spectral gap, then
\begin{equation} \label{eqn:rand-krylov-gap}
\Expect f_q \leq 2.589 \cdot \sqrt{n} \, \econst^{-2(q-1) \sqrt{\gamma}}
\quad\text{for $q \geq 1$.}
\end{equation}
\end{theorem}

The second result~\eqref{eqn:rand-krylov-gap} is a direct consequence
of a more detailed formula reported in \citeasnoun{KW92:Estimating-Largest}.

\subsubsection{Discussion}

Like the randomized power method, a randomized Krylov method can
also exhibit a burn-in period of $q \approx \log n$ steps.  Afterwards,
the result~\eqref{eqn:rand-krylov-nogap} shows that the error
decreases in proportion to $1/q^2$, which is much faster than the
$1/q$ rate achieved by the power method.  Furthermore,
the result~\eqref{eqn:rand-krylov-gap} shows that each
iteration decreases the error by a constant factor
$\econst^{-2\sqrt{\gamma}}$ where $\gamma$ is the spectral gap.
In contrast, the power method only decreases the error by a
constant factor $\econst^{-\gamma}$.

Theorem~\ref{thm:rand-krylov} admits the same kind of refinements
as Theorem~\ref{thm:rand-power}.  In particular, we can replace
the dimension $n$ with measures that reflect the spectral decay
of the input matrix.  See \citeasnoun{Tro18:Analysis-Randomized}
for details.

\subsubsection{Implementing Krylov methods}
\label{sec:rand-lanczos}

What do we have to pay for the superior performance of the
randomized Krylov method?  If we only need an estimate
of the maximum eigenvalue, without an associated eigenvector estimate,
the cost is almost the same as for the randomized power method!
On the other hand, if we desire the eigenvector estimate,
it is common practice to store a basis for the
Krylov subspace $K_q$.  This is a classic example
of a time--data tradeoff in computation.

We present pseudocode for the randomized Lanczos method,
which is an efficient formulation of the Krylov method.
Algorithm~\ref{alg:rand-lanczos} is a direct
implementation of the Lanczos recursion, but
it exhibits complicated performance in floating-point arithmetic.
Algorithm~\ref{alg:rand-lanczos} includes the option to
add full reorthogonalization; this step removes the
numerical shortcomings at a substantial price in arithmetic (and storage).

If we use the Lanczos method without orthogonalization,
then $q$ iterations require $q$ matrix--vector multiplies
with $\mtx{A}$ plus $O(qn)$ additional arithmetic.
The orthogonalization step adds a total of
$O(q^2 n)$ additional arithmetic.
Computing the maximum
eigenvalue (and eigenvector) of the tridiagonal matrix $\mtx{T}$
can be performed with $O(q)$ arithmetic~\cite[Sec.~8.4]{GVL13:Matrix-Computations-4ed}.

If we do not require the maximum eigenvector, the Lanczos
method without orthogonalization operates with storage
$O(n)$.  If we need the maximum eigenvector or we
add orthogonalization, the storage cost grows to
$O(qn)$.
It is possible to avoid the extra storage by recomputing the Lanczos vectors,
but this approach requires great care.
One of the main thrusts in the literature
on Krylov methods is to reduce these storage costs
while maintaining rapid convergence.

\begin{warning}[Lanczos method]
Algorithm~\ref{alg:rand-lanczos} should be used with caution.
For a proper discussion about designing Krylov methods,
we recommend the books
\cite{Par98:Symmetric-Eigenvalue,BDD+00:Templates-Solution,GVL13:Matrix-Computations-4ed}.
There is also some recent theoretical work on the numerical stability
of Lanczos methods~\cite{MMS18:Stability-Lanczos,CDST19:Rank-1-Sketch}.
\end{warning}

\begin{algorithm}[t]
\begin{algorithmic}[1]
\caption{\textit{Randomized Lanczos method (with full reorthogonalization).} \newline
\textbf{Use with caution!}  See Section~\ref{sec:rand-lanczos}.}
\label{alg:rand-lanczos}

\Require	Psd input matrix $\mtx{A} \in \Sym_n$, maximum number $q$ of iterations
\Ensure		Estimate $(\xi, \vct{y})$ for a maximum eigenpair of $\mtx{A}$
\Statex

\Function{RandomizedLanczos}{$\mtx{A}$, $q$}
	
\State	$q = \min(q, n-1)$
\State	$\mtx{Q}(:,1) = \texttt{randn}(n, 1)$
	\Comment	Starting vector $\vct{\omega}$ is random      %
\State	$\mtx{Q}(:,1) = \mtx{Q}(:,1) / \norm{\mtx{Q}(:,1)}$

\For	{$i = 1, \dots, q$}
\State	$\mtx{Q}(:, i+1) = \mtx{A} \mtx{Q}(:, i)$				%
\State	$\alpha_i = \real(\mtx{Q}(:, i)^* \mtx{Q}(:, i+1))$		%

\If		{$i = 1$}                     %
\State		$\mtx{Q}(:, i+1) = \mtx{Q}(:, i+1) - \alpha_i \mtx{Q}(:, i)$
\Else
\State	$\mtx{Q}(:, i+1) = \mtx{Q}(:, i+1) - \alpha_i \mtx{Q}(:, i) - \beta_{i-1} \mtx{Q}(:, i-1)$
\EndIf

\Statex
\Statex
\Comment	[opt] Reorthogonalize via double Gram--Schmidt

\State	$\mtx{Q}(:, i+1) = \mtx{Q}(:, i+1) - \mtx{Q}(:, 1:i) (\mtx{Q}(:, 1:i)^*\mtx{Q}(:, i+1))$
\State	$\mtx{Q}(:, i+1) = \mtx{Q}(:, i+1) - \mtx{Q}(:, 1:i) (\mtx{Q}(:, 1:i)^*\mtx{Q}(:, i+1))$

\Statex
\State  $\beta_i = \norm{ \mtx{Q}(:, i+1) }$            %

\If     {$\beta_i < \mu \sqrt{n}$} \textbf{break}
	\Comment	$\mu$ is machine precision
\EndIf

\State  $\mtx{Q}(:, i+1) = \mtx{Q}(:, i+1) / \beta_i$        %
\EndFor

\State	$\mtx{T} = \texttt{tridiag}(\beta(1:i-1), \alpha(1:i), \beta(1:i-1))$

\State	$[\mtx{V}, \mtx{D}] = \texttt{eig}(\mtx{T})$
\State	$[\xi, \texttt{ind}] = \min(\diag(\mtx{D}))$
\State	$\vct{y} = \mtx{Q}(:, 1:i) \, \mtx{V}(:, \texttt{ind})$
	\Comment	[opt] Estimate max eigenvector

\EndFunction
\end{algorithmic}
\end{algorithm}

\subsection{The minimum eigenvalue}

The randomized power method and the randomized Krylov method
can be used to estimate the minimum
eigenvalue $\lambda_{n}$ of the psd matrix $\mtx{A} \in \Sym_n$.

The first approach is to apply the randomized power
method to the shifted matrix $\nu \Id - \mtx{A}$,
where the shift is chosen so that $\nu \geq \lambda_1$.
In this case, the algorithm produces an approximation
for $\nu - \lambda_n$.  Note that the error in
Theorem~\ref{thm:rand-power} is relative to
$\nu - \lambda_n$, rather than $\lambda_n$.

The second approach begins with the computation
of the Krylov subspace $K_{q+1}$.  Instead of maximizing
the Rayleigh quotient over the Krylov subspace,
we minimize it:
$$
\zeta_q = \min_{\vct{u} \in K_{q+1}} \frac{ \vct{u}^*\mtx{A}\vct{u} }{\vct{u}^*\vct{u}}.
$$
This approach directly produces an estimate $\zeta_q$
for $\lambda_n$.  The Krylov subspace is invariant
under affine transformations of the spectrum of $\mtx{A}$,
so we can obtain an error bound for $\zeta_q$
by applying Theorem~\ref{thm:rand-krylov}
formally to $\lambda_n \Id - \mtx{A}$.

\begin{remark}[Inverses]
If we can apply the matrix inverse $\mtx{A}^{-1}$
to vectors, then we gain access to a wider class
of algorithms for computing the minimum eigenvalue,
including (shifted) inverse iteration and the
Rayleigh quotient iteration.
See \cite{Par98:Symmetric-Eigenvalue} and \cite{GVL13:Matrix-Computations-4ed}.
\end{remark}

\subsection{Block methods}

The basic power method and Krylov method can be
extended by applying the iteration simultaneously to a larger
number of (random) test vectors.  The resulting algorithms
are called \emph{subspace iteration} and
\emph{block Krylov methods}, respectively.
Historically, the reason for developing block
methods was to resolve repeated or clustered eigenvalues. %

In randomized linear algebra,
we discover additional motivations for developing
block methods.
When the test vectors are drawn at random,
block methods may converge slightly faster,
and they succeed with much higher probability.
On modern computer architectures, the cost
of a block method may be comparable with the
cost of a simple vector iteration, which
makes this modification appealing.
We will treat this class of algorithm more
thoroughly in Section~\ref{sec:random-rangefinder},
so we postpone a full discussion.
See also~\cite{Tro18:Analysis-Randomized}.

\subsection{Estimating trace functions}
\label{sec:slq}

Finally, we turn to the problem of estimating the trace
of a spectral function of a psd matrix.

\subsubsection{Overview}

Consider a psd matrix $\mtx{A} \in \Sym_n$
with eigenpairs $(\lambda_j, \vct{u}_j)$ for $j = 1, \dots, n$.
Let $f : \RR_+ \to \RR$ be a function,
and suppose that we wish to approximate
$$
\trace f(\mtx{A}) = \sum\nolimits_{j=1}^n f(\lambda_j).
$$
We outline an incredible approach to this problem, called
\emph{stochastic Lanczos quadrature} (SLQ), that marries
the randomized trace estimator (Section~\ref{sec:trace-est})
to the Lanczos iteration (Algorithm~\ref{alg:rand-lanczos}).

This algorithm was devised by~\cite{GM94:Matrices-Moments}.
Our presentation is based on~\citeasnoun[Sec.~10.2]{GVL13:Matrix-Computations-4ed}
and \citeasnoun{UCS17:Fast-Estimation}.
\citeasnoun{GM10:Matrices-Moments} provide a more complete
treatment, and \citeasnoun{MMS18:Stability-Lanczos}
give a theoretical discussion about stability.

Related ideas can be used to estimate the trace of a spectral
function of a rectangular matrix; that is, the sum of
a function of each singular value of the matrix.
For brevity, we omit all details on the rectangular case.

\subsubsection{Examples}

Computing the trace of a spectral function is a ubiquitous problem with
a huge number of applications.  Let us mention some of the key examples.

\begin{enumerate} \setlength{\itemsep}{1mm}
\item	For $f(\lambda) = \lambda^{-1}$, the resulting
trace function is the trace of the matrix inverse.
This computation arises in electronic structure calculations.

\item	For $f(t) = \log t$, the associated trace function
is the log-determinant.  This computation arises in
Gaussian process regression.

\item	For $f(t) = t^p$ with $p \geq 1$, the trace function
is the $p$th power of the Schatten $p$-norm.  SLQ offers a more powerful alternative
to the methods in Section~\ref{sec:schatten-p}.
\end{enumerate}

\noindent
We refer to~\citeasnoun{UCS17:Fast-Estimation} for additional discussion.

\subsubsection{Procedure}

Let us summarize the mathematical ideas that lead to SLQ.
As usual, we draw an isotropic random vector $\vct{\omega} \in \F^n$. %
Then the random variable
$$
X = \vct{\omega}^* f(\mtx{A}) \vct{\omega}
\quad\text{satisfies}\quad
\Expect X = \trace f(\mtx{A}).
$$
Using the spectral resolution of $\mtx{A}$,
we can rewrite $X$ in the form
$$
X = \sum\nolimits_{j=1}^n f(\lambda_j) \, \abs{\vct{u}_j^* \vct{\omega}}^2
	= \int_{\RR_+} f(\lambda) \, {\nu}(\diff\lambda)
$$
for an appropriate measure $\nu$ on $\RR_+$ that depends on $\mtx{A}$ and $\vct{\omega}$.
Although we cannot generally compute the integral directly, we can approximate
it by using a numerical quadrature rule:
$$
X \approx \sum\nolimits_{\ell=1}^{q+1} \tau_\ell^2 f(\theta_\ell) =: Z.
$$
What is truly amazing is that the weights $\tau_\ell^2$ and the nodes $\theta_\ell$
for the quadrature rule can be extracted from the tridiagonal matrix $\mtx{T} \in \Sym_{q+1}$
produced by $q$ iterations of the Lanczos iteration with starting
vector $\vct{\omega}$.
This point is not obvious, but a full explanation exceeds our scope.

SLQ approximates the trace function by averaging independent
copies of the simple approximation:
$$
\trace f(\mtx{A}) \approx \frac{1}{k} \sum\nolimits_{i=1}^k Z_i
\quad\text{where}\quad
\text{$Z_i \sim Z$ are iid.}
$$
The analysis of the SLQ approximation requires heavy machinery
from approximation theory.
See~\citeasnoun{GM10:Matrices-Moments}
and~\citeasnoun{UCS17:Fast-Estimation} for more details.

Algorithm~\ref{alg:slq} contains pseudocode for SLQ.
The dominant cost is
$O(kq)$ matrix--vector multiplies with $\mtx{A}$,
plus $O(k q^2)$ additional arithmetic.  We recommend
using structured random test vectors to reduce the variance
of the resulting approximation.  The storage cost is
$O(qn)$ numbers.

\begin{algorithm}[t]
\begin{algorithmic}[1]
\caption{\textit{Stochastic Lanczos quadrature.} \newline
See Section~\ref{sec:slq}.}
\label{alg:slq}

\Require	Psd input matrix $\mtx{A} \in \Sym_n$, function $f$, number $k$ of samples, number $q$ of Lanczos iterations
\Ensure		Estimate $\bar{Z}_k$ for $\trace f(\mtx{A})$.
\Statex

\Function{StochasticLanczosQuadrature}{$\mtx{A}$, $f$, $k$, $q$}
	
\For	{$i = 1, \dots, k$}
	\Comment	Extract $k$ independent samples
\State	Draw a random isotropic vector $\vct{\omega}_i \in \F^n$

\State	Form $\mtx{T} = \textsc{RandomizedLanczos}(\mtx{A}, \vct{\omega}_i, q)$ %
\Statex
	\Comment	Apply $q$ steps of Lanczos with starting vector $\vct{\omega}_i$ %
	
\State	$[\mtx{V}, \mtx{\Theta}] = \texttt{eig}(\mtx{T})$
	\Comment	Tridiagonal eigenproblem

\State	Extract nodes $\mtx{\Theta} = \diag(\theta_1, \dots, \theta_{q+1})$
\State	Extract weights $\vct{\delta}_1^* \mtx{V} = (\tau_1, \dots, \tau_{q+1})$

\State	Form the approximation $Z_i = \sum_{\ell=1}^{q+1} \tau_\ell^2 f(\theta_\ell)$

\EndFor

\State	Return $\bar{Z}_k = k^{-1} \sum_{i=1}^k Z_i$

\EndFunction
\end{algorithmic}
\end{algorithm}

\section{Matrix approximation by sampling}
\label{sec:matrix-mc}

In Section~\ref{sec:trace-est},
we have seen that it is easy to form an unbiased
estimator for the trace of a matrix.
By averaging multiple copies of the simple estimator,
we can improve the quality of the estimate.
The same idea applies in the context of
matrix approximation.  In this setting, the
goal is to produce a matrix approximation
that has more ``structure'' than
the target matrix.  The basic approach is
to find a structured unbiased estimator
for the matrix and to average multiple copies
of the simple estimator to improve the approximation quality.

This section outlines two instances of
this methodology.
First, we develop a toy algorithm
for approximate multiplication of matrices.
Second, we show how to
approximate a dense graph Laplacian
by a sparse graph Laplacian;
this construction plays a role in the
\textsc{SparseCholesky} algorithm
presented in Section~\ref{sec:sparse-cholesky}.
Section~\ref{sec:kernel} contains
another example, the method of
random features in kernel learning.

The material in this section is summarized from
the treatments of matrix concentration in
\cite{Tro15:Introduction-Matrix,Tro19:Matrix-Concentration-LN}.

\subsection{Empirical approximation}

We begin with a high-level discussion of the method
of empirical approximation of a matrix.  The examples
in this section are all instances of this general idea.

Let $\mtx{B} \in \F^{m \times n}$ be a target matrix
that we wish to approximate by a more ``structured'' matrix.
Suppose that we can express the matrix $\mtx{B}$ as
a sum of ``simple'' matrices:
\begin{equation} \label{eqn:B-simple-sum}
\mtx{B} = \sum\nolimits_{i=1}^I \mtx{B}_i.
\end{equation}
In the cases we will study, the summands $\mtx{B}_i$
will be sparse or low-rank.

Next, consider a probability distribution $\{ p_i : i = 1, \dots, I \}$
over the indices in the sum~\eqref{eqn:B-simple-sum}.  For now,
we treat this distribution as given.  Construct a random matrix
$\mtx{X} \in \F^{m \times n}$ that takes values
$$
\mtx{X} = p_i^{-1} \mtx{B}_i
\quad\text{with probability $p_i$ for each $i = 1, \dots, I$.}
$$
(Enforce the convention that $0/0 = 0$.)
It is clear that $\mtx{X}$ is an unbiased estimator for
$\mtx{B}$:
$$
\Expect \mtx{X} = \sum\nolimits_{i=1}^I (p_i^{-1} \mtx{B}_i) p_i
	= \mtx{B}.
$$
The random matrix $\mtx{X}$ enjoys the same kind of structure
as the summands $\mtx{B}_i$.  On the other hand,
a single draw of the matrix $\mtx{X}$ is rarely a good approximation
of the matrix $\mtx{B}$.

To obtain a better estimator for $\mtx{B}$, we average independent
copies of the initial estimator:
$$
\bar{\mtx{X}}_k = \frac{1}{k} \sum\nolimits_{i=1}^k \mtx{X}_i
\quad\text{where $\mtx{X}_i \sim \mtx{X}$ are iid.}
$$
By linearity of expectation, $\bar{\mtx{X}}_k$ is also unbiased:
$$
\Expect \bar{\mtx{X}}_k = \mtx{B}.
$$
If the parameter $k$ remains small, then $\bar{\mtx{X}}_k$
inherits some of the structure from the $\mtx{B}_i$.
The question is how many samples $k$ we need to ensure that
$\bar{\mtx{X}}_k$ also approximates $\mtx{B}$ well.

\begin{remark}[History]
Empirical approximation was first developed by Maurey
in unpublished work from the late 1970s on the geometry
of Banach spaces.  The idea was first broadcast by
\citeasnoun{Car85:Inequalities-Bernstein-Jackson},
in a paper on approximation theory.  Applications to randomized linear algebra
were proposed by \citeasnoun{FKV04:Fast-Monte-Carlo}
and by
\cite{DKM06:Fast-Monte-Carlo-I,DKM06:Fast-Monte-Carlo-II,DKM06:Fast-Monte-Carlo-III}.
More refined analyses of empirical matrix approximation were obtained in~\cite{RV07:Sampling-Large,Tro15:Introduction-Matrix}.
Many other papers consider specific applications of the same methodology.
\end{remark}

\subsection{The matrix Bernstein inequality}

The main tool for analyzing the sample average estimator $\bar{\mtx{X}}_k$
from the last section is a variant of the matrix Bernstein inequality.
The following result is drawn from \citeasnoun{Tro15:Introduction-Matrix}.

\begin{theorem}[Matrix Monte Carlo] \label{thm:mtx-sampling}
Let $\mtx{B} \in \F^{m \times n}$ be a fixed matrix.  Construct a
random matrix $\mtx{X} \in \F^{m \times n}$ that satisfies
$$
\Expect \mtx{X} = \mtx{B}
\quad\text{and}\quad
\norm{\mtx{X}} \leq R.
$$
Define the per-sample second moment:
$$
v(\mtx{X}) := \max\{ \norm{ \Expect[\mtx{X} \mtx{X}^*] }, \norm{\Expect[\mtx{X}^*\mtx{X}]} \}.
$$
Form the matrix sampling estimator
$$
\bar{\mtx{X}}_k = \frac{1}{k} \sum\nolimits_{i=1}^k \mtx{X}_i
\quad\text{where $\mtx{X}_i \sim \mtx{X}$ are iid.}
$$
Then
$$
\Expect \norm{ \bar{\mtx{X}}_k - \mtx{B} } \leq \sqrt{\frac{2 v(\mtx{X}) \log(m+n)}{k}} + \frac{2R \log(m+n)}{3k}.
$$
Furthermore, for all $t \geq 0$,
$$
\Prob{ \norm{ \bar{\mtx{X}}_k - \mtx{B} } \geq t } \leq (m + n) \exp\left( \frac{-kt^2/2}{v(\mtx{X}) + 2Rt/3} \right).
$$
\end{theorem}

To explain the meaning of this result, let us determine how large $k$ should be to
ensure that the expected approximation error lies below a positive threshold
$\eps$.  The bound
$$
k \geq \max\left\{ \frac{2v(\mtx{X})\log(m+n)}{\eps^2},\ \frac{2R \log(m+n)}{3\eps} \right\}
$$
implies that $\Expect \norm{\bar{\mtx{X}}_k - \mtx{B} } \leq \eps+\eps^{2}$.
In other words, the number $k$ of samples should be proportional to the larger
of the second moment $v(\mtx{X})$ and the upper bound $R$.

This fact points toward a disappointing feature of empirical matrix approximation:
To make $\eps$ small, the number $k$ of samples must increase with $\eps^{-2}$
and often with $\log (m + n)$ as well.
This phenomenon is an unavoidable consequence of the central limit theorem and
the geometry induced by the spectral norm.  It means that matrix sampling estimators are not
suitable for achieving high-precision approximations.  See~\cite[Sec.~6.2.3]{Tro15:Introduction-Matrix}
for further discussion.  The logarithmic terms are also necessary in the worst case.

\begin{remark}[History]
The matrix Bernstein inequality has a long history, outlined
in~\citeasnoun{Tro15:Introduction-Matrix}.  The earliest
related results concern uniform smoothness estimates \cite{TJ74:Moduli-Smoothness}
and Khintchine inequalities \cite{LP86:Inegalites-Khintchine}
for the Schatten classes.
A first application to statistics appears in~\citeasnoun{Rud99:Random-Vectors}.
Modern approaches are based on the matrix Laplace transform
method~\cite{AW02:Strong-Converse,Oli09:Concentration-Adjacency,Tro12:User-Friendly}
or on the method of exchangeable pairs~\cite{MJCFT14:Matrix-Concentration,Tro16:Expected-Norm}.
\end{remark}

\subsection{Approximate matrix multiplication}
\label{sec:approx-mtx-mult}

A first application of empirical matrix approximation is to approximate
the product of two matrices: $\mtx{M} = \mtx{BC}$ where $\mtx{B} \in \F^{m \times I}$
and $\mtx{C} \in \F^{I \times n}$.  Computing the product by direct
matrix--matrix multiplication requires $O(mnI)$ arithmetic operations.
When the inner dimension $I$ is very large as compared with $m$ and $n$,
we might try to reduce the cost by sampling.

In our own work, we have not encountered situations
where approximate matrix multiplication
is practical because the quality of the approximation is very low.
Nevertheless, the theory serves a dual purpose as the foundation
for subspace embedding by discrete sampling (Section~\ref{sec:coord-embed}).

\subsubsection{Matrix multiplication by sampling}

To simplify the analysis, let us pre-scale the matrices $\mtx{B}$ and $\mtx{C}$
so that each one has spectral norm equal to one:
$$
\norm{\mtx{B}} = \norm{\mtx{C}} = 1.
$$
This step can be performed efficiently using the spectral
norm estimators outlined in Section~\ref{sec:max-eig}.
This normalization will remain in force for the rest of Section \ref{sec:approx-mtx-mult}.

Observe that the matrix--matrix product satisfies
$$
\mtx{BC} = \sum\nolimits_{i=1}^I (\mtx{B})_{:i} (\mtx{C})_{i:},
$$
where $(\vct{B})_{:i}$ is the $i$th column of $\mtx{B}$ and
$(\mtx{C})_{i:}$ is the $i$th row of $\mtx{C}$.  This expression
motivates us to approximate the product by sampling terms
at random from the sum.

Let $\{ p_i : i = 1, \dots, I \}$ be a sampling distribution,
to be specified later.
Form a random rank-one unbiased estimator for the product:
$$
\mtx{X} = p_i^{-1} \cdot (\mtx{B})_{:i} (\mtx{C})_{i:}
\quad\text{with probability $p_i$ for each $i = 1, \dots, I$.}
$$
By construction, $\Expect \mtx{X} = \mtx{BC}$.
We can average $k$ independent copies of $\mtx{X}$
to obtain a better approximation $\bar{\mtx{X}}_k$
to the matrix product.  The cost of computing the
estimator $\bar{\mtx{X}}_k$ of the matrix product
explicitly is only $O(mnk)$
operations, so it is more efficient than the
full multiplication when $k \ll I$.

Theorem~\ref{thm:mtx-sampling} yields an easy analysis of this approach.
The conclusion depends heavily on the choice of sampling distribution.
Regardless, we can never expect to attain a high-accuracy approximation of the
product by sampling because the number $k$ of samples must
scale proportionally with the inverse square $\eps^{-2}$
of the accuracy parameter $\eps$.

\subsubsection{Uniform sampling}
\label{sec:mtx-mult-unif}

The easiest way to approximate matrix multiplication
is to choose uniform sampling probabilities:
$p_i = 1/I$ for each $i = 1, \dots, I$.  To analyze this case,
we introduce the \emph{coherence} parameter:
$$
\mu(\mtx{B}) := I \cdot \max_{i=1, \dots, I} \norm{(\mtx{B})_{:i}}^2.
$$
Up to scaling, this is the maximum squared norm of a column of $\mtx{B}$.
Since $\mtx{B} \in \F^{m \times I}$ and $\norm{\mtx{B}}=1$,
the coherence lies in the range $[m, I]$.
The difficulty of approximating matrix multiplication by uniform sampling
increases with the coherence of $\mtx{B}$ and $\mtx{C}^*$.

To invoke Theorem~\ref{thm:mtx-sampling}, observe that the per-sample second moment and the
spectral norm of the estimator $\mtx{X}$ satisfy the bound
$$
\max\{ v(\mtx{X}), \norm{\mtx{X}} \} \leq \max\{\mu(\mtx{B}), \mu(\mtx{C}^*)\}.
$$
Let $\eps \in (0, 1]$ be an accuracy parameter.  If
$$
k \geq 2 \eps^{-2} \max\{ \mu(\mtx{B}), \mu(\mtx{C}^*) \} \log(m + n),
$$
then
$$
\Expect \norm{ \bar{\mtx{X}}_k - \mtx{BC} } \leq 2\eps.
$$
In other words, the number $k$ of rank-one factors we need to
obtain a relative approximation of the matrix product
is proportional to the maximum coherence of $\mtx{B}$
and $\mtx{C}^*$.  In the best scenario, the number of
samples is proportional to $\max\{m,n\} \log(m+n)$;
in the worst case, the sample complexity can be as large as $I$.

\subsubsection{Importance sampling}
\label{sec:mtx-mult-import}

If the norms of the columns of the factors $\mtx{B}$ and $\mtx{C}^*$
vary wildly, we may need to use importance sampling to obtain
a nontrivial approximation bound.  Define the sampling probabilities
$$
p_i = \frac{\norm{(\mtx{B})_{:i}}^2 + \norm{(\mtx{C})_{:i}}^2}{\fnorm{\mtx{B}}^2 + \fnorm{\mtx{C}}^2}
\quad\text{for $i = 1, \dots, I$.}
$$
These probabilities are designed to balance terms arising from Theorem~\ref{thm:mtx-sampling}.
In most cases, we compute the sampling distribution by directly evaluating the formula,
at the cost of $O((m + n)I)$ operations.

With the importance sampling distribution,
the per-sample second moment and the spectral norm of the estimator
$\mtx{X}$ satisfy
$$
\max\{ v(\mtx{X}), \norm{\mtx{X}} \} \leq \frac{1}{2} (\srank(\mtx{B}) + \srank(\mtx{C})).
$$
The stable rank is defined in~\eqref{eqn:stable-rank}.
Let $\eps \in (0, 1]$ be an accuracy parameter.  If
$$
k \geq \eps^{-2} \left( \srank(\mtx{B}) + \srank(\mtx{C}) \right) \log(m + n),
$$
then
$$
\Expect \norm{ \bar{\mtx{X}}_k - \mtx{BC} } \leq 2 \eps.
$$
In other words, the number of rank-one factors we need to approximate
the matrix product by importance sampling is $\log(m+n)$ times
the total stable rank of the matrices.

The sample complexity bound for importance sampling always improves
over the bound for uniform sampling.  Indeed, $\srank(\mtx{B}) \leq \mu(\mtx{B})$
under the normalization $\norm{\mtx{B}}=1$.  There are also many
situations where the stable rank is smaller than either dimension of
the matrix.  These are the cases where one might consider using
approximate matrix multiplication by importance sampling.

\subsubsection{History}

Randomized matrix multiplication was proposed in \citeasnoun{CL99:Approximating-Matrix}.
It is implicit in~\citeasnoun{FKV04:Fast-Monte-Carlo},
while \citeasnoun{DKM06:Fast-Monte-Carlo-I} give a more explicit treatment.
The analysis here has its origins in~\citeasnoun{RV07:Sampling-Large},
and the detailed presentation is adapted from~\citeasnoun{Zou13:Randomized-Primitives}.
Another interesting approach to randomized matrix multiplication
appears in \cite{Pag11:Compressed-Matrix}.
See~\citeasnoun[Chap.~6 Notes]{Tro15:Introduction-Matrix} for further references.

\subsection{Approximating a graph by a sparse graph}
\label{sec:graph-sparsification}

As a second application of empirical matrix approximation,
we will show %
how to take a dense graph and find a sparse
graph that serves as a proxy.  This procedure operates by
replacing the Laplacian matrix of the graph with a sparse
Laplacian matrix.  Beyond its intrinsic interest as a fact about
graphs, this technique plays a central role in randomized solvers
for Laplacian linear systems (Section~\ref{sec:sparse-cholesky}).

\subsubsection{Graphs and Laplacians}

We will consider weighted, loop-free, undirected graphs
on the vertex set $V = \{1, \dots, n\}$.
We can specify a weighted graph $G$ on $V$ by means of
a nonnegative weight function $w : V \times V \to \RR_+$.
The graph is \emph{loop-free} when $w_{ii} = 0$ for each vertex $i \in V$.
The graph is \emph{undirected} if and only if
$w_{ij} = w_{ji}$ for each pair $(i, j)$.
The \emph{sparsity} of an undirected graph is the number
of strictly positive weights $w_{ij}$ with $i \leq j$.

Alternatively, we can work with graph Laplacians.  %
The \emph{elementary Laplacian} $\mtx{\Delta}_{ij}$ on a vertex pair
$(i, j) \in V \times V$ is the psd matrix
$$
\mtx{\Delta}_{ij} = (\vct{\delta}_{i} - \vct{\delta}_j)(\vct{\delta}_i - \vct{\delta}_j)^* \in \Sym_n(\RR).
$$
Let $w$ be the weight function of a loop-free, undirected graph $G$.
The Laplacian associated with the graph $G$ is the matrix
\begin{equation} \label{eqn:graph-laplacian}
\mtx{L}_{G} = \sum\nolimits_{1 \leq i < j \leq n} w_{ij} \mtx{\Delta}_{ij} \in \Sym_n(\RR).
\end{equation}
The Laplacian $\mtx{L}_G$ is psd because it is a nonnegative
linear combination of psd matrices.

The Laplacian of a graph is analogous to the Laplacian differential
operator.  You can think about $\mtx{L}_G$ as an analog of the
heat kernel that models the diffusion of a particle on the graph.
The Poisson problem $\mtx{L}_G \vct{x} = \vct{f}$ serves
as a primitive for answering a wide range of questions involving
undirected graphs~\cite{Ten10:Laplacian-Paradigm}.
Applications include clustering and partitioning data,
studying random walks on graphs, and solving finite-element
discretizations of elliptic PDEs.

\subsubsection{Spectral approximation}

Fix a parameter $\eps \in (0, 1)$.
We say that a graph $H$ is an $\eps$-spectral approximation of a graph $G$
if their Laplacians are comparable in the psd order:
\begin{equation} \label{eqn:spectral-approximation}
(1 - \eps) \, \mtx{L}_G \psdle \mtx{L}_H \psdle (1+\eps) \, \mtx{L}_G.
\end{equation}
The relation~\eqref{eqn:spectral-approximation} ensures that the
graph $H$ and the graph $G$ are close cousins.

In particular, under~\eqref{eqn:spectral-approximation},
the matrix $\mtx{L}_H$ serves as an excellent
preconditioner for the Laplacian $\mtx{L}_G$.  In other words,
if we can easily solve (consistent) linear systems of the form
$\mtx{L}_H \vct{y} = \vct{b}$, then we can just as easily solve
the Poisson problem $\mtx{L}_G \vct{x} = \vct{f}$.

For an arbitrary input graph $G$, we will demonstrate that
there is a \emph{sparse} graph $H$ that is a good spectral
approximation of $G$.  For several reasons, this construction
does not immediately lead to effective methods for
designing preconditioners. %
Nevertheless, related ideas have resulted in practical, fast
solvers for Poisson problems on undirected graphs.
We will make this connection in
Section~\ref{sec:sparse-cholesky}.

\subsubsection{The normalizing map}

It is convenient to present a few more concepts from spectral graph theory.
Let us introduce the \emph{normalizing map}
$$
\mtx{K}_G(\mtx{A}) := (\mtx{L}_G^{\pinv})^{1/2} \mtx{A} (\mtx{L}_G^{\pinv})^{1/2}
	\quad\text{for $\mtx{A} \in \Sym_n(\RR)$.}
$$
As usual, $(\cdot)^\pinv$ is the pseudoinverse and $(\cdot)^{1/2}$
is the psd square root of a psd matrix.
We use the normalizing map to compare graph Laplacians.
Indeed,
\begin{equation} \label{eqn:graph-comparison}
\norm{ \mtx{K}_G(\mtx{L}_H - \mtx{L}_G) } \leq \eps
\end{equation}
implies that the graph $H$ is an $\eps$-spectral approximation
of the graph $G$.
This claim follows easily from~\eqref{eqn:spectral-approximation}.

\subsubsection{Effective resistance}

Next, define the \emph{effective resistance} $\varrho_{ij}$
of a vertex pair $(i, j)$ in the graph $G$ as
\begin{equation} \label{eqn:effective-resistance}
\varrho_{ij} %
	:= \trace[ \mtx{K}_G(\mtx{\Delta}_{ij}) ] \geq 0
	\quad\text{for each $1 \leq i < j \leq n$.}
\end{equation}
We can compute the family of effective resistances in time
$O(n^3)$ by means of a Cholesky factorization
of $\mtx{L}_G$, although faster algorithms are now available
\cite{Kyn17:Approximate-Gaussian}.

To understand the terminology, let us regard the graph $G$
as an electrical network where $w_{ij}$ is the
conductivity of the wire connecting the vertex pair $(i, j)$.
The effective resistance $\varrho_{ij}$ is the resistance of the
entire electrical network $G$ against passing a unit
of current from vertex $i$ to vertex $j$.

\subsubsection{Sparsification by sampling}

We are now prepared to construct a sparse approximation
of a loop-free, undirected graph $G$ specified by the weight function $w$.
The representation~\eqref{eqn:graph-laplacian} of the graph
Laplacian immediately suggests that we can apply the empirical
approximation paradigm.  To do so, we must design an
appropriate sampling distribution.

Define the sampling probabilities
$$
p_{ij} = \frac{w_{ij} \, \varrho_{ij}}{\rank(\mtx{L}_G)}
\quad\text{for each $1 \leq i < j \leq n$.}
$$
It is clear that $p_{ij} \geq 0$.  With some basic matrix algebra,
we can also confirm that $\sum_{i < j} p_{ij} = 1$.

Following our standard approach, we construct the random matrix
$$
\mtx{X} = \frac{w_{ij}}{p_{ij}} \, \mtx{\Delta}_{ij}
\quad\text{with probability $p_{ij}$ for each $i < j$.}
$$
Next, we average $k$ independent copies of the estimator:
$$
\bar{\mtx{X}}_k = \frac{1}{k} \sum\nolimits_{i=1}^k \mtx{X}_i
\quad\text{where $\mtx{X}_i \sim \mtx{X}$ are iid.}
$$
By construction, the estimator is unbiased: $\Expect \bar{\mtx{X}}_k = \mtx{L}_G$.
Moreover, $\bar{\mtx{X}}_k$ is itself the graph Laplacian
of a (random) graph $H$ with sparsity at most $k$.
We just need to determine the number $k$ of samples
that are sufficient to make $H$ a good spectral approximation of $G$.

Given the effective resistances, computation of the sampling probabilities
requires $O(s)$ time, where $s$ is the sparsity of the graph.
The cost of sampling $k$ copies of $\mtx{X}$ is $\tilde{O}(s + k)$.
It is natural to represent $\bar{\mtx{X}}_k$ using a sparse matrix
data structure, with storage cost $O(k \log n)$.

\subsubsection{Analysis}

The analysis involves a small twist.
In view of~\eqref{eqn:graph-comparison}, we need to
demonstrate that $\mtx{K}_G(\mtx{L}_H) \approx \mtx{K}_G(\mtx{L}_G)$.
Therefore, instead of considering
the random matrix $\bar{\mtx{X}}_k$, we pass to the random matrix
$\mtx{K}_G(\bar{\mtx{X}}_k)$, which is an unbiased estimator
for $\mtx{K}_G(\mtx{L}_G)$.

Note that the random matrix $\mtx{K}_G(\mtx{X})$ satisfies %
$$
\max\{ v( \mtx{K}_G(\mtx{X}) ), \norm{ \mtx{K}_G(\mtx{X}) }_2 \} \leq \rank(\mtx{L}_G) < n.
$$
Suppose that we choose
$$
k \geq 3 \eps^{-2} n \log(2n)
\quad\text{where $\eps \in (0, 1)$.}
$$
Then Theorem~\ref{thm:mtx-sampling} implies
$$
\Expect \norm{ \mtx{K}_G( \bar{\mtx{X}}_k - \mtx{L}_G ) }_2 \leq \eps.
$$
We conclude that the random graph $H$ with Laplacian $\mtx{L}_H = \bar{\mtx{X}}_k$
has at most $k$ nonzero weights, and it is an $\eps$-spectral
approximation to the graph $G$.

In other words, every graph on $n$ vertices has a $(1/2)$-spectral
approximation with at most $12 n \log n$ nonzero weights.
Modulo the precise constant, this is the tightest result that can be
obtained if we form the graph $H$ via random sampling.

\subsubsection{History}

The idea of approximating a matrix in the spectral norm
by means of random sampling of entries was proposed by \citeasnoun{AM01:Fast-Computation}
and \citeasnoun{AM07:Fast-Computation}.
This work initiated a line of literature on matrix sparsification
in the randomized NLA community;
see \citeasnoun{Tro15:Introduction-Matrix} for more references.
Let us emphasize that these general approaches achieve much
weaker approximation guarantees than~\eqref{eqn:spectral-approximation}.

The idea of sparsifying a graph by randomly sampling
edges in proportion to the effective resistances
was developed by~\citeasnoun{SS11:Graph-Sparsification};
the analysis above is drawn from~\citeasnoun{Tro19:Matrix-Concentration-LN}.
A deterministic method for graph sparsification, with superior guarantees,
appears in~\citeasnoun{BSS14:Twice-Ramanujan}.

\section{Randomized embeddings}
\label{sec:gauss}

One of the core tools in randomized linear algebra
is \emph{randomized linear embedding},
often assigned the misnomer
\emph{random projection}.
The application of randomized embeddings
is often referred to as \emph{sketching}.

This section begins with a formal definition
of a randomized embedding.  Then we introduce
the Gaussian embedding, which is the simplest construction,
and we summarize its analysis. %
Randomized embeddings have a wide range
of applications in randomized linear algebra.
Some implications of this theory include the %
Johnson--Lindenstrauss lemma and a
simple construction of a subspace embedding.
We also explain why results for
Gaussians transfer to a wider setting.
Last, we give a short description of
random partial isometries,
a close cousin of Gaussian embeddings.

\subsection{What is a random embedding?}

Let $E \subseteq \F^n$ be a set,
and let $\eps \in (0, 1)$ be a distortion parameter.
We say that a linear map $\mtx{S} : \F^{n} \to \F^d$
is an ($\ell_2$) \emph{embedding} of $E$ with distortion $\eps$
when
\begin{equation} \label{eqn:embed-abstract}
(1 - \eps) \, \norm{\vct{x}} \leq \norm{ \mtx{S} \vct{x} } \leq (1 + \eps) \, \norm{ \vct{x} }
\quad\text{for all $\vct{x} \in E$.}
\end{equation}
It is sometimes convenient to abbreviate this kind of two-sided
inequality as $\norm{\mtx{S}\vct{x}} = (1 \pm \eps) \norm{\vct{x}}$.

We usually think about the case where $d \ll n$,
so the map $\mtx{S}$ enacts a dimension reduction. %
In other words, $\mtx{S}$
transfers data from the high-dimensional space $\F^n$
to the low-dimensional space $\F^d$.
As we will discuss, the low-dimensional representation of the data
can be used to obtain fast, approximate solutions
to computational problems.

The relation~\eqref{eqn:embed-abstract} expresses
the idea that the embedding $\mtx{S}$ should preserve the geometry
of the set $E$.  Unfortunately, we do not always know the set
$E$ in advance.  Moreover, we would like the map $\mtx{S}$ to
be easy to construct, and it should be computationally
efficient to apply $\mtx{S}$ to the data.
These goals may be in tension.

We can resolve this dilemma by drawing
the embedding $\mtx{S}$ from
a probability distribution.  Many
types of probability distributions serve.
In particular, we can use highly structured
random matrices (Section~\ref{sec:dimension-reduction}) that are easy to build,
to store, and to apply to vectors.
Section~\ref{sec:overdet-ls} presents a
case study about how random embeddings can be
applied to solve overdetermined least-squares
problems.

\subsection{Restricted singular values}

Our initial goal is to understand something
about the theoretical behavior of randomized embeddings.
To that end, let us introduce quantities that measure how
much an embedding distorts a set.
Let $\mtx{S} \in \F^{d \times n}$ be a linear map,
and let $E \subseteq \mathbb{S}^{n-1}(\F)$ be an arbitrary
subset of the unit sphere in $\F^n$.
The \emph{minimum} %
and \emph{maximum restricted singular value}
are, respectively, defined as
\begin{equation} \label{eqn:rsv}
\sigma_{\min}(\mtx{S}; E) := \min_{\vct{x} \in E}\ \norm{ \mtx{S} \vct{x} }
\quad\text{and}\quad
\sigma_{\max}(\mtx{S}; E) := \max_{\vct{x} \in E}\ \norm{ \mtx{S} \vct{x} }.
\end{equation}
If $E$ composes the entire unit sphere, then these quantities coincide
with the ordinary minimum and maximum singular value of $\mtx{S}$.
More generally, the restricted
singular values describe how much the linear map $\mtx{S}$
can contract or expand a point in $E$.

\begin{remark}[General sets]
In this treatment, we require $E$ to be a subset of the unit
sphere.  Related, but more involved, results hold when $E$ is a general set.
See~\citeasnoun{TOH14:Gaussian-Minmax} and \citeasnoun{OT18:Universality-Laws}
for more results and applications.
\end{remark}

\subsection{Gaussian embeddings}

Our theoretical treatment of random embeddings focuses on
the most highly structured case. %
A Gaussian embedding is a random matrix of the form
$$
\mtx{\Gamma} \in \F^{d \times n}
\quad\text{with iid entries $(\mtx{\Gamma})_{ij} \sim \textsc{normal}(0, d^{-1})$.}
$$
The cost of explicitly storing a Gaussian embedding is $O(dn)$,
and the cost of applying it to a vector is $O(dn)$.

The scaling of the matrix ensures that
$$
\Expect \norm{ \mtx{\Gamma} \vct{x} }^2 = \norm{ \vct{x} }^2
\quad\text{for each $\vct{x} \in \F^n$.}
$$
We wish to understand how large to choose the embedding
dimension $d$ so that the map $\mtx{\Gamma}$ approximately
preserves the norms of all points in a given set $E$.
We can do so by obtaining bounds for the restricted
singular values.  For a Gaussian embedding,
we will see that $\sigma_{\min}(\mtx{\Gamma}; E)$
and $\sigma_{\max}(\mtx{\Gamma}; E)$ are controlled
by the geometry of the set $E$. %

\begin{remark}[Why Gaussians?]
Gaussian embeddings admit a simple and beautiful analysis.
In our computational experience, many other embeddings exhibit
the same (universal) behavior as a Gaussian map.
In spite of that, the rigorous analysis of other types
of embeddings tends to be difficult,
even while it yields rather imprecise results.
The confluence of these facts motivates us to argue
that the Gaussian analysis provides enough insight
for many practical purposes.
\end{remark}

\subsection{The Gaussian width}

For the remainder of Section~\ref{sec:gauss}, we will work in the
real field ($\F = \RR$).
Given a set $E \subseteq \mathbb{S}^{n-1}(\RR)$,
define the \emph{Gaussian width} $w(E)$ via
$$
w(E) := \Expect \sup_{\vct{x} \in E}\ \ip{ \vct{g} }{ \vct{x} }
\quad\text{where $\vct{g} \in \RR^n$ is standard normal.}
$$
The Gaussian width is a measure of the content of the
set $E$. %
It plays a fundamental role
in the performance of randomized embeddings.

Here are some basic properties of the Gaussian width.

\begin{itemize} \setlength{\itemsep}{1mm}
\item	The width is invariant under rotations: $w(\mtx{Q}E) = w(E)$
for each orthogonal matrix $\mtx{Q}$.

\item	The width is increasing with respect to set inclusion: $E \subseteq F$
implies that $w(E) \leq w(F)$.

\item	The width lies in the range $0 \leq w(E) \leq \Expect \norm{\vct{g}} < \sqrt{n}$.
\end{itemize}

The width can be calculated accurately for many sets of interest.
In particular, if $L$ is an arbitrary $k$-dimensional subspace of $\RR^n$,
then
\begin{equation} \label{eqn:width-subspace}
\sqrt{k-1} < w( L \cap \mathbb{S}^{n-1} ) < \sqrt{k}.
\end{equation}
Indeed, it is productive to think about the \emph{squared}
width $w^2(E)$ as a measure of the ``dimension'' of the set $E$.

\begin{remark}[Statistical dimension]
The \emph{statistical dimension} is another measure
of content that is closely related to the squared Gaussian width.
The statistical dimension has additional geometric properties
that make it easier to work with in some contexts.
See~\cite{ALMT14:Living-Edge,MT13:Achievable-Performance,MT14:Steiner-Formulas}
and \cite{GNP17:Gaussian-Phase}
for more information.
\end{remark}

\subsection{Restricted singular values of Gaussian matrices}
\label{sec:rsv-gauss}

In a classic work on Banach space geometry,
\citeasnoun{Gor88:Milmans-Inequality} showed that the
Gaussian width controls both the minimum and maximum
restricted singular values of a subset of the sphere.

\begin{theorem}[Restricted singular values: Gaussian matrix] \label{thm:rsv-gauss}
Fix a subset $E \subseteq \mathbb{S}^{n-1}(\RR)$
of the unit sphere.
Draw a Gaussian matrix $\mtx{\Gamma} \in \RR^{d \times n}$
whose entries are iid $\textsc{normal}(0, d^{-1})$.
For all $t > 0$,
$$
\begin{aligned}
&\Prob{ \sigma_{\min}(\mtx{\Gamma}; E ) \leq 1 - \frac{w(E) + 1}{\sqrt{d}} - t }
	&\leq \econst^{-dt^2/2}; \\
&\Prob{ \sigma_{\max}(\mtx{\Gamma}; E) \geq 1 + \frac{w(E)}{\sqrt{d}} + t }
	&\leq \econst^{-dt^2/2}.
\end{aligned}
$$
\end{theorem}

\begin{proof}(Sketch)
The first inequality is a consequence of Gordon's minimax
theorem and Gaussian concentration.
The second inequality is essentially Chevet's theorem,
which follows from Slepian's lemma.
See~\citeasnoun[Chap.~3.3]{LT91:Probability-Banach}
for an overview of these ideas.
\end{proof}

Theorem~\ref{thm:rsv-gauss} yields the relations
$$
1 - \frac{w(E) + 1}{\sqrt{d}} \lessapprox \sigma_{\min}(\mtx{\Gamma}; E) \leq \sigma_{\max}(\mtx{\Gamma}; E) \lessapprox 1 + \frac{w(E)}{\sqrt{d}}.
$$
In other words, the embedding dimension should satisfy $d > (w(E) + 1)^2$
to ensure that the map $\mtx{\Gamma}$ is unlikely to annihilate any
point in $E$.   For this choice of $d$, the random embedding
is unlikely to dilate any point in $E$ by more than a factor of two.

As a consequence, we have reduced the problem of computing embedding
dimensions for Gaussian maps to the problem of computing Gaussian widths.
In the next two subsections, we work out
two important examples.

\begin{remark}[Optimality]
The statements in Theorem~\ref{thm:rsv-gauss} are nearly optimal.
One way to see this is to consider the set $E = \mathbb{S}^{n-1}(\RR)$,
for which the theorem implies that
$$
1 - \sqrt{n/d} \lessapprox \Expect \sigma_{\min}(\mtx{\Gamma})
	\leq \Expect \sigma_{\max}(\mtx{\Gamma}) \leq 1 + \sqrt{n/d}.
$$
The Bai--Yin law~\cite[Sec.~5.2]{BS10:Spectral-Analysis} confirms
that the first and last inequality are sharp %
as $n, d \to \infty$ with $n/d \to \mathrm{const} \in [0,1]$.

Moreover, if $E$ is spherically convex (i.e., the intersection of a convex
cone with the unit sphere), then the minimum restricted singular
value satisfies the reverse inequality
$$
\Prob{ \sigma_{\min}(\mtx{\Gamma}; E) \geq 1 - \frac{w(E)}{\sqrt{d}} + t }
	\leq 2\econst^{-dt^2/2}.
$$
This result is adapted from~\cite{TOH14:Gaussian-Minmax}.

In addition, fifteen years of computational experiments have also shown that
the predictions from Theorem~\ref{thm:rsv-gauss} are frequently sharp.
See~\citeasnoun{OT18:Universality-Laws} for some examples and references.
\end{remark}

\begin{remark}[History]
The application of Gaussian comparison theorems
in numerical analysis can be traced to work
in mathematical signal processing.
\citeasnoun{RV08:Sparse-Reconstruction}
used a corollary of Gordon's minimax theorem
to study $\ell_1$ minimization problems.
Significant extensions and improvements of this argument were made
by \citeasnoun{Sto10:l1-Optimization} and \citeasnoun{CRPW12:Convex-Geometry}.
\citeasnoun[Rem.~2.9]{ALMT14:Living-Edge} seem to have been the first
to recognize that Gordon's minimax theorem can be reversed
in the presence of convexity.  A substantial refinement
of this observation appeared in~\citeasnoun{TOH14:Gaussian-Minmax}.
There is a long series of follow-up works by Babak Hassibi's
group that apply this insight to other problems in
signal processing and communications.
\end{remark}

\subsection{Example: Johnson--Lindenstrauss}

As a first application of Theorem~\ref{thm:rsv-gauss}, let us explain how it
implies the classic dimension reduction result of~\citeasnoun{JL84:Extensions-Lipschitz}.

\subsubsection{Overview}

Let $\{ \vct{a}_1, \dots, \vct{a}_N \} \subset \RR^n$ be a discrete point set.
We would like to know when a Gaussian embedding
$\mtx{\Gamma} \in \RR^{d \times n}$ approximately preserves
all the pairwise distances between these points:
\begin{equation} \label{eqn:jl}
1 - \eps \leq \frac{\norm{ \mtx{\Gamma}(\vct{a}_i - \vct{a}_j) }}{ \norm{ \vct{a}_i - \vct{a}_j } }
	\leq 1 + \eps
	\quad\text{for all $i \neq j$.}
\end{equation}
The question is how large we must set the embedding dimension $d$
to achieve distortion $\eps \in (0, 1)$.

\subsubsection{Analysis}

We can solve this problem using the machinery described in Section~\ref{sec:rsv-gauss}.
Consider the set $E$ of normalized chords:
$$
E = \left\{ \frac{\vct{a}_i - \vct{a}_j}{\norm{\vct{a}_i - \vct{a}_j}} : 1 \leq i < j \leq N \right\}.
$$
By the definition~\eqref{eqn:rsv} of the restricted singular values,
$$
\sigma_{\min}(\mtx{\Gamma}; E) \leq \frac{\norm{ \mtx{\Gamma}(\vct{a}_i - \vct{a}_j) }}{ \norm{ \vct{a}_i - \vct{a}_j } }
	\leq \sigma_{\max}(\mtx{\Gamma}; E).
$$
Therefore, we can invoke Theorem~\ref{thm:rsv-gauss} to determine
how the embedding dimension controls the distortion.

Let us summarize the argument.
First, observe that the Gaussian width of the set $E$ satisfies
$$
w(E) = \Expect \max_{\vct{x} \in E}\ \ip{ \vct{g} }{ \vct{x} }
	\leq \sqrt{2 \log \# E}
	< 2 \sqrt{\log (N/2)}.
$$
As a consequence,
$$
\begin{aligned}
&\Prob{ \sigma_{\min}(\mtx{\Gamma}; E) \leq 1 - (1 + 2 \sqrt{\log(N/2)})/\sqrt{d} - t } &\leq \econst^{-d t^2 /2}; \\
&\Prob{ \sigma_{\max}(\mtx{\Gamma}; E) \geq 1 + 2 \sqrt{\log(N/2)}/\sqrt{d} + t } &\leq \econst^{-d t^2 /2}.
\end{aligned}
$$
To achieve distortion $\eps$ with high probability, it is sufficient to choose
$$
d \geq 8 \eps^{-2} \log N.
$$
In other words, the embedding dimension only needs to be \emph{logarithmic} in the cardinality $N$ of the point set.
With some additional calculation, we can also extract precise failure probabilities from this analysis.

\subsubsection{Discussion}

Let us close this example with a few comments.
In spite of its prominence, the Johnson--Lindenstrauss embedding lemma
is somewhat impractical.  Indeed, since the embedding
dimension $d$ is proportional to $\eps^{-2}$, it is a challenge to achieve
small distortions.  Even if we consider the setting where $\eps \approx 1$,
the uniform bound~\eqref{eqn:jl} may require the embedding dimension
to be prohibitively large.

As a step toward more applicable results, note
that the bound on the \emph{minimum} restricted singular value is more crucial
than the bound on the maximum restricted singular value, because
the former ensures that no two points coalesce after the random embedding.
Similarly, it is often more valuable to preserve the distances between
nearby points than between far-flung points.  This observation is the
starting point for the theory of locality sensitive
hashing~\cite{GIM99:Similarity-Search}.

\subsubsection{History}

\citeasnoun{JL84:Extensions-Lipschitz} were concerned with a problem in Banach space geometry,
namely the prospect of extending a Lipschitz function from a finite metric space into
a Hilbert space.  The famous lemma from their paper took on a life of its own
when~\citeasnoun{LLR95:Geometry-Graphs} used it to design efficient
approximation algorithms for some graph problems.
\citeasnoun{IM98:Approximate-Nearest} used random embeddings
to develop new algorithms for the approximate nearest neighbor problem.
\cite{AMS99:Space-Complexity,AGMS02:Tracking-Join} introduced the term
\emph{sketching}, and they showed how to use sketches to track streaming data.
Soon afterwards, \citeasnoun{PRTV00:Latent-Semantic} and \citeasnoun{FKV04:Fast-Monte-Carlo}
proposed using random embeddings and matrix sampling for low-rank matrix approximation,
bringing these ideas into the realm of computational linear algebra.

\subsection{Example: Subspace embedding}
\label{sec:subspace-embedding}

Next, we consider a question at the heart of
randomized linear algebra.  Can we embed an unknown
subspace into a lower-dimensional space?

\subsubsection{Overview}

Suppose that $L$ is a $k$-dimensional subspace in $\RR^n$.
We say that a dimension reduction map $\mtx{S}$
is a \emph{subspace embedding} for $L$ with distortion $\eps \in (0, 1)$ if
\begin{equation} \label{eqn:subspace-embedding}
 (1 - \eps) \, \norm{ \vct{x} }
	\leq \norm{ \mtx{S}\vct{x} }
	\leq (1 + \eps) \, \norm{ \vct{x} }
	\quad\text{for every $\vct{x} \in L$.}
\end{equation}
We say that $\mtx{S}$ is \emph{oblivious} if it can
be constructed without knowledge of the subspace $L$,
except for its dimension.

Two questions arise.  First, what types of dimension
reduction maps yield (oblivious) subspace embeddings?
Second, how large must we choose the embedding dimension
to achieve this outcome?

\subsubsection{Analysis}

Gaussian dimension reduction maps yield very
good oblivious subspace embeddings.  Theorem~\ref{thm:rsv-gauss}
easily furnishes the justification.
Consider the unit sphere in the subspace: $E = L \cap \mathbb{S}^{n-1}(\RR)$.
Then construct the Gaussian dimension reduction map $\mtx{\Gamma} \in \RR^{d \times n}$. %
In view of~\eqref{eqn:width-subspace}, we have
$$
\begin{aligned}
&\Prob{ \sigma_{\min}( \mtx{\Gamma}; E ) \leq 1 - (1+\sqrt{k})/\sqrt{d} - t } &\leq \econst^{-d t^2/2}; \\
&\Prob{ \sigma_{\max}( \mtx{\Gamma}; E ) \geq 1 + \sqrt{k}/\sqrt{d} + t } &\leq \econst^{-d t^2/2}. \\
\end{aligned}
$$
As a specific example, we can set the embedding dimension $d = 2k$
to ensure that $\norm{\mtx{\Gamma}\vct{x}} = (1 \pm 0.8) \norm{\vct{x}}$
simultaneously for all points $\vct{x} \in L$,
except with probability $\econst^{-\mathrm{c} k}$.
In some applications of subspace embeddings,
we can even choose the %
dimension as small as $d = k + 5$ or $d = k + 10$.

Many theoretical papers on randomized NLA use
subspace embeddings as a primitive for designing
algorithms for other linear algebra problems.
For example, Section~\ref{sec:overdet-ls}
describes several ways to use subspace embeddings to solve
overdetermined least-squares problems.

\subsubsection{History}

Subspace embeddings were explicitly introduced by~\citeasnoun{Sar06:Improved-Approximation};
see also \citeasnoun{DMM06:Subspace-Sampling}.  As work on randomized NLA
accelerated, researchers became interested in more structured types of
subspace embeddings; an early reference is~\citeasnoun{WLRT08:Fast-Randomized}.
Section~\ref{sec:dimension-reduction} covers these extensions.
See~\citeasnoun{2014_woodruff_sketching} for a theoretical perspective on
randomized NLA where subspace embeddings take pride of place.

\subsection{Universality of the minimum restricted singular value}

We have seen how to apply Gaussian dimension reduction
for embedding discrete point sets and for embedding subspaces.
Theorem~\ref{thm:rsv-gauss} contains precise theoretical
results on the behavior of Gaussian maps in terms of the
Gaussian width.  To what extent can we transfer this analysis
to other types of random embeddings?

The following theorem \cite[Thm.~9.1]{OT18:Universality-Laws}
shows that the bound on the \emph{minimum} restricted singular value
in Theorem~\ref{thm:rsv-gauss} is universal for a large class of random embeddings.
In particular, this class includes sparse random matrices,
whose nonzero entries compose a vanishing proportion of the total.

\begin{theorem}[Universality] \label{thm:universality}
Fix a set $E \subseteq \mathbb{S}^{n-1}$.  Let $\mtx{S} \in \RR^{d \times n}$
be a random matrix whose entries are independent random variables
that satisfy
$$
\Expect[ (\mtx{S})_{ij} ] = 0, \quad
\Expect[ (\mtx{S})_{ij}^2 ] = d^{-1}, \quad
\Expect[ (\mtx{S})_{ij}^5 ] \leq R. %
$$
When $d \leq n$, with high probability,
$$
\sigma_{\min}(\mtx{S}; E)
	\geq 1 - \frac{w(E)}{\sqrt{d}} - o(\sqrt{n/d}).
$$
The constant in $o(\sqrt{n/d})$ depends only on $R$.
A matching lower bound for $\sigma_{\min}(\mtx{S}; E)$
holds when $E$ is spherically convex.
\end{theorem}

In other words, if $E$ is a moderately large set,
the distribution of the entries of the random map
$\mtx{S}$ does not have an impact on the
embedding dimension $d$ sufficient to ensure no
point in $E$ is annihilated.

Theorem~\ref{thm:universality} is confirmed by extensive numerical
experiments~\cite{OT18:Universality-Laws},
which demonstrate that dimension reduction maps with
independent, standardized entries have identical performance for a wide range of examples.

It is perhaps surprising that the bound on the \emph{maximum}
restricted singular value from Theorem~\ref{thm:rsv-gauss} is not universal.
For some sets $E$, the quantity $\sigma_{\max}(\mtx{S}; E)$
depends heavily on the distribution of the entries of $\mtx{S}$.

\begin{remark}[Universality for least-squares]
\citeasnoun{DL18:Asymptotics-Sketching} give some asymptotic universality
results for random embeddings in the context of least-squares problems.
\end{remark}

\subsection{Random partial isometries}

Last, we consider a variant of Gaussian embedding that
is more suitable when the embedding dimension $d$
is close to the ambient dimension $n$.  In this
section, we allow the field $\F$ to be real or complex.

First, suppose that $d \leq n$, and let $\mtx{\Gamma} \in \F^{d \times n}$
be a Gaussian embedding.  Almost surely,
the co-range of $\mtx{\Gamma}$ is a uniformly random
$d$-dimensional subspace of $\F^n$.
Construct an embedding $\mtx{S} \in \F^{d \times n}$
with orthonormal rows that span the co-range
of $\mtx{\Gamma}$, for example by QR factorization.

Similarly, we can consider a Gaussian embedding $\mtx{\Gamma} \in \F^{d \times n}$
with $d \geq n$.  In this case, the range of $\mtx{\Gamma}$ is almost surely a uniformly
random $n$-dimensional subspace of $\F^d$.  Construct
an embedding $\mtx{S} \in \F^{d \times n}$ with orthonormal
columns that span the range of $\mtx{\Gamma}$, for example
by QR factorization.

In each case, we call $\mtx{S}$ a \emph{random partial isometry}.
The cost of storing a random partial isometry is $O(dn)$,
and the cost of applying it to a vector is $O(dn)$.
(We should warn the punctilious reader that QR factorization
of $\mtx{\Gamma}$ may not produce a matrix $\mtx{S}$ that is Haar-distributed
on the Stiefel manifold.  To achieve this guarantee, use the
algorithms from \citeasnoun{Mez07:How-to-Generate}.)

When $d \approx n$, random partial isometries are better embeddings
than Gaussian maps (because the nonzero singular values of a partial isometry are all equal).
When $d$ and $n$ are significantly different, the two
models are quite similar to each other.

\citeasnoun{TH15:Isotropically-Random} have established some
theoretical results on the embedding behavior of real
partial isometries $(\F = \RR)$.  Unfortunately,
the situation is more complicated
than in Theorem~\ref{thm:rsv-gauss}.
More relations between Gaussian matrices and partial
isometries follow from the Marcus--Pisier comparison
theorem~\cite{MP81:Random-Fourier};
see also~\citeasnoun{Tro12:Comparison-Principle}.

\section{Structured random embeddings}
\label{sec:dimension-reduction}

Gaussian embeddings and random partial isometries work extremely well.
But they are not suitable for all practical
applications because they are expensive to construct, to store,
and to apply to vectors.  Instead, we may prefer to implement
more structured embedding matrices that alleviate these burdens.

This section summarizes a number of constructions that
have appeared in the literature, with a focus on methods that
have been useful in applications.  Except as noted, these
approaches have the same practical performance as either a
Gaussian embedding or a random partial isometry.

Although many of these approaches are supported by theoretical analysis,
the results are far less precise than for Gaussian embeddings.
As such, we will not give detailed
mathematical statements about structured embeddings.
See Section~\ref{sec:how-in-theory} for a short discussion about how to
manage the lack of theoretical guarantees.

\subsection{General techniques}

Many types of structured random embeddings operate
on the same principle, articulated in \cite{AC09:Fast-Johnson-Lindenstrauss}.
When we apply the random embedding to a fixed vector,
it should homogenize (``mix'') the coordinates so that each one carries
about the same amount of energy.
Then the embedding can sample coordinates at random to extract a
lower-dimensional vector whose norm is proportional to the
norm of the original vector and has low variance.
Random embeddings differ in how they perform the initial mixing step.
Regardless of how it is done, mixing is very important for obtaining
embeddings that work well in practice.

With this intuition at hand, let us introduce a number
of pre- and post-processing transforms
that help us design effective random embeddings.
These approaches are used in many of the constructions below.

We say that a random variable is a \emph{random sign}
if it is $\textsc{uniform}\{ z \in \F : \abs{z} = 1 \}$.
A random matrix $\mtx{E} \in \F^{n \times n}$ is called
a \emph{random sign flip} if it is diagonal,
and the diagonal entries are iid random sign variables.

A \emph{random permutation} $\mtx{\Pi} \in \F^{n \times n}$ is
a matrix drawn uniformly at random from the set of permutation
matrices.  That is, each row and column of $\mtx{\Pi}$
has a single nonzero entry, which equals one, and all such
matrices are equally likely.

For $d \leq n$, a random matrix $\mtx{R} \in \F^{d \times n}$
is called a \emph{random restriction} if it selects $d$ uniformly
random entries from its input.  With an abuse of terminology,
we extend the definition to the case $d \geq n$ by making
$\mtx{R} \in \F^{d \times n}$ the matrix that embeds its input
into the first $n$ coordinates of the output.
That is, $(\mtx{R})_{ij} = 1$ when $i = j$ and
zero otherwise.

Random sign flips and permutations are useful for preconditioning
the input to a random embedding.  Random restrictions are
useful for reducing the dimension of a vector that has
already been homogenized.

\subsection{Sparse sign matrices}
\label{sec:sparse-map}

Among the earliest proposals for non-Gaussian embedding
is to use a sparse random matrix whose entries are
random signs.

Here is an effective construction of a sparse sign matrix
$\mtx{S} \in \F^{d \times n}$.  Fix a sparsity parameter
$\zeta$ in the range $2 \leq \zeta \leq d$.  The random
embedding takes the form
$$
\mtx{S} = \sqrt{\frac{n}{\zeta}} \begin{bmatrix} \vct{s}_1 & \dots & \vct{s}_n \end{bmatrix} \in \F^{d \times n}.
$$
The columns $\vct{s}_i \in \F^d$ are iid random vectors.  To construct each column,
we draw $\zeta$ iid random signs, and we situate them in $\zeta$ uniformly
random coordinates.  \citeasnoun{TYUC19:Streaming-Low-Rank}
recommend choosing $\zeta = \min\{d, 8\}$ in practice.

We can store a sparse embedding using about $O(\zeta n \log d)$
numbers.  We can apply it to a vector in $\F^n$ with $O(\zeta n)$
arithmetic operations.  The main disadvantage is that we must use sparse
data structures and arithmetic to achieve these benefits.  Sparse sign
matrices have similar performance to Gaussian embeddings.

\citeasnoun{Coh16:Nearly-Tight-Oblivious} has shown that a sparse sign
matrix serves as an oblivious subspace embedding with constant distortion
for an arbitrary $k$-dimensional subspace of $\RR^n$ when
the embedding dimension $d = O(k \log k)$ and
the per-column sparsity $\zeta = O(\log k)$.
It is conjectured that improvements are still possible.

\begin{remark}[History]
Sparse random embeddings emerged from the work of
\citeasnoun{Ach03:Database-Friendly-Random}
and \citeasnoun{CCF04:Finding-Frequent}.
For randomized linear algebra applications, sparse embeddings
were promoted in~\cite{CW13:Low-Rank-Approximation,MM13:Low-Distortion-Subspace,NN13:OSNAP-Faster}
and~\cite{Ura13:Fast-Randomized}.
Analyses of the embedding behavior of a sparse map appear in \cite{BDN15:Toward-Unified} and~\cite{Coh16:Nearly-Tight-Oblivious}.
\end{remark}

\subsection{Subsampled trigonometric transforms}
\label{sec:srtt}

Another type of structured randomized embeddings is designed
to mimic the performance of a random partial isometry.
One important class of examples consists of the subsampled
randomized trigonometric transforms (SRTTs).

To construct a random embedding $\mtx{S} \in \F^{d \times n}$
with $d \leq n$, we select a unitary trigonometric transform
$\mtx{F} \in \F^{n \times n}$.  Then we form
$$
\mtx{S} = \sqrt{\frac{n}{d}} \mtx{RFE\Pi},
$$
where $\mtx{R} \in \F^{d\times n}$ is a random restriction,
$\mtx{E} \in \F^{n \times n}$ is a random sign flip,
and $\mtx{\Pi} \in \F^{n \times n}$ is a random permutation.
Note that $\mtx{S}$ is a partial isometry.

The trigonometric transform $\mtx{F}$ can be any one
of the usual suspects.  In the complex case ($\F = \CC$),
we often use a discrete Fourier transform.
In the real case ($\F = \RR$), common choices
are the discrete cosine transform (DCT2) or
the discrete Hartley transform (DHT).
When $n$ is a power of two, we can consider a Walsh--Hadamard transform (WHT).
The paper~\cite{2010_avron_BLENDENPIK}
reports that the DHT is the best option
in the real case.

The cost of storing a SRTT is $O(n \log n)$,
and it can be applied to a vector in $O(n \log d)$
operations using a fast subsampled trigonometric transform algorithm.
The main disadvantage is that it requires
a good implementation of the fast transform.

\citeasnoun{Tro11:Improved-Analysis} has shown that
an SRTT serves as an oblivious subspace embedding
with constant distortion
for an arbitrary $k$-dimensional subspace of $\F^n$
provided that $d = O(k \log k)$.
This paper focuses on the Walsh--Hadamard transform,
but the analysis extends to other SRTTs.
In practice, it often suffices to choose $d = O(k)$,
but no rigorous justification is available.

\begin{remark}[Rerandomization]
It is also common to repeat the randomization and trigonometric transformations:
$$
\mtx{S} = \sqrt{\frac{n}{d}} \mtx{R}\mtx{F}\mtx{E}'\mtx{\Pi}'\mtx{FE\Pi},
$$
with an independent sign flip $\mtx{E}'$ and an independent permutation $\mtx{\Pi}'$.
This enhancement can make the embedding more robust, although it is not always necessary.
\end{remark}

\begin{remark}[History]
\citeasnoun{1995_parker_randombutterfly}
proposed the use of randomized trigonometric
transforms to precondition linear systems.
The idea of applying an SRTT for dimension
reduction appears in
\cite{2006_ailon_chazelle_FJLT} and \cite{AC09:Fast-Johnson-Lindenstrauss}.
\citeasnoun{WLRT08:Fast-Randomized}
develop algorithms for low-rank matrix
approximation based on SRTTs.
Embedding properties of an SRTT for general sets
follow from \cite{RV08:Sparse-Reconstruction} and \cite{KW11:New-Improved};
see \citeasnoun[Chap.~12]{FR13:Mathematical-Introduction}
or~\citeasnoun{PW15:Randomized-Sketches}.
\end{remark}

\subsection{Tensor random projections}
\label{sec:trp}

Next, we describe a class of random embeddings that are useful
for very large linear algebra and multilinear algebra problems.
This approach invokes tensor products to form
a random embedding for a high-dimensional space from
a family of random embeddings for lower-dimensional spaces.

Let $\mtx{S}_1 \in \F^{d \times m_1}$ and $\mtx{S}_2 \in \F^{d \times m_2}$
be statistically independent random embeddings.  We define
the \emph{tensor random embedding}
$$
\mtx{S} := \mtx{S}_1 \odot \mtx{S}_2 \in \F^{d \times n}
\quad\text{where}\quad
n = m_1 m_2
$$
to be the Khatri--Rao product of $\mtx{S}_1$ and $\mtx{S}_2$.  That is,
the $i$th row of $\mtx{S}$ is
$$
(\mtx{S})_{i:} = \begin{bmatrix} (\mtx{S}_1)_{i1} (\mtx{S}_2)_{i:} & \dots & (\mtx{S}_1)_{im} (\mtx{S}_2)_{i:} \end{bmatrix}
	\quad\text{for $i=1,\dots,d$.}
$$
Under moderate assumptions on the component embeddings
$\mtx{S}_1$ and $\mtx{S}_2$, the tensor random embedding
$\mtx{S}$ preserves the squared Euclidean norm of
an arbitrary vector in $\F^{n}$.

A natural extension of this idea is to draw many component embeddings $\mtx{S}_i \in \F^{d \times m_i}$
for $i = 1, \dots, k$ and to form the tensor random embedding
$$
\mtx{S} := \mtx{S}_1 \odot \mtx{S}_2 \odot \dots \odot \mtx{S}_k \in \F^{d \times n}
\quad\text{where}\quad
n = \prod\nolimits_{i=1}^k m_i.
$$
This embedding also inherits nice properties from its components.

The striking thing about this construction is that the tensor product
embedding operates on a \emph{much} larger space than the component embeddings.
The storage cost for the component embeddings is $O(d (\sum_{i=1}^k m_i))$,
or less.  We can apply the tensor random embedding to a vector directly
with $O(dn)$ arithmetic.
We can accelerate the process by using component embeddings that have fast transforms,
and we can obtain improvements for vectors that have a compatible tensor product structure.
Some theoretical analysis is available, but results are not yet complete.

\begin{remark}[History]
Tensor random embeddings were introduced by~\citeasnoun{KRSU10:Price-Privately}
on differential privacy.  They were first analyzed by~\citeasnoun{Rud12:Row-Products}.
The paper~\cite{SGTU18:Tensor-Random} proposed the application
of tensor random embeddings for randomized linear algebra;
some extensions appear in~\cite{JKW19:Faster-Johnson-Lindenstrauss} and \cite{MB19:Guarantees-Kronecker}.
See~\cite{BV19:Polynomial-Threshold} and \cite{Ver19:Concentration-Inequalities}
for related theoretical results.
\end{remark}

\subsection{Other types of structured random embeddings}

We have described the random embeddings that have received the
most attention in the NLA literature.  Yet there are other
types of random embeddings that may be useful in
special circumstances.  Some examples include
random filters \cite{TWDBB06:Random-Filters,KW11:New-Improved,RRT12:Restricted-Isometries,MRW18:Improved-Bounds},
the Kac random walk~\cite{Kac56:Foundations-Kinetic,Ros94:Random-Rotations,Oli09:Convergence-Equilibrium,PS17:Kacs-Walk},
and sequences of random reflections~\cite{Slo83:Encrypting-Random,Por96:Cutoff-Phenomenon}.
\citeasnoun{2009_edo_liberty_dissertation} discusses a number of other instances.

\subsection{Coordinate sampling}
\label{sec:coord-embed}

So far, we have discussed random embeddings that
mix up the coordinates of a vector.
It is sometimes possible to construct embeddings just
by sampling coordinates at random.
Coordinate sampling can be appealing in specialized situations
(e.g., kernel computations), where we only have access to individual entries
of the data.  On the other hand, this approach requires
strong assumptions, and it is far less reliable
than random embeddings that mix coordinates.
In this section, we summarize the basic facts about
subspace embedding via random coordinate sampling.

A note on terminology: we will use the term
\emph{coordinate sampling} to distinguish these
maps from random embeddings that mix coordinates.

\subsubsection{Coherence and leverage}
\label{sec:coherence_leverage}

Let $L \subset \F^{n}$ be a $k$-dimensional subspace.
The \emph{coherence} $\mu(L)$ of the subspace with respect
to the standard basis is
$$
\mu(L) := n \cdot \max_{i = 1, \dots, n} \norm{ \mtx{P}_L \vct{\delta}_i }^2,
$$
where $\mtx{P}_L \in \Sym_n$ is the orthogonal projector onto $L$
and $\vct{\delta}_i$ is the $i$th standard basis vector.
The coherence $\mu(L)$ lies in the range $[k, n]$.  The behavior
of coordinate sampling methods degrades as the coherence increases.

Next, define the (subspace) \emph{leverage score} distribution
with respect to the standard coordinate basis:
$$
p_i = \frac{1}{k} \norm{ \mtx{P}_L \vct{\delta}_i }^2
\quad\text{for $i = 1, \dots, n$.}
$$
It is straightforward to verify that
$(p_1, \dots, p_n)$ is a probability distribution.
In most applications, it is expensive to compute
or estimate subspace leverage scores because we
typically do not have a basis for the subspace $L$ at hand.

\subsubsection{Uniform sampling}

We can construct an embedding $\mtx{S} \in \F^{d \times n}$
by sampling each output coordinate uniformly at random.
That is, the rows of $\mtx{S}$ are iid, and each row
takes values $\vct{\delta}_i / \sqrt{d}$, each with probability $1/n$.
(We can also sample coordinates uniformly \emph{without} replacement; this
approach performs slightly better but requires more work to analyze.)

The embedding dimension $d$ must be chosen to ensure that
\begin{equation} \label{eqn:sample-embed}
\norm{\mtx{S} \vct{x} }^2 = (1 \pm \eps) \norm{ \vct{x} }^2
\quad\text{for all $\vct{x} \in L$.}
\end{equation}
To achieve this goal, it suffices that
$$
d \geq 2 \eps^{-2} \mu(L) \log(2k).
$$
In other words, the embedding dimension is proportional
to the coherence of the subspace, up to a logarithmic factor.
We expect uniform sampling to work well precisely when the
coherence is small ($\mu(L) \approx k$).

To prove this result, let $\mtx{U} \in \F^{n \times k}$ be an orthonormal
basis for the subspace $L$.  We can approximate the product $\Id_k = \mtx{U}^*\mtx{U}$
by sampling columns of $\mtx{U}$ uniformly at random.  The analysis in
Section~\ref{sec:mtx-mult-unif} furnishes the conclusion.

\subsubsection{Leverage score sampling}

When the coherence is large, it seems more natural to sample
with respect to the leverage score distribution $(p_1, \dots, p_n)$
described above.  That is, the embedding $\mtx{S} \in \F^{d \times n}$
has iid rows, and each row takes value $\vct{\delta}_i / \sqrt{d}$
with probability $p_i$.

To achieve the embedding guarantee~\eqref{eqn:sample-embed}
with this sampling distribution, we should choose the embedding dimension
$$
d \geq 2 \eps^{-2} k \log(2k).
$$
In other words, it suffices that the embedding dimension is proportional
to the dimension $k$ of the subspace, up to the logarithmic factor.
This result follows from the analysis of matrix multiplication
by importance sampling (Section~\ref{sec:mtx-mult-import}).

\subsubsection{Discussion}

Uniform sampling leads to an oblivious subspace embedding,
although it is not obvious how to select the embedding dimension
in advance because the coherence is usually not available.
Leverage score sampling is definitely not oblivious,
because we need to compute the sampling probabilities
(potentially at great cost).

In practice, uniform sampling works better than one might anticipate,
and it has an appealing computational profile.
As a consequence, it has become a workhorse for large-scale kernel computation;
see~\cite{KMT12:Sampling-Methods,2013_bach_sharp} and \cite{RCR17:FALKON-Optimal}.

Our experience suggests that leverage score sampling is rarely
a competitive method for constructing subspace embeddings,
especially once we take account of the effort required to
compute the sampling probabilities. We recommend using other
types of random embeddings (Gaussians, sparse maps, SRTTs)
in lieu of coordinate sampling whenever possible.

Although coordinate sampling may seem like a natural approach
to construct matrix approximations involving rows or columns,
we can obtain better algorithms for this problem by
using (mixing) random embeddings.  See Section~\ref{sec:natural} for details.

Coordinate sampling can be a compelling choice in situations where
other types of random embeddings are simply unaffordable. For instance,
a variant of leverage score sampling leads to effective
algorithms for kernel ridge regression;
see~\citeasnoun{RCCR18:Fast-Leverage} for evidence.
Indeed, in the context of kernel computations,
coordinate sampling and random features
may be the only tractable methods for
extracting information from the kernel matrix.
We discuss these ideas in Section~\ref{sec:kernel}.

An emerging research direction uses coordinate sampling for
solving certain kinds of continuous problems, such as function interpolation.
In this setting, sampling corresponds to function evaluation,
while mixing embeddings may lead to operations that are impossible
to implement in the continuous space.  For some examples, see
\cite{RW12:Sparse-Legendre,CDL13:Stability-Accuracy,HD15:Coherence-Motivated,RW16:Interpolation-Weighted,CM17:Optimal-Weighted,ABC19:Sequential-Sampling,AKM+19:Universal-Sampling} and \cite{CP19:Active-Regression}.

\subsubsection{History}

Most of the early theoretical computer science papers on randomized NLA rely on
coordinate sampling methods.  These approaches typically
construct an importance sampling distribution using the norms
of the rows or columns of a matrix.
For examples, see~\cite{FKV04:Fast-Monte-Carlo,2005_drineas_nystrom} and~\cite{DKM06:Fast-Monte-Carlo-I,DKM06:Fast-Monte-Carlo-II,DKM06:Fast-Monte-Carlo-III}.
These papers measure errors in the Frobenius norm. %
The first spectral norm analysis of coordinate sampling
appears in~\citeasnoun{RV07:Sampling-Large}.

Leverage scores are a classical tool in statistical regression,
used to identify influential data points.  %
\citeasnoun{2008_drineas_relative_error_CUR} proposed
the subspace leverage scores as a sampling distribution
for constructing low-rank matrix approximations.
\citeasnoun{2009_mahoney_CUR} identified the connection
with regression.  \citeasnoun{2011_mahoney_survey}
made a theoretical case for using leverage scores
as the basis for randomized NLA algorithms.
\citeasnoun{AM15:Fast-Randomized} introduced an alternative
definition of leverage scores for kernel ridge regression.

It is somewhat harder to trace the application of uniform sampling in randomized NLA. %
Several authors have studied the behavior of uniform
sampling in the context of Nystr{\"o}m approximation;
see \cite{WS01:Using-Nystrom,KMT12:Sampling-Methods} and \cite{Git13:Topics-Randomized}.
An analysis of uniform coordinate sampling is implicit in the theory on SRTTs;
see~\citeasnoun[Lemma 3.4]{Tro11:Improved-Analysis}.
See \citeasnoun{2017_kannan_vempala_acta} for more discussion about
sampling methods in NLA.

\subsection{But how does it work in theory?}
\label{sec:how-in-theory}

Structured random embeddings and random coordinate sampling
lack the precise guarantees that we can attribute
to Gaussian embedding matrices.  So how can we apply them with confidence?

First, we advocate using \textit{a posteriori} error estimators to assess
the quality of the output of a randomized linear algebra computation.
These error estimators are often quite cheap, yet they can give
(statistical) evidence that the computation was performed correctly.
We also recommend adaptive algorithms that can detect when the
accuracy is insufficient and make refinements.
With this approach, it is not pressing to produce theory that
justifies all of the internal choices (e.g., the specific type of random embedding)
in the NLA algorithm.  See Section~\ref{sec:error-est} for further discussion.

Even so, we would like to have \textit{a priori} predictions about
how our algorithms will behave.  Beyond that, we need reliable methods
for selecting algorithm parameters, especially in the streaming
setting where we cannot review the data and repeat the computation.

Here is one answer to these concerns.
As a practical matter, we can simply invoke the lessons from the Gaussian theory,
even when we are using a different type of random embedding.
The universality result, Theorem~\ref{thm:universality},
gives a rationale for this approach in one special case.
We also recommend undertaking computational experiments
to verify that the Gaussian theory gives an adequate description of
the observed behavior of an algorithm.

\begin{warning}[Coordinate sampling]
Mixing random embeddings perform similarly to Gaussian embeddings,
but coordinate sampling methods typically exhibit behavior that
is markedly worse.
\end{warning}

\section{How to use random embeddings}
\label{sec:overdet-ls}

Algorithm designers have employed random embeddings
for many tasks in linear algebra, optimization, and
related areas.
Methods based on random embedding
fall into three rough categories:
(1) sketch and solve, (2) iterative sketching,
and (3) sketch and precondition.  To draw distinctions
among these paradigms, we use each one to derive
an algorithm for solving an overdetermined least-squares
problem.

\subsection{Overdetermined least-squares}

Overdetermined least-squares problems
sometimes arise in statistics and data-analysis applications.
We may imagine that some of the data in these problems is redundant.
As such, it seems plausible that we could reduce the size of
the problem to accelerate computation without too much
loss in accuracy.

Consider a matrix $\mtx{A} \in \F^{m \times n}$ with $m \gg n$
and a vector $\vct{b} \in \F^m$.
An overdetermined least-squares problem has the form
\begin{equation} \label{eqn:overdet-ls}
\underset{\vct{x} \in \F^n}{\text{minimize}} \
	\frac{1}{2} \norm{ \mtx{A} \vct{x} - \vct{b} }^2.
\end{equation}
Following \citeasnoun{PW15:Randomized-Sketches}, we can also rewrite
the least-squares problem to emphasize the role of the matrix:
\begin{equation} \label{eqn:overdet-ls2}
\underset{\vct{x} \in \F^n}{\text{minimize}} \
	\frac{1}{2} \norm{ \mtx{A} \vct{x} }^2 - \ip{ \vct{x} }{ \mtx{A}^* \vct{b} }.
\end{equation}
Write $\vct{x}_{\star}$ for an arbitrary solution
to the problem~\eqref{eqn:overdet-ls}.

To make a clear comparison among algorithm design templates,
we will assume that $\mtx{A}$ is dense and unstructured.
In this case, the classical approach to solving~\eqref{eqn:overdet-ls}
is based on factorization of the coefficient matrix
(such as QR or SVD) at a cost of $O(mn^2)$ arithmetic operations.

When $\mtx{A}$ is sparse, we would typically use
iterative methods (such as CG), which have a different
computational profile.  For sparse matrices, we would
also make different design choices in a sketch-based algorithm.
Nevertheless, for simplicity, we will not discuss
the sparse case.

\subsection{Subspace embeddings for least-squares}

To design a sketching algorithm for the overdetermined
least-squares problem~\eqref{eqn:overdet-ls},
we need to construct a subspace embedding $\mtx{S} \in \F^{d \times m}$
that preserves the geometry of the range
of the matrix $\mtx{A} \in \F^{m \times n}$.
In some cases, we may also need the embedding to
preserve the range of the bordered matrix $\begin{bmatrix} \mtx{A} & \vct{b} \end{bmatrix} \in \F^{m \times (n+1)}$.

Since the matrix $\mtx{A}$ is dense and unstructured,
we will work with a structured subspace embedding,
such as an SRTT (Section~\ref{sec:srtt}).
More precisely, we will assume that evaluating
the product $\mtx{SA}$ costs only
$O(mn \log d)$ arithmetic operations.
The best theoretical results for these
structured sketches require that
the embedding dimension $d \sim n \log(n) / \eps^2$
to achieve distortion $\eps$,
although the logarithmic factor seems to be
unnecessary in practice.

Throughout this section, we use the heuristic
notation $\sim$ to indicate quantities that
are proportional.
We also write $\ll$ to mean
``much smaller than.''

\subsection{Sketch and solve}
\label{sec:sketchandsolve}

The sketch-and-solve paradigm maps the overdetermined least-squares
problem \eqref{eqn:overdet-ls} into a smaller space.
Then it uses the solution to the reduced problem
as a proxy for the solution to the original problem.
This approach can be very fast, and we only need
one view of the matrix $\mtx{A}$.  On the other hand,
the results tend to be very inaccurate.

Let $\mtx{S} \in \F^{d \times m}$ be a subspace embedding
for the range of $\begin{bmatrix} \mtx{A} & \vct{b} \end{bmatrix}$
with distortion $\eps$.
Consider the compressed least-squares problem
\begin{equation} \label{eqn:sketch-solve-ls}
\underset{\vct{x} \in \RR^n}{\text{minimize}} \
	\frac{1}{2} \norm{ \mtx{S} (\mtx{A} \vct{x} - \vct{b}) }^2.
\end{equation}
Since $\mtx{S}$ preserves geometry,
we may hope that the solution $\widehat{\vct{x}}$ to the sketched problem~\eqref{eqn:sketch-solve-ls}
can replace the solution $\vct{x}_{\star}$
to the original problem~\eqref{eqn:overdet-ls2}.
A typical theoretical bound is %
$$
\norm{ \mtx{A} \widehat{\vct{x}} - \vct{b} } \leq (1 + \eps) \norm{ \mtx{A} \vct{x}_{\star} - \vct{b} }
\quad\text{when $d \sim n \log(n) / \eps^2$.}
$$
See~\citeasnoun{Sar06:Improved-Approximation}.
Although the residuals are comparable, it need \emph{not}
be the case that $\widehat{\vct{x}} \approx \vct{x}_{\star}$,
even when the solution to~\eqref{eqn:overdet-ls} is unique.

The sketch-and-solve paradigm requires us to form the matrix $\mtx{SA}$
at a cost of $O(mn \log d)$ operations.
We would typically solve the (dense) reduced problem
with a direct method, using $O(dn^2)$ operations.
Assuming $d \sim n \log(n)/\eps^2$, the total arithmetic cost
is $O(mn \log(n/\eps^2) + n^3 \log(n)/\eps^2)$.

In summary, we witness an improvement in computational cost
over classical methods if $\log n \ll n \ll m/\log n$
and $\eps$ is constant.  But we must also be
willing to accept large errors, because we cannot make $\eps$ small.

\begin{remark}[History]
The sketch-and-solve paradigm is attributed
to \citeasnoun{Sar06:Improved-Approximation}.
It plays a major role in the theoretical
algorithms literature; see~\citeasnoun{2014_woodruff_sketching}
for advocacy.  It has also been proposed
for enormous problems that might otherwise
be entirely hopeless~\cite{2017_weare_randomized_iteration}.
\end{remark}

\subsection{Iterative sketching}
Iterative sketching attempts to remediate the poor accuracy of the sketch-and-solve paradigm
by applying it repeatedly to reduce the residual error.

First, we construct an initial solution $\vct{x}_0 \in \F^n$ using the sketch-and-solve paradigm
with a constant distortion embedding.  For each iteration $i$,
draw a fresh random subspace embedding $\mtx{S}_i \in \F^{d\times m}$ for $\range(\mtx{A})$,
with constant distortion.  We can solve
a sequence of least-squares problems
\begin{equation} \label{eqn:iterative-sketch-ls}
\underset{\vct{x} \in \RR^n}{\text{minimize}} \
	\frac{1}{2} \norm{ \mtx{S}_i \mtx{A}(\vct{x} - \vct{x}_{i-1}) }^2
	+ \ip{ \vct{x} - \vct{x}_{i-1} }{ \mtx{A}^* (\vct{b} - \mtx{A}\vct{x}_{i-1}) }.
\end{equation}
The solution $\vct{x}_i$ to this subproblem is fed into the next subproblem.
Without the sketch, each subproblem is equivalent to solving~\eqref{eqn:overdet-ls2}
with $\vct{b}$ replaced by the residual $\vct{r}_{i-1} = \vct{b} - \mtx{A} \vct{x}_{i-1}$.
The sketch $\mtx{S}_i$ preserves the geometry while reducing the problem size.
A typical theoretical error bound would be
$$
\norm{ \mtx{A}\vct{x}_j - \vct{b} } \leq (1 + \eps) \norm{\mtx{A}\vct{x}_{\star} - \vct{b}}
\quad\text{when $j \sim \log(1/\eps)$ and $d \sim n \log n$.}
$$
See~\cite{PW16:Iterative-Hessian} for related results.

In each iteration, the iterative sketching approach requires us to form $\mtx{S}_i \mtx{A}$
at a cost of $O(mn \log d)$.  We compute $\mtx{A}^* \vct{r}_{i-1}$ at a cost of $O(mn)$.
Although it is unnecessary to solve each subproblem accurately, we cannot obtain reliable
behavior without using a dense method at a cost of $O(dn^2)$ per iteration.  With
the theoretical parameter choices, the total arithmetic is
$O((mn + n^3) \log(n) \log(1/\eps))$ to achieve relative error $\eps$.

The interesting parameter regime is $\log n \ll n \ll m / \log n$, but we can now
allow $\eps$ to be tiny.  In this setting, iterative sketching costs slightly
more than the sketch-and-solve paradigm to achieve constant relative error,
while it is faster than the classical approach.  At the same time, it can produce
errors as small as traditional least-squares algorithms.  A shortcoming is that
this method requires repeated sketches of the matrix $\mtx{A}$.

For overdetermined least-squares, we can short-circuit the iterative sketching approach.
In this setting, we can sketch the input matrix just once and factorize it.
We can use the same factorized sketch in each iteration to solve the subproblems faster.
For problems more general than least-squares, it may be necessary to extract
a fresh sketch at each iteration, as we have done here.

\begin{remark}[History]
Iterative sketching can be viewed as an extension
of stochastic approximation methods from optimization,
for example stochastic gradient descent \cite{Bot10:Large-Scale-Machine}.
In the context of randomized NLA, these algorithms
first appeared in the guise of the randomized Kaczmarz
iteration~\cite{SV09:Randomized-Kaczmarz};
see Section~\ref{sec:rk}.
\citeasnoun{GR15:Randomized-Iterative} reinterpreted
randomized Kaczmarz as an iterative sketching method
and developed generalizations.
\citeasnoun{PW16:Iterative-Hessian} proposed a
similar method for solving overdetermined
least-squares problems with constraints;
they observed that better numerical performance
is obtained by sketching~\eqref{eqn:overdet-ls2}
instead of~\eqref{eqn:overdet-ls}.
\end{remark}

\subsection{Sketch and precondition}
\label{sec:sketchandprecond}

The sketch-and-precondition paradigm uses random embedding to find
a proxy for the input matrix.  We can use this proxy to precondition
a classical iterative algorithm so it converges in a minimal number
of iterations.

Let $\mtx{S} \in \F^{d \times m}$ be a subspace embedding for $\range(\mtx{A})$
with constant distortion.  Compress
the input matrix $\mtx{A}$, and then compute a (pivoted) QR factorization:
$$
\mtx{Y} = \mtx{SA}
\quad\text{and}\quad
\mtx{Y} = \mtx{QR}.
$$
Since $\mtx{S}$ preserves the range of $\mtx{A}$ when $d \sim n \log n$,
we anticipate that $\mtx{Y}^* \mtx{Y} \approx \mtx{A}^*\mtx{A}$.
As a consequence, $\mtx{A}\mtx{R}^\pinv$ should be
close to an isometry.
Thus, we can pass to the preconditioned problem
\begin{equation} \label{eqn:sketch-precond-ls}
\underset{\vct{x} \in \RR^n}{\text{minimize}} \
	\frac{1}{2} \norm{ (\mtx{A}\mtx{R}^\pinv)(\mtx{R}\vct{x}) - \vct{b} }^2.
\end{equation}
Construct an initial solution $\vct{x}_0 \in \F^n$ using the sketch-and-solve paradigm
with the embedding $\mtx{S}$.  From this starting point, we solve~\eqref{eqn:sketch-precond-ls}
using preconditioned LQSR.  The $j$th iterate satisfies
$$
\norm{ \mtx{A} \vct{x}_j - \vct{b} } \leq (1 + \eps) \norm{ \mtx{A} \vct{x}_{\star} - \vct{b} }
\quad\text{when $j \sim \log(1/\eps)$.} %
$$
This statement is a reinterpretation of the theory in~\citeasnoun{RT08:Fast-Randomized}.

The cost of sketching the input matrix and performing
the QR decomposition is $O(mn \log d + dn^2)$.
Afterwards, we pay $O(mn)$ for each iteration
of PCG.  With the theoretical parameter settings, the
total cost is $O(mn \log(nB/\eps) + n^3 \log n)$
operations.

Once again, the interesting regime is $\log n \ll n \ll m / \log n$,
and the value of $\eps$ can be very small.
For overdetermined least-squares,
this approach is faster than both the sketch-and-solve
paradigm and the iterative sketching paradigm.
The sketch-and-precondition approach leads to errors
that are comparable with classical linear algebra algorithms,
but it may be a factor of $n / \log(n)$ faster.  On the
other hand, it requires repeated applications of the
matrix $\mtx{A}$.

\begin{remark}[History]
The randomized preconditioning idea was proposed by~\citeasnoun{RT08:Fast-Randomized}.
\citeasnoun{2010_avron_BLENDENPIK} demonstrate that least-squares algorithms
based on randomized preconditioning can beat the highly engineered
software in LAPACK.  The same method drives the algorithms
in~\citeasnoun{MSM14:LSRN-Parallel}.  \citeasnoun{Avr18:Randomized-Riemannian}
contains a recent summary of existing randomized preconditioning methods.
\end{remark}

\subsection{Comparisons}

If we seek a high-precision solution to
a dense, unstructured, overdetermined least-squares problem,
randomized preconditioning leads to the most efficient
existing algorithm.
For the same problem, if we can only view the input matrix once,
then the sketch-and-solve paradigm still allows us to obtain
a low-accuracy solution.
Although iterative sketching is less efficient
than its competitors in this setting,
it remains useful for solving
constrained least-squares problems,
and it has further connections with optimization.

\subsection{Summary}

From the perspective of a numerical analyst,
randomized preconditioning and iterative sketching
should be the preferred methods for designing
sketching algorithms
because they allow for high precision.
The sketch-and-solve approach is appropriate
only when data access is severely constrained.

In spite of this fact, a majority of the literature
on randomized NLA develops algorithms based on the
sketch-and-solve paradigm.  There are far
fewer works on randomized preconditioning
or iterative sketching.  This discrepancy
points to an opportunity for further research.

\section{The randomized rangefinder}
\label{sec:random-rangefinder}

A core challenge in linear algebra is to find a subspace
that captures a lot of the action of a matrix.
We call this the \emph{rangefinder} problem.
As motivation for considering this problem, we will
use the rangefinder %
to derive the randomized SVD algorithm.
Then we will introduce several randomized algorithms for computing
the rangefinder primitive, along with theoretical guarantees for these methods.
These algorithms all make use of random embeddings, but their
performance depends on more subtle features than the basic
subspace embedding property.

In Section~\ref{sec:error-est}, we will complement the
algorithmic discussion with details about error estimation
and adaptivity for the rangefinder primitive.
In Sections~\ref{sec:natural}--\ref{sec:full},
we will see that the subspace produced by the rangefinder
can be used as a primitive for
other linear algebra computations. %

Most of the material in this section is adapted
from our papers~\cite{HMT11:Finding-Structure}
and~\cite{HMST11:Algorithm-Principal}.
We have also incorporated more recent perspectives.

\subsection{The rangefinder: Problem statement}
\label{sec:rrf-overview}

Let $\mtx{B} \in \F^{m \times n}$ be an input matrix,
and let $\ell \leq \min\{m, n\}$ be the subspace dimension.
The goal of the rangefinder problem is to produce an orthonormal
matrix $\mtx{Q} \in \F^{m \times \ell}$ whose range
aligns with the dominant left singular vectors of $\mtx{B}$.

To measure the quality of $\mtx{Q}$,
we use the spectral norm error
\begin{equation} \label{eqn:rrf-error}
\norm{ \mtx{B} - \mtx{QQ}^* \mtx{B} }
	= \norm{ (\Id - \mtx{QQ}^*) \mtx{B} }.
\end{equation}
If the error measure~\eqref{eqn:rrf-error} is small,
then the rank-$\ell$ matrix $\widehat{\mtx{B}} = \mtx{QQ}^* \mtx{B}$
can serve as a proxy for $\mtx{B}$.
See Section~\ref{sec:rsvd} for an important application.

\begin{algorithm}[t]
\begin{algorithmic}[1]
\caption{\textit{The randomized rangefinder.} \newline
Implements the procedure from Section~\ref{sec:rrf-pseudo}.}
\label{alg:random-rangefinder}

\Require	Input matrix $\mtx{B} \in \F^{m \times n}$, subspace dimension $\ell$
\Ensure		Orthonormal matrix $\mtx{Q} \in \F^{m \times \ell}$
\Statex

\Function{RandomRangefinder}{$\mtx{B}$, $\ell$}

\State	Draw a random matrix $\mtx{\Omega} \in \F^{n \times \ell}$
\State	Form $\mtx{Y} = \mtx{B\Omega}$
\State	Compute $[\mtx{Q}, \sim] = \texttt{qr\_econ}(\mtx{Y})$

\EndFunction
\end{algorithmic}
\end{algorithm}

\subsubsection{The randomized rangefinder: A pseudoalgorithm}
\label{sec:rrf-pseudo}

Using randomized methods, it is remarkably easy to find
an initial solution to the rangefinder problem.
We simply multiply the target matrix by a random embedding
and then orthogonalize the resulting matrix.

More rigorously: consider a target matrix $\mtx{B} \in \F^{m \times n}$
and a subspace dimension $\ell$.  We draw a random test matrix
$\mtx{\Omega} \in \F^{n \times \ell}$, where $\mtx{\Omega}^*$ is a mixing
random embedding.  We form the product $\mtx{Y} = \mtx{B}\mtx{\Omega} \in \F^{m \times \ell}$.
Then we compute an orthonormal basis $\mtx{Q} \in \F^{m \times \ell}$
for the range of $\mtx{Y}$ using a QR factorization method.
See Algorithm~\ref{alg:random-rangefinder} for pseudocode.

In a general setting, the arithmetic cost of this procedure is dominated by
$\bigO(mn\ell)$ operations for the matrix--matrix
multiplication.  The QR factorization of $\mtx{Y}$ requires $\bigO(m\ell^2)$
arithmetic, and we also need to simulate the $n \times \ell$
random matrix $\mtx{\Omega}$.  Economies are possible when
either $\mtx{B}$ or $\mtx{\Omega}$ admits fast multiplication.

\subsubsection{Practicalities}

To implement Algorithm~\ref{alg:random-rangefinder} effectively,
several computational aspects require attention.

\begin{itemize} \setlength{\itemsep}{1mm}
\item	\textbf{How do we choose the subspace dimension?}  If we have advance knowledge of the ``effective rank''
$r$ of the target matrix $\mtx{B}$, the theory (Theorem~\ref{thm:rrf-gauss} and Corollary~\ref{cor:rrf-power-gauss})
indicates that we can select the subspace dimension $\ell$ to be just slightly larger, say,
$\ell = r + p$ where $p = 5$ or $p = 10$.  The value $p$ is called the \emph{oversampling} parameter.
Alternatively, we can use an error estimator to decide
when the computed subspace $\mtx{Q}$ is sufficiently accurate; see Section~\ref{sec:rrf-error}.

\item	\textbf{What kind of random matrix?}
We can use most types of mixing random embeddings to implement
Algorithm~\ref{alg:random-rangefinder}.  We highly recommend Gaussians and random partial isometries (Section~\ref{sec:gauss}).
Sparse maps, SRTTs, and tensor random embeddings (Section~\ref{sec:dimension-reduction}) also work very well.
In practice, all these approaches exhibit similar behavior;
see Section~\ref{sec:rrf-dimred} for more discussion.
We present analysis only for Gaussian dimension reduction
because it is both simple and precise.

\item	\textbf{Matrix multiplication.}  The randomized rangefinder is powerful
because most of the computation takes place in the matrix multiplication step,
which is a highly optimized primitive on most computer systems.
When the target matrix admits fast matrix--vector multiplications
(e.g., due to sparsity), the rangefinder can exploit this property.

\item	\textbf{Powering.}  As we will discuss in Sections~\ref{sec:rrf-subspace}
and~\ref{sec:rrf-krylov}, it is often beneficial to enhance Algorithm~\ref{alg:random-rangefinder}
by means of powering or Krylov subspace techniques.

\item	\textbf{Orthogonalization.}  The columns of the matrix $\mtx{Y}$ tend to be
strongly aligned, so it is important to use a numerically stable
orthogonalization procedure~\cite[Chap.~5]{GVL13:Matrix-Computations-4ed}, such
as Householder reflectors, double Gram--Schmidt, or rank-revealing QR.
The rangefinder algorithm is also a natural place to invoke a
TSQR algorithm~\cite{DGHL12:Communication-Optimal-Parallel}.
\end{itemize}

\noindent
See~\citeasnoun{HMT11:Finding-Structure} for much more information.

\subsection{The randomized singular value decomposition (RSVD)}
\label{sec:rsvd}

Before we continue with our discussion of the rangefinder,
let us summarize one of the key applications: the randomized
SVD algorithm.

Low-rank approximation problems often arise %
when a user seeks an incomplete matrix factorization that exposes structure,
such as a truncated eigenvalue decomposition or a partial QR factorization.
The randomized rangefinder, described in Section \ref{sec:rrf-pseudo},
can be used to perform the heavy lifting in these computations.
Afterwards, we perform some light post-processing to reach the desired
factorization.

To illustrate how this works,
suppose that %
we want to compute an approximate rank-$\ell$ truncated singular value decomposition
of the input matrix $\mtx{B} \in \F^{m \times n}$.  That is,
$$
\mtx{B} \approx \mtx{U}\mtx{\Sigma}\mtx{V}^{*},
$$
where $\mtx{U} \in \F^{m \times \ell}$ and $\mtx{V} \in \F^{n \times \ell}$ are orthonormal matrices %
and $\mtx{\Sigma} \in \F^{\ell \times \ell}$ is a diagonal matrix whose diagonal entries
approximate the largest singular values of $\mtx{B}$.

Choose a target rank $\ell$, and suppose that $\mtx{Q} \in \F^{m \times \ell}$
is a computed solution to the rangefinder problem.
The rangefinder furnishes an approximate rank-$\ell$ factorization
of the input matrix: $\mtx{B} \approx \mtx{Q}\bigl(\mtx{Q}^{*}\mtx{A}\bigr)$.
To convert this representation into a truncated SVD,
we just compute an economy-size SVD of the matrix $\mtx{C} := \mtx{Q}^{*}\mtx{A} \in \F^{\ell\times n}$
and consolidate the factors.

In symbols, once $\mtx{Q}$ is available, the computation proceeds as follows:
\begin{align*}
\mtx{B} \approx&\ \mtx{Q}\mtx{Q}^{*}\mtx{B}  &&\{\mbox{matrix--matrix multiplication: }\mtx{C} = \mtx{Q}^{*}\mtx{B}\} \\
=&\ \mtx{Q}\mtx{C}  &&\{\mbox{economy-size SVD: }\mtx{C} = \widehat{\mtx{U}}\mtx{\Sigma}\mtx{V}^{*}\} \\
=&\ \mtx{Q}\widehat{\mtx{U}}\mtx{\Sigma}\mtx{V}^{*}  &&\{\mbox{matrix--matrix multiplication: }\mtx{U} = \mtx{Q}\widehat{\mtx{U}}\} \\
=&\ \mtx{U}\mtx{\Sigma}\mtx{V}^{*}.
\end{align*}
After the rangefinder step,
the remaining computations are all exact (modulo floating-point arithmetic errors).  Therefore,
$$
\|\mtx{B} - \mtx{U}\mtx{\Sigma}\mtx{V}^{*}\| = \|\mtx{B} - \mtx{Q}\mtx{Q}^{*}\mtx{B}\|.
$$
In words, the accuracy of the approximate SVD is determined entirely
by the error in the rangefinder computation!

Empirically, the smallest singular values and singular vectors
of the approximate SVD contribute to the accuracy of the approximation,
but they are not good estimates for the true singular values
and vectors of the matrix.
Therefore, it can be valuable to truncate the rank by zeroing out the
smallest computed singular values.  We omit the details.
See \cite{HMT11:Finding-Structure,Gu15:Subspace-Iteration} and \cite{TYUC19:Streaming-Low-Rank}
for further discussion.

Algorithm~\ref{alg:rsvd} contains pseudocode for the randomized SVD.
In a general setting, the dominant cost after the rangefinder step
is the matrix--matrix multiply, which requires $\bigO(mn\ell)$
operations.  The storage requirements are $\bigO((m+n)\ell)$
numbers.

\begin{algorithm}[t]
\begin{algorithmic}[1]
\caption{\textit{Randomized singular value decomposition (RSVD).} \newline
Implements the procedure from Section~\ref{sec:rsvd}.}
\label{alg:rsvd}

\Require	Input matrix $\mtx{B} \in \F^{m \times n}$, factorization rank $\ell$
\Ensure		Orthonormal matrices $\mtx{U} \in \F^{m \times \ell}$, $\mtx{V} \in \F^{n \times \ell}$ and a diagonal matrix $\mtx{\Sigma} \in \F^{\ell \times \ell}$
such that $\mtx{B} \approx \mtx{U}\mtx{\Sigma}\mtx{V}^{*}$.
\Statex

\Function{RSVD}{$\mtx{B}$, $\ell$}

\State	$\mtx{Q} = \textsc{RandomRangefinder}(\mtx{B}, \ell)$
	\Comment Algorithm \ref{alg:random-rangefinder}
\State	$\mtx{C} = \mtx{Q}^{*}\mtx{B}$
\State	$[\widehat{\mtx{U}}, \mtx{\Sigma}, \mtx{V}] = \texttt{svd\_econ}(\mtx{C})$
\State $\mtx{U} = \mtx{Q}\widehat{\mtx{U}}$
\State	[optional] Truncate the factorization to rank $r \leq \ell$

\EndFunction
\end{algorithmic}
\end{algorithm}

The rangefinder primitive allows us to perform other matrix computations as well.
For example, in Section~\ref{sec:natural}, we explain how to use the rangefinder
to construct %
matrix factorizations where a subset of the rows/columns are picked to form a basis for the row/column spaces.
This approach gives far better results than the more obvious randomized algorithms based
on coordinate sampling.

\subsection{The rangefinder and Schur complements}

Why does the randomized rangefinder work?
We will demonstrate that the procedure has its most natural
expression in the language of Schur complements.
This point is implicit in the
analysis in \citeasnoun{HMT11:Finding-Structure},
and it occasionally appears more
overtly in the literature (e.g., in~\citeasnoun{Git13:Topics-Randomized} and~\citeasnoun{TYUC17:Fixed-Rank-Approximation}).
Nevertheless, this connection has not
been explored in a systematic way.

\begin{proposition}[Rangefinder: Schur complements] \label{prop:rrf-schur}
Let $\mtx{Y} = \mtx{BX}$ for an arbitrary test matrix $\mtx{X} \in \F^{n \times \ell}$,
and let $\mtx{P}_{\mtx{Y}}$ be the orthogonal projector onto the range of $\mtx{Y}$.
Define the approximation error as
$$
\mtx{E} := \mtx{E}(\mtx{B}, \mtx{X}) := (\Id - \mtx{P}_{\mtx{Y}}) \mtx{B}.
$$
Then the squared error can be written as a Schur complement~\eqref{eqn:schur-complement}:
$$
\abs{\mtx{E}}^2 := \mtx{E}^* \mtx{E} %
	= (\mtx{B}^* \mtx{B}) / \mtx{X}.
$$
We emphasize that $\abs{\mtx{E}}^2$ is a psd matrix, not a scalar.
\end{proposition}

\begin{proof}
This result follows from a short calculation.
We can write the orthogonal projector $\mtx{P}_{\mtx{Y}}$ in the form
$$
\mtx{P}_{\mtx{Y}} = (\mtx{BX})((\mtx{B}\mtx{X})^* (\mtx{BX}))^\pinv (\mtx{BX})^*.
$$
Abbreviating $\mtx{A} = \mtx{B}^* \mtx{B}$, we have
$$
\mtx{E}^* \mtx{E} = \mtx{B}^* (\Id - \mtx{P}_{\mtx{Y}}) \mtx{B}
	= \mtx{A} - (\mtx{A} \mtx{X})(\mtx{X}^* \mtx{A} \mtx{X})^\pinv (\mtx{AX})^*.
$$
This is precisely the definition~\eqref{eqn:schur-complement} of the Schur complement $\mtx{A} / \mtx{X}$.
\end{proof}

Proposition~\ref{prop:rrf-schur} gives us access to the deep theory
of Schur complements~\cite{Zha05:Schur-Complement}.  In particular,
we have a beautiful monotonicity property that follows instantly
from~\citeasnoun[Thm.~5.3]{And05:Schur-Complements}.

\begin{corollary}[Monotonicity] \label{cor:rrf-monotone}
Suppose that $\mtx{B}^* \mtx{B} \psdle \mtx{C}^* \mtx{C}$
with respect to the semidefinite order $\psdle$.
For each fixed test matrix $\mtx{X}$,
$$
\begin{aligned}
\abs{ \mtx{E}(\mtx{B}, \mtx{X}) }^2
	&= (\mtx{B}^* \mtx{B}) / \mtx{X} \\
	&\psdle (\mtx{C}^* \mtx{C}) / \mtx{X}
	= \abs{ \mtx{E}(\mtx{C}, \mtx{X}) }^2
\end{aligned}
$$
\end{corollary}

In particular, the error increases if we increase any singular value
of $\mtx{B}$ while retaining the same right singular vectors;
the error decreases if we decrease any singular value of $\mtx{B}$.
The left singular vectors do not play a role here.
This observation allows us to identify which target matrices are
hardest to approximate.

\begin{example}[Extremals] \label{ex:extremal}
Consider the parameterized matrix $\mtx{B}(\vct{\sigma}) = \mtx{U} \diag(\vct{\sigma}) \mtx{V}^* \in \F^{n \times n}$,
where $\mtx{U}, \mtx{V}$ are unitary.
Suppose that we fix $\sigma_1$ and $\sigma_{k+1}$.
For each test matrix $\mtx{X}$, the error $\abs{\mtx{E}(\mtx{B}(\vct{\sigma}), \mtx{X})}^2$
is maximal in the semidefinite order when
$$
\vct{\sigma} = (\underbrace{\sigma_1, \dots, \sigma_1}_k, \underbrace{\sigma_{k+1}, \dots, \sigma_{k+1}}_{n-k}).
$$
\end{example}

It has long been appreciated that Example~\ref{ex:extremal}
is the hardest matrix to approximate; cf.~\citeasnoun[Sec.~5, Ex.~4, 5]{2006_martinsson_random1_orig}.
The justification of this insight is new.

\subsection{{A priori} error bounds}

Proposition~\ref{prop:rrf-schur} shows that the error in the rangefinder
procedure can be written as a Schur complement.  Incredibly, the Schur
complement of a psd matrix with respect to a random subspace tends to be
quite small.  In this section, we summarize a theoretical analysis,
due to~\citeasnoun{HMT11:Finding-Structure}, that explains why this
claim is true.

\subsubsection{Master error bound}

First, we present a deterministic upper bound on the error incurred
by the rangefinder procedure.  This requires some notation.

Without loss of generality, we may assume that $m = n$ by extending $\mtx{B}$ with zeros.
For any $k \leq \ell$, construct a partitioned SVD of the target matrix:
$$
\mtx{B} = \mtx{U} %
	\begin{bmatrix} \mtx{\Sigma}_1 & \\ & \mtx{\Sigma}_{2} \end{bmatrix} %
	\begin{bmatrix} \mtx{V}_1 & \mtx{V}_{2} \end{bmatrix}^*
	\quad\text{with $\mtx{\Sigma}_1 \in \RR^{k \times k}$ and
	$\mtx{V}_1 \in \F^{n \times k}$.}
$$
The factors $\mtx{U}, \mtx{\Sigma}, \mtx{V}$ are all square matrices.
As usual, the entries of $\mtx{\Sigma} = \diag(\sigma_1, \sigma_2, \dots)$
are arranged in weakly decreasing order.  So $\mtx{\Sigma}_1$ lists the first $k$ singular
values, and $\mtx{\Sigma}_2$ lists the remaining $n - k$ singular values.
The matrix $\mtx{V}_1$ contains the first $k$ right singular vectors;
the matrix $\mtx{V}_2$ contains the remaining $n - k$ right singular vectors.

For any test matrix $\mtx{X} \in \F^{n \times \ell}$, define
$$
\mtx{X}_1 = \mtx{V}_1^* \mtx{X}
\quad\text{and}\quad
\mtx{X}_{2} = \mtx{V}_{2}^* \mtx{X}.
$$
These matrices reflect the alignment of the test matrix $\mtx{X}$
with the matrix $\mtx{V}_1$ of dominant right singular
vectors of $\mtx{B}$.  We assume $\mtx{X}_1$ has full row rank.

With this notation, we can present a strong deterministic bound
on the error in Algorithm~\ref{alg:random-rangefinder}.

\begin{theorem}[Rangefinder: Deterministic bound] \label{thm:rrf-deterministic}
Let $\mtx{Y} = \mtx{BX}$ be the sample matrix obtained by testing
$\mtx{B}$ with $\mtx{X}$.  With the notation and assumptions above,
for all $k \leq \ell$,
\begin{equation} \label{eqn:rrf-determ}
\norm{ (\Id - \mtx{P}_{\mtx{Y}}) \mtx{B} }
	\leq \sigma_{k+1} + \norm{ \mtx{\Sigma}_{2} \mtx{X}_{2} \mtx{X}_1^\pinv }.
\end{equation}
A related inequality holds for every
quadratic unitarily invariant norm.
\end{theorem}

The result and its proof are drawn
from~\citeasnoun[Thm.~9.1]{HMT11:Finding-Structure}.
The same bound was obtained independently
in~\citeasnoun{BMD09:Improved-Approximation}
by means of a different technique.

Theorem~\ref{thm:rrf-deterministic} leads to sharp bounds on the performance of the
rangefinder in most situations of practical interest.
Let us present a sketch of the argument.
Our approach can be modified to obtain
matching lower and upper bounds,
but they do not give any additional
insight into the performance.

\begin{proof}
In view of Proposition~\ref{prop:rrf-schur},
we want to bound the spectral norm of
$$
\abs{(\Id - \mtx{P}_{\mtx{Y}}) \mtx{B}}^2
	= (\mtx{B}^*\mtx{B}) / \mtx{X}.
$$
First, change coordinates so that the right singular vectors
of $\mtx{B}$ are the identity: $\mtx{V} = \Id$.
In particular, $\mtx{B}^* \mtx{B} = \mtx{\Sigma}^2$ is diagonal.
By homogeneity of~\eqref{eqn:rrf-determ}, we may assume that $\sigma_1 = 1$.
Next, using Corollary~\ref{cor:rrf-monotone},
we may also assume that $\sigma_1 = \dots = \sigma_k = 1$,
which leads to the worst-case error.  %
Thus, it suffices to bound the spectral norm of the psd matrix
$$
\mtx{S} := \begin{bmatrix} \Id_k & \mtx{0} \\ \mtx{0} & \mtx{\Sigma}_{2}^2\end{bmatrix} / \mtx{X}.
$$

To accomplish this task, we may as well take the Schur complement of the diagonal matrix with respect
to a test matrix that has a smaller range than $\mtx{X}$; see~\citeasnoun[Thm.~5.9(iv)]{And05:Schur-Complements}.
Define $\widetilde{\mtx{X}} := \mtx{X} \mtx{X}_{1}^\pinv = [\Id_k; \mtx{X}_{2}\mtx{X}_1^\pinv]$.
Since $\range(\widetilde{\mtx{X}}) \subseteq \range(\mtx{X})$,
$$
\mtx{S} \psdle \begin{bmatrix} \Id_k & \mtx{0} \\ \mtx{0} & \mtx{\Sigma}_{2}^2\end{bmatrix}  / \widetilde{\mtx{X}}
	=: \widetilde{\mtx{S}}.
$$
Using the definition of the Schur complement, we can write out
the matrix on the right-hand side in block form.  With the abbreviation
$\mtx{F} := \mtx{\Sigma}_{2} \mtx{X}_{2} \mtx{X}_1^\pinv$,
$$
\widetilde{\mtx{S}}
	= \begin{bmatrix} \Id - (\Id + \mtx{FF}^*)^{-1} & \star \\ \star & \mtx{\Sigma}_{2}^2 - \mtx{F}(\Id + \mtx{FF}^*)^{-1} \mtx{F}^* \end{bmatrix}.
$$
The $\star$ symbol denotes matrices that do not play a role in the rest
of the argument.  We can bound the block matrix above in the psd order:
$$
\widetilde{\mtx{S}} \psdle \begin{bmatrix} \mtx{FF}^* & \star \\ \star & \mtx{\Sigma}_{2}^2 \end{bmatrix}
$$
The inequality for the top-left block holds because $1 - (1+a)^{-1} \leq a$ for all numbers $a \geq 0$.
Last, take the spectral norm:
$$
\norm{ \mtx{S} }
	\leq \norm{ \widetilde{\mtx{S}} }
	\leq \lnorm{ \begin{bmatrix} \mtx{FF}^* & \star \\ \star & \mtx{\Sigma}_{2}^2 \end{bmatrix} }
	\leq \norm{\mtx{FF}^*} + \norm{\mtx{\Sigma}_{2}^2}.
$$
This bound is stronger than the stated result.
\end{proof}

\subsubsection{Gaussian test matrices}

We can obtain precise results for the behavior
of the randomized rangefinder when the test
matrix is (real) standard normal.
Let us present a variant of~\citeasnoun[Thm.~10.1]{HMT11:Finding-Structure}.

\begin{theorem}[Rangefinder: Gaussian analysis] \label{thm:rrf-gauss}
Fix a matrix $\mtx{B} \in \RR^{m \times n}$ with singular values
$\sigma_1 \geq \sigma_2 \geq \dots$.
Draw a standard normal test matrix $\mtx{\Omega} \in \F^{n \times \ell}$,
and construct the sample matrix $\mtx{Y} = \mtx{B\Omega}$.
Choose $k < \ell - 1$, and introduce the random variable
$$
Z = \norm{ \mtx{\Gamma}^\pinv }
\quad\text{where $\mtx{\Gamma} \in \RR^{k \times \ell}$ is standard normal.}
$$
Then the expected error in the random rangefinder satisfies
$$
\Expect \norm{(\Id - \mtx{P}_{\mtx{Y}}) \mtx{B}}
	\leq \left( 1 + \sqrt{\frac{k}{\ell-k-1}} \right) \sigma_{k+1}
		+ (\Expect Z) \left(\sum\nolimits_{j > k} \sigma_{j}^2\right)^{1/2}.
$$
\end{theorem}

In other words, the randomized rangefinder computes
an $\ell$-dimensional subspace that captures as much of the
action of the matrix $\mtx{B}$ as the best $k$-dimensional
subspace.  If we think about $k$ as fixed and $\ell$ as the variable,
we only need to choose $\ell$ slightly larger than $k$ to enjoy this outcome.

The error is comparable with $\sigma_{k+1}$, the error in the best rank-$k$
approximation, provided that the tail singular values
$\sigma_{j}$ for $j > k$ have small $\ell_2$ norm.
This situation occurs, for example, when $\mtx{B}$
has a rapidly decaying spectrum.

\begin{proof}
Here is a sketch of the argument.
Since the test matrix $\mtx{\Omega}$ is standard normal, the matrices
$\mtx{\Omega}_{1} := \mtx{V}_1^* \mtx{\Omega}$ and $\mtx{\Omega}_{2} := \mtx{V}_2^*\mtx{\Omega}$
are independent standard normal matrices because $\mtx{V}_1$ and $\mtx{V}_2$ are orthonormal
and mutually orthogonal.
Using Chevet's theorem~\cite[Prop.~10.1]{HMT11:Finding-Structure},
$$
\begin{aligned}
\Expect \norm{ \mtx{\Sigma}_{2} \mtx{\Omega}_{2} \mtx{\Omega}_1^\pinv }
	&= \Expect_{\mtx{\Omega}_{1}} \Expect_{\mtx{\Omega}_{2}} \big[ \norm{ \mtx{\Sigma}_{2} \mtx{\Omega}_{2} \mtx{\Omega}_1^\pinv } \big] \\
	&\leq \Expect \big[ \norm{\mtx{\Sigma}_{2}} \fnorm{\mtx{\Omega}_1^\pinv} + \fnorm{\mtx{\Sigma}_2} \norm{\mtx{\Omega}_1^\pinv} \big] \\
	&\leq \sqrt{\frac{k}{\ell-k-1}} \norm{ \mtx{\Sigma}_{2} } + (\Expect Z) \fnorm{\mtx{\Sigma}_{2}}.
\end{aligned}
$$
The last inequality involves a well-known estimate for the trace of an
inverted Wishart matrix~\cite[Prop.~10.2]{HMT11:Finding-Structure}.
\end{proof}

To make use of the result, we simply insert estimates
for the expectation of the random variable $Z$.  For instance,
$$
\begin{aligned}
\Expect Z \leq \frac{\econst\sqrt{\ell}}{\ell - k}
\quad\text{when $2 \leq k < \ell$}
\quad\text{and}\quad
\Expect Z \approx \frac{1}{\sqrt{\ell} - \sqrt{k}}
\quad\text{for $k \ll \ell$.}
\end{aligned}
$$
These estimates lead to very accurate performance bounds
across a wide selection of matrices and parameters.

The rangefinder also operates in the regime
$k \in \{ \ell - 1, \ell \}$.  In this case,
it attains significantly larger errors.
A heuristic is
$$
\norm{(\Id - \mtx{P}_{\mtx{Y}}) \mtx{B}}
	\lessapprox (1 + k) \sigma_{k+1} + \sqrt{k} \left( \sum\nolimits_{j > k} \sigma_{j}^2 \right)^{1/2}.
$$
This point follows because $\fnorm{\mtx{\Omega}_1^\pinv} \approx k$ and $\norm{\mtx{\Omega}_1^\pinv} \approx \sqrt{k}$
when $k \approx \ell$.

For relevant results about Gaussian random matrices,
we refer the reader
to~\cite{Ede89:Eigenvalues-Condition,DS01:Local-Operator,CD05:Condition-Numbers,BS10:Spectral-Analysis}
and~\cite{HMT11:Finding-Structure}.

\subsection{Other test matrices}
\label{sec:rrf-dimred}

In many cases, it is too expensive to use Gaussian test matrices to implement
Algorithm~\ref{alg:random-rangefinder}.  Instead, we may prefer to apply
(the adjoint of) one of the structured
random embeddings discussed in Section~\ref{sec:dimension-reduction}.

\subsubsection{Random embeddings for the rangefinder}

Good alternatives to Gaussian test matrices include the following.

\begin{itemize} \setlength{\itemsep}{1mm}

\item	\textbf{Sparse maps.}  Sparse dimension reduction maps work well
in the rangefinder procedure, even if the input matrix is sparse.
The primary shortcoming is the need to use sparse data structures and arithmetic.
See Section~\ref{sec:sparse-map}.

\item	\textbf{SRTTs.}  In practice, subsampled randomized trigonometric transforms perform slightly
better than Gaussian maps.  The main difficulty is that the implementation requires  fast trigonometric
transforms.  See Section~\ref{sec:srtt}.

\item	\textbf{Tensor product maps.}  Emerging evidence suggests that tensor
product random projections are also effective in practice.  See Section~\ref{sec:trp}.
\end{itemize}

Some authors have proposed using random coordinate sampling
to solve the rangefinder problem.  We cannot recommend
this approach unless it is impossible to use one
of the random embeddings described above.
See Section~\ref{sec:coord-embed} for a discussion of
random coordinate sampling and situations where it
may be appropriate.

\subsubsection{Universality}

In practice, if the test matrix is a mixing random embedding, the error in the rangefinder
is somewhat insensitive to the precise distribution of the test matrix.
In this case, we can use the Gaussian theory to obtain good heuristics about
the performance of other types of embeddings.
Regardless, we always recommend using \textit{a posteriori} error estimates
to validate the performance of the rangefinder method, as well as downstream
matrix approximations; see Section~\ref{sec:error-est}.

\subsubsection{Aside: Subspace embeddings}

If we merely assume that the test matrix is a subspace embedding,
then we can still perform a theoretical analysis of the rangefinder algorithm.
Here is a typical result, adapted from~\citeasnoun[Thm.~11.2]{HMT11:Finding-Structure}.

\begin{theorem}[Rangefinder: SRTT] \label{thm:rrf-srtt}
Fix a matrix $\mtx{B} \in \F^{m \times n}$ with singular values $\sigma_1 \geq \sigma_2 \geq \dots$.
Choose a natural number $k$, and draw an SRTT embedding matrix $\mtx{\Omega} \in \F^{n \times \ell}$ where
$$
\ell \geq 8 (k + 8 \log(kn)) \log k.
$$
Construct the sample matrix $\mtx{Y} = \mtx{B\Omega}$.  Then
$$
\norm{ (\Id- \mtx{P}_{\mtx{Y}}) \mtx{B} } \leq (1 + 3 \sqrt{n / \ell}) \cdot \sigma_{k+1},
$$
with failure probability at most $\bigO(k^{-1})$.
\end{theorem}

\begin{proof}
(Sketch) %
By Theorem~\ref{thm:rrf-deterministic},
$$
\norm{ (\Id- \mtx{P}_{\mtx{Y}}) \mtx{B} }
	\leq \left[ 1 + \norm{ \mtx{\Omega}_{2} } \norm{ \mtx{\Omega}_1^\pinv } \right] \sigma_{k+1}.
$$
With the specified choice of $\ell$, the test matrix $\mtx{\Omega}$ is likely to be
an oblivious subspace embedding of the $k$-dimensional subspace $\mtx{V}_1$
with distortion $1/3$.  Thus, the matrix $\mtx{\Omega}_1^{\pinv}$ has spectral norm bounded by $3$.
The matrix $\sqrt{\ell/n} \, \mtx{\Omega}$ is orthonormal, so the spectral norm of $\mtx{\Omega}_{2}$
is bounded by $\sqrt{n/\ell}$.
\end{proof}

The lower bound in the subspace embedding property~\eqref{eqn:subspace-embedding}
is the primary fact about the SRTT used in the proof.
But this is only part of the reason that the rangefinder works.
Accordingly, the outcome of this ``soft'' analysis is qualitatively
weaker than the ``hard'' analysis in Theorem~\ref{thm:rrf-gauss}.
The resulting bound does not explain the actual (excellent) performance
of Algorithm~\ref{alg:random-rangefinder} when implemented with an SRTT.

\subsection{Subspace iteration}
\label{sec:rrf-subspace}

Theorem~\ref{thm:rrf-gauss} shows that the basic randomized rangefinder procedure,
Algorithm~\ref{alg:random-rangefinder}, can be effective for target matrices
$\mtx{B}$ with a rapidly decaying spectrum.  Nevertheless, in many applications,
we encounter matrices that do not meet this criterion.
As in the case of spectral norm estimation (Section~\ref{sec:max-eig}),
we can resolve the problem by powering the matrix.

\subsubsection{Rangefinder with powering}
\label{sec:poweringq}

Let $\mtx{B} \in \F^{m \times n}$ be a fixed input matrix,
and let $q$ be a natural number.  Let $\mtx{\Omega} \in \F^{m \times \ell}$
be a random test matrix.  We form the sample matrix
$$
\mtx{Y} = (\mtx{BB}^*)^q \mtx{\Omega}
$$
by repeated multiplication.  Then we compute an orthobasis $\mtx{Q} \in \F^{m \times \ell}$
for the range of $\mtx{Y}$ using a QR factorization method.

See Algorithm~\ref{alg:power-rangefinder} for pseudocode.  In general,
the arithmetic cost is dominated by the $\bigO(qmn\ell)$
cost of the matrix--matrix multiplications.
Economies are possible when $\mtx{B}$ admits fast multiplication.
Unfortunately, when powering is used, it is not possible to fundamentally accelerate
the computation by using a structured test matrix $\mtx{\Omega}$.

Algorithm~\ref{alg:power-rangefinder} coincides with the classic
subspace iteration algorithm with a random start.
Historically, subspace iteration was regarded as
a method for spectral computations. The block size $\ell$
was often chosen to be quite small, say $\ell = 3$ or $\ell = 4$,
because the intention was simply to resolve singular values with
multiplicity greater than one.

The randomized NLA literature contains several new insights about the
behavior of randomized subspace iteration.
It is now recognized that \emph{iteration is not required}.
In practice, $q = 2$ or $q = 3$ is entirely adequate to
solve the rangefinder problem to fairly high accuracy.
The modern perspective also emphasizes the value of
running subspace iteration with a very large block
size $\ell$ to obtain matrix approximations.

\begin{remark}[History]
\citeasnoun{2009_szlam_power} introduced the idea
of using randomized subspace iteration to obtain
matrix approximations by selecting a large block
size $\ell$ and a small power $q$.
\citeasnoun{HMT11:Finding-Structure} refactored
and simplified the algorithm, and they presented
a complete theoretical justification for the
approach.
Subsequent analysis appears in \citeasnoun{Gu15:Subspace-Iteration}.
\end{remark}

\subsubsection{Analysis}

The analysis of the powered rangefinder is an easy consequence
of the following lemma \cite[Prop.~8.6]{HMT11:Finding-Structure}.

\begin{lemma}[Powering] \label{lem:powering}
Let $\mtx{B} \in \F^{m \times n}$ be a fixed matrix,
and let $\mtx{P} \in \F^{m \times m}$ be an orthogonal projector.
For any number $q \geq 1$,
$$
\norm{ (\Id - \mtx{P}) \mtx{B} }^{2q}
	\leq \norm{ (\Id - \mtx{P}) (\mtx{BB}^*)^q }.
$$
\end{lemma}

\begin{proof}
This bound follows immediately from the Araki--Lieb--Thirring
inequality \citeasnoun[Thm.~IX.2.10]{Bha97:Matrix-Analysis}.
\end{proof}

Theorem~\ref{thm:rrf-gauss} gives us bounds for the right-hand
side of the inequality in Lemma~\ref{lem:powering}
when $\mtx{P}$ is the orthogonal projector onto the
subspace generated by the powered rangefinder.

\begin{corollary}[Powered rangefinder: Gaussian analysis] \label{cor:rrf-power-gauss}
Under the same conditions as Theorem~\ref{thm:rrf-gauss},
let $\mtx{Y} = (\mtx{BB}^*)^q \mtx{\Omega}$
be the sample matrix computed by Algorithm~\ref{alg:power-rangefinder}.
Then
\begin{multline*}
\Expect \norm{ (\Id - \mtx{P}_{\mtx{Y}}) \mtx{B} }
	\leq \left( \Expect \norm{ (\Id - \mtx{P}_{\mtx{Y}}) \mtx{B} }^{2q} \right)^{1/(2q)} \\
	\leq \left[ \left(1 + \sqrt{\frac{k}{\ell-k-1}} \right) \sigma_{k+1}^{2q}
		+ (\Expect Z) \left( \sum\nolimits_{j > k} \sigma_j^{4q} \right)^{1/2} \right]^{1/(2q)}.
\end{multline*}
\end{corollary}

In other words, powering the matrix $\mtx{B}$ drives the error in the rangefinder
to $\sigma_{k+1}$ exponentially fast as the parameter $q$ increases.
In cases where the target matrix has some spectral decay, it suffices
to take $q = 2$ or $q = 3$ to achieve satisfactory results.
For matrices with a flat spectral tail, however,
we may need to set $q \approx \log \min\{m,n\}$
to make the error a constant multiple of $\sigma_{k+1}$.

\begin{algorithm}[t]
\begin{algorithmic}[1]
\caption{\textit{The powered randomized rangefinder.} \newline
Implements the procedure from Section~\ref{sec:rrf-subspace}.}
\label{alg:power-rangefinder}

\Require	Input matrix $\mtx{B} \in \F^{m \times n}$, target rank $\ell$, depth $q$.
\Ensure		Orthonormal matrix $\mtx{Q} \in \F^{m \times \ell}$
\Statex

\Function{PowerRangefinder}{$\mtx{B}$, $\ell$, $q$}

\State	Draw a random matrix $\mtx{\Omega} \in \F^{m \times \ell}$
\State	$\mtx{Y}_0 = \mtx{\Omega}$
\For{$i = 1, \dots, q$}
	\State	$[\mtx{Y}_{i-1}, \sim] = \texttt{qr\_econ}(\mtx{Y}_{i-1})$
	\State	$\mtx{Y}_i = \mtx{B} (\mtx{B}^* \mtx{Y}_{i-1})$
\EndFor
\State	$[\mtx{Q}, \sim] = \texttt{qr\_econ}(\mtx{Y}_q)$
\EndFunction
\end{algorithmic}
\end{algorithm}

\subsection{Block Krylov methods}
\label{sec:rrf-krylov}

As in the case of spectral norm estimation (Section~\ref{sec:max-eig}),
we can achieve more accurate results with Krylov subspace methods.
Nevertheless, this improvement comes at the cost of additional
storage and more complicated algorithms. %

\subsubsection{Rangefinder with a Krylov subspace}

Let $\mtx{B} \in \F^{m \times n}$ be a fixed input matrix,
and let $q$ be a natural number.  Let $\mtx{\Omega} \in \F^{m \times \ell}$
be a random test matrix.  We can form the extended sample matrix
$$
\mtx{Y} = \begin{bmatrix} \mtx{\Omega} & (\mtx{B}^*\mtx{B}) \mtx{\Omega} & \dots & (\mtx{B}^*\mtx{B})^q \mtx{\Omega} \end{bmatrix}.
$$
Then we compute an orthobasis $\mtx{Q} \in \F^{m \times (q+1)\ell}$ using
a QR factorization method.

See Algorithm~\ref{alg:krylov-rangefinder} for pseudocode.  This dominant
source of arithmetic is the $\bigO(qmn\ell)$ cost of matrix--matrix
multiplication.  The QR factorization now requires $\bigO(q^2 \ell^2 m)$
operations, a factor of $q^2$ more than the powered rangefinder.
We also need to store $\bigO(qm\ell)$ numbers,
which is roughly a factor $q$ more than the powered rangefinder.
Nevertheless, there is %
evidence that we can balance the values
of $\ell$ and $q$ to make the computational cost of the Krylov method
comparable with the cost of the power method---and still achieve higher accuracy. %

Historically, block Lanczos methods were used for spectral computations
and for SVD computations.  The block size $\ell$ was typically chosen
to be fairly small, say $\ell = 3$ or $\ell = 4$, with the goal of resolving singular values with
multiplicity greater than one.  The depth $q$ of the iteration was
usually chosen to be quite large.  There is theoretical and empirical
evidence that this parameter regime is the most efficient for resolving the
largest singular values to high accuracy~\cite{YGL18:Superlinear-Convergence}.

The randomized NLA literature has recognized that there are still potential
advantages to choosing the block size $\ell$ to be very large and to
choose the depth $q$ to be quite small \cite{HMST11:Algorithm-Principal,MM15:Randomized-Block}.
This parameter regime leads to algorithms that are more efficient on modern computer architectures,
and it still works extremely well for matrices with a modest amount of spectral
decay \cite{Tro18:Analysis-Randomized}.  The contemporary literature
also places a greater emphasis on the role of block Krylov methods
for computing matrix approximations.

\begin{remark}[History]
Block Lanczos methods, which are an efficient implementation
of the block Krylov method for a symmetric matrix,
were proposed by \citeasnoun{CD74:Block-Generalization} and by \citeasnoun{GU77:Block-Lanczos}.
The extension to the rectangular case appears in \citeasnoun{GLO81:Block-Lanczos}.
These algorithms have received renewed attention, beginning
with \cite{HMST11:Algorithm-Principal}.
The theoretical analysis of randomized block Krylov methods is much more difficult
than the analysis of randomized subspace iteration.
See \cite{MM15:Randomized-Block,YGL18:Superlinear-Convergence} and~\cite{Tro18:Analysis-Randomized}
for some results.
\end{remark}

\begin{algorithm}[t]
\begin{algorithmic}[1]
\caption{\textit{The Krylov randomized rangefinder.} \newline
Implements the procedure from Section~\ref{sec:rrf-krylov}. \newline
\textbf{Use with caution!} Unreliable in floating-point arithmetic.}
\label{alg:krylov-rangefinder}

\Require	Input matrix $\mtx{B} \in \F^{m \times n}$, target rank $\ell$, depth $q$
\Ensure		Orthonormal matrix $\mtx{Q} \in \F^{m \times 2(q+1)\ell}$
\Statex

\Function{KrylovRangefinder}{$\mtx{B}$, $\ell$, $q$}

\State	Draw a random matrix $\mtx{\Omega} \in \F^{n \times \ell}$
\State	$[\mtx{Y}_0,\sim] = \texttt{qr\_econ}(\mtx{\Omega})$
\For{$i = 1, \dots, q$}
	\State	$[\mtx{Y}_{i-1}, \sim] = \texttt{qr\_econ}(\mtx{Y}_{i-1})$
	\State	$\mtx{Y}_i = \mtx{B} (\mtx{B}^* \mtx{Y}_{i-1})$
\EndFor
\State	$[\mtx{Q}, \sim] = \texttt{qr\_econ}([\mtx{Y}_0, \dots, \mtx{Y}_q])$
\EndFunction
\end{algorithmic}
\end{algorithm}

\subsubsection{Alternative bases}

Algorithm~\ref{alg:krylov-rangefinder} computes a monomial basis for
the Krylov subspace, which is a poor choice numerically.  Let us mention
two more practical alternatives.

\begin{enumerate} \setlength{\itemsep}{1mm}
\item	\textbf{Block Lanczos.}  The classical approach uses the block Lanczos
iteration with reorthogonalization to compute a Lanczos-type basis for the Krylov subspace.
Algorithm~\ref{alg:rrf-lanczos} contains pseudocode for this approach adapted
from \citeasnoun{GLO81:Block-Lanczos}.  The basis computed by this algorithm
has the property that the blocks $\mtx{Q}_i \in \F^{m \times \ell}$ of the
sample matrix are mutually orthogonal.  The cost is similar to the cost
of computing the monomial basis, but there are advantages when the
subspace is used for spectral computations (Section~\ref{sec:max-eig}).
Indeed, we can approximate the largest singular values of $\mtx{A}$
by means of the largest singular values of the band matrix
$$
\mtx{R} = \begin{bmatrix}
\mtx{R}_1 & \mtx{R}_2^* \\
& \mtx{R}_3 & \mtx{R}_4^* \\
&& \ddots & \ddots \\
&&& \mtx{R}_{2q-1} & \mtx{R}_{2q}^* \\
&&&& \mtx{R}_{2q+1}
\end{bmatrix}.
$$
(The notation in this paragraph corresponds to the quantities computed in
Algorithm~\ref{alg:rrf-lanczos}.)

\item	\textbf{Chebyshev.}  Suppose that we have a good upper bound for
the spectral norm of the target matrix. (For example, we can obtain one using
techniques from Section~\ref{sec:max-eig}.)  %
Then we can compute a Chebyshev basis for the Krylov subspace.
The advantage of this approach is that we can postpone all normalization
and orthogonalization steps to the end of the computation,
which is beneficial for distributed computation.
This approach, which is being presented for the first time,
is inspired by~\citeasnoun{JC91:Parallelizable-Restarted}.
See Algorithm~\ref{alg:rrf-chebyshev} for pseudocode.
\end{enumerate}

\subsubsection{Analysis}

We are not aware of a direct analysis of Krylov subspace methods
for solving the rangefinder problem.  One may extract some bounds
from analysis of randomized SVD algorithms.
The following result is adapted from~\citeasnoun{Tro18:Analysis-Randomized}.

\begin{theorem}[Krylov rangefinder: Gaussian analysis]
Under the conditions in Theorem~\ref{thm:rrf-gauss},
let $\mtx{Y}$ be the sample matrix computed by Algorithm~\ref{alg:krylov-rangefinder}.
For $0 \leq \eps \leq 1/2$,
$$
\Expect \norm{ (\Id - \mtx{P}_{\mtx{Y}}) \mtx{B} }^2
	\leq \left[ 1 + 2\eps + \frac{9nk(\ell - k)}{\ell - k - 2} \cdot \econst^{-4q\sqrt{\eps}} \right] \sigma_{k+1}^2.
$$
\end{theorem}

In other words, the Krylov method can drive the error bound for
the rangefinder to $\bigO(\eps)$ by using
a Krylov subspace with depth  $q \approx \log (n/\eps) / \sqrt{\eps}$.
In contrast, the power method needs about
$q \approx (\log n) / \eps$ iterations to reach the same target.
The difference can be very substantial when $\eps$ is small.

\subsubsection{Spectral computations}

The Krylov rangefinder can also be used for highly accurate computation of the singular values
of a general matrix and the eigenvalues of a self adjoint matrix.
See~\cite{MM15:Randomized-Block,YGL18:Superlinear-Convergence} and~\cite{Tro18:Analysis-Randomized}
for discussion and analysis.

\begin{algorithm}[t]
\begin{algorithmic}[1]
\caption{\textit{The Lanczos randomized rangefinder.} \newline
Implements the Krylov rangefinder (Section~\ref{sec:rrf-krylov})
with block Lanczos bidiagonalization.}
\label{alg:rrf-lanczos}

\Require	Input matrix $\mtx{B} \in \F^{m \times n}$, target rank $\ell$, depth $q$
\Ensure		Orthonormal matrix $\mtx{Q} \in \F^{m \times (q+1)\ell}$
\Statex

\Function{LanczosRangefinder}{$\mtx{B}$, $\ell$, $q$}

\State	Draw a random matrix $\mtx{\Omega} \in \F^{m \times \ell}$

\State	$[\mtx{Q}_0, \sim] = \texttt{qr\_econ}(\mtx{\Omega})$
\State	$\mtx{W}_0 = \mtx{B}\mtx{Q}_0$
\State	$[\mtx{P}_0, \mtx{R}_1] = \texttt{qr\_econ}(\mtx{W}_0)$
\For{$i = 1, \dots, q$}
	\State	$\mtx{Z}_i = \mtx{B}^* \mtx{P}_{i-1} - \mtx{Q}_{i-1} \mtx{R}_{2i-1}^*$
		\Comment	Lanczos recursion, part 1
	\For{$j = 0, \dots, i-1$}
		\Comment	Double Gram--Schmidt
		\State	$\mtx{Z}_i = \mtx{Z}_i - \mtx{Q}_{j} (\mtx{Q}_{j}^* \mtx{Z}_i)$
		\State	$\mtx{Z}_i = \mtx{Z}_i - \mtx{Q}_{j} (\mtx{Q}_{j}^* \mtx{Z}_i)$
	\EndFor
	\State	$[\mtx{Q}_i, \mtx{R}_{2i}] = \texttt{qr\_econ}(\mtx{Z}_i)$
	\State	$\mtx{W}_i = \mtx{B}\mtx{Q}_i - \mtx{P}_{i-1} \mtx{R}_{2i}^*$
		\Comment	Lanczos recursion, part 2
	\For{$j = 0, \dots, i-1$}
		\Comment	Double Gram--Schmidt
		\State	$\mtx{W}_i = \mtx{W}_i - \mtx{P}_{j} (\mtx{P}_j^* \mtx{W}_i)$
		\State	$\mtx{W}_i = \mtx{W}_i - \mtx{P}_{j} (\mtx{P}_j^* \mtx{W}_i)$
	\EndFor
	\State	$[\mtx{P}_i, \mtx{R}_{2i+1}] = \texttt{qr\_econ}(\mtx{W}_i)$
\EndFor

\State	$\mtx{Q} = [ \mtx{Q}_0, \dots, \mtx{Q}_{q} ]$
\EndFunction
\end{algorithmic}
\end{algorithm}

\begin{algorithm}[t]
\begin{algorithmic}[1]
\caption{\textit{The Chebyshev randomized rangefinder.} \newline
Implements the Krylov rangefinder (Section~\ref{sec:rrf-krylov})
with a Chebyshev basis.} %
\label{alg:rrf-chebyshev}

\Require	Input matrix $\mtx{B} \in \F^{m \times n}$, target rank $\ell$, depth $q$, and norm bound $\norm{\mtx{B}} \leq \nu$
\Ensure		Orthonormal matrix $\mtx{Q} \in \F^{m \times (q+1)\ell}$
\Statex

\Function{ChebyshevRangefinder}{$\mtx{B}$, $\ell$, $q$}

\State	Draw a random matrix $\mtx{\Omega} \in \F^{m \times \ell}$

\State	$[\mtx{Y}_0, \sim] = \texttt{qr\_econ}(\mtx{\Omega})$
\State	$\mtx{Y}_1 = (2/\nu) \mtx{B} (\mtx{B}^*\mtx{Y}_0) - \mtx{Y}_0$

\For{$i = 2, \dots, q$}
	\State	$\mtx{Y}_i = (4/\nu)\mtx{B} (\mtx{B}^*\mtx{Y}_{i-1}) - 2\mtx{Y}_{i-1} - \mtx{Y}_{i-2}$
		\Comment	Chebyshev recursion
\EndFor

\State	$[\mtx{Q}, \sim] = \texttt{qr\_econ}([\mtx{Y}_0, \dots, \mtx{Y}_q])$
\EndFunction
\end{algorithmic}
\end{algorithm}

\section{Error estimation and adaptivity}
\label{sec:error-est}

The theoretical analysis of the randomized rangefinder in Section \ref{sec:random-rangefinder}
describes with great precision when the procedure is effective, and what errors to expect. From a
practical point of view, however, the usefulness of this analysis is limited by the fact that
we rarely have advance knowledge of the singular values of the matrix to be approximated. In this section,
we consider the more typical situation where we are given a matrix $\mtx{A} \in \F^{m\times n}$
and a tolerance $\varepsilon$, and it is our job to find a low-rank factorization of $\mtx{A}$
that is accurate to within precision $\varepsilon$. In other words, part of the task to be solved
is to determine the $\varepsilon$-rank of $\mtx{A}$.

To complete this job, we must equip the rangefinder with an \textit{a posteriori} error estimator:
given an orthonormal matrix $\mtx{Q} \in \F^{m\times \ell}$ whose columns form a putative basis for
the range of $\mtx{A}$, it estimates the corresponding approximation error $\|\mtx{A} - \mtx{Q}\mtx{Q}^{*}\mtx{A}\|$.

To solve the fixed-error approximation problem,
the idea is to start with a lowball ``guess'' at the rank, run the rangefinder, and then check to
see if we are within the requested tolerance. If the answer is negative, then there are several
strategies for how to proceed. Typically, we would just draw more samples to enrich the basis we
already have on hand.  But in some circumstances, it may be better to start over or to try an
alternative approach, such as increasing the amount of powering that is done.

\subsection{A posteriori error estimation}
\label{sec:rrf-error}

Let $\mtx{A} \in \F^{m \times n}$ be a target matrix,
and let $\mtx{Q} \in \F^{m \times \ell}$ be an orthonormal matrix
whose columns may or may not form a good basis for the range of $\mtx{A}$.
We typically think of $\mtx{Q}$ as being the output of one of the rangefinder
algorithms described in Section \ref{sec:random-rangefinder}.
Our goal is now to produce an inexpensive and reliable estimate of the error
$
\triplenorm{ (\Id - \mtx{QQ}^*) \mtx{A} }$
with respect to some norm $\triplenorm{\cdot}$.

To do so, we draw on the techniques for
norm estimation by sampling (Sections~\ref{sec:trace-est}--\ref{sec:schatten-p}).
The basic idea is to collect a (small) auxiliary sample
\begin{equation}
\label{eq:auxiliarysample}
\mtx{Z} = \mtx{A} \mtx{\Phi},
\end{equation}
where the test matrix $\mtx{\Phi} \in \F^{n \times s}$ is drawn from a
Gaussian distribution. We assume that $\mtx{\Phi}$ is statistically independent
from whatever process was used to compute $\mtx{Q}$.
Then
\begin{equation}
\label{eq:errorsample}
(\Id - \mtx{QQ}^*) \mtx{Z} = (\Id - \mtx{QQ}^*) \mtx{A \Phi} \in \F^{m \times s}.
\end{equation}
is a random sample of the error in the approximation.
We can now use any of the methods from Sections~\ref{sec:trace-est}--\ref{sec:schatten-p}
to estimate the error from this sample.  For example,
$$
\fnorm{ (\Id - \mtx{QQ}^*) \mtx{A} }^2
	\approx \frac{1}{s} \fnorm{(\Id - \mtx{QQ}^*) \mtx{Z}}^2.
$$
The theory in Section~\ref{sec:trace-est} gives precise tail
bounds for this estimator.
Similarly, we can approximate the Schatten $4$-norm by computing a sample variance.
Beyond that, the Valiant--Kong estimator (Section~\ref{sec:valiant-kong})
allows us to approximate higher-order Schatten norms. %

The cost of extracting the auxiliary sample (\ref{eq:auxiliarysample}) is almost
always much smaller than the cost of running the rangefinder itself;
it involves only $s$ additional matrix--vector multiplications,
where $s$  can be thought of as a small fixed number, say $s=10$.

\subsection{A certificate of accuracy for structured random matrices}
\label{sec:certificate}

The idea of drawing an auxiliary sample (\ref{eq:auxiliarysample}) for purposes of
error estimation is particularly appealing when the rangefinder implementation
involves a structured random test matrix (see~Section \ref{sec:rrf-dimred}). %
These structured random maps can be much faster than Gaussian random matrices,
while producing errors that are just as small \cite[Sec.~7.4]{HMT11:Finding-Structure}.
Their main weakness is that they come with far weaker \textit{a priori} error guarantees;
see Theorem~\ref{thm:rrf-srtt}.

Now, consider a situation where we use a structured random matrix to compute an
approximate basis $\mtx{Q}$ for the range of a matrix $\mtx{A}$
(via Algorithm \ref{alg:random-rangefinder}), and a small Gaussian
random matrix $\mtx{\Phi}$ to draw an auxiliary sample (\ref{eq:auxiliarysample}) from
$\mtx{A}$. The additional cost of extracting the ``extra'' sample is small, both in
terms of practical execution time and in terms of the asymptotic cost estimates (which
would generally remain unchanged). However, the computed output can now be relied on
with supreme confidence, since it is backed up by the strong theoretical results that govern Gaussian matrices.

One may even push these ideas further and use the ``certificate of accuracy'' to gain
confidence in randomized sampling methods based on heuristics or educated guesses about
the matrix being estimated.  For instance, one may observe that matrices that arise
in some application typically have low coherence (Section~\ref{sec:coord-embed}),
and one may implement a fast uniform sampling strategy that only works in this setting.
No matter how unsafe the sampling strategy, %
we can trust the \textit{a posteriori} error estimator when
it promises that the computed factorization is sufficiently accurate.

The general idea of observing the action of a residual on random vectors in order to get
an estimate for its magnitude can be traced back at least as far as \citeasnoun{Gir89:Fast-Monte-Carlo}.
Here, we followed the discussion in \citeasnoun[Sec.~14]{2018_PCMI_martinsson} and
\citeasnoun[Sec.~6]{TYUC19:Streaming-Low-Rank}.

\begin{remark}[Rank doubling]
Let us consider what should be done if the \textit{a posteriori} error estimator tells us that
a requested tolerance has not yet been met. Since many structured random matrices have the
unfortunate property that it is not an easy matter to recycle the sample already computed
in order to build a larger one, it often makes sense to simply start from scratch, but
doubling the number of columns in the test matrix. Such a strategy of doubling the rank at each
attempt typically does not change the order of the dominant term in the asymptotic cost. It may, however,
be somewhat wasteful from a practical point of view.
\end{remark}

\subsection{Adaptive rank determination using Gaussian test matrices}
\label{sec:adaptivebasic}

When a Gaussian test matrix is used, we can incorporate \textit{a posteriori} error estimation %
into the rangefinder algorithm with negligible increase in the amount of computation.

To illustrate, let us first consider a situation
where we are given a matrix $\mtx{A}$, and we use the randomized rangefinder (Algorithm \ref{alg:random-rangefinder})
with a Gaussian random matrix to find an orthonormal basis for its approximate range.
We do not know the numerical rank of $\mtx{A}$ in advance, so we use a number $\ell$ of samples
that we believe is likely to be more than enough.

In order to include \textit{a posteriori} error estimation, we \textit{conceptually} split
the test matrix so that
\begin{equation}
\label{eq:courteney1}
\mtx{\Omega} = \bigl[\mtx{\Omega}_{1}\quad\mtx{\Omega}_{2}\bigr],
\end{equation}
where $\mtx{\Omega}_{1}$ holds the first $\ell-s$ columns of $\mtx{\Omega}$. The thin
sliver $\mtx{\Omega}_{2}$ that holds the last $s$ columns will temporarily play the
part of the independent test matrix $\mtx{\Phi}$.
(Note that the matrices $\mtx{\Omega}_i$ defined here are different from those
in the proof of Theorem~\ref{thm:rrf-gauss}.)

The sample matrix inherits a corresponding
split
\begin{equation}
\label{eq:courteney2}
\mtx{Y} =
\mtx{A}\mtx{\Omega} =
\bigl[\mtx{A}\mtx{\Omega}_{1}\quad\mtx{A}\mtx{\Omega}_{2}\bigr] =:
\bigl[\mtx{Y}_{1}\quad\mtx{Y}_{2}\bigr].
\end{equation}
In order to orthonormalize the columns of $\mtx{Y}$ to form an approximate basis for the
column space, we perform an unpivoted QR factorization: %
$$
\mtx{Y} =
\bigl[\mtx{Q}_{1}\quad\mtx{Q}_{2}\bigr]\,
\left[\begin{array}{cc}
\mtx{R}_{11} & \mtx{R}_{12} \\
\mtx{0}      & \mtx{R}_{22}
\end{array}\right],
$$
where the partitioning conforms with (\ref{eq:courteney2}).  Thus,
$\mtx{Y}_{1} = \mtx{Q}_{1}\mtx{R}_{11}$ and $\mtx{Y}_{2} = \mtx{Q}_{1}\mtx{R}_{12} + \mtx{Q}_{2}\mtx{R}_{22}$.
Now, observe that
$$
\mtx{Q}_{2}\mtx{R}_{22} =
\mtx{Y}_{2} - \mtx{Q}_{1}\mtx{R}_{12} =
\mtx{Y}_{2} - \mtx{Q}_{1}\mtx{Q}_{1}^{*}\mtx{Y}_{2} =
\bigl(\mtx{I} - \mtx{Q}_{1}\mtx{Q}_{1}^{*})\mtx{A}\mtx{\Omega}_{2}.
$$
In other words, the matrix $\mtx{Q}_{2}\mtx{R}_{22}$ is a sample of the residual
error resulting from using $\mtx{\Omega}_{1}$ as the test matrix. We can analyze
the sample $\mtx{Q}_{2}\mtx{R}_{22}$ using the techniques described in Section
\ref{sec:rrf-error} to derive an estimate on the norm of the residual error.
If the resulting estimate is small enough, we can
confidently trust the computed factorization. (When the rangefinder is used as a
preliminary step towards computing a partial SVD of the matrix, we may as well use
the full orthonormal basis in $\mtx{Q}$ in any steps that follow.)

If the error estimate computed is larger than what is acceptable, then additional
work must be done, typically by drawing additional samples to enrich the basis already
computed, as described in the next section.

\subsection{An incremental algorithm based on Gaussian test matrices}
\label{sec:randQB}
The basic error estimation procedure outlined in Section \ref{sec:adaptivebasic} is
appropriate when it is not onerous to draw a large number $\ell$ of samples
and when the \textit{a posteriori} error merely serves as an
insurance policy against a rare situation where the singular
values decay more slowly than expected. In this section, we describe a technique that is
designed for  situations where we have no notion what the rank may
be in advance. The idea is to build the approximate basis incrementally by drawing
and processing one batch of samples at a time, while monitoring the errors as we go.
The upshot is that this computation can be organized in such a way that the total cost is
essentially the same as it would have been had we known the numerical rank in advance.

We frame the rangefinder problem as usual: for a given matrix $\mtx{A}$ and a given
tolerance $\tau$, we seek to build an orthonormal matrix $\mtx{Q}$ such that
$$
\|\mtx{A} - \mtx{Q}\mtx{Q}^{*}\mtx{A}\| \leq \tau.
$$
The procedure that we describe will be controlled by a tuning parameter
$b$ that specifies how many columns we process at a time. If we
choose $b$ too large, we may overshoot the numerical rank of $\mtx{A}$ and perform
more work than necessary. If $b$ is too small, computational efficiency may suffer;
see Section \ref{sec:blocking}.  In many environments, picking $b$ between $10$ and
$100$ would be about right.

While the \textit{a posteriori} error estimator signals that the error tolerance
has not been met, the incremental rangefinder successively draws blocks of $b$ Gaussian random vectors,
computes the corresponding samples, and adds them to the basis.  See Algorithm
\ref{alg:IncRangeFinder}.
To understand how the method works,
observe that, after line 8 has been executed, the matrix $\mtx{Y}$ holds the sample
\begin{equation}
\label{eq:logitech7}
\mtx{Y} = \bigl(\mtx{I}  - \mtx{Q}\mtx{Q}^{*}\bigr)\mtx{A}\mtx{\Omega}
\end{equation}
from the residual $\bigl(\mtx{I}  - \mtx{Q}\mtx{Q}^{*}\bigr)\mtx{A}$,
where $\mtx{Q}$ is the cumulative basis that has been built at that point
and where $\mtx{\Omega}$ is an $n\times b$ matrix drawn from a Gaussian distribution. Since
(\ref{eq:logitech7}) holds, we can estimate $\|(\mtx{I}  - \mtx{Q}\mtx{Q}^{*}\bigr)\mtx{A}\|$
using the techniques described in Section \ref{sec:rrf-error}.

\begin{algorithm}[t]
\begin{algorithmic}[1]
\caption{\textit{Incremental rangefinder.} \newline
Implements the first procedure from Section~\ref{sec:randQB}. \newline
This technique builds an orthonormal basis for a given matrix by
processing blocks of vectors at a time.
The method stops when an \textit{a posteriori}
error estimator indicates that a specified tolerance has been met.
}
\label{alg:IncRangeFinder}

\Require	Target matrix $\mtx{A} \in \F^{m \times n}$, tolerance $\tau \in \mathbb{R}_{+}$, block size $b$
\Ensure		Orthonormal matrix $\mtx{Q}$ such that $\|\mtx{A} - \mtx{Q}\mtx{Q}^{*}\mtx{A}\| \leq \tau$ with high probability
\Statex

\Function{IncrementalRangefinder}{$\mtx{A}$, $\tau$, b}

\State	$\mtx{Y} = \mtx{A}\mtx{\Omega}$\Comment Draw $\mtx{\Omega} \in \F^{n\times b}$ from a Gaussian distribution
\State	$[\mtx{Q}_{1},\sim] = \texttt{qr\_econ}(\mtx{Y})$
\State  $i=1$
\While{$\texttt{norm\_est}(\mtx{Y}) > \tau$} \Comment Norm estimator in Section \ref{sec:rrf-error}
  \State    $i = i+1$
  \State	$\mtx{Y} = \mtx{A}\mtx{\Omega}$\Comment Draw $\mtx{\Omega} \in \F^{n\times b}$ from a Gaussian distribution
  \State	$\mtx{Y} = \mtx{Y} - \sum_{j=1}^{i-1}\mtx{Q}_{j}\bigl(\mtx{Q}_{j}^{*}\mtx{Y}\bigr)$
  \label{line:Yafter}
  \State	$[\mtx{Q}_{i},\sim] = \texttt{qr\_econ}(\mtx{Y})$
\EndWhile
\State $\mtx{Q} = \bigl[\mtx{Q}_{1}\quad\mtx{Q}_{2}\quad \cdots \quad \mtx{Q}_{i-1}\bigr]$
\EndFunction
\end{algorithmic}
\end{algorithm}

In situations where the matrix $\mtx{A}$ is small enough to fit in RAM, it often
makes sense to explicitly update it after every step. The benefit to doing so is
that one can then determine the norm of the remainder matrix explicitly.
Algorithm \ref{alg:IncRangeFinderUpdating}
summarizes the resulting procedure. After line 9 of the algorithm has been executed,
the formula (\ref{eq:logitech7}) holds because, at this point in the computation,
$\mtx{A}$ has been overwritten by $\bigl(\mtx{I}  - \mtx{Q}\mtx{Q}^{*}\bigr)\mtx{A}$.
Algorithm \ref{alg:IncRangeFinderUpdating} relates to Algorithm \ref{alg:IncRangeFinder}
in the same way that modified Gram--Schmidt relates to classical Gram--Schmidt.

\begin{algorithm}[t]
\begin{algorithmic}[1]
\caption{\textit{Incremental rangefinder with updating.} \newline
Implements the second procedure from Section~\ref{sec:randQB}. \newline
This variation of Algorithm \ref{alg:IncRangeFinder} is suitable when the given
matrix $\mtx{A}$ is small enough that it can be updated explicitly. The method
builds an approximate factorization $\mtx{A} \approx \mtx{Q}\mtx{B}$ that is \textit{guaranteed}
to satisfy $\|\mtx{A} - \mtx{Q}\mtx{B}\| \leq \tau$ for a requested tolerance $\tau$.
}
\label{alg:IncRangeFinderUpdating}

\Require	Target matrix $\mtx{A} \in \F^{m \times n}$, tolerance $\tau \in \mathbb{R}_{+}$, block size $b$.
\Ensure		Orthonormal matrix $\mtx{Q}$ and a matrix $\mtx{B}$ such that $\|\mtx{A} - \mtx{Q}\mtx{B}\| \leq \tau$.
\Statex

\Function{IncrementalRangefinderWithUpdating}{$\mtx{A}$, $\tau$, b}

\State	$\mtx{Y} = \mtx{A}\mtx{\Omega}$\Comment Draw $\mtx{\Omega} \in \F^{n\times b}$ from a Gaussian distribution.
\State	$[\mtx{Q}_{1},\sim] = \texttt{qr\_econ}(\mtx{Y})$
\State  $\mtx{B}_{1} = \mtx{Q}_{1}^{*}\mtx{A}$
\State  $\mtx{A} = \mtx{A} - \mtx{Q}_{1}\mtx{B}_{1}$
\State  $i=1$
\While{$\|\mtx{A}\| > \tau$}\Comment Use an inexpensive norm such as Frobenius
  \State  $i           = i+1$
  \State  $\mtx{Y}     = \mtx{A}\mtx{\Omega}$\Comment Draw $\mtx{\Omega} \in \F^{n\times b}$ from a Gaussian distribution.
  \State  $[\mtx{Q}_{i},\sim] = \texttt{qr\_econ}(\mtx{Y})$
  \State  $\mtx{B}_{i} = \mtx{Q}_{i}^{*}\mtx{A}$
  \State  $\mtx{A}     = \mtx{A} - \mtx{Q}_{i}\mtx{B}_{i}$
\EndWhile
\State $\mtx{Q} = \bigl[\mtx{Q}_{1}\quad\mtx{Q}_{2}\quad \cdots \quad \mtx{Q}_{i}\bigr]$
\State $\mtx{B} = \bigl[\mtx{B}_{1}^{*}\quad\mtx{B}_{2}^{*}\quad \cdots \quad \mtx{B}_{i}^{*}\bigr]^{*}$
\EndFunction
\end{algorithmic}
\end{algorithm}

For matrices whose singular values decay slowly, incorporating a few steps of power iteration
(as described in Section \ref{sec:rrf-subspace}) is very beneficial. In this environment, it
is often necessary to incorporate additional reorthonormalizations to combat loss of orthogonality
due to round-off errors. Full descriptions of the resulting techniques can be found in
\citeasnoun{2015_martinsson_blocked}.

A version of Algorithm \ref{alg:IncRangeFinderUpdating} suitable for sparse matrices or matrices
stored out-of-core is described in \citeasnoun{2017_yu_singlepass_randQB}. This variant reorganizes
the computation to avoid the explicit updating step and to reduce the communication requirements overall.
Observe that it is still possible to evaluate the Frobenius norm of the residual exactly (without
randomized estimation) by using the identity %
\begin{multline*}
\fnorm{ \mtx{A} }^2 =
\fnorm{ \bigl(\mtx{I} - \mtx{Q}\mtx{Q}^{*}\bigr)\mtx{A} }^{2} +
\fnorm{ \mtx{Q}\mtx{Q}^{*}\mtx{A} }^{2} \\
=
\fnorm{ \bigl(\mtx{I} - \mtx{Q}\mtx{Q}^{*}\bigr)\mtx{A} }^{2} +
\fnorm{ \mtx{Q}\mtx{B} }^{2} =
\fnorm{ \bigl(\mtx{I} - \mtx{Q}\mtx{Q}^{*}\bigr)\mtx{A} }^{2} +
\fnorm{ \mtx{B} }^{2},
\end{multline*}
where the first equality holds since the column spaces of $\bigl(\mtx{I} - \mtx{Q}\mtx{Q}^{*}\bigr)\mtx{A}$
and $\mtx{Q}\mtx{Q}^{*}\mtx{A}$ are orthogonal and where the third equality holds since $\mtx{Q}$ is orthonormal.
For an error measure, we can use the resulting relationship:
$$
\fnorm{ \bigl(\mtx{I} - \mtx{Q}\mtx{Q}^{*}\bigr)\mtx{A} } =
\sqrt{ \fnorm{\mtx{A} }^{2} - \fnorm{\mtx{B} }^{2}},
$$
observing that both $\fnorm{ \mtx{A} }$ and $\fnorm{ \mtx{B} }$ can be computed explicitly.

\section{Finding natural bases: QR, ID, and CUR}
\label{sec:natural}

In Section \ref{sec:random-rangefinder},
we explored a number of efficient techniques
for building a tall thin matrix $\mtx{Q}$
whose columns form an approximate basis
for the range of an input matrix $\mtx{A}$
that is numerically rank-deficient.
The columns of $\mtx{Q}$ are orthonormal,
and they are formed as linear combinations
of many columns from the matrix $\mtx{A}$.

It is sometimes desirable to work with a basis
for the range that consists of a subset of
the columns of $\mtx{A}$ itself.  In this
case, one typically has to give up on the
requirement that the basis vectors be orthogonal.
We gain the advantage of a basis that shares
properties with the original matrix,
such as sparsity or nonnegativity.
Moreover, for purposes of data interpretation
and analysis, it can be very useful to identify
a subset of the columns that distills the
information in the matrix.

In this section, we start by describing some popular
matrix decompositions that use ``natural'' basis
vectors for the column space, for the row space,
or for both.  We show how these matrices can be
computed somewhat efficiently by means of
slight modifications to classical deterministic
techniques.  Then we describe how to combine
deterministic and randomized methods to
obtain algorithms with superior performance.

\subsection{The CUR decomposition, and three flavors of interpolative decompositions}
\label{sec:CURandIDdef}

To introduce the low-rank factorizations that we investigate in this section,
we describe how they can be used to represent an $m\times n$ matrix $\mtx{A}$ of
\textit{exact} rank $k$, where $k < \min(m,n)$. This is an artificial
setting, but it allows us to convey the key ideas using a minimum
of notational overhead.

A basic interpolative decomposition (ID) of a matrix $\mtx{A}$ with exact
rank $k$ takes the form
\begin{equation}
\label{eq:defID1}
\begin{array}{cccc}
\mtx{A} &=& \mtx{C} &\mtx{Z},\\
m\times n && m\times k & k\times n
\end{array}
\end{equation}
where the matrix $\mtx{C}$ is given by a subset of the columns of $\mtx{A}$
and where $\mtx{Z}$ is a matrix that contains the $k\times k$ identity matrix
as a submatrix.
The fact that the decomposition \eqref{eq:defID1} exists is an immediate consequence
of the definition of rank. A more significant observation is that there exists a
factorization of the form \eqref{eq:defID1} that is \textit{well-conditioned},
in the sense that no entry of $\mtx{Z}$ is larger than one in modulus.  This claim can be
established through an application of Cramer's rule.
See \cite{tyrt1997} and \cite{2006_martinsson_skeletonization}.

The factorization \eqref{eq:defID1} uses a subset of the columns of $\mtx{A}$ to
span its column space. Of course, there is an analog factorization that
uses a subset of the rows of $\mtx{A}$ to span the row space.
We write this as
\begin{equation}
\label{eq:defID2}
\begin{array}{cccc}
\mtx{A} &=& \mtx{X} &\mtx{R},\\
m\times n && m\times k & k\times n
\end{array}
\end{equation}
where $\mtx{R}$ is a matrix consisting of $k$ rows of $\mtx{A}$, and where
$\mtx{X}$ is a matrix that contains the $k\times k$ identity matrix.

For bookkeeping purposes, we introduce index vectors $J_{\rm s}$ and $I_{\rm s}$
that identify the columns and rows chosen in the factorizations \eqref{eq:defID1}
and \eqref{eq:defID2}. To be precise, let $J_{\rm s} \subset \{1,2,\dots,n\}$
denote the index vector of length $k$ such that
$$
\mtx{C} = \mtx{A}(:,J_{\rm s}).
$$
Analogously, we let $I_{\rm s} \subset \{1,2,\dots,m\}$ denote the index vector
for which
$$
\mtx{R} = \mtx{A}(I_{\rm s},:).
$$
The index vectors $I_{\rm s}$ and $J_{\rm s}$ are often referred to as \emph{skeleton} index vectors,
whence the subscript ``s''.  This terminology arises from the original literature about
these factorizations~\cite{tyrt1997}.

A related two-sided factorization is based on extracting a row/column submatrix.
In this case, the basis vectors for the row and column space are
less interpretable.  More precisely,
\begin{equation}
\label{eq:defID3}
\begin{array}{ccccc}
\mtx{A} &=& \mtx{X} & \mtx{A}_{\rm s} & \mtx{Z},\\
m\times n && m\times k & k\times k & k\times n
\end{array}
\end{equation}
where $\mtx{X}$ and $\mtx{Z}$ are the same matrices as those that appear in
\eqref{eq:defID1} and \eqref{eq:defID2}, and where $\mtx{A}_{\rm s}$ is the
$k\times k$ submatrix of $\mtx{A}$ given by
$$
\mtx{A}_{\rm s} = \mtx{A}(I_{\rm s},J_{\rm s}).
$$
To distinguish among these variants,
we refer to \eqref{eq:defID1} as a \textit{column ID},
to \eqref{eq:defID2} as a \textit{row ID},
and to \eqref{eq:defID3} as a \textit{double-sided ID}.

We introduce a fourth factorization, often called the %
\textit{CUR decomposition}.  For a matrix of exact rank $k$, it
takes the form
\begin{equation}
\label{eq:CUR}
\begin{array}{ccccc}
\mtx{A} &=& \mtx{C} & \mtx{U} & \mtx{R},\\
m\times n && m\times k & k\times k & k\times n
\end{array}
\end{equation}
where $\mtx{C}$ and $\mtx{R}$ are the matrices that appeared in
\eqref{eq:defID1} and \eqref{eq:defID2}, which consist of $k$ columns
and $k$ rows of $\mtx{A}$, and where $\mtx{U}$ is a
small matrix that links them together. In the present case, where
$\mtx{A}$ has exact rank $k$, the matrix $\mtx{U}$ must take the form
\begin{equation}
\label{eq:CURformulaeasy}
\mtx{U} = \bigl(\mtx{A}(I_{\rm s},J_{\rm s})\bigr)^{-1}.
\end{equation}
The factorizations \eqref{eq:defID3} and \eqref{eq:CUR} are related
through the formula
$$
\mtx{A} =
\mtx{X}\mtx{A}_{\rm s}\mtx{Z} =
\underbrace{\bigl(\mtx{X}\mtx{A}_{\rm s}\bigr)}_{=\mtx{C}}\,
\underbrace{\mtx{A}_{\rm s}^{-1}}_{=\mtx{U}}\,
\underbrace{\bigl(\mtx{A}_{\rm s}\mtx{Z}\bigr)}_{=\mtx{R}}.
$$

Comparing the two formats, we see that the CUR \eqref{eq:CUR} has
an advantage in that it requires very little storage.
As long as $\mtx{A}$ is stored explicitly (or is
easy to retrieve), the CUR factorization \eqref{eq:CUR} is determined
by the index vectors $I_{\rm s}$ and $J_{\rm s}$ and
the linking matrix $\mtx{U}$. If $\mtx{A}$ is not readily available,
then, in order to use the CUR, we need to evaluate and store the
matrices $\mtx{C}$ and $\mtx{R}$. When $\mtx{A}$ is sparse, the latter approach
can still be more efficient than storing the matrices $\mtx{X}$ and $\mtx{Z}$.

A disadvantage of the CUR factorization \eqref{eq:CUR} is that, when
the singular values of $\mtx{A}$ decay rapidly, the factorization \eqref{eq:CUR}
is typically numerically ill-conditioned. The reason is that,
whenever the factorization is a good representation of $\mtx{A}$, the
singular values of $\mtx{A}_{\rm s}$ should approximate the $k$
dominant singular values of $\mtx{A}$, so the singular values of
$\mtx{U}$ end up approximating the \textit{inverses} of these singular
values. This means that $\mtx{U}$ will have elements of magnitude $1/\sigma_{k}$,
which is clearly undesirable when $\sigma_{k}$ is small. In contrast, the
ID \eqref{eq:defID3} is numerically benign.

In the numerical analysis literature, what we refer to as an interpolative decomposition is
often called a \emph{skeleton factorization} of $\mtx{A}$. This term dates
back at least as far as \citeasnoun{tyrt1997}, where the term \emph{pseudo-skeleton} was
used for the CUR decomposition \eqref{eq:CUR}.

\begin{remark}[Storage-efficient ID]
We mentioned that the matrices $\mtx{X}$ and $\mtx{Z}$ that appear in the ID
are almost invariably dense, which appears to necessitate the storage of
$(m+n)k$ floating-point numbers for the double-sided ID. Observe, however,
that these matrices satisfy the relations
$$
\mtx{X} = \mtx{C}\mtx{A}_{\rm s}^{-1},
\qquad\mbox{and}\qquad
\mtx{Z} = \mtx{A}_{\rm s}^{-1}\mtx{R}.
$$
This means that as long as we store the index vectors $I_{\rm s}$
and $J_{\rm s}$, the matrices $\mtx{X}$ and $\mtx{Z}$ can be applied
on the fly whenever needed, and do not need to be explicitly formed.
\end{remark}

\subsection{Approximate rank}

In practical applications, the situation we considered in Section \ref{sec:CURandIDdef}
where a matrix has exact rank $k$ is rare. Instead, we typically work with a matrix
whose singular values decay fast enough that it is advantageous to form a
low-rank approximation.  Both the ID and the CUR can be used in this environment, but now
the discussion becomes slightly more involved.

To illustrate, let us consider a situation where we
are given a tolerance $\varepsilon$, and we seek to compute an approximation $\mtx{A}_{k}$
of rank $k$, with $k$ as small as possible, such that $\norm{ \mtx{A} - \mtx{A}_{k} } \leq \varepsilon$.
If $\mtx{A}_k$ is a truncated singular value decomposition, %
then the Eckart--Young theorem implies that the rank $k$ of the approximation $\mtx{A}_k$
will be minimal.
When we use an approximate algorithm, such as the RSVD (Section~\ref{sec:rsvd}),
we may not find the exact optimum, but we typically get very close.

What happens if we seek an ID $\mtx{A}_{k}$ that approximates $\mtx{A}$ to a fixed tolerance?
There is no guarantee that the rank $k$ (that is, the number of rows or columns involved)
will be close to the rank of the truncated SVD.  How close can we get in practice?

When the singular values of $\mtx{A}$ decay rapidly, then the minimal rank attainable
by an approximate ID is close to what is attainable with an SVD.  Moreover,
the algorithms we will describe for computing an ID produce an answer that is close
to the optimal one.

When the singular values decay slowly, however,
the difference in rank between the optimal ID and the optimal SVD
can be quite substantial \cite{gu1996}.
On top of that, the algorithms used to compute the ID can result in answers that
are still further away from the optimal value \cite{2005_martinsson_skel}.

When the CUR decomposition is used in an environment of approximate rank, standard
algorithms start by determining index sets $I_{\rm s}$ and $J_{\rm s}$ that
identify the spanning rows and columns, and then proceed to the problem of
finding a ``good'' linking matrix
$\mtx{U}_{\rm s}$. One could still use the formula \eqref{eq:CURformulaeasy}, but
this is rarely a good idea. The most obvious reason is that the matrix
$\mtx{A}(I_{\rm s},J_{\rm s})$ need not be invertible in this situation.
Indeed, when randomized sampling is used to find the index sets, it is common
practice to compute index vectors that hold substantially more elements than is
theoretically necessary, which can easily make $\mtx{A}(I_{\rm s},J_{\rm s})$
singular or, at the very least, highly ill-conditioned.
In this case, a better approximation is given by
\begin{equation}
\label{eq:CURformula}
\mtx{U} = \mtx{C}^{\dagger} \mtx{A} \mtx{R}^{\dagger},
\end{equation}
with $\mtx{C}^{\dagger}$ and $\mtx{R}^{\dagger}$ the pseudoinverses of $\mtx{C}$ and $\mtx{R}$.
(As always, pseudoinverses should be applied numerically by computing a QR or SVD factorization.)

\subsection{Deterministic methods, and the connection to column-pivoted QR}
\label{sec:detID}

A substantial amount of research effort has been dedicated to the question of
how to find a set of good spanning columns and/or rows of a given matrix. It
is known that the task of finding the absolutely optimal one is combinatorially
hard, but efficient algorithms exist that are guaranteed to produce
a close-to-optimal answer~\cite{gu1996}. In this subsection, we briefly discuss some
deterministic methods that work well for dense matrices of modest size.
In Section \ref{sec:randID}, we will show how these methods can be
combined with randomized techniques to arrive at algorithms that work well
for general matrices, whether they are small or huge, sparse or dense,
available explicitly or not, etc.

Perhaps the most obvious deterministic method for computing an ID
is the classical Gram--Schmidt process, which selects
the columns or rows in a greedy fashion. Say we are interested in the column
ID \eqref{eq:defID1} of a given matrix $\mtx{A}$. The Gram--Schmidt procedure
first grabs the largest column and places it in the first column of
$\mtx{C}$.  Then it projects the remaining columns onto the orthogonal complement
of the one that was picked.  It places the largest of the resulting columns
in the second column of $\mtx{C}$, and so on.

In the traditional numerical linear algebra literature, it is customary
to formulate the Gram--Schmidt process as a column-pivoted QR (CPQR) decomposition.
After $k$ steps, this factorization results in a partial decomposition
of $\mtx{A}$ such that
\begin{equation}
\label{eq:partialCPQR}
\begin{array}{cccccccccccccccc}
\mtx{A} & \mtx{\Pi} &=& \mtx{Q} & \mtx{S} & + & \mtx{E}, \\
m\times n & n\times n&& m\times k & k\times n && m\times n
\end{array}
\end{equation}
where the columns of $\mtx{Q}$ form an orthonormal basis for the space spanned
by the $k$ selected columns of $\mtx{A}$, where $\mtx{S}$ is upper-triangular,
where $\mtx{E}$ is a ``remainder matrix'' holding what remains of the $n-k$ columns
of $\mtx{A}$ that have not yet been picked, and where $\mtx{\Pi}$ is a permutation matrix
that reorders the columns of $\mtx{A}$ in such a way that the $k$ columns picked are
the first $k$ columns of $\mtx{A}\mtx{\Pi}$. (We use the letter \mtx{S} for the
upper-triangular factor in lieu of the more traditional \mtx{R} to avoid confusion with the
matrix $\mtx{R}$ holding spanning rows in \eqref{eq:defID2} and \eqref{eq:CUR}.)

In order to convert \eqref{eq:partialCPQR} into the ID \eqref{eq:defID1},
we split off the first $k$ columns of $\mtx{S}$ into a $k\times k$
upper-triangular matrix $\mtx{S}_{11}$, so that
$$
\mtx{S} =
\kbordermatrix{&
k & n-k\\
k &\mtx{S}_{11} & \mtx{S}_{12}}.
$$
Upon multiplying \eqref{eq:partialCPQR} by $\mtx{\Pi}^{*}$ from the right, we obtain
\begin{equation}
\label{eq:postCPQR}
\mtx{A} =
\underbrace{\mtx{Q}\mtx{S}_{11}}_{=:\mtx{C}}\,
\underbrace{\bigl[\mtx{I}_{k}\qquad \mtx{S}_{11}^{-1}\mtx{S}_{12}\bigr]\mtx{\Pi}^{*}}_{=:\mtx{Z}}
+
\mtx{E}\mtx{\Pi}^{*}.
\end{equation}
We recognize equation \eqref{eq:postCPQR} as the ID \eqref{eq:defID1}, with the only
difference that there is now a remainder term that results from the fact that $\mtx{A}$
is only approximately rank-deficient. (Observe that the remainder terms in
\eqref{eq:partialCPQR} and in \eqref{eq:postCPQR} are identical, up to a permutation of the columns.)

A row ID can obviously be computed by applying Gram--Schmidt to the rows of $\mtx{A}$
instead of the columns. Alternatively, one may express this as a column-pivoted QR
factorization of $\mtx{A}^{*}$ instead of $\mtx{A}$.

In order to build a double-sided ID,
one starts by computing a single-sided ID. If $m\geq n$, it is best to start
with a column ID of $\mtx{A}$ to determine $J_{\rm s}$ and $\mtx{Z}$.
Then we perform a row ID \textit{on the rows of $\mtx{A}(:,J_{\rm s})$} to determine
$I_{\rm s}$ and $\mtx{X}$.

Finally, in order to build a CUR factorization of $\mtx{A}$,
we can easily convert the double-sided ID to a CUR factorization using \eqref{eq:CURformulaeasy}
or \eqref{eq:CURformula}.

Detailed descriptions of all algorithms can be found in Sections 10 \& 11 of
\citeasnoun{2018_PCMI_martinsson}, while analysis and numerical results are
given in \citeasnoun{2014_martinsson_CUR}. A related set of deterministic techniques that
are efficient and often result in slightly higher quality spanning sets than column pivoting
are described in \citeasnoun{2016_sorensen_CUR}. Techniques based
on optimized spanning volumes of submatrices are described in
\cite{2010_oseledets_tyrtyshnikov_find_submatrix} and \cite{2012_thurau_deterministic_CUR}.

\begin{remark}[Quality of ID]
In this section, we have described simple methods based on the column-pivoted QR factorization
for computing a CUR decomposition, as well as all three flavors of interpolatory decompositions.
In discussing the quality of the resulting factorizations, we will address two questions:
(1) How close to minimal is the resulting approximation error? (2) How well-conditioned
are the basis matrices?

The NLA literature contains a detailed study of both questions.
This inquiry was instigated by Kahan's construction of
matrices for which CPQR performs very poorly \cite[Sec.~5]{1966_kahan_NLA}.
\citeasnoun{gu1996} provided a comprehensive analysis of the situation and
presented an algorithm whose asymptotic complexity in typical environments
is only slightly worse than that of CPQR and that is guaranteed to produce near-optimal results.

In practice, CPQR works well.  In almost all cases, it yields factorizations that are close to optimal.
Moreover, it gives well-conditioned factorizations as long as orthonormality
of the basis is scrupulously maintained \cite{2006_martinsson_skeletonization,2005_martinsson_skel}.

A more serious problem with the ID and the CUR is that these decompositions
can exhibit much larger approximation errors than the SVD when
the input matrix has slowly decaying singular values.  This issue persists even
when the optimal index sets are used.
\end{remark}

\subsection{Randomized methods for finding natural bases}
\label{sec:randID}
The deterministic techniques for computing an ID or a CUR decomposition in Section
\ref{sec:detID} work very well for small, dense matrices. In this section, we describe
randomized methods that work much better for matrices that are sparse or are just very
large.

To be concrete, we consider the problem of finding a vector $I_{\rm s}$
that identifies a set of rows that form a good basis for the row space of a given matrix $\mtx{A}$.
To do so, we use the randomized rangefinder to build a matrix $\mtx{Y}$ whose columns accurately
span the column space of $\mtx{A}$ as in Section \ref{sec:random-rangefinder}. Since $\mtx{Y}$ is far
smaller than $\mtx{A}$, we can use the deterministic methods in Section \ref{sec:detID} to
find a set $I_{\rm s}$ of rows of $\mtx{Y}$ that form a basis for the row space of $\mtx{Y}$.
Next, we establish a simple but perhaps non-obvious fact: the set $I_{\rm s}$ also identifies
a set of rows of $\mtx{A}$ that form a good basis for the row space of $\mtx{A}$.

To simplify the argument, let us first suppose
that we are given an $m\times n$ matrix $\mtx{A}$ of \textit{exact} rank $k$, and that
we have determined by some means (say, the randomized rangefinder) an $m\times k$
matrix $\mtx{Y}$ whose columns span the column space of $\mtx{A}$. Then $\mtx{A}$
admits by definition a factorization
\begin{equation}
\label{eq:YF}
\begin{array}{cccccccccc}
\mtx{A} &=& \mtx{Y} & \mtx{F}, \\
m\times n && m\times k & k\times n
\end{array}
\end{equation}
for some matrix $\mtx{F}$. Now compute a row ID of $\mtx{Y}$, by performing Gram--Schmidt
on its rows, as described in Section \ref{sec:detID}. The result is a matrix $\mtx{X}$ and
an index vector $I_{\rm s}$ such that
\begin{equation}
\label{eq:Yid}
\begin{array}{cccccccccc}
\mtx{Y} &=& \mtx{X} & \mtx{Y}(I_{\rm s},:). \\
m\times k && m\times k & k\times k
\end{array}
\end{equation}
The claim is now that $\{I_{\rm s},\mtx{X}\}$ is automatically a row ID of $\mtx{A}$ as well.
To prove this, observe that
\begin{align*}
\mtx{X}\mtx{A}(I_{\rm s},:) &= \mtx{X}\mtx{Y}(I_{\rm s},:)\mtx{F} &&\mbox{\{use \eqref{eq:YF} restricted to the rows in $I_{\rm s}$\}} \\
&= \mtx{Y}\mtx{F} && \mbox{\{use \eqref{eq:Yid}\}} \\
&= \mtx{A} &&\mbox{\{use \eqref{eq:YF}\}}
\end{align*}
The key insight here is simple and powerful:
\textit{In order to compute a row ID of a matrix $\mtx{A}$, the only
information needed is a matrix $\mtx{Y}$ whose columns span the column space
of $\mtx{A}$.}

The task of finding a matrix $\mtx{Y}$ such that \eqref{eq:YF} holds to high
accuracy is particularly well suited for the randomized rangefinder described in
Section \ref{sec:random-rangefinder}. Putting everything together, we obtain
Algorithm \ref{alg:randomizedID}. When a Gaussian random matrix
is used, the method has complexity $O(mnk)$.

\begin{remark}[$O(mn\log k)$ complexity methods]
An interesting thing happens if we replace the Gaussian random
matrix $\mtx{\Omega}$ in Algorithm \ref{alg:randomizedID} with a structured random matrix, as
described in Section \ref{sec:dimension-reduction}: Then $\mtx{Y}$ is computed
at cost $O(mn\log k)$, and every step after that has cost $O((m+n)k^{2})$ or less.
\end{remark}

\begin{algorithm}[t]
\begin{algorithmic}[1]
\caption{\textit{Randomized ID.} \newline
Implements the procedure from Section~\ref{sec:randID}. \newline
The function \texttt{ID\_row} refers to any algorithm for computing a row-ID,
so that given a matrix $\mtx{B}$ and a rank $k$, calling $[I_{\rm s},\mtx{X}] = \texttt{row\_ID}(\mtx{B},k)$
results in an approximate factorization $\mtx{B} \approx \mtx{X}\mtx{B}(I_{\rm s},\colon)$.
The techniques based on column-pivoted QR described in Section \ref{sec:detID} work well.
}
\label{alg:randomizedID}

\Require	Matrix $\mtx{A} \in \F^{m\times n}$, target rank $k$, oversampling parameter $p$.
\Ensure		An $m\times k$ interpolation matrix $\mtx{X}$ and an index vector $I_{\rm s}$ such that $\mtx{A} \approx \mtx{X}\mtx{A}(I_{\rm s},\colon)$.
\Statex

\Function{RandomizedID}{$\mtx{A}$, k, p}

\State Draw an $n\times (k+p)$ test matrix $\mtx{\Omega}$, e.g., from a Gaussian distribution
\State Form the sample matrix $\mtx{Y} = \mtx{A}\mtx{\Omega}$ \Comment Powering may be used
\State Form an ID of the $n\times (k+p)$ sample matrix: $[I_{\rm s},\mtx{X}] = \texttt{ID\_row}(\mtx{Y},k)$
\EndFunction
\end{algorithmic}
\end{algorithm}

\subsection{Techniques based on coordinate sampling}

To find natural bases for a matrix, it is tempting just to
sample coordinates from some probability distribution on
the full index vector.
Some advantages and disadvantages of this approach were discussed
in Section \ref{sec:coord-embed}.

In the current context, the main appeal of coordinate sampling
is that the cost is potentially lower than
the techniques described in this section---provided that
we do not need to expend much effort to compute the
sampling probabilities.
This advantage can be decisive
in applications where mixing
random embeddings are too expensive.

Coordinate sampling has several disadvantages in comparison
to using mixing random embeddings.  Coordinate sampling typically
results in worse approximations for a given budget of rows or
columns.  Moreover, the quality of the approximation obtained
from coordinate sampling tends to be highly variable.
These vulnerabilities are less pronounced when
the matrix has very low coherence, so that uniform sampling works well.
There are also a few specialized situations where we can compute
subspace leverage scores efficiently.
(Section \ref{sec:coherence_leverage} defines coherence and leverage scores.)

In certain applications, a hybrid approach can work well. First, form an
initial approximation by drawing a very large subset of columns using a cheap coordinate sampling
method.  Then slim it down using the techniques described here, based on mixing random embeddings.
An example of this methodology appears in \citeasnoun{LBKL15:Large-Scale-Nystrom}.

There is a distinct class of methods, based on coresets,
that explicitly takes advantage of coordinate structure
for computing matrix approximations.
For example, see~\citeasnoun{2016_feldman_dimensionality}.
These techniques can be useful for processing enormous
matrices that are very sparse.  On the other hand, they
may require larger sets of basis vectors to achieve the
same quality of approximation.

\section{Nystr{\"o}m approximation}
\label{sec:nystrom}

We continue our discussion of matrix approximation
with the problem of finding a low-rank approximation
of a positive semidefinite (PSD) matrix.  There is an
elegant randomized method for accomplishing this goal
that is related to our solution to the rangefinder problem.

\subsection{Low-rank PSD approximation}
\label{sec:nystrom-overview}

Let $\mtx{A} \in \Sym_n$ be a PSD matrix.
For a rank parameter $k$, the goal is to produce
a rank-$k$ PSD matrix $\widehat{\mtx{A}}_k \in \Sym_n$ %
that approximates $\mtx{A}$ nearly as well as the
best rank-$k$ matrix:
$$
\norm{ \mtx{A} - \widehat{\mtx{A}}_k } \lesssim \sigma_{k+1}.
$$
To obtain the approximation, we will adapt the
randomized rangefinder method
(Algorithm~\ref{alg:random-rangefinder}).

\subsection{The Nystr{\"o}m approximation}

The most natural way to construct a low-rank approximation
of a PSD matrix is via the Nystr{\"o}m method.
Let $\mtx{X} \in \F^{n \times \ell}$ be an
arbitrary test matrix.  The \emph{Nystr{\"o}m approximation}
of $\mtx{A}$ with respect to $\mtx{X}$
is the PSD matrix
\begin{equation} \label{eqn:nystrom-def}
\mtx{A}\nys{\mtx{X}}
	:= (\mtx{AX})(\mtx{X}^* \mtx{AX})^\pinv (\mtx{AX})^*.
\end{equation}
An alternative presentation of this formula is
$$
\mtx{A}\nys{\mtx{X}}
	= \mtx{A}^{1/2} \mtx{P}_{\mtx{A}^{1/2} \mtx{X}} \mtx{A}^{1/2},
$$
where $\mtx{P}_{\mtx{Y}}$ is the orthogonal projector onto
the range of $\mtx{Y}$.  In particular, the Nystr{\"o}m approximation
only depends on the range of the matrix $\mtx{X}$.

The Nystr{\"o}m approximation~\eqref{eqn:nystrom-def} is closely related to the
Schur complement~\eqref{eqn:schur-complement} of $\mtx{A}$ with respect to $\mtx{X}$.
Indeed,
$$
\mtx{A} / \mtx{X}
	= \mtx{A} - \mtx{A} \nys{\mtx{X}}
	= \mtx{A}^{1/2} (\Id - \mtx{P}_{\mtx{A}^{1/2} \mtx{X}}) \mtx{A}^{1/2}.
$$
That is, the Schur complement of $\mtx{A}$ with respect to $\mtx{X}$
is precisely the error in the Nystr{\"o}m approximation.

Proposition~\ref{prop:rrf-schur} indicates that the
Nystr{\"o}m decomposition is also connected with our
approach to solving the rangefinder problem.  We immediately
perceive the opportunity to use a random test matrix $\mtx{X}$
to form the Nystr{\"o}m approximation.
Let us describe how this choice leads to algorithms for
computing a near-optimal low-rank approximation of
the matrix $\mtx{A}$.

\subsection{Randomized Nystr{\"o}m approximation algorithms}
\label{sec:nystrom-alg}

Here is a simple and effective procedure for computing a rank-$k$
PSD approximation of the PSD matrix $\mtx{A} \in \Sym_n$. %

First, draw a random test matrix $\mtx{\Omega} \in \F^{n \times \ell}$,
where $\ell \geq k$.  Form the sample matrix $\mtx{Y} = \mtx{A\Omega} \in \F^{n \times \ell}$.
Then compute the Nystr{\"o}m approximation
\begin{equation} \label{eqn:Ahat-nys-bad}
\widehat{\mtx{A}} = \mtx{A}\nys{ \mtx{\Omega} } = \mtx{Y} (\mtx{\Omega}^* \mtx{Y})^\pinv \mtx{Y}^*.
\end{equation}
The initial approximation $\widehat{\mtx{A}}$ has rank $\ell$.  To truncate the rank to $k$,
we just report a best rank-$k$ approximation $\widehat{\mtx{A}}_k$
of the initial approximation $\widehat{\mtx{A}}$ with respect to the Frobenius norm;
see \cite{TYUC17:Fixed-Rank-Approximation,PB19:Improved-Fixed-Rank}
and \cite{WGM19:Scalable-k-Means}.

Let us warn the reader that the formula~\eqref{eqn:Ahat-nys-bad}
is not suitable for numerical computation.  See Algorithm~\ref{alg:random-nystrom}
for a numerically stable implementation adapted from~\cite{LLS+17:Algorithm-971}
and \cite{TYUC17:Fixed-Rank-Approximation}.
In general, the matrix--matrix multiply with $\mtx{A}$ dominates the cost,
with $\bigO(n^2 \ell)$ arithmetic operations; this expense can be reduced
if either $\mtx{A}$ or $\mtx{\Omega}$ admits fast multiplication.
The approximation steps involve $\bigO(n \ell^2)$ arithmetic.
Meanwhile, storage costs are $\bigO(n\ell)$.

Another interesting aspect of Algorithm~\ref{alg:random-nystrom} is that
it only uses linear information about the matrix $\mtx{A}$.  Therefore,
it can be implemented in the one-pass or the streaming data model.
Remark~\ref{rem:streaming} gives more details about matrix approximation
in the streaming model.
See Section~\ref{sec:streaming-kpca} for an application to kernel
principal component analysis.

\begin{algorithm}[t]
\begin{algorithmic}[1]
\caption{\textit{Randomized Nystr{\"o}m approximation.} \newline
Implements the procedure from Section~\ref{sec:nystrom-alg}.}
\label{alg:random-nystrom}

\Require	Psd target matrix $\mtx{A} \in \Sym_n(\F)$, rank $k$ for approximation,
number $\ell$ of samples
\Ensure		Rank-$k$ PSD approximation $\widehat{\mtx{A}}_k \in \Sym_n(\F)$ expressed in factored form
$\widehat{\mtx{A}}_k = \mtx{U \Lambda U}^*$ where $\mtx{U} \in \F^{n \times k}$ is orthonormal
and $\mtx{\Lambda} \in \Sym_k(\F)$ is nonnegative and diagonal.

\Statex

\Function{RandomNystr{\"o}m}{$\mtx{A}$, $k$, $\ell$}

\State	Draw random matrix $\mtx{\Omega} \in \F^{n \times \ell}$
\State	Form $\mtx{Y} = \mtx{A\Omega}$

\State	$\nu = \sqrt{n} \, \texttt{eps}(\texttt{norm}(\mtx{Y}))$
\Comment	Compute shift

\State	$\mtx{Y}_{\nu} = \mtx{Y} + \nu \mtx{\Omega}$
\Comment	Samples of shifted matrix

\State	$\mtx{C} = \texttt{chol}(\mtx{\Omega}^* \mtx{Y}_{\nu})$

\State	$\mtx{B} = \mtx{Y}_{\nu} \mtx{C}^{-1}$
\Comment	Triangular solve!

\State	$[\mtx{U}, \mtx{\Sigma}, \sim] = \texttt{svd}(\mtx{B})$
\Comment	Dense SVD

\State	$\mtx{\Lambda} = \max\{0, \mtx{\Sigma}^2 - \nu \Id\}$
\Comment	Remove shift

\State	$\mtx{U} = \mtx{U}(:, 1:k)$ and $\mtx{\Lambda} = \mtx{\Lambda}(1:k, 1:k)$
\Comment	Truncate rank to $k$

\EndFunction
\end{algorithmic}
\end{algorithm}

\subsection{Analysis}

The randomized Nystr{\"o}m method enjoys the same kind of guarantees as the randomized
rangefinder.  The following result~\cite[Thm.~4.1]{TYUC17:Fixed-Rank-Approximation}
extends earlier contributions from~\cite{HMT11:Finding-Structure} and
\cite{Git13:Topics-Randomized}. %

\begin{theorem}[Nystr{\"o}m: Gaussian analysis] \label{thm:rand-nys}
Fix a PSD matrix $\mtx{A} \in \Sym_n(\F)$ with eigenvalues $\lambda_1 \geq \lambda_2 \geq \dots$.
Let $1 \leq k < \ell \leq n$.
Draw a random test matrix $\mtx{\Omega} \in \F^{n \times \ell}$ that is standard normal.
Then the rank-$k$ PSD approximation $\widehat{\mtx{A}}_k$ computed by
Algorithm~\ref{alg:random-nystrom} satisfies
$$
\Expect \norm{ \mtx{A} - \widehat{\mtx{A}}_k }
	\leq \lambda_{k+1} + \frac{k}{\ell - k - 1} \left( \sum\nolimits_{j > k} \lambda_j \right).
$$
\end{theorem}

In other words, the computed rank-$k$ approximation $\widehat{\mtx{A}}_k$ achieves almost
the error $\lambda_{k+1}$ in the optimal rank-$k$ approximation of $\mtx{A}$.  The error
declines as the number $\ell$ of samples increases and as the $\ell_1$ norm of the
tail eigenvalues decreases.   %

When the target matrix $\mtx{A}$ has a sharply decaying spectrum,
Theorem~\ref{thm:rand-nys} can be pessimistic.
See~\citeasnoun[Sec.~4]{TYUC17:Fixed-Rank-Approximation} for additional theoretical results.

\subsection{Powering}
\label{sec:nystrom-power}

We can reduce the error in the randomized
Nystr{\"o}m approximation by powering the input matrix,
much as subspace iteration improves the
performance in the randomized rangefinder (Section~\ref{sec:rrf-subspace}).

Let $\mtx{A} \in \Sym_n$ be a PSD matrix.  Draw a random test matrix
$\mtx{\Omega} \in \F^{n \times \ell}$.  For a natural number $q$,
we compute $\mtx{Y} = \mtx{A}^q \mtx{\Omega}$ by repeated multiplication.
Then the Nystr{\"o}m approximation of the input matrix takes the form
$$
\widehat{\mtx{A}} = \big[ (\mtx{A}^q) \nys{ \mtx{\Omega} } \big]^{1/q}
	= \big[ \mtx{Y} (\mtx{\Omega}^* \mtx{Y})^\pinv \mtx{Y}^* \big]^{1/q}.
$$
This approach requires very careful numerical implementation.
As in the case of subspace iteration, it drives down the
error exponentially fast as $q$ increases.  We omit the details.

\subsection{History}

The Nystr{\"o}m approximation was developed in the context of
integral equations~\cite{Nys30:Uber-Praktische}.  It has had a substantial
impact in machine learning, beginning with the work of
~\citeasnoun{WS01:Using-Nystrom} on randomized low-rank approximation
of kernel matrices.
Section~\ref{sec:kernel} contains a discussion of this literature.
Note that the Nystr{\"o}m approximation of a kernel matrix is almost always
computed with respect to a random coordinate subspace,
in contrast to the uniformly random subspace induced by a Gaussian test matrix.

Algorithmic and theoretical results on Nystr{\"o}m approximation
with respect to general test matrices have appeared in a number of papers,
including \cite{HMT11:Finding-Structure,Git13:Topics-Randomized,LLS+17:Algorithm-971}
and \cite{2017_tropp_practical_sketching}.

\section{Single-view algorithms}
\label{sec:singlepass}

In this section, we will describe a remarkable class of algorithms that
are capable of computing a low-rank approximation of a matrix that is
so large that it cannot be stored at all.

We will consider the specific problem of computing an approximate
singular value decomposition of a matrix $\mtx{A} \in \F^{m\times n}$
under the assumption that we are allowed to view each entry of $\mtx{A}$ only once
and that we cannot specify the order in which they are viewed.
To the best of our knowledge, no deterministic techniques
can carry off such a computation without \textit{a priori}
information about the singular vectors of the matrix.

For the case where $\mtx{A}$ is psd, we have already seen a single-view algorithm:
the Nystr\"om technique of Algorithm~\ref{alg:random-nystrom}. Here, we concentrate on the more difficult
case of general matrices.  This presentation is adapted from \citeasnoun[Sec.~5.5]{HMT11:Finding-Structure}
and the papers \cite{TYUC17:Practical-Sketching} and
\cite{TYUC19:Streaming-Low-Rank}.

\subsection{Algorithms}
\label{sec:singleviewalgorithms}

In the basic RSVD algorithm (Section \ref{sec:rsvd}),
we view each element of the given matrix $\mtx{A}$ at least twice.
In the first view, we form
a sample matrix $\mtx{Y} = \mtx{A}\mtx{\Omega}$ for a given test matrix $\mtx{\Omega}$. We
orthonormalize the columns of $\mtx{Y}$ to form the matrix $\mtx{Q}$ and then visit $\mtx{A}$
again to form a second sample $\mtx{C} = \mtx{Q}^{*}\mtx{A}$. The columns of $\mtx{Y}$ form
an approximate basis for the column space of $\mtx{A}$, and the columns of $\mtx{C}$ form an
approximate basis for the row space.

In the single-view framework, we can only visit $\mtx{A}$
once, which means that we must sample both the row and the column space simultaneously. To this
end, let us draw tall thin random matrices
\begin{equation}
\label{eq:single1}
\mtx{\Upsilon} \in \F^{m\times \ell}
\qquad\mbox{and}\qquad
\mtx{\Omega} \in \F^{n\times \ell}
\end{equation}
and then form the two corresponding sample matrices
\begin{equation}
\label{eq:single2}
\mtx{X} = \mtx{A}^{*}\mtx{\Upsilon} \in \F^{n\times \ell}
\qquad\mbox{and}\qquad
\mtx{Y} = \mtx{A}\mtx{\Omega} \in \F^{m\times \ell}.
\end{equation}
In (\ref{eq:single1}), we draw a number $\ell$ of samples that is slightly larger than
the rank $k$ of the low-rank approximation that we seek. (Section \ref{sec:singleapriori}
gives details about how to choose $\ell$.) Observe that both $\mtx{X}$ and $\mtx{Y}$ can be
formed in a single pass over the matrix $\mtx{A}$.

Once we have seen the entire matrix, the next step is to orthonormalize the columns
of $\mtx{X}$ and $\mtx{Y}$ to obtain orthonormal matrices
\begin{equation}
\label{eq:single3}
[\mtx{P},\sim] = \texttt{qr\_econ}(\mtx{X}),
\qquad\mbox{and}\qquad
[\mtx{Q},\sim] = \texttt{qr\_econ}(\mtx{Y}).
\end{equation}
At this point, $\mtx{P}$ and $\mtx{Q}$ hold approximate bases for the row and column spaces
of $\mtx{A}$, so we anticipate that
\begin{equation}
\label{eq:single4}
\mtx{A} \approx \mtx{Q}\mtx{Q}^{*}\mtx{A}\mtx{P}\mtx{P}^{*} = \mtx{Q}\mtx{C}\mtx{P}^{*},
\end{equation}
where we defined the ``core'' matrix
\begin{equation}
\label{eq:single5}
\mtx{C} := \mtx{Q}^{*}\mtx{A}\mtx{P} \in \F^{\ell\times \ell}.
\end{equation}
Unfortunately, since we cannot revisit $\mtx{A}$,
we are not allowed to form $\mtx{C}$ directly by applying formula \eqref{eq:single5}.

Instead, we develop a relation that $\mtx{C}$ must satisfy approximately,
which allows us to estimate $\mtx{C}$ from the quantities we have on hand.
To do so, right-multiply the definition \eqref{eq:single5}
by $\mtx{P}^{*}\mtx{\Omega}$ to obtain %
\begin{equation}
\label{eq:single6}
\mtx{C}\bigl(\mtx{P}^{*}\mtx{\Omega}\bigr) = \mtx{Q}^{*}\mtx{A}\mtx{P}\bigl(\mtx{P}^{*}\mtx{\Omega}\bigr).
\end{equation}
Inserting the approximation $\mtx{A}\mtx{P}\mtx{P}^{*} \approx \mtx{A}$ into (\ref{eq:single6}), we find that
\begin{equation}
\label{eq:single7}
\mtx{C}\bigl(\mtx{P}^{*}\mtx{\Omega}\bigr) \approx \mtx{Q}^{*}\mtx{A}\mtx{\Omega} = \mtx{Q}^{*}\mtx{Y}.
\end{equation}
In (\ref{eq:single7}), all quantities except $\mtx{C}$ are known explicitly, which means that we can
solve it, in the least-squares sense, to arrive at an estimate
\begin{equation}
\label{eq:single8}
\mtx{C}_{\rm approx} = \bigl(\mtx{Q}^{*}\mtx{Y}\bigr)\bigl(\mtx{P}^{*}\mtx{\Omega}\bigr)^{\pinv},
\end{equation}
where $\pinv$ denotes the Moore-Penrose pseudoinverse.  As always, the pseudoinverse is applied
by means of an orthogonal factorization.  Once $\mtx{C}_{\rm approx}$ has been
computed via (\ref{eq:single8}), we obtain the rank-$\ell$ approximation
\begin{equation}
\label{eq:single8b}
\mtx{A} \approx \mtx{Q}\mtx{C}_{\rm approx}\mtx{P}^{*},
\end{equation}
which we can convert into an approximate SVD using the standard postprocessing steps. For additional
implementation details, see \citeasnoun[Sec.~5.5]{HMT11:Finding-Structure}
and \citeasnoun[Sec.~5]{2018_PCMI_martinsson}.  Extensions of this approach,
with theoretical analysis, appear in \citeasnoun{2017_tropp_practical_sketching}.

Recently, \citeasnoun{TYUC19:Streaming-Low-Rank} have demonstrated that the numerical performance of the single-view algorithm
can be improved by extracting a third sketch of $\mtx{A}$ that is independent from $\mtx{X}$ and $\mtx{Y}$.
The idea is to draw tall thin random matrices
$$
\mtx{\Phi} \in \F^{m\times s},
\qquad\mbox{and}\qquad
\mtx{\Psi} \in \F^{n\times s},
$$
where $s$ is another oversampling parameter.  Then we form a ``core sketch''
\begin{equation}
\label{eq:single9}
\mtx{Z} = \mtx{\Phi}^{*}\mtx{A}\mtx{\Psi} \in \F^{s\times s}.
\end{equation}
This extra data allows us to derive an alternative equation for the core matrix $\mtx{C}$.
We left- and right-multiply the definition \eqref{eq:single5} by $\mtx{\Phi}^{*}\mtx{Q}$
and $\mtx{P}^{*}\mtx{\Psi}$ to obtain the relation
\begin{equation}
\label{eq:single10}
\bigl(\mtx{\Phi}^{*}\mtx{Q}\bigr)\mtx{C}\bigl(\mtx{P}^{*}\mtx{\Psi}\bigr) =
\mtx{\Phi}^{*}\mtx{Q}\mtx{Q}^{*}\mtx{A}\mtx{P}\mtx{P}^{*}\mtx{\Psi}.
\end{equation}
Inserting the approximation $\mtx{A} \approx \mtx{Q}\mtx{Q}^{*}\mtx{A}\mtx{P}\mtx{P}^{*}$ into
\eqref{eq:single10}, we find that
\begin{equation}
\label{eq:single11}
\bigl(\mtx{\Phi}^{*}\mtx{Q}\bigr)\mtx{C}\bigl(\mtx{P}^{*}\mtx{\Psi}\bigr) \approx
\mtx{\Phi}^{*}\mtx{A}\mtx{\Psi} = \mtx{Z}.
\end{equation}
An improved approximation to the core matrix $\mtx{C}$ results by solving
\eqref{eq:single11} in a least-squares sense; to wit,
$\mtx{C} = (\mtx{\Phi}^{*}\mtx{Q})^{\pinv}\mtx{Z}(\mtx{P}^{*}\mtx{\Psi})^{\pinv}$.

\begin{remark}[Streaming algorithms] \label{rem:streaming}
Single-view algorithms are related to the streaming
model of computation \cite{2005_muthukrishnan_stream}.
\citeasnoun{2009_clarkson_woodruff} were the first
to explicitly study matrix computations in streaming
data models.

One important streaming model poses the assumption that the input matrix
$\mtx{A}$ is presented as a sequence %
of innovations:
$$
\mtx{A} = \mtx{H}_{1} + \mtx{H}_{2} + \mtx{H}_{3} + \cdots
$$
Typically, each update $\mtx{H}_{i}$ is simple; for instance,
it may be sparse or low-rank.
The challenge is that the full matrix $\mtx{A}$ is too large to be stored.
Once an innovation $\mtx{H}_{i}$ has been processed, it cannot be retained.
This is called the ``turnstile'' model in the theoretical computer science literature.

The algorithms described in this section handle this difficulty by creating a
random linear transform $\mathcal{S}$ that maps $\mtx{A}$ down to a low-dimensional sketch
that is small enough to store.  What we actually retain in memory is the evolving sketch
of the input:
$$
\mathcal{S}(\mtx{A}) =
\mathcal{S}(\mtx{H}_{1}) +
\mathcal{S}(\mtx{H}_{2}) +
\mathcal{S}(\mtx{H}_{3}) + \cdots.
$$
In Algorithm \ref{alg:SingleViewSVD}, we instantiate $\mathcal{S}$ by drawing the
random matrices $\mtx{\Upsilon}$, $\mtx{\Omega}$, $\mtx{\Phi}$, and $\mtx{\Psi}$, and
then work with the sketch
$$
\mathcal{S}(\mtx{H}) =
\big(\mtx{\Upsilon}^{*}\mtx{H},\,\mtx{H}\mtx{\Omega},\,\mtx{\Phi}^{*}\mtx{H}\mtx{\Psi} \big).
$$
The fact that the sketch is a \textit{linear} map is essential here.
\citeasnoun{LNW14:Turnstile-Streaming} prove that randomized
linear embeddings are essentially the only kind of algorithm
for handling the turnstile model.
In contrast, the sketch implicit in the RSVD algorithm from Section \ref{sec:rsvd} is
a quadratic or higher-order polynomial in the input matrix.
\end{remark}

\begin{algorithm}[t]
\begin{algorithmic}[1]
\caption{\textit{Single-view SVD.} \newline
Implements the algorithm from Section~\ref{sec:singleviewalgorithms}. \newline
This algorithm computes an approximate partial singular value decomposition of a given matrix
$\mtx{A}$, under the constraint that each entry of $\mtx{A}$ may be viewed only once.
}
\label{alg:SingleViewSVD}

\Require	Target matrix $\mtx{A} \in \F^{m \times n}$, rank $k$, sampling sizes $\ell$ and $s$.
\Ensure		Orthonormal matrices $\mtx{U} \in \F^{m\times k}$ and $\mtx{V} = \F^{n\times k}$, and a diagonal matrix $\mtx{\Sigma} \in \F^{k\times k}$ such that
$\mtx{A} \approx \mtx{U}\mtx{\Sigma}\mtx{V}^{*}$.
\Statex

\Function{SingleViewSVD}{$\mtx{A}$, $k$, $\ell$, $s$}

\State	Draw test matrices $\mtx{\Upsilon} \in \F^{m\times \ell}$, $\mtx{\Omega} \in \F^{n\times \ell}$,
             $\mtx{\Phi} \in \F^{m\times s}$, $\mtx{\Psi} \in \F^{n\times s}$
\State  Form $\mtx{X} = \mtx{A}^{*}\mtx{\Upsilon}$, $\mtx{Y} = \mtx{A}\mtx{\Omega}$, $\mtx{Z} = \mtx{\Phi}^{*}\mtx{A}\mtx{\Psi}$
\Comment Viewing $\mtx{A}$ only once!
\State  $[\mtx{P},\sim] = \texttt{qr\_econ}(\mtx{X})$, $[\mtx{Q},\sim] = \texttt{qr\_econ}(\mtx{Y})$
\State  $\mtx{C} = \bigl(\mtx{\Phi}^{*}\mtx{Q}\bigr)^{\pinv}\,\mtx{Z}\,\bigl(\mtx{P}^{*}\mtx{\Psi}\bigr)^{\pinv}$
\Comment Execute using a least-squares solver
\State  $[\widehat{\mtx{U}},\widehat{\mtx{\Sigma}},\widehat{\mtx{V}}] = \texttt{svd}(\mtx{C})$
\Comment A full SVD
\State $\mtx{U} = \mtx{Q}\widehat{\mtx{U}}(\colon,1:k)$,
       $\mtx{V} = \mtx{P}\widehat{\mtx{V}}(\colon,1:k)$,
       $\mtx{\Sigma} = \widehat{\mtx{\Sigma}}(1:k,1:k)$
\EndFunction
\end{algorithmic}
\end{algorithm}

\begin{remark}[Single-view versus out-of-core algorithms]
In principle, the methods discussed in this section can also be used in situations where
a matrix is stored in slow memory, such as a spinning disk hard drive, or on a
distributed memory system.  However, one has to carefully weigh whether the decrease
in accuracy and increase in uncertainty that is inherent to single-view algorithms
is worth the cost savings. As a general matter, revisiting the matrix at least once
is advisable whenever it is possible.
\end{remark}

\subsection{Error estimation, parameter choices and truncation}
\label{sec:singleapriori}

In the single-view computing environment, one must choose sampling parameters before
the computation starts, and there is no way to revisit these choices after data
has been gathered. This constraint makes \textit{a priori} error analysis particularly
important, because we need guidance on how large to make the sketches given
some prior knowledge about the spectral decay of the input matrix.
To illustrate
how this may work, let us cite \citeasnoun[Thm.~5.1]{TYUC19:Streaming-Low-Rank}:

\begin{theorem}[Single-view SVD: Gaussian analysis]
\label{thm:singleview}
Suppose that Algorithm \ref{alg:SingleViewSVD} is executed for an input matrix
$\mtx{A} \in \mathbb{C}^{m\times n}$ and for sampling parameters $s$ and $\ell$
that satisfy $s \geq 2\ell$. When the test matrices are drawn from a standard
normal distribution, the computed matrices $\mtx{P}$, $\mtx{C}$, and $\mtx{Q}$ satisfy
$$
\Expect \fnorm{ \mtx{A} - \mtx{Q}\mtx{C}\mtx{P}^{*} }^{2}
\leq
\frac{s}{s-\ell}\,
\min_{k < \ell}\left(\frac{\ell+k}{\ell-k}\sum_{j=k+1}^{\min(m,n)}\sigma_{j}^{2}\right).
$$
As usual, $\sigma_j$ is the $j$th largest singular value of $\mtx{A}$.
A very similar bound holds for the real field.
\end{theorem}

This result suggests that more aggressive oversampling is
called for in the single-view setting, as compared to the basic rangefinder problem.
For instance, if we aim for an approximation error
that is comparable to the best possible approximation with rank $k$,
then we might choose $\ell = 4k$ and $s=8k$ to obtain
$$
\mathbb{E}\|\mtx{A} - \mtx{Q}\mtx{C}\mtx{P}^{*}\|_{\rm F}^{2}
\leq
\frac{10}{3}\sum_{j=k+1}^{\min(m,n)}\sigma_{j}^{2} =
\frac{10}{3}\|\mtx{A} - \mtx{A}_{k}\|_{\rm F}^{2},
$$
where $\mtx{A}_{k}$ is the best possible rank-$k$ approximation of $\mtx{A}$.
As usual, the likelihood of large deviations from the expectation is negligible.
(For contrast, recall that the basic rangefinder algorithm often works well
when we select $\ell = k+5$ or $\ell=k+10$.)

Besides computing $\mtx{P}$, $\mtx{C}$, and $\mtx{Q}$, Algorithm \ref{alg:SingleViewSVD}
also prunes the approximation $\mtx{A} \approx \mtx{Q}\mtx{C}\mtx{P}^{*}$ by computing
an SVD of $\mtx{C}$ (line 6) and then throwing out the trailing $\ell-k$ modes (line 7).
The motivation for this truncation is that the approximation $\mtx{A} \approx \mtx{Q}\mtx{C}\mtx{P}^{*}$,
tends to capture the dominant singular modes of $\mtx{A}$ well, but the trailing ones have
very low accuracy. The same thing happens with the basic RSVD (Section~\ref{sec:rsvd}),
but the phenomenon is more pronounced in the single-view environment, in part because $\ell$
is substantially larger than $k$. Theorem \ref{thm:singleview} can be applied
to prove that the truncated factorization is as accurate as one can reasonably hope for;
see~\citeasnoun[Cor.~5.5]{TYUC19:Streaming-Low-Rank} for details.

\begin{remark}[Spectral norm bounds?]
Theorem~\ref{thm:singleview} provides a Frobenius norm
error bound for a matrix approximation algorithm.
For our survey, this is a \emph{rara avis in terra}.  Unfortunately,
relative error spectral norm error bounds are not generally possible
in the streaming setting~\cite[Chap.~6]{2014_woodruff_sketching}.
\end{remark}

\subsection{Structured test matrices}

Algorithm~\ref{alg:SingleViewSVD} can -- and should -- be implemented with structured
test matrices, rather than Gaussian test matrices.  This modification is especially
appealing in the single-view environment, where storage is often the main bottleneck.

For instance, consider the parameter selections $\ell = 4k$ and $s = 8k$ that we referenced
above. Then the four test matrices consist of $12k(m+n)$ floats that must be stored, and
the sketches add another $4k(m+n) + 64k^{2}$ floats.  Since $m$ and $n$ can be huge, these numbers
could severely limit the rank $k$ of the final matrix approximation.

If we swap out the Gaussian matrices for structured random matrices,
we can almost remove the cost associated with storing the test matrices.
In particular, the addition of the core sketch \eqref{eq:single9}
has a very light memory footprint because the sketch itself only uses $\bigO(k^{2})$ floats.
Empirically, when we use a structured random matrix, such as a sparse sign matrix (Section~\ref{sec:sparse-map})
or an SRTT (Section~\ref{sec:srtt}), the observed errors
are more or less indistinguishable from the errors attained with Gaussian test matrices.
See~\citeasnoun[Supplement, Figs.~SM2--7]{TYUC19:Streaming-Low-Rank}.

\subsection{A posteriori error estimation}
\label{sec:singleaposteriori}

In order to reduce the uncertainty associated with the single-view algorithms described in this
section, the ``certificate of accuracy'' technique described in Section \ref{sec:certificate} is
very useful.

Recall that the idea is to draw a separate test matrix whose only purpose is to
provide an independent estimate of the error in the computed solution.  This additional
test matrix can be \textit{very} thin (say 5 or 10 columns wide), and it still provides a dependable
bound on the computed error. These techniques can be incorporated without any difficulty in the
single-view environment, as outlined in \citeasnoun[Sec.~6]{TYUC19:Streaming-Low-Rank}.

Let us mention one caveat.  In the single-view environment,
we have no recourse when the \textit{a posteriori} error
estimator signals that the approximation error is unacceptable.
On the other hand, it is reassuring that the algorithm
can sound a warning that it has not met the desired accuracy.

\subsection{History}
\label{sec:singlehistory}

To the best of our knowledge,
\citeasnoun[Sec.~5.2]{WLRT08:Fast-Randomized}
described the first algorithm that can compute a low-rank matrix approximation
in the single-view computational model.
Their paper introduced the idea of independently sampling the row- and the column-space of a matrix,
as summarized in formulas \eqref{eq:single1}--\eqref{eq:single8b}.  This approach inspired the
single-view algorithms presented in \citeasnoun[Sec.~5.5]{HMT11:Finding-Structure}.
It is interesting that the primary objective of \citeasnoun{WLRT08:Fast-Randomized}
was to reduce the asymptotic flop count of the computation through the use of structured random test matrices.

\citeasnoun{2009_clarkson_woodruff} gave an explicit discussion of randomized NLA
in a streaming computational model.  They independently proposed a variant of the
algorithm from~\cite[Sec.~5.5]{HMT11:Finding-Structure}.
Later contributions to the field appeared in
\cite{LNW14:Turnstile-Streaming,2016_boutsidis_optimal,2016_feldman_dimensionality,GLPW16:Frequent-Directions} and \cite{2017_tropp_practical_sketching}.
The idea of introducing an additional sketch such as \eqref{eq:single9} to capture the ``core'' matrix
was proposed by \citeasnoun{2016_upadhyay_streaming}. \citeasnoun{TYUC19:Streaming-Low-Rank} have provided improvements of his approach, further analysis, and computational considerations.

\section{Factoring matrices of full or nearly full rank}
\label{sec:full}

So far, we have focused on techniques for computing low-rank approximations
of an input matrix. We will now upgrade to techniques for computing \textit{full}
rank-revealing factorizations such as the column-pivoted QR (CPQR) decomposition.

Classical deterministic techniques for computing these factorizations proceed through
a sequence of rank-one updates to the matrix, making them communication-intensive
and slow when executed on modern hardware. Randomization allows the algorithms to be
reorganized so that the vast majority of the arithmetic takes place inside matrix--matrix
multiplications, which greatly accelerates the execution speed.

When applied to an $n \times n$ matrix, most of the algorithms described in this section
have the same $\bigO(n^{3})$ asymptotic complexity as traditional methods;
the objective is to improve the practical speed by reducing communication costs.
However, randomization also allows us to incorporate Strassen-type techniques
to accelerate the matrix multiplications in a numerically stable manner,
attaining an overall cost of $\bigO(n^{\omega})$ where $\omega$
is the exponent of square matrix--matrix multiplication.

As well as the CPQR decomposition, we will consider algorithms for computing factorizations
of the form $\mtx{A} = \mtx{U}\mtx{R}\mtx{V}^{*}$ where $\mtx{U}$ and $\mtx{V}$ are
unitary matrices and $\mtx{R}$ is upper-triangular. Factorizations of this form can be
used for almost any task where either the CPQR or the SVD is currently used. The
additional flexibility allows us to improve on both the computational speed and on the
rank-revealing qualities of the factorization.

Sections \ref{sec:rankrevealing}--\ref{sec:demmelURV} introduce the key concepts
by describing a simple algorithm for computing a rank-revealing factorization
of a matrix.  This method is both faster than traditional column-pivoted QR
and better at revealing the spectral properties of the matrix.
Sections \ref{sec:classicCPQR}--\ref{sec:randUTV} are more technical; %
they describe how randomization can be used to resolve a longstanding challenge
of how to \textit{block} a classical algorithm for computing a column-pivoted QR decomposition by
applying groups of Householder reflectors simultaneously. They also describe how these ideas
can be extended to the task of computing a URV factorization.

\subsection{Rank-revealing factorizations}
\label{sec:rankrevealing}

Before we discuss algorithms, let us first define what we mean when we say that a factorization
is \textit{rank-revealing.} Given an $m\times n$ matrix $\mtx{A}$, we will consider factorizations
of the form
\begin{equation}
\label{eq:fullfactintro}
\begin{array}{ccccccccccccccccccc}
\mtx{A} &=& \mtx{U} &\mtx{R}&\mtx{V}^{*},\\
m\times n && m\times c & c\times n & n\times n
\end{array}
\end{equation}
where $c = \min(m,n)$, where $\mtx{U}$ and $\mtx{V}$ are orthonormal,
and where $\mtx{R}$ is upper-triangular (or banded upper-triangular).
We want the factorization to reveal the numerical rank of
$\mtx{A}$ in the sense that we obtain a near-optimal
approximation of $\mtx{A}$ when we truncate~\eqref{eq:fullfactintro}
to any level $k$.
That is,
\begin{equation}
\label{eq:rankreveal}
\|\mtx{A} - \mtx{U}(\colon,1:k)\mtx{R}(1:k,:)\mtx{V}^{*}\|
\approx
\inf\{\|\mtx{A} - \mtx{B}\|\,\colon\,\mtx{B}\mbox{ has rank }k\}
\end{equation}
for $k \in \{1,2,\dots,c\}$.
The factorization \eqref{eq:fullfactintro} can be viewed either as a generalization
of the SVD (for which $\mtx{R}$ is diagonal) or as a generalization of the
column-pivoted QR factorization (for which $\mtx{V}$ is a permutation matrix).

A factorization such as \eqref{eq:fullfactintro} that satisfies \eqref{eq:rankreveal}
is very handy.  It can be used to solve ill-conditioned linear
systems or least-squares problems; it can be used for estimating the
singular spectrum of $\mtx{A}$ (and all Schatten $p$-norms); and it
provides orthonormal bases for approximations to the four fundamental subspaces of
the matrix. Finally, it can be used to compute approximate low-rank factorizations
efficiently in situations where the numerical rank of the matrix is not that much
smaller than $m$ or $n$.  In contrast, all techniques described up to now
are efficient only when the numerical rank $k$ satisfies $k \ll \min(m,n)$.

\subsection{Blocking of matrix computations}
\label{sec:blocking}

A well-known feature of modern computing is that we can execute increasingly many floating-point
operations because CPUs are gaining more cores while GPUs and other accelerators are becoming
more affordable and more energy-efficient.
In contrast, the cost of communication (data transfer up and down levels of a memory hierarchy,
among servers, and across networks, etc.) is declining very slowly.
As a result, reducing communication is often the key to accelerating numerical
algorithms in the real world.

In the context of matrix computations, the main reaction to this development
has been to cast linear-algebraic operations as operating on blocks of the matrix,
rather than on individual entries or individual columns and rows
\cite{1998_dongarra_blocking,1993_stewart_blockQR_SVD,2015_blockQR_SISC}.
The objective is to reorganize an algorithm so that the majority of flops can be executed using
highly efficient algorithms for matrix--matrix multiplication (BLAS3),
rather than the slower methods for matrix--vector multiplications (BLAS2).

Unfortunately, it turns out that classical algorithms for computing rank-revealing factorizations
of matrices are very challenging to block. Column-pivoted QR proceeds through a sequence of rank-1
updates to the matrix.  The next pivot cannot be found until the previous update has been applied.
Techniques for computing an SVD of a matrix start by reducing the matrix to bidiagonal form.
Then they iterate on the bidiagonal matrix to drive it towards diagonal form. Both steps are
challenging to block.

To emphasize just how much of a difference blocking makes, let us peek ahead at the plot in
Figure \ref{fig:demmel_URV}. This graph shows computational times for computing certain matrix
factorizations versus matrix size on a standard desktop PC. In particular, look
at the times for column-pivoted QR (red solid line) and for unpivoted QR (red dashed line).
The asymptotic flop counts of these two algorithms are identical in the dominant term.
Yet the unpivoted factorization can easily be blocked, which means that it executes
one order of magnitude faster than the pivoted one.

\subsection{The powerURV algorithm}
\label{sec:powerURV}

There is a simple randomized algorithm for computing a rank-revealing factorization
of a matrix that perfectly illustrates the power of randomization to reduce communication.
Our starting point is an algorithm proposed by \citeasnoun{2007_demmel_fast_linear_algebra_is_stable}.
Given an $m\times n$ matrix $\mtx{A}$, typically with $m \geq n$, it proceeds as follows.
\begin{enumerate} \setlength{\itemsep}{1mm}
\item Draw an $n\times n$ matrix $\mtx{\Omega}$ from a standard normal distribution.
\item Perform an unpivoted QR factorization of $\mtx{\Omega}$ so that $[\mtx{V},\sim] = \texttt{qr}(\mtx{\Omega})$.
\item Perform an unpivoted QR factorization of $\mtx{A}\mtx{V}$ so that $[\mtx{U},\mtx{R}] = \texttt{qr}(\mtx{A}\mtx{V})$.
\end{enumerate}
Observe that the purpose of steps (1) and (2) is simply to generate a
matrix $\mtx{V}$ whose columns serve as a ``random'' orthonormal basis.
It is easily verified that the matrices $\mtx{U}$, $\mtx{R}$, and $\mtx{V}$ satisfy
\begin{equation}
\label{eq:basicURV}
\mtx{A} = \mtx{U}\mtx{R}\mtx{V}^{*}.
\end{equation}
The factorization (\ref{eq:basicURV}) is rank-revealing in theory
(see \citeasnoun[Theorem 5.2]{2007_demmel_fast_linear_algebra_is_stable} and \citeasnoun{2019_demmel_URV}),
but it does not reveal the rank particularly well in practice.

The cost to compute (\ref{eq:basicURV}) is dominated by the cost to perform two unpivoted QR
factorizations, and one matrix--matrix multiplication. (Simulating the random matrix
$\mtx{\Omega}$ requires only $\bigO(n^{2})$ flops.)

To improve the rank-revealing ability of the factorization, one can incorporate a small number of power iteration steps
\cite{2018_martinsson_powerurv}, so that step (2) in the recipe gets modified to
\begin{equation}
\label{eq:powerURV}
[\mtx{V},\sim] = \texttt{qr}((\mtx{A}^{*}\mtx{A})^{q}\mtx{\Omega}),
\end{equation}
where $q$ is a small integer. In practice, $q=1$ or $q=2$ is often enough to dramatically improve
the accuracy of the computation. Of course, incorporating power iteration increases the cost of
the procedure by adding $2q$ additional matrix--matrix multiplications. When the singular values
of $\mtx{A}$ decay rapidly, reorthonormalization in between each application of $\mtx{A}$ is
sometimes required to avoid loss of accuracy due to floating-point arithmetic.

Algorithm \ref{alg:powerURV} summarizes the techniques introduced in this section.
The method is simple, and easy to code. It requires far more
flops than traditional methods for computing rank-revealing factorizations,
yet it is faster in practice. For instance, if $m=n$ and $q=2$, then
powerURV requires $\approx 5n^{3}$ flops versus $0.5\,n^{3}$ flops for CPQR,
but Figure \ref{fig:demmel_URV} shows that powerURV is still faster. This
is noteworthy, since powerURV with $q=2$ does a \textit{far} better job
at revealing the numerical rank of $\mtx{A}$ than CPQR, as shown in Figure \ref{fig:powerURVaccuracy}.
See \citeasnoun{2018_martinsson_powerurv} for details.

\begin{algorithm}[t]
\begin{algorithmic}[1]
\caption{\textit{powerURV.} \newline
This algorithm computes a rank-revealing URV factorization of a given matrix
$\mtx{A}$; see Section \ref{sec:powerURV}. Reorthonormalization may be required
between applications of $\mtx{A}$ and $\mtx{A}^{*}$ on line 3 to combat round-off errors.
}
\label{alg:powerURV}

\Require	Target matrix $\mtx{A} \in \F^{m \times n}$ for $m\geq n$, power parameter $q$
\Ensure		Orthonormal matrices $\mtx{U} \in \F^{m\times n}$ and $\mtx{V} = \F^{n\times n}$, and upper triangular $\mtx{R} \in \F^{n\times n}$ such that
$\mtx{A} = \mtx{U}\mtx{R}\mtx{V}^{*}$
\Statex

\Function{powerURV}{$\mtx{A}$, $q$}

\State	Draw a test matrix $\mtx{\Omega} \in \F^{n\times n}$ from a standard normal distribution
\State  $[\mtx{V},\sim] = \texttt{qr\_econ}(\bigl(\mtx{A}^{*}\mtx{A}\bigr)^{q}\mtx{\Omega})$ \Comment Unpivoted QR
\State  $[\mtx{U},\mtx{R}] = \texttt{qr\_econ}(\mtx{A}\mtx{V})$ \Comment Unpivoted QR
\EndFunction
\end{algorithmic}
\end{algorithm}

\begin{figure}
\centering
\setlength{\unitlength}{1mm}
\begin{picture}(120,53)
\put(05,02){\includegraphics[width=115mm]{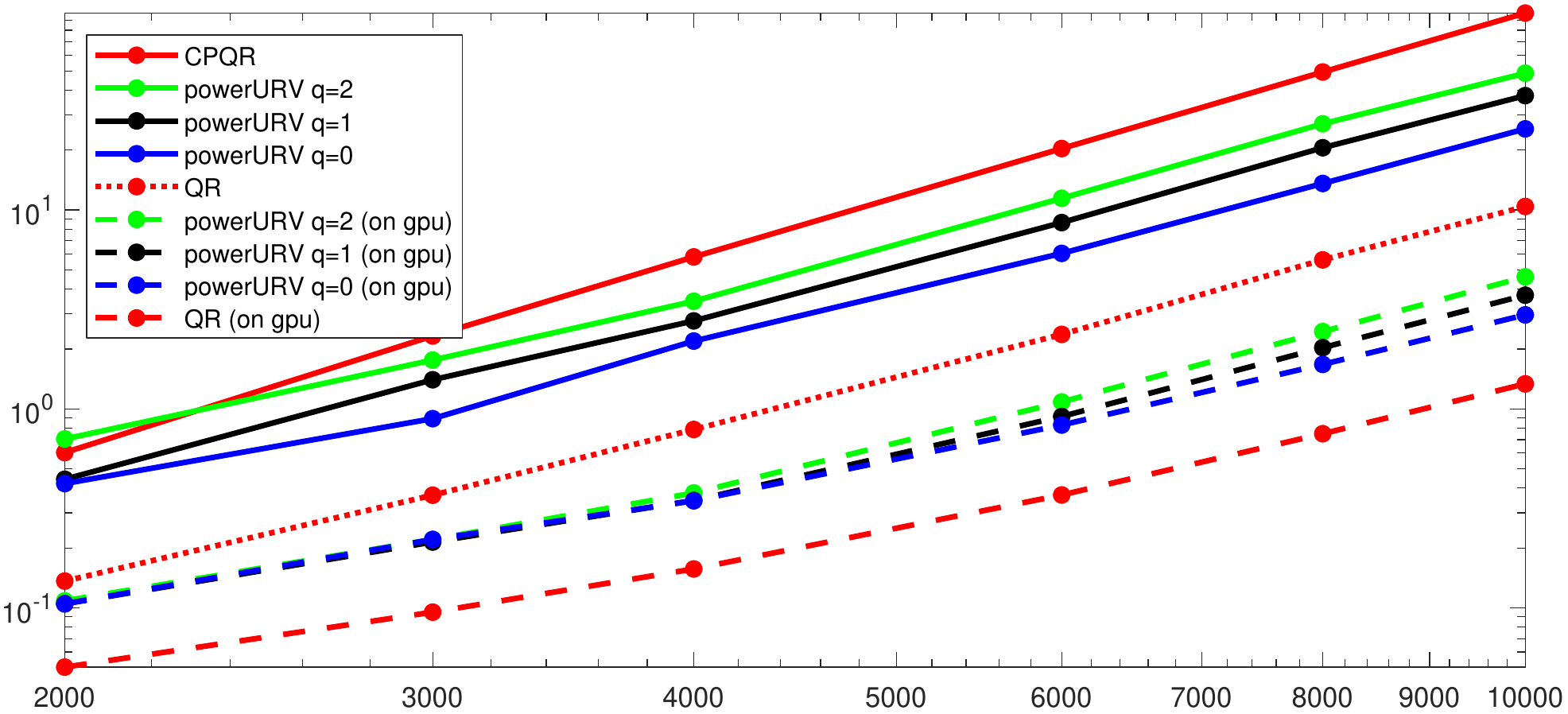}}
\put(00,04){\rotatebox{90}{\small\textit{Computational time in seconds}}}
\put(60,00){$n$}
\end{picture}
\caption{Computational times required for column-pivoted QR (CPQR) and unpivoted QR (QR)
of an $n\times n$ real matrix using MATLAB on an Intel i7-8700k CPU. We see that the unpivoted factorization is an
order of magnitude faster, despite having the identical asymptotic flop count. The graph
also shows the times required for the randomized rank-revealing factorization described in
Section \ref{sec:demmelURV}, executed both on a CPU (solid lines) and an Nvidia Titan V GPU (dashed lines).}
\label{fig:demmel_URV}
\end{figure}

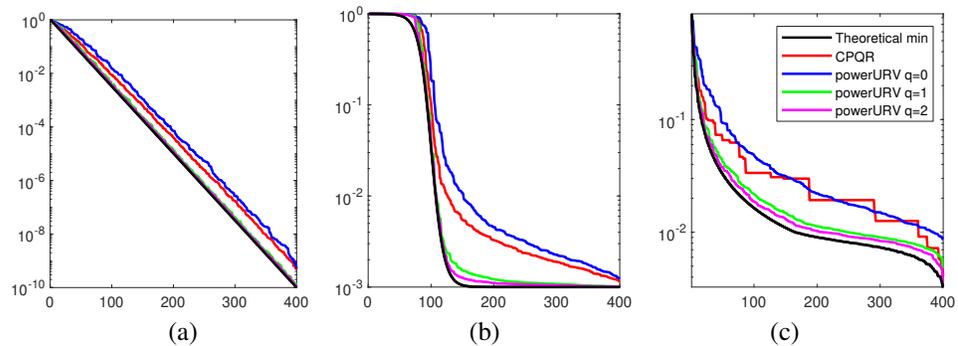
\begin{figure}
\centering
\setlength{\unitlength}{1mm}
\begin{picture}(123,45)
\put(-3,04){\includegraphics[width=40mm]{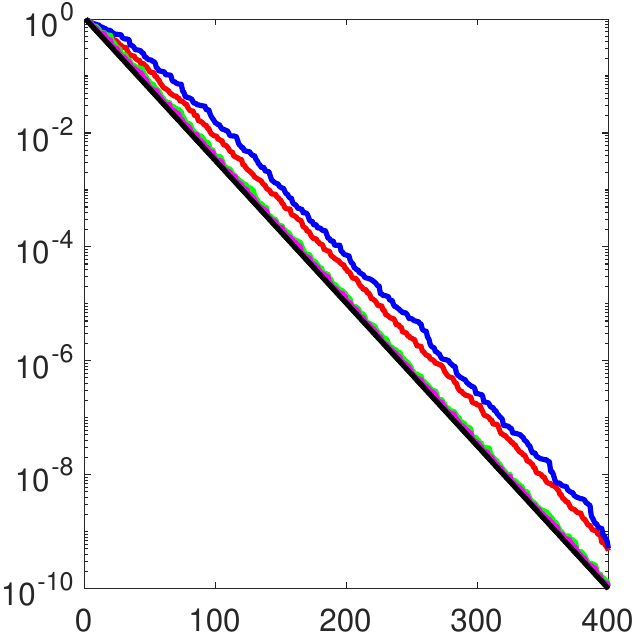}}
\put(40,04){\includegraphics[width=40mm]{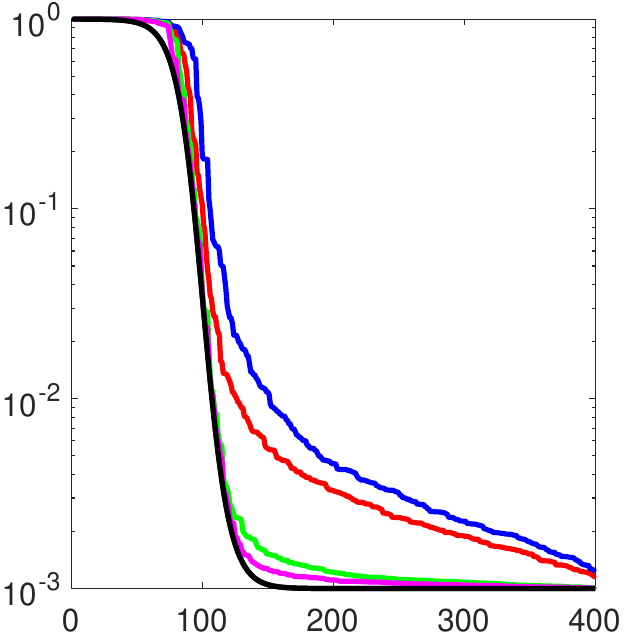}}
\put(83,04){\includegraphics[width=40mm]{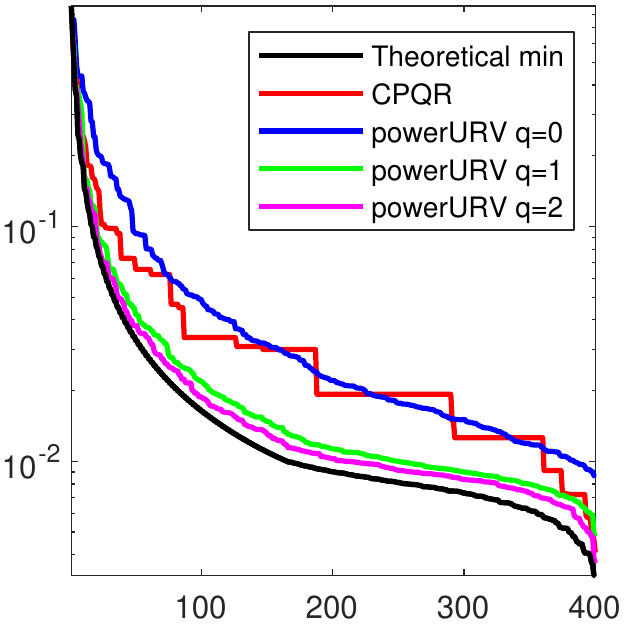}}
\put(18,00){(a)}
\put(58,00){(b)}
\put(98,00){(c)}
\end{picture}
\caption{The rank-revealing ability of CPQR and powerURV for different values
of the power parameter $q$, as discussed in Section \ref{sec:powerURV}.
The error
$e_{k} = \|\mtx{A} - \mtx{U}(:,1:k)\mtx{R}(1:k,:)\mtx{V}^{*}\|$ (see (\ref{eq:rankreveal})),
is plotted versus $k$ for three different matrices $\mtx{A}$ of size $400\times 400$.
The black lines plot the theoretical minimal values $\sigma_{k+1}$. (a) A matrix
whose singular values decay rapidly; we see that all methods perform well.
(b) A matrix whose singular values plateau; we see that CPQR performs poorly,
and so does the randomized method unless powering is used.
(c) A discretized boundary integral operator whose singular values decay slowly;
we again see the high precision of powerURV for $q=1$ and $q=2$.}
\label{fig:powerURVaccuracy}
\end{figure}

\subsection{Computing a rank-revealing factorization of an $n\times n$ matrix in less than $\bigO(n^{3})$ operations}
\label{sec:demmelURV}

The basic version of the randomized URV factorization algorithm described in Section \ref{sec:powerURV} was
originally proposed by \citeasnoun{2007_demmel_fast_linear_algebra_is_stable} for purposes loftier
than practical acceleration.
Indeed, randomization allows us to exploit fast matrix--matrix multiplication primitives
to design accelerated algorithms for other NLA problems, such as constructing rank-revealing factorizations.
The main point of this research is that, whenever the
fast matrix--matrix multiplication is stable, the computation
of a rank-revealing factorization is stable too.

Let us be more precise. \citeasnoun{2007_demmel_fast_linear_algebra_is_stable} embark from
the observation that there exist algorithms%
\footnote{The celebrated method of \citeasnoun{1969_strassen} has exponent $\omega = \log_{2}(7) = 2.807\cdots$.
It is a compelling algorithm in terms of both its numerical stability and its practical speed, even for modest
matrix sizes. More exotic algorithms, such as the Coppersmith--Winograd method and variants, attain
complexity of about $\omega \approx 2.37$, but they are not considered to be practically useful.}
for multiplying two $n\times n$ matrices using $\bigO(n^{\omega})$ flops, where $\omega < 3$.
Once such an algorithm is available, one can stably perform a whole range of other standard matrix operations
at the same asymptotic complexity. The idea is to apply a divide-and-conquer approach
that moves essentially all flops into the matrix--matrix multiplication. This approach
turns out to be relatively straightforward for decompositions that do not reveal the numerical
rank such as the unpivoted QR factorization.  It is harder to implement, however,
for (pivoted) rank-revealing factorizations.

\subsection{Classical column-pivoted QR}
\label{sec:classicCPQR}

The powerURV algorithm described in Section \ref{sec:powerURV} can be very
effective, but it operates on the whole matrix
at once, and it cannot be used to compute a partial factorization.
In the remainder of this section, we describe algorithms that build a
rank-revealing factorization incrementally.  These methods enjoy the property that the
factorization can be halted once a specified tolerance has been met.

We start off this discussion by reviewing a classical (deterministic)
method for computing a column-pivoted QR factorization.
This material is elementary, but the discussion serves to set
up a notational framework that lets us describe the
randomized version succinctly in Section \ref{sec:randCPQR}.
Suppose that we are given an $m\times n$ matrix $\mtx{A}$ with $m\geq n$.
We seek a factorization of the form
\begin{equation}
\label{eq:CPQRbasic}
\begin{array}{cccccccccc}
\mtx{A} &=& \mtx{Q} & \mtx{R} & \mtx{\Pi}^{*},\\
m\times n && m\times n & n\times n & n\times n
\end{array}
\end{equation}
where $\mtx{Q}$ has orthonormal columns,
where $\mtx{\Pi}$ is a permutation matrix, and
where $\mtx{R}$ is upper-triangular with diagonal elements that decay in magnitude
so that $|\mtx{R}(1,1)| \geq |\mtx{R}(2,2)| \geq |\mtx{R}(3,3)| \geq \cdots$.
The factors are typically built through a sequence of steps, where $\mtx{A}$
is driven to upper-triangular form one column at a time.

To be precise, we start by forming the matrix $\mtx{A}_{0} = \mtx{A}$.
Then we proceed using the iteration formula
$$
\mtx{A}_{j} = \mtx{Q}_{j}^{*}\mtx{A}_{j-1}\mtx{\Pi}_{j},
$$
where $\mtx{\Pi}_{j}$ is a permutation matrix that swaps the $j$th column of
$\mtx{A}_{j-1}$ with the column in $\mtx{A}_{j-1}(:,j:n)$ that has the largest
magnitude, and where $\mtx{Q}_{j}$ is a Householder reflector that zeros out
all elements below the diagonal in the $j$th column of $\mtx{A}_{j-1}\mtx{\Pi}_{j}$;
see Figure \ref{fig:classicCPQR}.
Once the process concludes, the relation \eqref{eq:CPQRbasic} holds for
$$
\mtx{Q}   = \mtx{Q}_{n}\mtx{Q}_{n-1}\mtx{Q}_{n-2}\dots\mtx{Q}_{1},\quad
\mtx{R}   = \mtx{A}_{n},\quad
\mtx{\Pi} = \mtx{\Pi}_{n}\mtx{\Pi}_{n-1}\mtx{\Pi}_{n-2}\dots\mtx{\Pi}_{1}.
$$

This algorithm is well understood, and it is ubiquitous in numerical computations.
For exotic matrices, it can produce factorizations that are quite far from optimal
\cite[Sec.~5]{1966_kahan_NLA}, but it typically works very well for many tasks.
For instance, it serves for revealing the numerical rank of a
matrix or for solving an ill-conditioned linear system.
However, a serious drawback to this algorithm is that
it fundamentally consists of a sequence of $n-1$ steps
(or $n$ steps if $m > n$),
where a large part of the matrix is updated in each step.

\begin{figure}
\centering
\begin{tabular}{cccc}
\includegraphics[width=20mm]{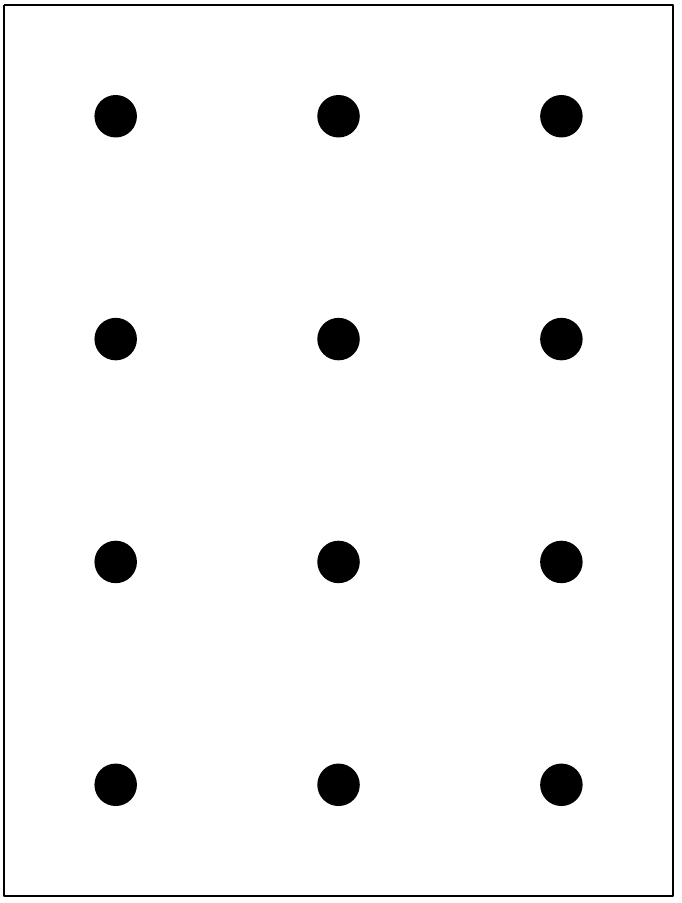}
&
\includegraphics[width=20mm]{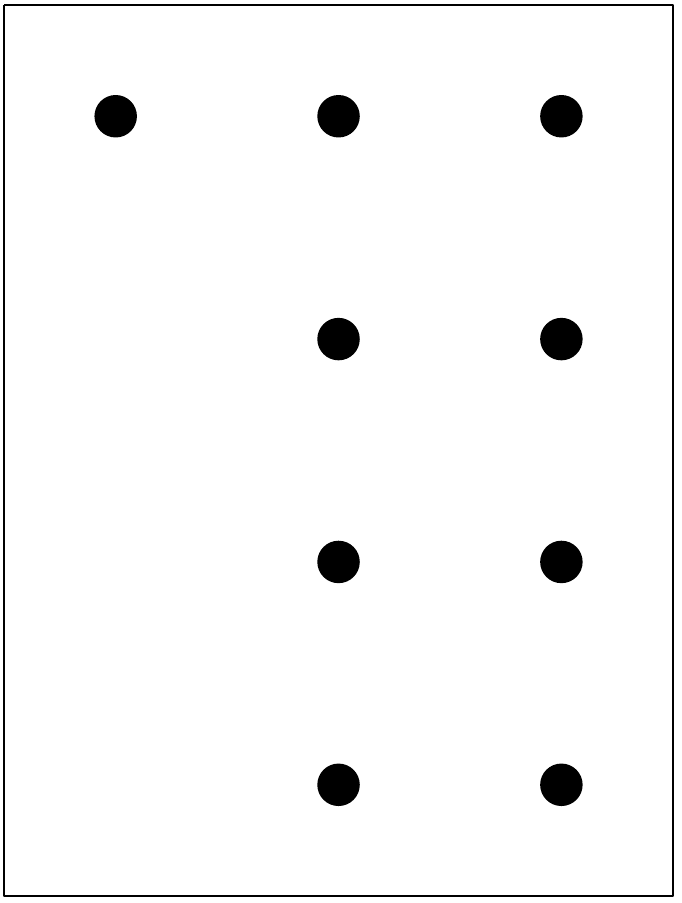}
&
\includegraphics[width=20mm]{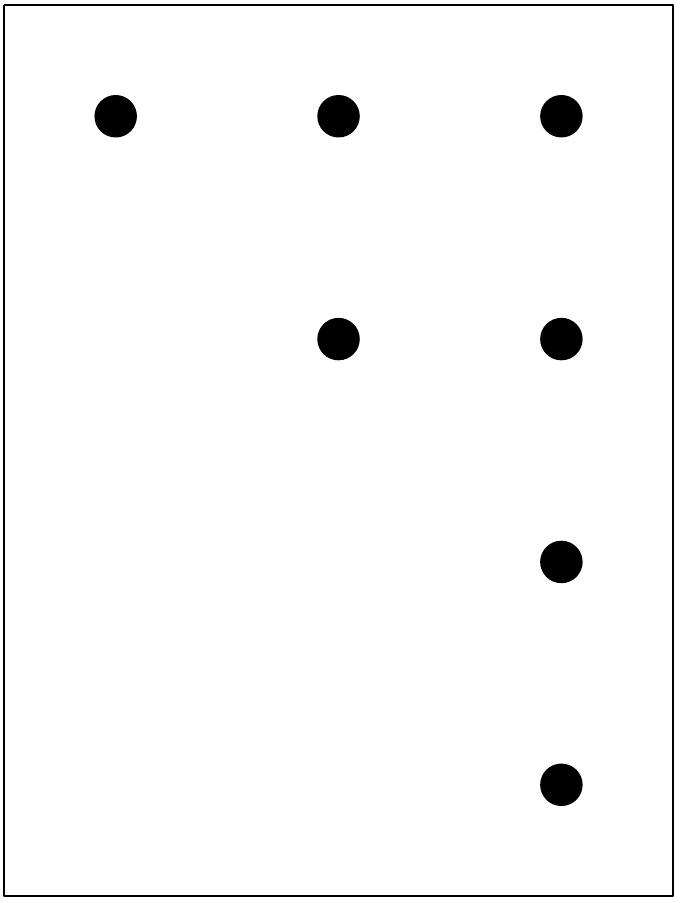}
&
\includegraphics[width=20mm]{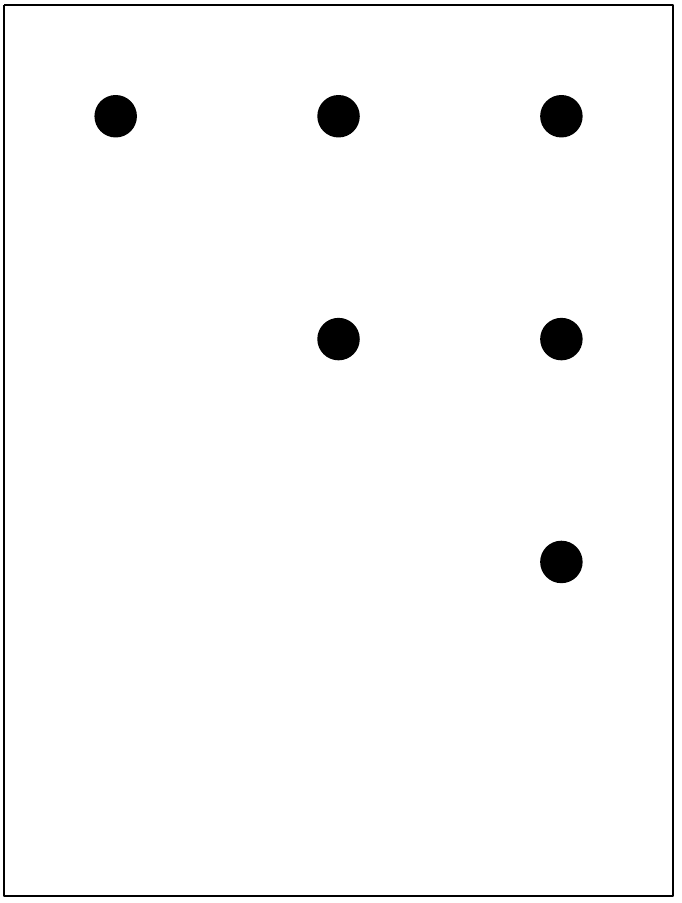}
\\
$\mtx{A}_{0} = \mtx{A}$
&
$\mtx{A}_{1} = \mtx{Q}_{1}^{*}\mtx{A}_{0}\mtx{\Pi}_{1}$
&
$\mtx{A}_{2} = \mtx{Q}_{2}^{*}\mtx{A}_{1}\mtx{\Pi}_{2}$
&
$\mtx{A}_{3} = \mtx{Q}_{3}^{*}\mtx{A}_{2}\mtx{\Pi}_{3}$
\end{tabular}
\caption{The figure shows the sparsity pattern of a $4\times 3$ matrix
$\mtx{A}$ as it is driven to upper-triangular form in the column-pivoted
QR factorization algorithm described in Section \ref{sec:classicCPQR}.
The process takes three steps in this case, and step $j$ involves the
application of a permutation matrix $\mtx{\Pi}_{j}$
from the right, and by a Householder reflector $\mtx{Q}_{j}$ from the left.}
\label{fig:classicCPQR}
\end{figure}

\subsection{A randomized algorithm for computing a CPQR decomposition}
\label{sec:randCPQR}

Our next objective is to recast the algorithm for computing a CPQR decomposition
that was introduced in Section \ref{sec:classicCPQR} so that it works with ``panels''
of $b$ contiguous columns, as shown in Figure \ref{fig:randCPQR}.
The difficulty is to find a set of $b$ pivot vectors without updating
the matrix between each selection. Fortuitously, the randomized algorithm for
interpolatory decomposition (Section~\ref{sec:randID}) is well adapted for
this task.
Indeed, a set of $b$ columns that forms a good basis for the column space also forms a
good set of pivot vectors.

To be specific, let us describe how to pick the first group of $b$ pivot columns for an
$m\times n$ matrix $\mtx{A}$. Adapting the ideas in Section \ref{sec:randID}, we draw a
Gaussian random matrix $\mtx{\Omega}$ of size $(b+p) \times m$, where $p$ is a small
oversampling parameter. We form a sample matrix
$$
\begin{array}{ccccccccccc}
\mtx{Y} &=& \mtx{\Omega} & \mtx{A}, \\
(b+p)\times n && (b+p) \times m & m\times n
\end{array},
$$
and then we execute $b$ steps of column-pivoted QR on the matrix $\mtx{Y}$
(either Householder or Gram--Schmidt is fine for this step).
The resulting $b$ pivot columns turn out to be good pivot columns for $\mtx{A}$
as well. Once these $b$ columns have been moved to
the front of $\mtx{A}$, we perform a local CPQR factorization of this panel.
We update the remaining $n-b$ columns using the computed Householder reflectors.

We could then proceed using exactly the same method a second time: draw a $(b+p) \times (m-b)$
Gaussian random matrix $\mtx{\Omega}$, form a $(b+p) \times (n-b)$ sample matrix $\mtx{Y}$,
perform classical CPQR on $\mtx{Y}$, and so on.  However, there is a shortcut.
We can update the sample matrix that was used in the first step,
which renders the overhead cost induced by randomization almost negligible
\cite[Sec.~4]{2017_blockQR_ming_article}.

\begin{figure}
\centering
\begin{tabular}{cccc}
\includegraphics[width=20mm]{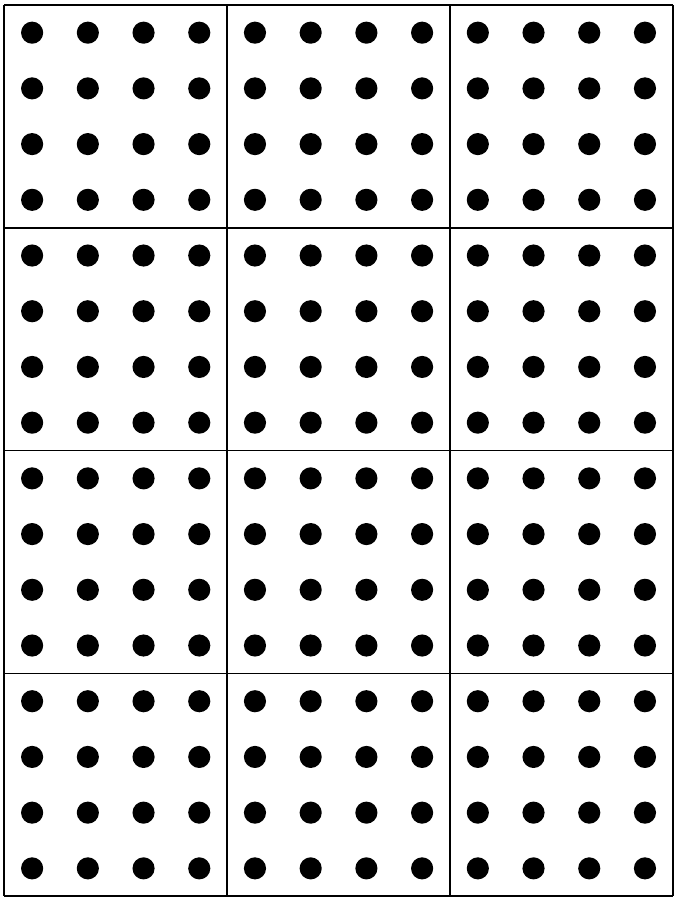}
&
\includegraphics[width=20mm]{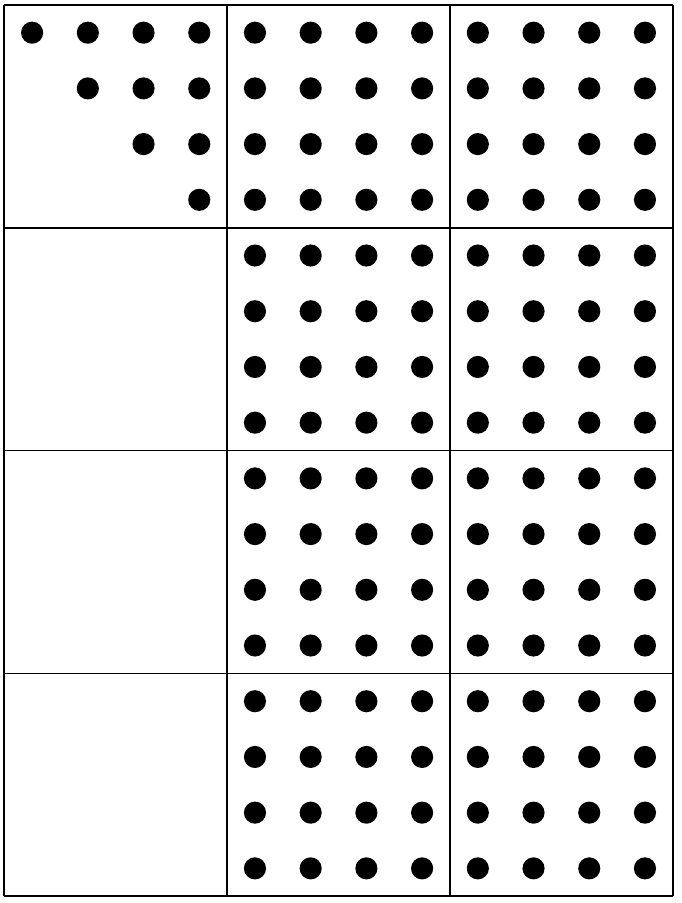}
&
\includegraphics[width=20mm]{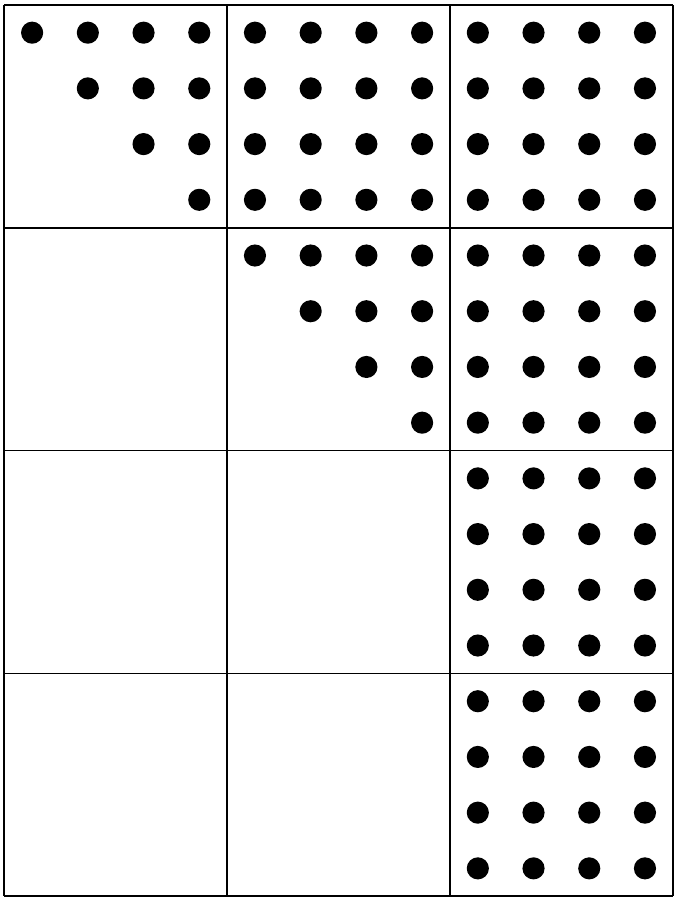}
&
\includegraphics[width=20mm]{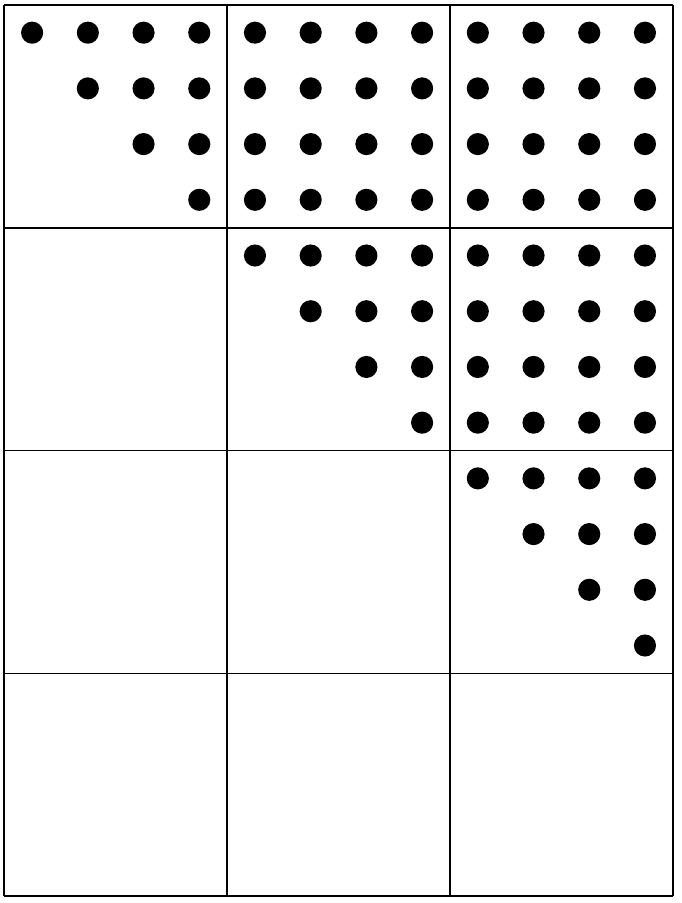}
\\
$\mtx{A}_{0} = \mtx{A}$
&
$\mtx{A}_{1} = \mtx{Q}_{1}^{*}\mtx{A}_{0}\mtx{\Pi}_{1}$
&
$\mtx{A}_{2} = \mtx{Q}_{2}^{*}\mtx{A}_{1}\mtx{\Pi}_{2}$
&
$\mtx{A}_{3} = \mtx{Q}_{3}^{*}\mtx{A}_{2}\mtx{\Pi}_{3}$
\end{tabular}
\caption{The sparsity pattern of a matrix $\mtx{A}$ consisting of $4\times 3$
blocks, each of size $3\times 3$, as it undergoes the blocked version of the Householder QR
algorithm described in Section \ref{sec:randCPQR}. Each matrix $\mtx{Q}_{j}$ is
a product of three Householder reflectors. The difficulty in building
an algorithm of this type is to find groups of pivot vectors \textit{before}
applying the corresponding Householder reflectors.}
\label{fig:randCPQR}
\end{figure}

\begin{figure}[b]
\centering
\setlength{\unitlength}{1mm}
\begin{picture}(100,54)
\put(07,05){\includegraphics[width=87mm]{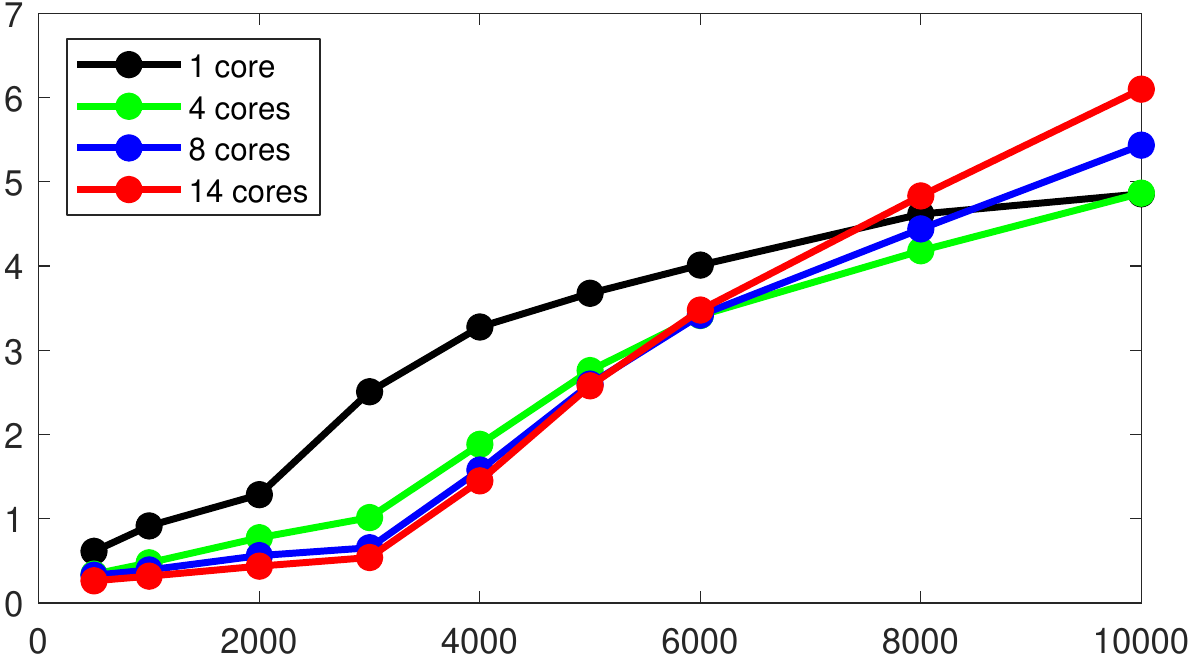}}
\put(40,00){\textit{Matrix size $n$.}}
\put(00,17){\rotatebox{90}{\textit{Speedup factor.}}}
\end{picture}
\caption{Speedup of the randomized algorithm for computing a column-pivoted QR decomposition
described in Section \ref{sec:randCPQR}, relative to LAPACK's faster routine (dgeqp3) as implemented
in the Intel MKL library (version 11.2.3), running on an Intel Xeon E5-2695 v3 processor.}
\label{fig:HQRRPspeedup}
\end{figure}

Extensive numerical work has demonstrated dramatic acceleration
over deterministic algorithms. Figure \ref{fig:HQRRPspeedup} draws on data from
\citeasnoun{2015_blockQR_SISC} that illustrates the acceleration over a state-of-the-art
software implementation of the classical CPQR method.
Computer experiments also show that the randomized scheme
chooses pivot columns whose quality is almost indistinguishable
from those chosen by traditional pivoting, in the sense that the
relation \eqref{eq:rankreveal} holds to about the same accuracy.
(However, the diagonal entries of $\mtx{R}$ do not strictly decay
in magnitude across the block boundaries.)

To understand the behavior of the algorithm, it is helpful to think about two extreme cases.
In the first, suppose that the singular values of $\mtx{A}$ decay very rapidly. Here,
the analysis in Section \ref{sec:random-rangefinder} can be modified to show that, for any $j$,
the first $j$ pivot columns chosen by the randomized algorithm is likely to span the column space nearly
as well as the optimal set of $j$ columns.  Therefore, they are excellent pivot vectors.
At the other extreme, suppose that the singular values of $\mtx{A}$ hardly decay at all.
In this case, the randomized method may pick a completely different set of pivot vectors
than the deterministic method, but this outcome is unproblematic because we can take any
group of columns as pivot vectors.
Of course, the interesting cases are intermediate between these two extremes.  It turns
out that the randomized methods work well regardless of how rapidly the singular values decay.
For a detailed analysis, see \cite{2017_blockQR_ming_article,2017_gu_langou,2015_gu_melgaard,2018_gu_flipflop}.

\begin{remark}[History]
Finding a blocked version of the CPQR algorithm described in Section \ref{sec:classicCPQR}
has remained an open challenge in NLA for some time \cite{1998_quintana_panelRRQR1,2015_demmel_communication_avoiding}.
The randomized technique described in this section was introduced in
\citeasnoun{2015_blockQR}, while the updating technique was
described in \citeasnoun{2015_blockQR_ming}. For full details, see
\cite{2015_blockQR_SISC} and \cite{2017_blockQR_ming_article}.
\end{remark}

\subsection{A randomized algorithm for computing a URV decomposition}
\label{sec:randUTV}

In this section, we describe an incremental randomized algorithm for computing the
URV factorization (\ref{eq:fullfactintro}). Let us recall that for $\mtx{A} \in \F^{m\times n}$,
with $m \geq n$, this factorization takes the form
\begin{equation}
\label{eq:UTVrepeat}
\begin{array}{cccccccccc}
\mtx{A} &=& \mtx{U} & \mtx{R} & \mtx{V}^{*},\\
m\times n && m\times n & n\times n & n\times n
\end{array}
\end{equation}
where $\mtx{U}$ and $\mtx{V}$ have orthonormal columns and where $\mtx{R}$ is
upper-triangular.

The algorithm we describe is blocked, and executes efficiently on
modern computing platforms. It is similar in speed to the randomized CPQR
described in Section \ref{sec:randCPQR}, and it shares the advantage that the
decomposition is built incrementally so that the process can be stopped once
a requested accuracy has been met. However, the URV factorization offers
compelling advantages: (1) It is almost as good at revealing the numerical rank
as the SVD (unlike the CPQR). (2) The URV factorization provides us with orthonormal
basis vectors for both the column and the row spaces. (3) The off-diagonal entries
of $\mtx{R}$ are very small in magnitude. (4) The diagonal entries of
$\mtx{R}$ form excellent approximations to the singular values of $\mtx{A}$.

The randomized algorithm for computing a URV factorization we present follows
the same algorithmic template as the randomized CPQR described in Section \ref{sec:randCPQR}.
It drives $\mtx{A}$ to upper-triangular form one block at a time, but it
replaces the permutation matrices $\mtx{\Pi}_{j}$ in the CPQR with general
unitary matrices $\mtx{V}_{j}$.
Using this increased freedom, we can obtain a factorization
where all diagonal blocks are themselves diagonal matrices
and all off-diagonal elements have small magnitude.
Figure \ref{fig:blockUTV} summarizes the process.

\begin{figure}
\setlength{\unitlength}{1mm}
\begin{picture}(125,40)
\put(000,04){\includegraphics[width=24mm]{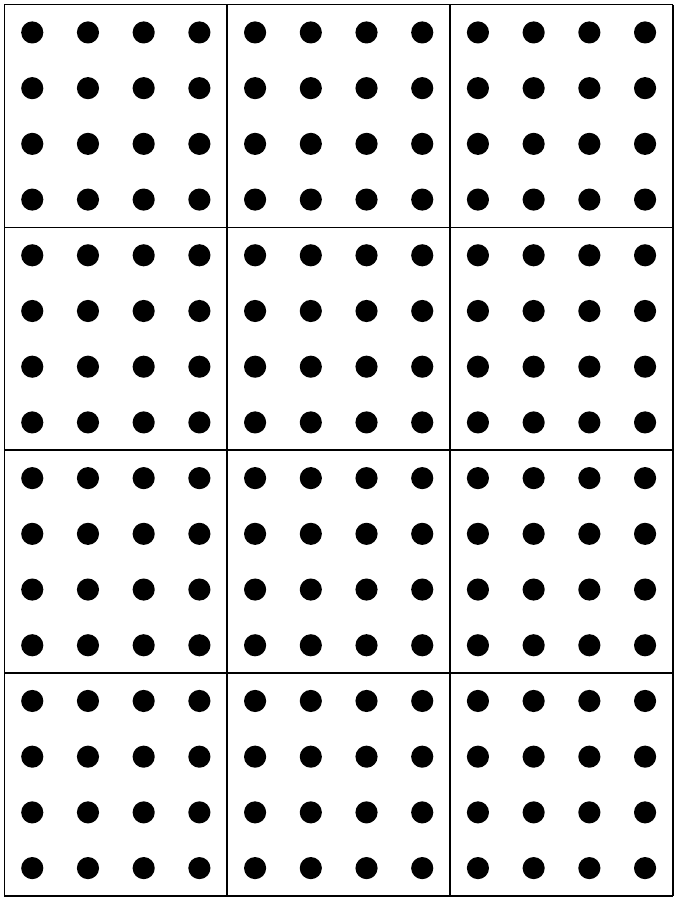}}
\put(032,04){\includegraphics[width=24mm]{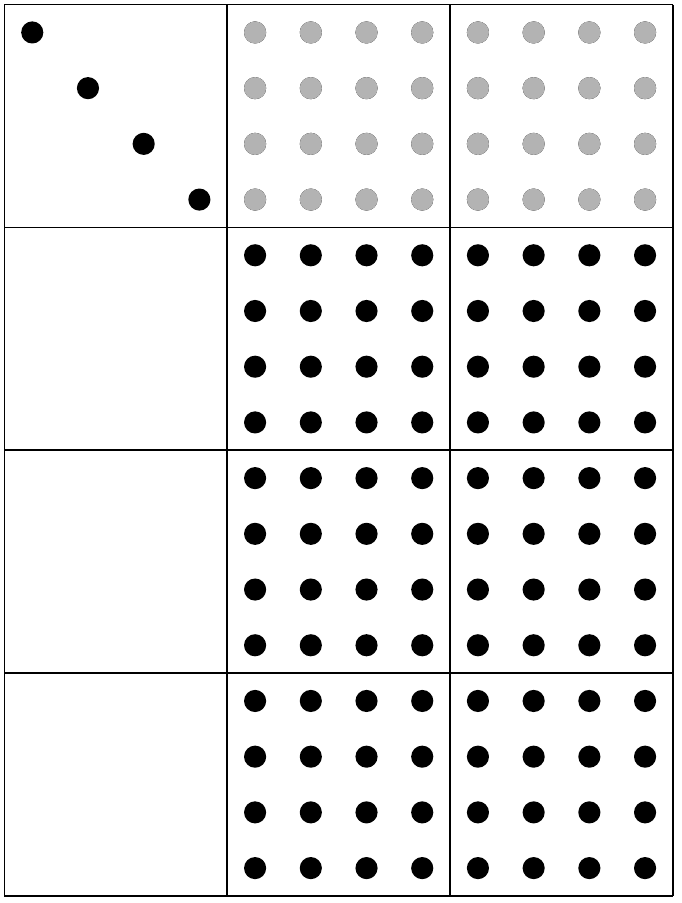}}
\put(064,04){\includegraphics[width=24mm]{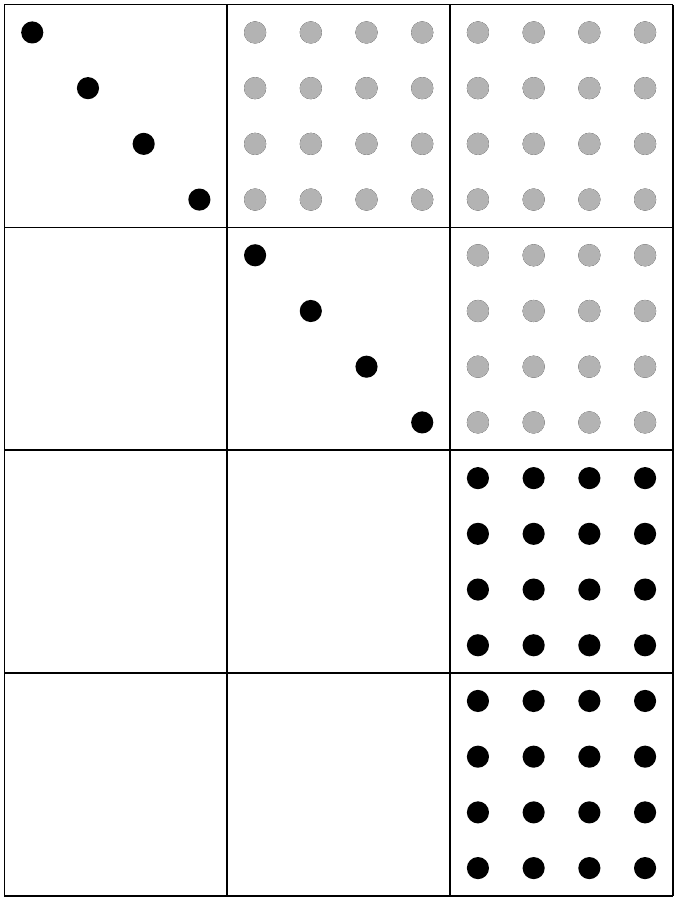}}
\put(096,04){\includegraphics[width=24mm]{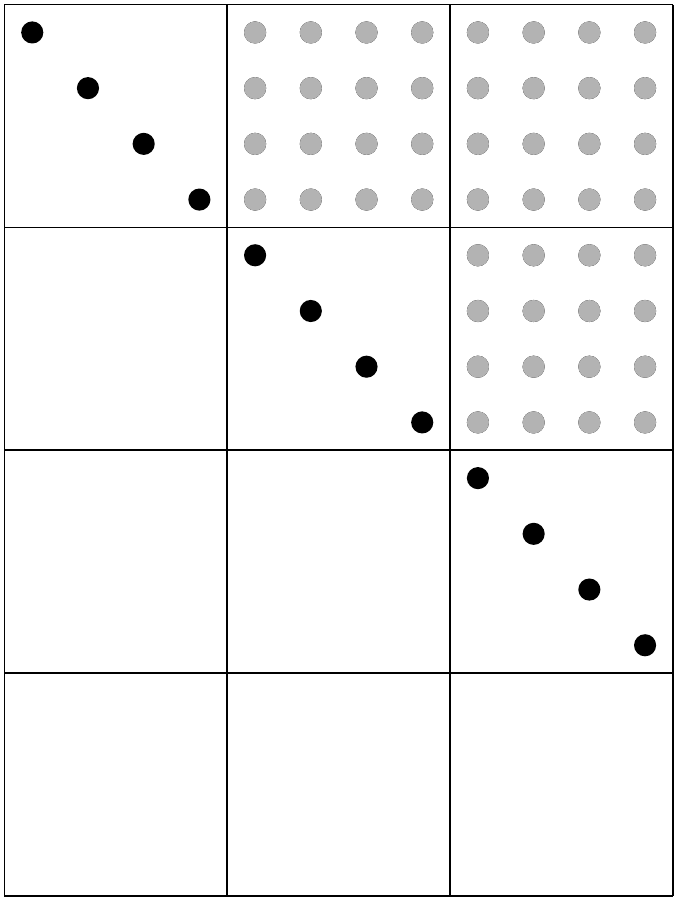}}
\put(007,00){\footnotesize $\mtx{A}_{0} = \mtx{A}$}
\put(035,00){\footnotesize $\mtx{A}_{1} = \mtx{U}_{1}^{*}\mtx{A}_{0}\mtx{V}_{1}$}
\put(067,00){\footnotesize $\mtx{A}_{2} = \mtx{U}_{2}^{*}\mtx{A}_{1}\mtx{V}_{2}$}
\put(099,00){\footnotesize $\mtx{A}_{3} = \mtx{U}_{3}^{*}\mtx{A}_{2}\mtx{V}_{3}$}
\end{picture}
\caption{Sparsity pattern of a matrix being driven to upper-triangular form in
the randomized URV factorization algorithm described in Section \ref{sec:randUTV};
see~Figure \ref{fig:randCPQR}. The matrices $\mtx{U}_{i}$ and $\mtx{V}_{i}$ are
now more general unitary transforms (consisting in the bulk of Householder
reflectors). The entries shown as gray are nonzero,
but are typically very small in magnitude.}
\label{fig:blockUTV}
\end{figure}
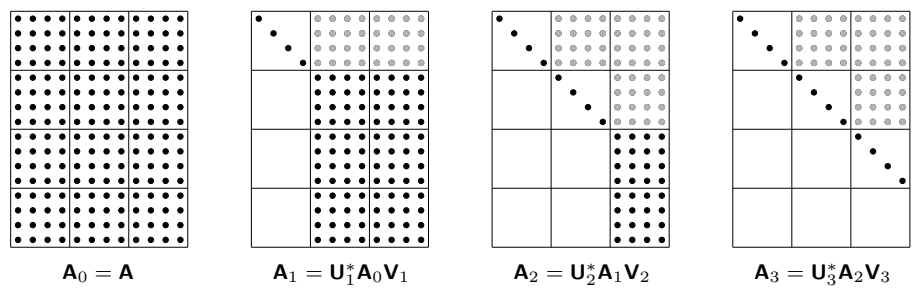

To provide details on how the algorithm works, suppose that we are given
an $m\times n$ matrix $\mtx{A}$ and a block size $b$. In the first step of the process,
our objective is then to build unitary matrices $\mtx{U}_{1}$ and $\mtx{V}_{1}$ such that
$$
\mtx{A} = \mtx{U}_{1}\mtx{A}_{1}\mtx{V}_{1}^{*},
$$
where $\mtx{A}_{1}$ has the block structure
$$
\mtx{A}_{1} =
\left[\begin{array}{cc}
\mtx{A}_{1,11} & \mtx{A}_{1,12} \\
\mtx{0}        & \mtx{A}_{1,22}
\end{array}\right] =
\raisebox{-9mm}{\includegraphics[height=18mm]{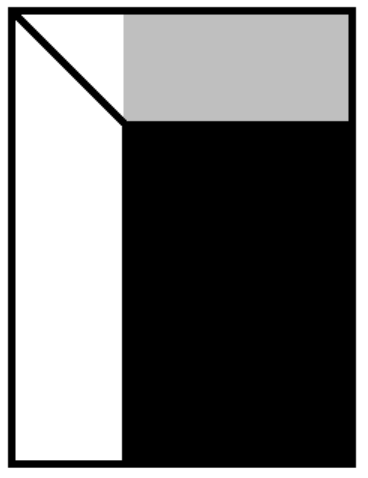}},
$$
so that the $b\times b$ matrix $\mtx{A}_{1,11}$ is diagonal and the entries of $\mtx{A}_{1,12}$ are
small in magnitude. To build $\mtx{V}_{1}$, we use the randomized power
iteration described in Section \ref{sec:rrf-subspace} to find a basis that approximately spans
the same space as the top $b$ right singular vectors of $\mtx{A}$. To be
precise, we form the sample matrix
$$
\mtx{Y} = \mtx{\Omega}\mtx{A}\bigl(\mtx{A}^{*}\mtx{A}\bigr)^{q},
$$
where $\mtx{\Omega}$ is a Gaussian random matrix of size $b \times m$ and
where $q$ is a parameter indicating the number of steps of power iteration
taken. We then perform an unpivoted QR factorization of the \textit{rows} of
$\mtx{Y}$ to form a matrix $\widetilde{\mtx{V}}$ whose first $b$ columns form
an orthonormal basis for the column space of $\mtx{Y}$. We then apply $\widetilde{\mtx{V}}$
from the right to form the matrix $\mtx{A}\widetilde{\mtx{V}}$, and we perform an unpivoted
QR factorization of the first $b$ columns of $\mtx{A}\widetilde{\mtx{V}}$. This results
in a new matrix
$$
\widetilde{\mtx{A}} = \bigl(\widetilde{\mtx{U}}\bigr)^{*}\,\mtx{A}\,\widetilde{\mtx{V}}
$$
that has the block structure
$$
\widetilde{\mtx{A}} =
\left[\begin{array}{cc}
\widetilde{\mtx{A}}_{1,11} & \widetilde{\mtx{A}}_{1,12} \\
\mtx{0}                & \widetilde{\mtx{A}}_{1,22}
\end{array}\right] =
\raisebox{-9mm}{\includegraphics[height=18mm]{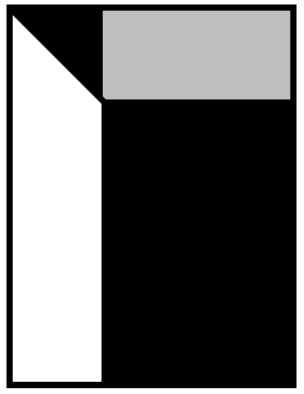}}.
$$
The top left $b\times b$ block $\widetilde{\mtx{A}}_{1,11}$ is upper-triangular,
and the bottom left block is zero. The entries of $\widetilde{\mtx{A}}_{1,12}$ are typically small in magnitude.
Next, we compute a full SVD of the block $\widetilde{\mtx{A}}_{1,11}$:
$$
\widetilde{\mtx{A}}_{1,11} = \widehat{\mtx{U}}\mtx{D}_{11}\widehat{\mtx{V}}^{*}.
$$
This step is inexpensive because $\widetilde{\mtx{A}}_{1,11}$ has size $b\times b$, where $b$ is small.
As a final step, we form the transformation matrices
$$
\mtx{U}_{1} =
\widetilde{\mtx{U}}
\left[\begin{array}{cc}
\widehat{\mtx{U}} & \mtx{0} \\
\mtx{0} & \mtx{I}_{m-b}
\end{array}\right],
\qquad\mbox{and}\qquad
\mtx{V}_{1} =
\widetilde{\mtx{V}}
\left[\begin{array}{cc}
\widehat{\mtx{V}} & \mtx{0} \\
\mtx{0} & \mtx{I}_{n-b}
\end{array}\right]
$$
and set
$$
\mtx{A}_{1} = \mtx{U}_{1}^{*}\mtx{A}\mtx{V}_{1}.
$$
The result of this process is that the diagonal entries of $\mtx{D}_{11}$ typically
give accurate approximations to the first $b$ singular values of $\mtx{A}$, and the
``remainder'' matrix $\mtx{A}_{1,22}$ has spectral norm that is similar to $\sigma_{b+1}$.
Thus,
$$
\|\mtx{A}_{1,22}\| \approx \inf\{\|\mtx{A} - \mtx{B}\|\,\colon\,\mtx{B}\mbox{ has rank }b\}.
$$
This process corresponds to the first step in Figure \ref{fig:blockUTV}.
The succeeding iterations execute the same process, at each step working
on the remaining lower-right part of the matrix that has not yet been driven to upper-triangular form.
We refer to \citeasnoun{2017_martinsson_UTV} and \citeasnoun{2018_PCMI_martinsson} for details.

From a theoretical point of view, the first step of the URV factorization described is well
understood since it is mathematically equivalent to the randomized power iteration described
in Section \ref{sec:rrf-subspace}. We observe that the first $b$ columns of $\mtx{U}_{1}$ do a better job of
spanning the column space of $\mtx{A}$ than the first $b$ columns of $\mtx{V}_{1}$ do for
spanning the row space; the reason for this asymmetry is that by forming the product
$\mtx{A}\mtx{V}_{1}$, we in effect perform an additional step of the power iteration
\cite[Sec.~6]{2018_martinsson_powerurv}.

An important feature of the method described in this section is that it is incremental, and
it can be halted once a given computational tolerance has been met. This feature has been a key
competitive advantage of the column-pivoted QR decomposition, and it is often cited as the motivation for
using CPQR. The method described in this section has almost all the advantages of the randomized
Householder CPQR factorization (it is blocked, it is incremental, and it executes very fast in
practice), while resulting in a factorization that is far closer to optimal in revealing the
rank.

\begin{remark}[Related work]
The idea of loosening the requirements on the factors in a rank-revealing factorization and
searching for a decomposition such as (\ref{eq:UTVrepeat}) is well explored in the literature
\cite{1999_hansen_UTVtools,1994_stewart_UTV,1995_elden_downdate_UTV}
and \cite{1998_stewart_volume1}.
Deterministic techniques
for computing the URV decomposition are described in \cite{1999_hansen_UTVtools,1999_stewart_QLP};
these algorithms combine some of the
appealing qualities of the SVD (high accuracy in revealing the rank) with some of the
appealing qualities of CPQR (the possibility of halting the execution once a requested tolerance
has been met). However, they were not blocked, and therefore they were subject to the same liabilities as
deterministic algorithms for computing the SVD and the CPQR. A more recent use of randomization
in this context is described in \cite{2018_gu_flipflop}.
\end{remark}

\section{General linear solvers}
\label{sec:linear-solve}

Researchers are currently exploring randomized algorithms for
solving linear systems, such as
\begin{equation}
\label{eq:Ax=b}
\mtx{A}\vct{x} = \vct{b},
\end{equation}
where $\mtx{A}$ is a given coefficient matrix and $\vct{b}$ is a given vector.
This section describes a few probabilistic approaches for solving \eqref{eq:Ax=b}.
For the most part, we restrict our attention to the case where $\mtx{A}$ is square
and the system is consistent, but we will also touch on linear regression problems. %
Research on randomized linear solvers has not progressed as rapidly as some other
areas of randomized NLA, so the discussion here is more preliminary
than other parts of this survey.

\subsection{Background: Iterative solvers}

It is important to keep in mind that existing iterative solvers often
work exceptionally well. Whenever $\mtx{A}$ is well-conditioned or, more generally, whenever
its spectrum is ``clustered,'' Krylov solvers such as the conjugate gradient (CG) algorithm
or GMRES tend to converge very rapidly.  For practical purposes, the cost of solving \eqref{eq:Ax=b}
is no larger than the cost of a handful of matrix--vector multiplications with $\mtx{A}$.
In terms of speed, it is very difficult to beat these techniques.
Consequently, we focus on the cases where known iterative methods converge slowly
and where we cannot deploy standard preconditioners to resolve the problem.

Having limited ourselves to this situation, %
the choice of solver for \eqref{eq:Ax=b} will depend on properties of the coefficient matrix:
Is it dense or sparse?  Does it fit in RAM?  Do we have access
to individual matrix entries?  Can we apply $\mtx{A}$ to a vector?
We will consider several of these environments.

\subsection{Accelerating solvers based on dense matrix factorizations}
\label{sec:parkeretc}

As it happens, one of the early examples of randomization in NLA
was a method for accelerating the solution of a dense linear system \eqref{eq:Ax=b}.
\citeasnoun{1995_parker_randombutterfly} observed that
we can precondition a linear system by left and right multiplying
the coefficient matrix by random unitary matrices $\mtx{U}$ and $\mtx{V}$.
With probability $1$, we can solve the resulting system
\begin{equation}
\label{eq:Ax=b_precond}
\bigl(\mtx{U}\mtx{A}\mtx{V}^{*}\bigr)\,\bigl(\mtx{V}\vct{x}\bigr) = \mtx{U}\vct{b}
\end{equation}
by Gaussian elimination \emph{without pivoting}.
More precisely, Parker proved that, almost surely, blocked Gaussian elimination will not
encounter a degenerate diagonal block.

Blocked Gaussian elimination without pivoting is substantially faster than
ordinary Gaussian elimination for two reasons: matrix operations
are more efficient than vector operations on modern computers,
and we avoid the substantial communication costs that arise when
we search for pivots.  (Section~\ref{sec:blocking} contains
more discussion about blocking.)
Parker also observed that structured random matrices
(such as the randomized trigonometric transforms from Section~\ref{sec:srtt})
allow us to perform the preconditioning step at lower cost
than the subsequent Gaussian elimination procedure.

\citeasnoun{1995_parker_randombutterfly} inspired many subsequent papers,
including
\cite{2014_li_random_butterfly_pivoting,2017_trogdon_random_butterfly,DGHL12:Communication-Optimal-Parallel,2017_baboulin_GPU} and \cite{2017_pan_randomized_gaussian_elim}.
Another related direction concerns the smoothed analysis of Gaussian elimination
undertaken in \cite{SST06:Smoothed-Analysis}.

As we saw in Section~\ref{sec:full}, randomization can be used to accelerate
the computation of rank-revealing factorizations of the matrix $\mtx{A}$.
In this context, randomness allows us to block the factorization method,
which increases its practical speed, even though the overall arithmetic cost remains at $\bigO(n^3)$.
Randomized rank-revealing factorizations are ideal for solving ill-conditioned
linear systems because they allow the user to stabilize the computation
by avoiding subspaces associated with small singular values.

For instance, suppose that we have computed a singular value decomposition (SVD):
$$
\mtx{A} =
\mtx{U}\mtx{D}\mtx{V}^{*} =
\sum_{j=1}^{n}\sigma_{j}\,\vct{u}_{j}\vct{v}_{j}^{*}.
$$
Let us introduce a truncation parameter $\eps$ and ignore all singular
modes where $\sigma_j \leq \eps$.  Then the stabilized solution to \eqref{eq:Ax=b}
is
$$
\vct{x}_{\eps} = \sum_{j \,:\, \sigma_{j} > \eps}\frac{1}{\sigma_{j}}\,\vct{v}_{j}\vct{u}_{j}^{*}\vct{b}.
$$
By allowing the residual
to take a nonzero value, we can ensure that $\vct{x}_\eps$ does not include large
components that contribute little toward satisfying the original equation.
The randomized URV decomposition, described in Section~\ref{sec:randUTV},
can also be used for stabilization, and we can compute it much faster than
an SVD.

\begin{remark}[Are rank-revealing factorizations needed?]
In some applications, computing a rank-revealing factorization is overkill for
purposes of solving the linear system (\ref{eq:Ax=b}). In particular, if we compute an unpivoted QR
decomposition of $\mtx{A}$, then it is easy to block both the factorization and the solve stages so that
very high speed is attained. This process is provably backwards stable, which is sometimes all
that is needed. (In practice, partially pivoted LU can often be used in an analogous manner, despite
being theoretically unstable.)

In contrast, when the actual entries of the computed solution $\vct{x}_{\rm approx}$ matter
(as opposed to the value of $\mtx{A}\vct{x}_{\rm approx}$),
a stabilized solver is generally preferred.
As a consequence, column-pivoted QR is often cited as a method of choice
for ill-conditioned problems in situations where an SVD is not affordable.
\end{remark}

\begin{remark}[Strassen accelerated solvers]
We saw in Section \ref{sec:demmelURV} that randomization has enabled us to compute a
rank-revealing factorization
of an $n\times n$ matrix in less than $\mathcal{O}(n^{3})$ operations. The idea was to use
randomized preconditioning as in (\ref{eq:Ax=b_precond}), and then accelerate an unpivoted
factorization of the resulting coefficient matrix using fast algorithms for the matrix-matrix
multiplication such as Strassen \cite{2007_demmel_fast_linear_algebra_is_stable}.
This methodology can of course be immediately applied to the
task of solving ill-conditioned linear systems.
For improved numerical stability, a few steps of power iteration can be incorporated to this approach; see \eqref{eq:powerURV}.
\end{remark}

\subsection{Sketch and precondition}
\label{sec:sketchtoprecond}

Another approach to preconditioning is to look for a random transformation
of the linear system that makes an iterative linear solver converge more quickly.
Typically, these preconditioning transforms need to cluster the eigenvalues
of the matrix.

The most successful example of this type of randomized preconditioning
does not concern square systems, but rather highly overdetermined least-squares problems.
See Section~\ref{sec:sketchandprecond} \textit{et seq.}~for a discussion of this idea.
This type of randomized preconditioning can greatly enhance the robustness and power of
``asynchronous'' solvers for communication-constrained environments \cite{2015_avron_revisiting}.
Related techniques for kernel ridge regression are described in \citeasnoun{2017_avron_kernel_ridge_regression}.
For linear systems involving high-dimensional tensors, see \citeasnoun{2016_kressner_precond_tensor}.

For square linear systems, the search for randomized preconditioners has been less fruitful.
Section~\ref{sec:sparse-cholesky} outlines the main success story.  Nevertheless, techniques
already at hand can be very helpful for solving linear systems in special situations,
which we illustrate with a small example.

Consider the task of solving \eqref{eq:Ax=b} for a positive-definite (PD) coefficient matrix $\mtx{A}$.
In this environment, the iterative method of choice is the
conjugate gradient (CG) algorithm~\cite{1952_hestenes_CG}.
A detailed convergence analysis for CG is available; for example, see~\cite[Sec.~38]{1997_trefethen_bau}.
In a nutshell, CG converges rapidly when the eigenvalues of $\mtx{A}$ are clustered,
as in Figure~\ref{fig:eval_distributions_for_CG}(a).  Therefore, our task is to find
a matrix $\mtx{M}$ for which $\mtx{M}^{-1}$ can be applied rapidly to vectors and
for which $\mtx{M}^{-1/2} \mtx{A} \mtx{M}^{-1/2}$ has a tightly clustered spectrum.

In a situation where $\mtx{A}$ has a few eigenvalues that are larger than the others
(Figure~\ref{fig:eval_distributions_for_CG}(b)),
randomized algorithms for low-rank approximation provide excellent preconditioners.
For instance, we can use the randomized Nystr{\"o}m method (Section~\ref{sec:nystrom})
to compute an approximation
\begin{equation}
\label{eq:spd_exact_lowrank}
\mtx{A} \approx \mtx{U}\mtx{D}\mtx{U}^{*}
\end{equation}
where $\mtx{D} \in \RR_{+}^{k\times k}$ is a diagonal matrix whose entries hold approximations to the largest $k$
eigenvalues of $\mtx{A}$, and where $\mtx{U} \in \F^{m\times k}$ is an orthonormal matrix holding the corresponding
approximate eigenvectors.  We then form a preconditioner for $\mtx{A}$ by setting
$$
\mtx{M} = (1/\alpha)\,\mtx{U}\mtx{D}\mtx{U}^{*} + \bigl(\mtx{I} - \mtx{U}\mtx{U}^{*}\bigr).
$$
It is trivial to invert $\mtx{M}$ because
$\mtx{M}^{-1} = \alpha\mtx{U}\mtx{D}^{-1}\mtx{U}^{*} + \bigl(\mtx{I} - \mtx{U}\mtx{U}^{*}\bigr)$.
Now, if \eqref{eq:spd_exact_lowrank} captured the top $k$ eigenmodes of $\mtx{A}$ exactly, then the
preconditioned coefficient matrix $\mtx{M}^{-1/2}\mtx{A}\mtx{M}^{-1/2}$ would have the same eigenvectors as $\mtx{A}$,
but with the top $k$ eigenvalues replaced by $\alpha$ and the remaining eigenvalues unchanged.
By setting $\alpha = \lambda_{k}$, say, the spectrum of $\mtx{M}^{-1/2}\mtx{A}\mtx{M}^{-1/2}$
would become far more tightly clustered. In reality, the columns of $\mtx{U}$ do not exactly align
with the eigenvectors of $\mtx{A}$.  Even so, the accuracy will be good for the eigenvectors
associated with the top eigenvalues, which is what matters.

\begin{figure}
\begin{tabular}{p{27mm}p{27mm}p{27mm}p{27mm}}
\includegraphics[width=30mm]{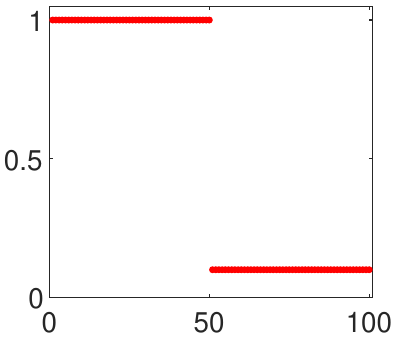} &
\includegraphics[width=30mm]{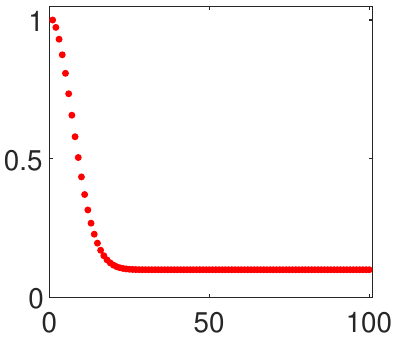} &
\includegraphics[width=30mm]{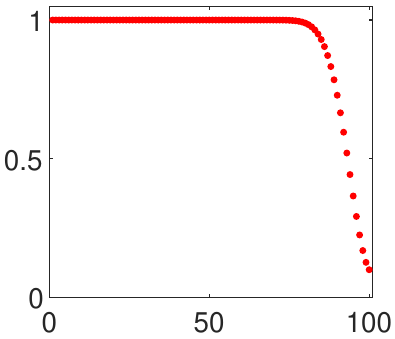} &
\includegraphics[width=30mm]{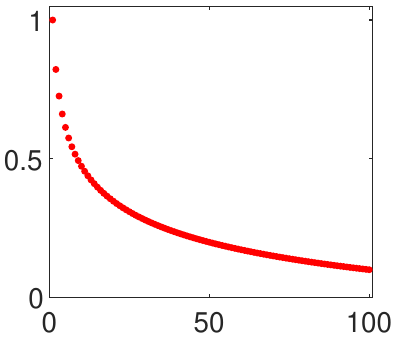} \\
\centering (a) & \centering (b) & \centering (c) & \centering (d)
\end{tabular}
\caption{The eigenvalues of four different PD matrices that all have condition number $10$
(since $\lambda_{\rm max}=1$ and $\lambda_{\rm min}=0.1$). As discussed in Section
\ref{sec:sketchtoprecond}, the difficulty of solving the corresponding linear systems using
conjugate gradients differ significantly between these cases.
(a) For this matrix, CG converges in two iterations, without the need for preconditioners.
(b) When the spectrum has some large outliers, the randomized preconditioner outlined in Section \ref{sec:sketchtoprecond}
works well. (c,d) Finding randomized preconditioners for matrices with spectra like these remains an open research problem.}
\label{fig:eval_distributions_for_CG}
\end{figure}

\subsection{The randomized Kaczmarz method and its relatives}
\label{sec:rk}

The Kaczmarz method is an iterative algorithm for solving linear
systems that is typically used for large, overdetermined problems
with inconsistent equations.  Randomized variants of the Kaczmarz
method have received a lot of attention in recent years,
in part because of close connections to stochastic gradient
descent (SGD) algorithms for solving least-squares problems.

To explain the idea, consider a (possibly inconsistent) linear system
\begin{equation} \label{eq:Ax=b_kacz}
\mtx{A}^* \vct{x} \approx \vct{b}
\quad\text{where}\quad
\mtx{A}^* \in \F^{m \times n}.
\end{equation}
The basic Kaczmarz algorithm starts with an initial guess $\vct{x}_0 \in \F^n$
for the solution.  At each iteration $t$, we select a new index $j = j(t) \in \{1, \dots, m\}$,
and we make the update
\begin{equation} \label{eq:kacz_trad}
\vct{x}_{t+1} = \vct{x}_t + \frac{\vct{b}(j) - \ip{\mtx{A}(:, j)}{\vct{x}_t}}{\norm{\mtx{A}(:,j)}^2} \mtx{A}(:, j).
\end{equation}
The rule~\eqref{eq:kacz_trad} has a simple interpretation:
it ensures that $\vct{x}_{t+1}$ is the closest point to $\vct{x}_t$
in the hyperplane containing solutions to the linear equation
determined by the $j$th equation in the system.

In implementing this method, we must choose a control mechanism
that determines the next index.  A simple and robust approach is
to cycle through the rows consecutively; that is, $j(t) = t \bmod m$.
Another effective, but expensive, option is to select the equation
with the largest violation.

The randomized Kaczmarz (RK) algorithm uses a probabilistic control
mechanism instead.  This kind of approach also has a long history,
and it is useful in cases where cyclic control is ineffective.
The RK method has received renewed attention owing to work of
\citeasnoun{SV09:Randomized-Kaczmarz}.
They proposed sampling each $j(t)$ independently at random,
with the probability of choosing the $i$th equation proportional to the
squared $\ell_2$ norm of the $i$th column of $\mtx{A}$.
They proved that this version of RK converges linearly with a rate determined by the
(Demmel) condition number of the matrix $\mtx{A}$.  Later,
it was recognized that this approach is just a particular
instantiation of SGD for the least-squares problem~\eqref{eq:Ax=b_kacz}.
See \citeasnoun{2014_needell_randomized_kaczmarz}, which draws
on results from \citeasnoun{2011_bach_moulines}.

There are many subsequent papers that have built on the RK
approach for solving inconsistent linear systems.
\citeasnoun{LL10:Randomized-Methods} observed that related
ideas can be used to design randomized Gauss--Seidel
and randomized Jacobi iterations.
\citeasnoun{2014_tropp_needell_paved} studied a blocked version
of the RK algorithm, which is practically more efficient
for many of the same reasons that other blocked algorithms
work well (Section~\ref{sec:blocking}).

\citeasnoun{2015_richtarik_randomized_kaczmarz} observed
that the RK algorithm is a particular type of iterative
sketching method.  Based on this connection, they proposed
a generalization.  At each iteration, we draw an
independent random embedding $\mtx{S}_t \in \F^{\ell \times m}$.
The next iterate is chosen by solving the least-squares problem
\begin{equation}
\label{eq:Ax=b_kacz_proj0}
\vct{x}_{t} = \argmin\nolimits_{\vct{y}} \norm{ \vct{x}_{t-1} - \vct{y} }^2
\quad\text{subject to}\quad
\mtx{S}_t \mtx{A}^* \vct{y} = \mtx{S}_t \vct{b}.
\end{equation}
The idea is to choose the dimension $\ell$ sufficiently small that
the sketched least-squares problem can be solved explicitly
using a direct method (e.g., QR factorization).
This flexibility leads to algorithms that converge more rapidly in practice
because the sketch $\mtx{S}_t$ can mix equations instead of just sampling.
Later, \citeasnoun{RT17:Stochastic-Reformulations}
showed that this procedure can be accelerated to achieve
rates that depend on the \emph{square root} of an appropriate condition
number; see also \citeasnoun{GHRS18:Accelerated-Stochastic}.

\section{Linear solvers for graph Laplacians}
\label{sec:sparse-cholesky}

In this section we describe the randomized algorithm,
\textsc{SparseCholesky}, which can efficiently solve
a linear system whose coefficient matrix is a graph Laplacian matrix.
Up to a polylogarithmic factor, this method achieves the minimum
possible runtime and storage costs.  This algorithm has the potential
to accelerate many types of computations involving graph Laplacian
matrices.

The \textsc{SparseCholesky} algorithm was developed by
\citeasnoun{KS16:Approximate-Gaussian} and further
refined by \citeasnoun{Kyn17:Approximate-Gaussian}.
These approaches are based on earlier work from Dan Spielman's
group, notably the paper of \citeasnoun{KLP+16:Sparsified-Cholesky}.
The presentation here is adapted from~\citeasnoun{Tro19:Matrix-Concentration-LN}.

\subsection{Overview}

We begin with a high-level approach for solving Laplacian
linear systems.  The basic idea is to construct a preconditioner
using a randomized variant of the incomplete Cholesky method.
Then we can solve the original linear system
by means of the preconditioned conjugate gradient (PCG) algorithm.

\subsubsection{Approximate solutions to the Poisson problem}

Let $\mtx{L} \in \Sym_n(\RR)$ be the Laplacian matrix of a weighted, loop-free, undirected graph
on the vertex set $V = \{1, \dots, n\}$.  We write $m$ for the number of edges in the graph,
i.e., the sparsity of the graph.
For simplicity, we will also assume that the graph is \emph{connected};
equivalently, $\ker(\mtx{L}) = \lspan\{ \mathbf{1} \}$ where
$\mathbf{1} \in \RR^n$ is the vector of ones.
See Section~\ref{sec:graph-sparsification} for definitions.
See Figure~\ref{fig:big-dipper} for an illustration of
an unweighted, undirected graph.

The basic goal is to find the unique solution $\vct{x}_{\star}$ to the Poisson problem
$$
\mtx{L} \vct{x} = \vct{f}
\quad\text{where}\quad
\mathbf{1}^* \vct{f} = {0}
\quad\text{and}\quad
\mathbf{1}^* \vct{x} = {0}.
$$
For a parameter $\eps > 0$, we can relax this requirement by asking
instead for an approximate solution $\vct{x}_{\eps}$ that satisfies
$$
\norm{ \vct{x}_{\eps} - \vct{x}_{\star} }_{\mtx{L}}
	\leq \eps \, \norm{ \vct{x}_{\star} }_{\mtx{L}}.
$$
We have written $\norm{ \vct{x} }_{\mtx{L}} := (\vct{x}^* \mtx{L} \vct{x})^{1/2}$
for the energy seminorm induced by the Laplacian matrix.

\subsubsection{Approximate Cholesky decomposition}

Imagine that we can efficiently construct a
sparse, approximate Cholesky decomposition of the Laplacian matrix $\mtx{L}$.
More precisely, we seek a morally lower-triangular matrix $\mtx{C}$
that satisfies
\begin{equation} \label{eqn:approx-cholesky}
0.5 \, \mtx{L} \psdle \mtx{CC}^* \psdle 1.5 \, \mtx{L}
	\quad\text{where}\quad
	\texttt{nnz}(\mtx{C}) = \bigO(m \log^2 n).
\end{equation}
In other words, there is a known permutation of the rows
that brings the matrix $\mtx{C}$ into lower-triangular form.
As usual, $\psdle$ is the semidefinite order, and
\texttt{nnz} returns the number of nonzero entries in a matrix.

This section describes an algorithm, called \textsc{SparseCholesky},
that can complete the task outlined in the previous paragraph.
This algorithm is motivated by the insight
that we can produce a sparse approximation
of a Laplacian matrix by random sampling (Section~\ref{sec:graph-sparsification}).
The main challenge is to obtain sampling probabilities without extra computation.
The resulting method can be viewed as a randomized variant of the incomplete
Cholesky factorization \cite[Sec.~11.5.8]{GVL13:Matrix-Computations-4ed}.

\subsubsection{Preconditioning}

Given the sparse, approximate Cholesky factor $\mtx{C}$, we can precondition the Poisson problem:
\begin{equation} \label{eqn:precond-poisson}
(\mtx{C}^\pinv \mtx{L} \mtx{C}^{*\pinv})(\mtx{C}^* \vct{x}) = (\mtx{C}^\pinv \vct{f}).
\end{equation}
When~\eqref{eqn:approx-cholesky} holds, the matrix $\mtx{C}^\pinv \mtx{L} \mtx{C}^{*\pinv}$
has condition number $\kappa \leq 3$.

Therefore, we can solve the preconditioned system~\eqref{eqn:precond-poisson} quickly using the PCG
algorithm \cite[Sec.~11.5]{GVL13:Matrix-Computations-4ed}.
If the initial iterate $\vct{x}_0 = \vct{0}$, then $j$ steps
of PCG produce an iterate $\vct{x}_j$ that satisfies
$$
\norm{ \vct{x}_j - \vct{x}_{\star} }_{\mtx{L}}
	\leq 2 \left[ \frac{\sqrt{\kappa} - 1}{\sqrt{\kappa} + 1} \right]^j
	\norm{ \vct{x}_{\star} }_{\mtx{L}}
	< 3^{1-j} \norm{ \vct{x}_{\star} }_{\mtx{L}}.
$$
As a consequence, we can achieve relative error $\eps$ after $1 + \log_3 (1/\eps)$ iterations.
Each iteration requires a matrix--vector product with $\mtx{L}$
and the solution of a (consistent) linear system $\mtx{CC}^* \vct{u} = \vct{y}$.
We can perform these steps in $\bigO(m \log^2 n)$ operations per iteration.
Indeed, $\mtx{L}$ has only $\bigO(m)$ nonzero entries.
The matrix $\mtx{C}$ is morally triangular with $\bigO(m \log^2 n)$
nonzero entries, so we can apply $(\mtx{CC}^*)^{\pinv}$ using two triangular
solves.

\begin{remark}[Triangular solve]
To solve a consistent linear system $(\mtx{CC}^*) \vct{u} = \vct{y}$,
we can apply $\mtx{C}^\pinv$ by triangular elimination and then
apply $\mtx{C}^{*\pinv}$ by triangular elimination.
The four subspace theorem ensures that solution to the
first problem renders the second problem consistent too.
Since $\ker(\mtx{CC}^*) = \lspan\{\vct{1}\}$, we can
enforce consistency numerically by removing the constant
component of the input $\vct{y}$.  Similarly, we can
remove the constant component of the output $\vct{u}$
to ensure it belongs to the correct space.
\end{remark}

\subsubsection{Main results}

The following theorem describes the performance of the \textsc{SparseCholesky} procedure.
This is the main result from~\citeasnoun{KS16:Approximate-Gaussian}.

\begin{theorem}[SparseCholesky] \label{thm:sparse-cholesky}
Let $\mtx{L}$ be the Laplacian of a connected graph on $n$ vertices, with $m$ weighted edges.
With high probability, the \textsc{SparseCholesky} algorithm produces a morally lower-triangular
matrix $\mtx{C}$ that satisfies
$$
0.5 \, \mtx{L} \psdle \mtx{CC}^* \psdle 1.5 \, \mtx{L}.
$$
The matrix $\mtx{C}$ has $\bigO(m \log^2 n)$ nonzero entries.  The expected running
time is $\bigO(m \log^3 n)$ operations.
\end{theorem}

In view of our discussion about PCG, we arrive at the following statement
about solving the Poisson problem.

\begin{corollary}[Poisson problem]
Suppose the \textsc{SparseCholesky} algorithm delivers
an approximation $\mtx{L} \approx \mtx{CC}^*$ that
satisfies~\eqref{eqn:approx-cholesky}.
Then we can solve each consistent linear system $\mtx{L}\vct{x} = \vct{f}$
to relative error $\eps$ in the seminorm $\norm{\cdot}_{\mtx{L}}$
using at most $1 + \log_3(1/\eps)$ iterations of PCG, each with a cost
of $\bigO(m \log^2 n)$ arithmetic operations.
\end{corollary}

\subsubsection{Discussion}

As we have mentioned, the Poisson problem %
serves as a primitive for undertaking many computations on
undirected graphs~\cite{Ten10:Laplacian-Paradigm}.
Potential applications include clustering, analysis of random walks,
and finite-element discretizations of elliptic PDEs.

The \textsc{SparseCholesky} algorithm achieves a near-optimal runtime
and storage guarantee for the Poisson problem on a graph.
Indeed, for a general graph with $m$ edges,
any algorithm must use $\bigO(m)$ storage and arithmetic.
There is a proof that the cost $\bigO(m \log^{1/2}(n))$
is achievable in theory \cite{CKM+14:Solving-SDD},
but the resulting methods are currently impractical.
Meanwhile, the simplicity of the \textsc{SparseCholesky}
method makes it a candidate for real-world computation.

For particular classes of Laplacian matrices, existing solvers
can be very efficient. Optimized sparse direct solvers
\cite{2016_acta_sparse_direct_survey} work very well for small-
and medium-size problems, but they typically have superlinear scaling,
which renders them unsuitable for truly large-scale problems.
Iterative methods such as multigrid or preconditioned
Krylov solvers can attain linear complexity for important
classes of problems, in particular for sparse systems arising
from the discretization of elliptic PDEs. However, we are not
aware of competing methods that provably enjoy near-optimal
complexity for all problems.

We regard the \textsc{SparseCholesky} algorithm as one of the
most dramatic examples of how randomization has the potential to accelerate
basic linear algebra computations, both in theory and in practice.

\subsection{Cholesky decomposition of a graph Laplacian}

To begin our explanation of the \textsc{SparseCholesky} algorithm,
let us summarize what happens when we apply the
standard Cholesky decomposition method to a graph Laplacian.

\subsubsection{The Laplacian of a multigraph}

For technical reasons, related to the design and analysis of the algorithm,
we need to work with multigraphs instead of ordinary graphs.
In the discussion, we will point out specific places where
this generality is important.

Consider a weighted, undirected \emph{multigraph} $G$,
defined on the vertex set $V = \{1, \dots, n\}$.
Each edge $e = \{u, v\}$ is an unordered pair of vertices;
we typically use the abbreviated notation $e = uv = vu$.
We introduce a weight function
$w_G$ that assigns a positive weight to each edge $e$
in the multigraph $G$.  Since $G$ is a multigraph,
there may be many multiedges, with distinct weights,
connecting the same two vertices.

Taking some notational liberties, we will identify the
multigraph $G$ with its Laplacian matrix $\mtx{L}$,
which we express in the form
\begin{equation} \label{eqn:multi-laplacian}
\mtx{L} = \sum\nolimits_{e \in \mtx{L}} w_{\mtx{L}}(e) \, \mtx{\Delta}_{e}.
\end{equation}
The matrix $\mtx{\Delta}_{e}$ is the elementary Laplacian on the
vertex pair $(u, v)$ that composes the edge $e = uv$.  That is,
$$
\mtx{\Delta}_e := \mtx{\Delta}_{uv} := (\vct{\delta}_{u} - \vct{\delta}_v)(\vct{\delta}_{u} - \vct{\delta}_v)^*
\quad\text{where $e = uv$.}
$$
The sum in~\eqref{eqn:multi-laplacian} takes place over all the multiedges $e$
in the multigraph $\mtx{L}$, so the same elementary Laplacian may
appear multiple times with different weights.

\subsubsection{Stars and cliques}
\label{sec:star-clique}

To describe the Cholesky algorithm on a graph, we need to
introduce a few more concepts from graph theory.
Define the \emph{degree} and the \emph{total weight} of a vertex $u$ in the multigraph $\mtx{L}$ to be
$$
\operatorname{deg}_{\mtx{L}}(u) := \sum\nolimits_{e=uv \in \mtx{L}} 1
\quad\text{and}\quad
w_{\mtx{L}}(u) := \sum\nolimits_{e = uv \in \mtx{L}} w_{\mtx{L}}(e).
$$
In other words, the degree of $u$ is the total number of multiedges $e$ that contain $u$.
The total weight of $u$ is the sum of the weights of the multiedges $e$ that contain $u$.

Let $u$ be a fixed vertex.  The \emph{star}
induced by $u$ is the Laplacian
$$
\textsc{star}(u, \mtx{L}) := \sum\nolimits_{e = uv \in \mtx{L}} w_{\mtx{L}}(e) \, \mtx{\Delta}_e.
$$
In words, the star includes precisely those multiedges $e$ in the multigraph $\mtx{L}$
that contain the vertex $u$.

The \emph{clique} induced by $u$ is defined implicitly as the correction
that occurs when we take the Schur complement~\eqref{eqn:schur-complement} of the Laplacian with
respect to the coordinate $u$:
$$
\mtx{L} / \vct{\delta}_u %
	=: (\mtx{L} - \textsc{star}(u, \mtx{L})) + \textsc{clique}(u, \mtx{L}).
$$
Recall that $\vct{\delta}_u$ is the standard basis vector in coordinate $u$.
By direct calculation, one may verify that
$$
\textsc{clique}(u, \mtx{L}) = \frac{1}{2 w_{\mtx{L}}(u)} \sum\limits_{e_1 = uv_1 \in \mtx{L}} \sum\limits_{e_2 = uv_2 \in \mtx{L}}
	w_{\mtx{L}}(e_1) w_{\mtx{L}}(e_2) \mtx{\Delta}_{v_1v_2}.
$$
Each sum takes place over all multiedges $e$ in $\mtx{L}$ that contain the vertex $u$.
It can be verified that the clique is also the Laplacian of a weighted multigraph.

Figures~\ref{fig:big-dipper} and~\ref{fig:star-clique} contain an illustration of a (simple) graph,
along with the star and clique induced by eliminating a vertex.  In our more general setting,
the edges in the star and clique would have associated weights.
These diagrams are courtesy of Richard Kueng.

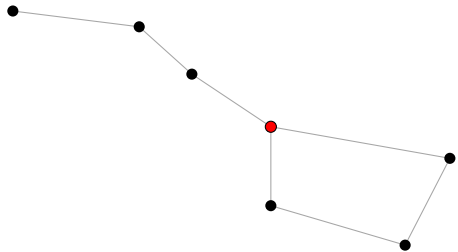
\begin{figure}
\centering
\begin{tikzpicture}[baseline,scale=0.7]
\coordinate (S1) at (0,0);
\coordinate (S2) at (-1.5,1);
\coordinate (S3) at (0,-1.5);
\coordinate (S4) at (2.55,-2.25);
\coordinate (S5) at (3.4,-0.6);
\coordinate (S6) at (-2.5,1.9);
\coordinate (S7) at (-4.9, 2.2);
\draw[black!35!white] (S1) -- (S2);
\draw[black!35!white] (S1) -- (S3);
\draw[black!35!white] (S3) -- (S4);
\draw[black!35!white] (S4) -- (S5);
\draw[black!35!white] (S5) -- (S1);
\draw[black!35!white] (S2) -- (S6);
\draw[black!35!white] (S6) -- (S7);
\draw[fill=red] (S1) circle (3pt);
\fill[black] (S2) circle (3pt);
\fill[black] (S3) circle (3pt);
\fill[black] (S4) circle (3pt);
\fill[black] (S5) circle (3pt);
\fill[black]  (S6) circle (3pt);
\fill[black]  (S7) circle (3pt);
\end{tikzpicture}
\caption{\emph{A combinatorial graph.} The Ursa Major graph with a distinguished vertex (the star Megrez) highlighted in red.}
\label{fig:big-dipper}
\end{figure}

\begin{figure}
\centering
\begin{tabular}{cc}
\begin{tikzpicture}[baseline,scale=0.7]
\coordinate (S1) at (0,0);
\coordinate (S2) at (-1.5,1);
\coordinate (S3) at (0,-1.5);
\coordinate (S4) at (2.55,-2.25);
\coordinate (S5) at (3.4,-0.6);
\coordinate (S6) at (-2.5,1.9);
\coordinate (S7) at (-4.9, 2.2);
\draw[black!35!white]  (S3) -- (S4);
\draw[black!35!white] (S2) -- (S6);
\draw[black!35!white] (S6) -- (S7);
\draw[black!35!white] (S4) -- (S5);
\fill[black!35!white] (S4) circle (3pt);
\fill[black!35!white] (S6) circle (3pt);
\fill[black!35!white] (S7) circle (3pt);
\draw[thick] (S1) -- (S2);
\draw[thick] (S1) -- (S3);
\draw[thick] (S5) -- (S1);
\draw[fill=red] (S1) circle (3pt);
\draw[fill=black] (S2) circle (3pt);
\draw[fill=black] (S3) circle (3pt);
\draw[fill=black] (S5) circle (3pt);
\end{tikzpicture} &
\begin{tikzpicture}[baseline,scale=0.7]
\coordinate (S1) at (0,0);
\coordinate (S2) at (-1.5,1);
\coordinate (S3) at (0,-1.5);
\coordinate (S4) at (2.55,-2.25);
\coordinate (S5) at (3.4,-0.6);
\coordinate (S6) at (-2.5,1.9);
\coordinate (S7) at (-4.9, 2.2);
\draw[black!35!white]  (S1) -- (S2);
\draw[black!35!white]  (S1) -- (S3);
\draw[black!35!white]  (S3) -- (S4);
\draw[black!35!white]  (S4) -- (S5);
\draw[black!35!white]  (S5) -- (S1);
\draw[black!35!white]  (S2) -- (S6);
\draw[black!35!white]  (S6) -- (S7);
\fill[white] (S1) circle (3pt);
\draw[fill=red,opacity=0.35] (S1) circle (3pt);
\fill[black!35!white] (S4) circle (3pt);
\fill[black!35!white]  (S6) circle (3pt);
\fill[black!35!white]  (S7) circle (3pt);
\draw[thick] (S2) -- (S3);
\draw[thick] (S2) -- (S5);
\draw[thick] (S3) -- (S5);
\fill[black] (S2) circle (3pt);
\fill[black]  (S3) circle (3pt);
\fill[black] (S5) circle (3pt);
\end{tikzpicture}
\end{tabular}
\caption{\emph{Illustration of a star and clique in the Ursa Major graph.}
(a) The star induced by the red vertex consists of the three solid black edges.
(b) The clique induced by the red vertex consists of the three solid black edges.
These edges are added when the red vertex is eliminated.
} \label{fig:star-clique}
\end{figure}
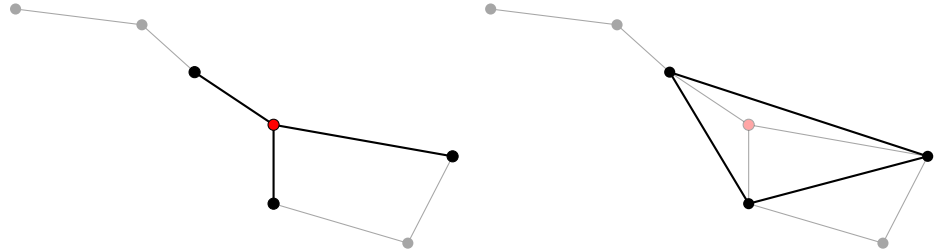

\subsubsection{Graphs and Cholesky}

With the notation introduced in Section~\ref{sec:star-clique},
we can present the graph-theoretic interpretation of the Cholesky algorithm
as it applies to the Laplacian $\mtx{L}$ of a weighted multigraph.

Define the initial Laplacian $\mtx{S}_0 := \mtx{L}$.  In each step $i = 1, 2, \dots, n$, we
select a new vertex $u_i$.  We extract the associated column of the current Laplacian:
$$
\vct{c}_i := \frac{1}{\sqrt{(\mtx{S}_{i-1})_{u_i u_i}}} \mtx{S}_{i-1} \vct{\delta}_{u_i}.
$$
We compute the Schur complement with respect to the vertex $u_i$:
$$
\mtx{S}_i := \mtx{S}_{i-1} / \vct{\delta}_{u_i} = (\mtx{S}_{i-1} - \textsc{star}(u_i, \mtx{S}_{i-1})) + \textsc{clique}(u_i, \mtx{S}_{i-1}).
$$
In other words, we remove the star induced by $u_i$ and replace it with
the clique induced by $u_i$.
Each $\mtx{S}_i$ is the Laplacian of a multigraph; it has no multiedge
that contains any one of the vertices $u_1, \dots, u_i$.  Therefore, we have reduced the
size of the problem.

After $n$ steps, the Cholesky factorization is determined by a vector
$\vct{\pi} = (u_1, u_2, \dots, u_n)$ that holds the chosen indices
and a matrix $\mtx{C} = \begin{bmatrix} \vct{c}_1 & \dots & \vct{c}_n \end{bmatrix}$
such that $\mtx{C}(\vct{\pi}, :)$ is lower-triangular.
The algorithm ensures that we have the exact decomposition
$$
\mtx{L} = \mtx{CC}^*.
$$
This is the (pivoted) Cholesky factorization of the Laplacian matrix.

To choose a vertex to eliminate, the classical approach is to find
a vertex with minimum degree or with minimum total weight.
Or one may simply select one of the remaining vertices at random.

\subsubsection{Computational costs}

The cost to compute a Cholesky factorization $\mtx{L} = \mtx{C}\mtx{C}^{*}$
of a Laplacian matrix $\mtx{L}$ is typically superlinear in $n$, with a worst-case cost of $\bigO(n^{3})$
arithmetic and $\bigO(n^{2})$ storage. This is the reason that $\mtx{C}$ is
less sparse than $\mtx{L}$: the clique that is introduced at an elimination
step has more edges than the star that it replaces, a phenomenon referred
to as \emph{fill-in}. The exact growth in the number of nonzero entries depends
on the sparsity pattern of $\mtx{L}$ and on the chosen elimination order. For
special cases, using a nested dissection ordering can provably improve on the
worst-case estimates \cite{2016_acta_sparse_direct_survey}. For instance, if $\mtx{L}$ results from the
finite-difference or finite-element discretization of an elliptic PDE, then
$\texttt{nnz}(\mtx{C}) = \bigO(n\log(n))$ in two dimensions and
$\texttt{nnz}(\mtx{C}) = \bigO(n^{4/3})$ in three.

For general graphs, one path towards improving the efficiency of the Cholesky factorization procedure is to randomly approximate
the clique by sampling, in order to curb the fill-in. Section~\ref{sec:graph-sparsification} already
indicates that this innovation may be possible, provided that we can find
a way to obtain sampling probabilities.

\subsection{The \textsc{SparseCholesky} algorithm}

We are now prepared to present the \textsc{SparseCholesky} procedure,
which uses randomized sampling to compute a sparse, approximate
Cholesky factorization.

\subsubsection{Procedure}

Let $\mtx{L}$ be the Laplacian of a weighted
multigraph on $V = \{1, \dots, n\}$.  We
perform the following steps.

\begin{enumerate} \setlength{\itemsep}{1mm}
\item	\textbf{Preprocessing.}	Split each multiedge $e = uv$ in $\mtx{L}$ into
$R = \lceil 64 \log^2(\econst n) \rceil$
multiedges,
each connecting $\{u, v\}$, and each with weight $w_{\mtx{L}}(e) / R$.
The purpose of this step is to control the effective resistance~\eqref{eqn:effective-resistance}
of each multiedge at the outset of the algorithm.
Note that this splitting results in a weighted multigraph,
even if we begin with a simple graph.

\item	\textbf{Initialization.} Form the initial Laplacian $\mtx{S}_0 = \mtx{L}$ and the list of remaining vertices $F_0 = V$.

\item	\textbf{Iteration.}  For each $i = 1, 2, \dots, n$:

\begin{enumerate} \setlength{\itemsep}{1mm}
\item	\textbf{Select a vertex.}  Choose a vertex $u_i$ uniformly at random from $F_{i-1}$.  Remove this vertex from the list: $F_i = F_{i-1} \setminus \{ u_i \}$.

\item	\textbf{Extract the column.}  Copy the normalized $u_i$th column from the current Laplacian:
$$
\vct{c}_i = \frac{1}{\sqrt{(\mtx{S}_{i-1})_{u_i u_i}}} \mtx{S}_{i-1} \vct{\delta}_{u_i}.
$$
Set $\vct{c}_i = \vct{0}$ if the denominator equals zero.

\item	\textbf{Sampling the clique.}  Construct the Laplacian $\mtx{K}_i$ of a random sparse approximation of $\textsc{clique}(u_i, \mtx{S}_{i-1})$.
We will detail this procedure in Section~\ref{sec:clique-sample}.

\item	\textbf{Approximate Schur complement.}  Form
$$
\mtx{S}_i = (\mtx{S}_{i-1} - \textsc{star}(u_i, \mtx{S}_{i-1})) + \mtx{K}_i.
$$
\end{enumerate}

\item	\textbf{Decomposition.}  Collate the columns $\vct{c}_i$ into the Cholesky factor
$$
\mtx{C} = \begin{bmatrix} \vct{c}_1 & \dots & \vct{c}_n \end{bmatrix}.
$$
Define the row permutation $\pi(i) = u_i$ for each $i$.
\end{enumerate}

Once these operations are complete, $\mtx{C}$ is a sparse, morally lower-triangular matrix.
It is also very likely that $\mtx{L} \approx \mtx{CC}^*$.  Theorem~\ref{thm:sparse-cholesky}
makes a rigorous accounting of these claims.

\subsubsection{Clique sampling}
\label{sec:clique-sample}

The remaining question is how to construct a random approximation $\mtx{K}$ of a clique
$\textsc{clique}(u, \mtx{S})$.  Here is the procedure:

\begin{enumerate}
\item	\textbf{Probabilities.}  Construct a probability mass $\vct{p}$ such that
$$
p(e) = \frac{w_{\mtx{S}}(e)}{w_{\mtx{S}}(u)}
\quad\text{for each $e \in \textsc{star}(u, \mtx{S})$.}
$$

\item	\textbf{Sampling.}  For each $i = 1, \dots, d = \operatorname{deg}_{\mtx{S}}(u)$,

\begin{enumerate}
	\item	Draw a random multiedge $e_1 = uv_1$ from the multiedges in $\textsc{star}(u, \mtx{S})$
according to the probability mass $\vct{p}$.

\item	Draw a second random multiedge $e_2 = uv_2$ from the multiedges in $\textsc{star}(u, \mtx{S})$
according to the uniform distribution.

\item	Form the random Laplacian matrix of a new multiedge:
$$
\mtx{X}_i = \frac{w_{\mtx{S}}(e_1) \, w_{\mtx{S}}(e_2)}{w_{\mtx{S}}(e_1) + w_{\mtx{S}}(e_2)} \mtx{\Delta}_{v_1 v_2}.
$$
\end{enumerate}

\item	\textbf{Approximation.}  Return $\mtx{K} = \sum_{i=1}^d \mtx{X}_i$.
\end{enumerate}

The key fact about this construction is that it produces an unbiased estimator $\mtx{K}$
of the clique:
$$
\Expect \mtx{K} = \textsc{clique}(u, \mtx{S}).
$$
Furthermore, %
each summand $\mtx{X}_i$ creates a multiedge with uniformly bounded effective resistance~\eqref{eqn:effective-resistance}.
This property persists as the \textsc{SparseCholesky} algorithm executes,
and it ensures that the random matrix $\mtx{K}$ has controlled variance.
Note that the sampling procedure can result in several edges between the same
pair of vertices, which is another reason we need the multigraph formalism.

Next, observe that the number $d$ of multiedges in $\mtx{K}$ is no greater
than the number $d$ of multiedges in the star that we are removing from $\mtx{S}$.
(For comparison, note that the full clique has $d^2$ multiedges.)
As a consequence, the clique approximation is inexpensive to construct.
Moreover, the total number of multiedges in the Laplacian can
only decrease as the \textsc{SparseCholesky} algorithm proceeds.

\subsubsection{Analysis}

The analysis of \textsc{SparseCholesky} is well beyond the scope of this paper.
The key technical tool is a concentration inequality for
matrix-valued martingales that was derived in
\cite{Oli09:Concentration-Adjacency} and \cite{Tro11:Freedmans-Inequality}.
To activate this result, \cite{KS16:Approximate-Gaussian} use the fact that the
random clique approximation is unbiased and low-variance,
conditional on previous choices made by the algorithm.
The proof also relies heavily on the fact that we eliminate a random vertex
at each step of the iteration.
For the technical details,
see~\cite{KS16:Approximate-Gaussian,Kyn17:Approximate-Gaussian}
and~\cite{Tro19:Matrix-Concentration-LN}.

\subsubsection{Implementation}

The \textsc{SparseCholesky} algorithm is fairly simple to describe,
but it demands some care to develop an implementation that achieves
the runtime guarantees stated in Theorem~\ref{thm:sparse-cholesky}.

The most important point is that we need to use data structures
for weighted multigraphs.
One method is to maintain the vertex--multiedge adjacency matrix,
along with a list of weights.  This approach requires sparse matrix libraries,
including efficient iterators over the rows and columns.

A secondary point is that we need fast methods for constructing
finite probability distributions and sampling from them repeatedly.
See \citeasnoun{BP17:Efficient-Sampling}.

It is unlikely that the \textsc{SparseCholesky} procedure will fail
to produce a factor $\mtx{C}$ that satisfies~\eqref{eqn:approx-cholesky}.
Even so, it is reassuring to know that we can detect failures.  Indeed,
we can estimate the extreme singular values of the preconditioned
linear system~\eqref{eqn:precond-poisson} using the methods from
Section~\ref{sec:max-eig}.

The failure probability can also be reduced
by modifying the random vertex selection rule.
Instead, we can draw a random vertex whose total weight
is at most twice the average total weight of the remaining vertices
\cite{Kyn17:Approximate-Gaussian}.
In practice, it may suffice to use the classical elimination rules
based on minimum degree or minimum total weight.

The main shortcoming of the \textsc{SparseCholesky} procedure arises
from the initialization step, where we split each multiedge into $\bigO(\log^2 n)$ pieces.
This step increases the storage and computation costs of the algorithm
enough to make it uncompetitive (e.g.~with fast direct Poisson solvers)
for some problem instances.  At present, it is unclear whether
the initialization step can be omitted or relaxed,
while maintaining the reliability and correctness of the algorithm.

\section{Kernel matrices in machine learning}
\label{sec:kernel}

Randomized NLA algorithms have played a major role in
developing scalable kernel methods for large-scale
machine learning.
This section contains a brief introduction to kernels.
Then it treats two probabilistic techniques that have had
an impact on kernel matrix computations:
Nystr{\"o}m approximation by random coordinate sampling \cite{WS01:Using-Nystrom}
and empirical approximation by random features \cite{RR08:Random-Features}.

The literature on kernel methods is truly vast,
so we cannot hope to achieve comprehensive coverage within our survey.
There are also many computational considerations and
learning-theoretic aspects that fall outside the
realm of NLA.  Our goal is simply to give a taste of the
ideas, along with a small selection of key references.

\subsection{Kernels in machine learning}

We commence with a crash course on kernels
and their applications in machine learning.
The reader may refer to \citeasnoun{SS01:Learning-Kernels}
for a more complete treatment.

\subsubsection{Kernel functions and kernel matrices}
\label{sec:kernel-functions}

Let $\mathcal{X}$ be a set, called the \emph{input space}
or \emph{data space}.
Suppose that we acquire a finite set of observations
$\{ \vct{x}_1, \dots, \vct{x}_n \} \subset \mathcal{X}$.
We would like to use the observed data to perform
learning tasks.

One approach is to introduce a \emph{kernel function}:
$$
k : \mathcal{X} \times \mathcal{X} \to \F.
$$
The value $k(\vct{x}, \vct{y})$ of the kernel function
is interpreted as a measure of similarity between two data points
$\vct{x}$ and $\vct{y}$.
We can tabulate the pairwise similarities of the observed
data points in a \emph{kernel matrix}:
$$
(\mtx{K})_{ij} := k( \vct{x}_i, \vct{x}_j )
\quad\text{for $i, j = 1, \dots, n$.}
$$
The kernel matrix is an analog of the Gram matrix of a
set of vectors in a Euclidean space.
In Sections~\ref{sec:kpca} and~\ref{sec:krr},
we will explain how to use the matrix $\mtx{K}$
to solve some core problems in data analysis.

The kernel function is required to be \emph{positive definite}.
That is, for each natural number $n$ and each set
$\{\vct{x}_1, \dots, \vct{x}_n\} \subset \mathcal{X}$
of observations, the associated kernel matrix
$$
\mtx{K} = \big[ k(\vct{x}_i, \vct{x}_j) \big]_{i, j = 1, \dots, n} \in \Sym_n
\quad\text{is psd.}
$$
In particular, $k(\vct{x}, \vct{x}) \geq 0$ for all $\vct{x} \in \mathcal{X}$.
The kernel must also be (conjugate) symmetric in its arguments:
$k(\vct{x}, \vct{y}) = k(\vct{y}, \vct{x})^*$ for all $\vct{x}, \vct{y} \in \mathcal{X}$.
These properties mirror the properties of a Gram matrix.
In Section~\ref{sec:kernel-examples}, we give some
examples of positive definite kernel functions.

\subsubsection{The feature space}

It is common to present kernel functions using the
theory of \emph{reproducing kernel Hilbert spaces}.
This approach gives an alternative interpretation
of the kernel function as the inner product defined
on a feature space.  We give a very brief
treatment, omitting all technical details.

Let $\mathcal{F}$ be a Hilbert space, called the \emph{feature space}.
We introduce a \emph{feature map} $\Phi: \mathcal{X} \to \mathcal{F}$,
which maps a point in the input space to a point in the feature space.
Heuristically, the feature map extracts information from
a data point that is relevant for learning applications.

Under mild conditions, we can construct a positive definite
kernel function $k$ from the feature map:
$$
k(\vct{x}, \vct{y}) = \ip{ \Phi(\vct{x}) }{ \Phi(\vct{y}) }
\quad\text{for all $\vct{x}, \vct{y} \in \mathcal{X}$.}
$$
In other words, the kernel function reports the inner product
between the features associated with the data points
$\vct{x}$ and $\vct{y}$.
Conversely, a positive definite kernel always induces a feature
map into an appropriate feature space.

\subsubsection{Examples of kernels}
\label{sec:kernel-examples}

Kernel methods are powerful because we can select or design
a kernel that automatically extracts relevant feature information
from our data.  This approach applies in all sorts of domains,
including images and text and DNA sequences.  Let us present
a few kernels that commonly arise in applications.
See~\citeasnoun{SS01:Learning-Kernels} for many additional examples
and references.

\begin{example}[Inner product kernel]
The simplest example of a kernel is the ordinary inner product.
Let $\mathcal{X} = \F^d$.  Evidently,
$$
k(\vct{x}, \vct{y}) = \ip{\vct{x}}{\vct{y}}
\quad\text{for $\vct{x}, \vct{y} \in \F^d$}
$$
is a positive definite kernel.
\end{example}

\begin{example}[Angular similarity] \label{ex:angsim}
Another simple example is the angular similarity map.
Let $\mathcal{X} = \mathbb{S}^{d-1}(\RR) \subset \RR^d$.
This kernel is given by the formula
$$
k(\vct{x}, \vct{y}) = \frac{2}{\pi} \arcsin\ \ip{ \vct{x} }{ \vct{y} }
\quad\text{for $\vct{x}, \vct{y} \in \mathcal{X}$.}
$$
This kernel is positive definite because of Schoenberg's theorem~\cite{Sch42:Positive-Definite}.
We will give a short direct proof in Example~\ref{ex:angsim-rf}.
\end{example}

\begin{example}[Polynomial kernels]
Let $\mathcal{X}$ be a subset of $\F^d$.
For a natural number $p$, the inhomogeneous polynomial kernel is
$$
k(\vct{x}, \vct{y}) = (1 + \ip{\vct{x}}{\vct{y}})^p.
$$
This kernel is also positive definite because of Schoenberg's theorem;
see~\citeasnoun{KK12:Random-Feature}.  There is a short direct proof
using the Schur product theorem.
\end{example}

\begin{example}[Gaussian kernel] \label{ex:rbf}
An important example is the Gaussian kernel.
Let $\mathcal{X} = \F^d$.
For a bandwidth parameter $\sigma > 0$, define
$$
k( \vct{x}, \vct{y} ) = \exp\left(-\frac{\norm{ \vct{x} - \vct{y} }^2 }{ 2 \sigma^2 } \right)
	\quad\text{for $\vct{x}, \vct{y} \in \F^d$.}
$$
This kernel is positive definite because of Bochner's theorem~\cite{Boc33:Monotone-Funktionen}.
We will give a short direct proof in Example~\ref{ex:rbf-rf}.
\end{example}

\subsubsection{The kernel trick}

As we have mentioned, kernels can be used for a wide range of tasks
in machine learning.
\citeasnoun[Rem.~2.8]{SS01:Learning-Kernels} state the key idea succinctly:

\begin{quotation}
Given an algorithm which is formulated in terms of a positive definite kernel $k$,
one can construct an alternative algorithm by replacing $k$ with another
positive definite kernel $\tilde{k}$.
\end{quotation}

In particular, any algorithm that can be formulated in terms of the inner product
kernel applies to every other kernel.  That is to say, an algorithm for Euclidean
data
that depends only on the Gram matrix can be implemented with a
kernel matrix instead.  The next two subsections give two specific examples of this
methodology; there are many other applications.

\subsubsection{Kernel PCA}
\label{sec:kpca}

Given a set of observations in a Euclidean space,
principal component analysis (PCA) searches for orthogonal directions
in which the data has the maximum variability.
The nonlinear extension,
kernel PCA (KPCA), was proposed in \citeasnoun{SSM96:Nonlinear-Component};
see also~\citeasnoun[Chap.~14]{SS01:Learning-Kernels}.

Let $\{\vct{x}_1, \dots, \vct{x}_n \} \subset \mathcal{X}$ be a set
of observations.  For a kernel $k$ associated with a feature map $\Phi$,
construct the kernel matrix $\mtx{K} \in \Sym_n$ associated with the observations.
For a natural number $\ell$, we compute a truncated eigenvalue decomposition
of the kernel matrix:
$$
\mtx{K} = \sum_{i=1}^{\ell} \lambda_i \, \vct{u}_i \vct{u}_i^*.
$$
Each unit-norm eigenvector $\vct{u}_i$ determines a direction
$(n\lambda_{i})^{-1/2}\sum_{j=1}^n \vct{u}_i(j) \, \Phi(\vct{x}_j)$
of high variability in the feature space, called
the $i$th kernel principal component.

To find the projection of a new point $\vct{x} \in \mathcal{X}$ onto the
$i$th kernel principal component, we embed it into the feature space
via $\Phi(\vct{x})$ and compute the inner product with
the $i$th kernel principal component.
In terms of the kernel function,
$$
\mathrm{PC}_i(\vct{x}) :=
\frac{1}{\sqrt{n\lambda_{i}}}\sum_{j=1}^n \vct{u}_i(j) \, k(\vct{x}, \vct{x}_j).
$$
We can summarize the observation $\vct{x}$ with the vector
$$
(\mathrm{PC}_1(\vct{x}), \dots, \mathrm{PC}_{\ell}(\vct{x})) \in \mathbb{F}^\ell.
$$
This representation provides a data-driven feature that can be used
for downstream learning tasks.

In practice, it is valuable to center the feature space representation of
the data, which requires a simple modification of the kernel matrix.
We also need to center each observation before
computing its projection onto the kernel principal components.
See~\citeasnoun[App.~1]{SSM96:Nonlinear-Component} for details. %

\subsubsection{Kernel ridge regression}
\label{sec:krr}

Given a set of labeled observations in a Euclidean space, ridge regression
uses regularized least-squares to model the labels as a linear functional
of the observations.
The nonlinear extension of this approach
is called \emph{kernel ridge regression} (KRR).
We refer to \citeasnoun[Chap.~4]{SS01:Learning-Kernels} for a more detailed treatment,
including an interpretation in terms of a nonlinear feature map.

Let $\{ (\vct{x}_i, y_i) : i = 1, \dots, n \} \subset \mathcal{X} \times \F$
be a set of paired observations.  For a kernel $k$, construct the kernel
matrix $\mtx{K} \in \Sym_n$ associated with the observations $\vct{x}_i$
(but not the numerical values $y_i$).
For a regularization parameter $\tau > 0$, the kernel ridge regression problem
takes the form
$$
\underset{\vct{\alpha} \in \F^n}{\text{minimize}} \quad \frac{1}{n} \sum_{i=1}^n \left[ y_i - (\mtx{K} \vct{\alpha})_i \right]^2
	+ \frac{\tau}{2} \vct{\alpha}^* \mtx{K} \vct{\alpha}.
$$
The solution to this optimization problem is obtained by solving an ordinary linear system:
$$
( \mtx{K} + \tau n \, \Id ) \vct{\alpha} = \vct{y}
\quad\text{where}\quad
\vct{y} = (y_1, \dots, y_n).
$$
Let $\widehat{\vct{\alpha}}$ be the solution to this system.

Given a new observation $\vct{x} \in \mathcal{X}$,
we can make a prediction $\widehat{y} \in \F$
for its label via the formula
$$
\widehat{y}(\vct{x}) %
	:= \sum_{j=1}^n \widehat{\alpha}_j \, k(\vct{x}, \vct{x}_j).
$$
In practice, the regularization parameter $\tau$
is chosen by cross-validation
with a holdout set of the paired observations.

\subsubsection{The issue}

Kernel methods are powerful tools for data analysis.  Nevertheless,
in their native form, they suffer from two weaknesses.

First, it is very expensive to compute the kernel matrix explicitly.
For example, if points in the data space $\mathcal{X}$
have a $d$-dimensional parameterization,
we may expect that it will cost $\bigO(d)$ arithmetic
operations to evaluate the kernel a single time.
Therefore, the cost of forming the kernel matrix $\mtx{K}$
for $n$ observations is $\bigO(n^2 d)$.

Second, after computing the kernel matrix $\mtx{K}$, it remains expensive to
perform the linear algebra required by kernel methods.  Both
KPCA and KRR require $\bigO(n^3)$
operations if we use direct methods.

The poor computational profile of kernel methods limits our ability to
use them directly for large-scale data applications.

\subsubsection{The solution}

Fortunately, there is a path forward.  To implement kernel methods,
we simply need to \emph{approximate} the kernel matrix
\cite[Sec.~10.2]{SS01:Learning-Kernels}.
Surprisingly,
using the approximation often results in \emph{better} learning outcomes
than using the exact kernel matrix.
Even a poor approximation of the kernel can suffice
to achieve near-optimal performance, both in theory and in practice
\cite{2013_bach_sharp,RCR15:Less-More,RCR17:FALKON-Optimal}.
Last, working with a structured approximation of the kernel can accelerate
the linear algebra computations dramatically.

Randomized algorithms provide several effective tools for approximating
kernel matrices.  Since we pay a steep price for each kernel evaluation,
we need to develop algorithms that explicitly control this cost.
The rest of this section describes two independent approaches. %
In Section~\ref{sec:nystrom-kernel}, we present coordinate sampling algorithms for Nystr{\"o}m approximations,
while Section~\ref{sec:random-features} develops the method of random features.

\begin{remark}[Function approximation]
Let us remark that approximation of the kernel matrix is usually incidental to the goals
of learning theory.  In many applications, such as KRR, we actually need to
approximate a function on the input space.  The sampling complexity of the
latter task may be strictly lower than the complexity
of approximating the full kernel matrix.  We cannot
discuss this issue in detail because it falls outside
the scope of NLA.
\end{remark}

\subsection{Coordinate Nystr{\"o}m approximation of kernel matrices}
\label{sec:nystrom-kernel}

One way to approximate a kernel matrix is to form
a Nystr{\"o}m decomposition with respect to a judiciously
chosen coordinate subspace.  A natural idea is to draw
these coordinates at random.  This basic technique
was proposed by \citeasnoun{WS01:Using-Nystrom}.

Coordinates play a key role here
because we only have access to individual
entries of the kernel matrix.  There is no
direct way to compute a matrix--vector product
with the kernel matrix, so we cannot easily
apply the more effective constructions of random embeddings
(e.g., Gaussians or sparse maps or SRTTs).
Indeed, kernel computation is the primary setting where
coordinate sampling is a practical idea.

\subsubsection{Coordinate Nystr{\"o}m approximation}

Suppose that $\{\vct{x}_1, \dots, \vct{x}_n \} \subset \mathcal{X}$
is a collection of observations.  Let $\mtx{K}$ be the psd
kernel matrix associated with some kernel function $k$.

Given a set $I \subseteq \{1, \dots, n\}$ consisting of $r$ indices,
we can form a psd Nystr{\"o}m approximation of the kernel matrix:
$$
\mtx{K}\nys{I} := \mtx{K}(:, I) \, \mtx{K}(I, I)^\pinv \, \mtx{K}(I, :).
$$
This matrix is equivalent to the Nystr{\"o}m decomposition~\eqref{eqn:nystrom-def}
with respect to a test matrix $\mtx{X}$ whose range is $\lspan \{ \vct{\delta}_i : i \in I \}$.

To obtain $\mtx{K}\nys{I}$, the basic cost is $nr$ kernel evaluations,
which typically require $\mathcal{O}(nrd)$ operations for a $d$-dimensional input space $\mathcal{X}$.
We typically do not form the pseudoinverse directly, but rather use the factored
form of the Nystr{\"o}m approximation for downstream calculations.

For kernel problems, it is common to regularize the coordinate Nystr{\"o}m approximation.
One approach replaces the core matrix $\mtx{K}(I, I)$ with its truncated eigenvalue decomposition
before computing the pseudoinverse; for example, see \citeasnoun{2005_drineas_nystrom}.
The RSVD algorithm (Section~\ref{sec:rsvd}) has been proposed for this purpose
\cite{LBKL15:Large-Scale-Nystrom}.
When it is computationally feasible, we recommend taking a truncated eigenvalue decomposition
of the full Nystr{\"o}m approximation $\mtx{K}\nys{I}$, rather than just the core;
see \citeasnoun{TYUC17:Fixed-Rank-Approximation}.

\subsubsection{Greedy selection of coordinates}

Recall from Section~\ref{sec:nystrom}
that the error $\mtx{K} / I := \mtx{K} - \mtx{K}\nys{I}$ in the Nystr{\"o}m decomposition
is simply the Schur complement of $\mtx{K}$ with respect to the coordinates in $I$.

This connection suggests that we should use a pivoted Cholesky method or a pivoted QR
algorithm to select the coordinates $I$.  These techniques lead to Nystr{\"o}m approximations with
superior learning performance; for example,
see~\cite{FS01:Efficient-SVM,BJ05:Predictive-Low-Rank} and~\cite{2013_bach_sharp}.
Unfortunately the $\mathcal{O}(n^2 r)$ cost is prohibitive in applications.
See~\citeasnoun[Sec.~10.2]{SS01:Learning-Kernels} for some randomized strategies
that can reduce the expense.

\subsubsection{Ridge leverage scores}

A natural approach to selecting the coordinate set $I$ is to perform
randomized sampling.  To describe these approaches, we need to take
a short detour.

Fix a regularization parameter $\tau > 0$.
Consider the smoothed projector:
$$
h_{\tau}(\mtx{K}) := \mtx{K} (\mtx{K} + \tau n \Id)^{-1}.
$$
The number of  effective degrees of freedom at regularization level $\tau$ is
$$
\nu_{\rm eff} := \trace h_{\tau}(\mtx{K}). %
$$
The maximum marginal number of degrees of freedom at regularization level $\tau$ is
$$
\nu_{\rm mof} := n \cdot \max_{i = 1, \dots, n} (h_{\tau}(\mtx{K}))_{ii}.
$$
Observe that $\nu_{\rm eff} \leq \nu_{\rm mof}$.
The statistic $\nu_{\rm mof}$ is analogous
with the coherence that appears
in our initial discussion of coordinate sampling
(Section~\ref{sec:coord-embed}).

The \emph{ridge leverage scores} at regularization level $\tau$ are the (normalized)
diagonal entries of the smoothed projector:
$$
p_i = \frac{(h_{\tau}(\mtx{K}))_{ii}}{\nu_{\rm eff}}
\quad\text{for $i = 1, \dots, n$.}
$$
Evidently, $(p_1, \dots, p_n)$ is a probability distribution.
The ridge leverage scores and related quantities are expensive to compute
directly, but there are efficient algorithms for approximating them well.
These approximations suffice for applications.
See Section~\ref{sec:ridge-leverage} for more discussion.

\begin{remark}[History]
\citeasnoun{2013_bach_sharp} identified the core role
of the smoothed projector for KRR.
\citeasnoun{AM15:Fast-Randomized} proposed the definition
of the ridge leverage scores and described a simple
method for approximating them.
At present, the most practical algorithm for approximating ridge leverage scores
appears in \citeasnoun{RCCR18:Fast-Leverage}.
\citeasnoun{MM17:Recursive-Sampling} recognize that
ridge leverage scores also have relevance for KPCA.
\end{remark}

\subsubsection{Uniform sampling}

The simplest way to select a set $I$ of $r$ coordinates
for the Nystr{\"o}m approximation $\mtx{K}\nys{I}$ is to draw
the set uniformly at random.  Although this approach
seems na{\"i}ve, it can be surprisingly effective in
practice.  The main failure mode occurs when there
are a few significant observations that make
outsize contributions to the kernel matrix;
uniform sampling is likely to miss these influential
data points.

\citeasnoun{2013_bach_sharp} proves that we can achieve
optimal learning guarantees for KRR with a uniformly sampled
Nystr{\"o}m approximation.  It suffices that the
number $r$ of coordinates is proportional to
$\nu_{\rm mof} \log n$.  A similar result holds
for KPCA.

Nystr{\"o}m approximation with uniform coordinate sampling
was proposed by \citeasnoun{WS01:Using-Nystrom}.
The theoretical and numerical performance
of this approach has been studied in many subsequent works, including
\cite{KMT12:Sampling-Methods,Git13:Topics-Randomized,2013_bach_sharp}
and \cite{RCR17:FALKON-Optimal}.

\subsubsection{Sampling with ridge leverage scores}
\label{sec:ridge-leverage}

Suppose that we have computed an approximation
of the ridge leverage score distribution.
We can construct a coordinate set $I$
for the Nystr{\"o}m approximation $\mtx{K}\nys{I}$
by sampling $r$ coordinates independently
at random from the ridge leverage score distribution.
Properly implemented, this method is unlikely to miss
influential observations.

\citeasnoun{AM15:Fast-Randomized} prove that we can achieve
optimal learning guarantees for KRR by ridge leverage score
sampling.  It suffices that the number $r$ of sampled
coordinates is proportional to $\nu_{\rm eff} \log n$.
This bound improves over the uniform sampling bound.
\citeasnoun{MM17:Recursive-Sampling} give related
theoretical results for KPCA.

Effective algorithms for estimating ridge leverage scores
are based on multilevel procedures that sequentially
improve the ridge leverage score estimates.
The basic idea is to start with a small uniform sample of coordinates, which we use to approximate the smoothed projector for a very large regularization parameter $\tau_0$. From this smoothed projector, we estimate the ridge leverage scores at level $\tau_0$. We then sample a larger set of coordinates non-uniformly using the approximate ridge leverage score distribution at level $\tau_0$. These samples allow us to approximate the smoothed projector at level $\tau_1 = \mathrm{const.} \tau_0$ for a constant smaller than one. We obtain an estimate for the ridge leverage score distribution at level $\tau_1$. This process is repeated.
In this way, the sampling and the matrix approximation
are intertwined.  See \cite{MM17:Recursive-Sampling}
and \cite{RCCR18:Fast-Leverage}.

\citeasnoun{MM17:Recursive-Sampling} provide empirical evidence
that ridge leverage score sampling is more efficient than
uniform sampling for KPCA, including the cost of the ridge
leverage score approximations.
Likewise, \citeasnoun{RCCR18:Fast-Leverage} report empirical evidence
that ridge leverage score sampling is more efficient
than uniform sampling for KRR.

\subsection{Random features approximation of kernels}
\label{sec:random-features}

A second approach to kernel approximation is based on
the method of empirical approximation (Section~\ref{sec:matrix-mc}).
This technique constructs a random rank-one matrix that
serves as an unbiased estimator for the kernel matrix.
By averaging many copies of the estimator, we can obtain
superior approximations of the kernel matrix.
The individual rank-one components
are called \emph{random features}.

\citeasnoun{Nea96:Priors-Infinite} proposed the idea
of using empirical approximation for kernels arising
in Gaussian process regression.
Later, \citeasnoun{RR08:Random-Features}
and \citeasnoun{RR08:Weighted-Sums}
developed empirical approximations for translation-invariant
kernels and Mercer kernels, and they coined the term
``random features.''  Our presentation is based on
an abstract formulation of the random feature method
from \citeasnoun[Sec.~6.5]{Tro15:Introduction-Matrix};
see also \citeasnoun{Bac17:Equivalence-Kernel}.

In this subsection, we introduce the idea of a random
feature map, along with some basic examples.
We explain how to use random feature maps to
construct empirical approximations of a kernel matrix,
and we give a short analysis.
Afterwards, we summarize two randomized NLA methods
for improving the computational profile of random features.

\subsubsection{Random feature maps}

In many cases, a kernel function $k$ on a domain $\mathcal{X}$
can be written as an expectation, and we can exploit this representation to obtain
empirical approximations of the kernel matrix.

Let $\mathcal{W}$ be a probability space equipped with a probability measure $\rho$.
Assume that there is a bounded function
$$
\psi : \mathcal{X} \times \mathcal{W} \to \{ z \in \CC : \abs{z} \leq b \}
$$
with the reproducing property
\begin{equation} \label{eqn:repro-prop}
k(\vct{x}, \vct{y}) = \int \psi(\vct{x}; \vct{w}) \, \psi(\vct{y}; \vct{w})^* \, \rho(\diff{\vct{w}})
\quad\text{for all $\vct{x}, \vct{y} \in \mathcal{X}$.}
\end{equation}
The star ${}^*$ denotes the conjugate of a complex number.
We call $(\psi, \rho)$ a \emph{random feature map} for the kernel function $k$.
As we will see in Section~\ref{sec:rf-kernel}, a kernel that admits a random feature map must be positive definite.

It is not obvious that we can equip kernels of practical interest with
random feature maps, so let us offer a few concrete examples.

\begin{example}[The inner product kernel]
\label{ex:innerprodkernel}

There are many ways to construct a random feature map
for the inner product kernel on $\F^d$.  One simple example is
$$
\psi(\vct{x}; \vct{w}) = \ip{ \vct{x} }{ \vct{w} }
\quad\text{with $\vct{w} \sim \textsc{normal}(\mtx{0}, \Id_d)$.}
$$
To check that this map satisfies the reproducing property~\eqref{eqn:repro-prop},
just note that $\vct{w}$ is isotropic: $\Expect[ \vct{ww}^* ] = \Id$.
This formulation is closely related to the theory of random embeddings (Sections~\ref{sec:gauss} and \ref{sec:dimension-reduction})
and to approximate matrix multiplication (Section~\ref{sec:approx-mtx-mult}).
\end{example}

\begin{example}[Angular similarity kernel]
\label{ex:angsim-rf}

For the angular similarity map defined in Example~\ref{ex:angsim},
we can construct a random feature map using an elegant fact from geometry.
Indeed, the function
$$
\psi(\vct{x}; \vct{w}) = \operatorname{sgn}\ \ip{\vct{x}}{\vct{w}}
\quad\text{with $\vct{w} \sim \textsc{uniform}(\mathbb{S}^{d-1}(\RR))$}
$$
gives a random feature map for the angular similarity kernel.
As a consequence, the angular similarity kernel is positive definite.
\end{example}

\begin{example}[Translation-invariant kernels]
\label{ex:rbf-rf}

A kernel function $k$ on $\F^d$ is called
\emph{translation-invariant} if it has the form
$$
k( \vct{x}, \vct{y} ) = \phi( \vct{x} - \vct{y} )
\quad\text{for all $\vct{x}, \vct{y} \in \F^d$.}
$$
A classic result from analysis, Bochner's theorem~\cite{Boc33:Monotone-Funktionen},
gives a characterization of these kernels.
A kernel is continuous, positive definite and translation-invariant
if and only if it is the Fourier transform of a positive probability measure $\rho$
on $\F^d$:
$$
\phi( \vct{x} - \vct{y} ) = c_{\phi} \int \econst^{\iunit \ip{\vct{x}}{\vct{w}}} \econst^{-\iunit \ip{\vct{y}}{\vct{w}}} \, \rho(\diff{\vct{w}}).
$$
The constant $c_{\phi}$ is a normalizing factor that depends only on $\phi$,
and $\iunit$ is the imaginary unit.

Bochner's theorem immediately delivers a random feature map
for the translation-invariant kernel $k$:
$$
\psi( \vct{x}; \vct{w} ) = \sqrt{c_{\phi}} \, \econst^{\iunit \ip{\vct{x}}{\vct{w}}}
\quad\text{where $\vct{w} \sim \rho$.}
$$
This was one of the original examples
of a random feature map~\cite{RR08:Random-Features}.
When working with data in $\RR^d$, the construction
can also be modified to avoid complex values.

The key example of a positive definite, translation-invariant kernel is the
Gaussian kernel on $\F^d$, defined in Example~\ref{ex:rbf}.
The Gaussian kernel is derived from the function
$$
\phi( \vct{x} ) = \econst^{-\norm{\vct{x}}^2 / (2 \sigma^2)}
\quad\text{where the bandwidth $\sigma > 0$.}
$$
The associated random feature map is
$$
\psi( \vct{x}; \vct{w} ) = \econst^{\iunit \ip{\vct{x}}{\vct{w}}}
\quad\text{where}\quad
\vct{w} \sim \textsc{normal}(\vct{0}, \sigma^{-2} \Id ) \in \F^d.
$$
This fact is both beautiful and useful because of the ubiquity
of the Gaussian kernel in data analysis.
\end{example}

There are many other kinds of kernels that admit random feature maps.
Random feature maps for dot product kernels were obtained in
\cite{KK12:Random-Feature,PP13:Fast-Scalable} and \cite{HXGD14:Compact-Random}.
For nonstationary kernels, see \cite{SR15:Generalized-Spectral} and
\cite{TFSB18:Spatial-Mapping}. %
Catalogs of examples appear in \citeasnoun[App.~E]{RR17:Generalization-Properties}
and \citeasnoun{Bac17:Equivalence-Kernel}.

\subsubsection{Random features and kernel matrix approximation}
\label{sec:rf-kernel}

We can use the random feature map to construct an empirical approximation of the kernel matrix
$\mtx{K} \in \Sym_n$ induced by the dataset $\{ \vct{x}_1, \dots, \vct{x}_n \}$.
To do so, we draw a random variable $\vct{w} \in \mathcal{W}$ with
the distribution $\rho$.  Then we form a random vector
$$
\vct{z} = \begin{bmatrix} z_1 \\ \vdots \\ z_n \end{bmatrix}
	= \begin{bmatrix} \psi(\vct{x}_1; \vct{w}) \\ \vdots \\ \psi(\vct{x}_n; \vct{w}) \end{bmatrix} \in \F^n.
$$
Note that we are using the \emph{same} random variable $\vct{w}$ for each data point.
A realization of the random vector $\vct{z}$ is called a \emph{random feature}.
The reproducing property~\eqref{eqn:repro-prop} ensures that the random feature
verifies the identity
$$
(\mtx{K})_{ij} = k(\vct{x}_i, \vct{x}_j)
	= \int \psi(\vct{x}_i;\vct{w}) \, \psi(\vct{x}_j;\vct{w})^* \, \rho(\diff{\vct{w}})
	= \Expect[ z_i \cdot z_j^* ].
$$
In matrix form,
$$
\mtx{K} = \Expect[ \vct{zz}^* ].
$$
Therefore, the random rank-one psd matrix $\mtx{Z} = \vct{zz}^*$
is an unbiased estimator for the kernel matrix.
The latter display proves that
a kernel $k$ must be positive definite
if it admits a random feature map.

To approximate the kernel matrix, we can average $r$ copies of the rank-one estimator:
$$
\bar{\mtx{Z}}_r:= \frac{1}{r} \sum\nolimits_{i=1}^r \mtx{Z}_i
\quad\text{where $\mtx{Z}_i \sim \mtx{Z}$ are iid.}
$$
If the points in the input space $\mathcal{X}$ are parameterized by $d$ numbers,
each random feature typically requires $\bigO(nd)$ arithmetic.
The total cost of forming the approximation $\bar{\mtx{Z}}_r$ is thus $\bigO(rnd)$.
When $r \ll n$, we can obtain substantial improvements over the direct approach
of computing the kernel matrix $\mtx{K}$ explicitly at a cost of $\bigO(n^2 d)$.

\subsubsection{Analysis of the random feature approximation}

How many random features are enough to approximate the
kernel matrix in spectral norm?  Theorem~\ref{thm:mtx-sampling} delivers bounds. %

For simplicity, assume that the kernel satisfies $k(\vct{x}, \vct{x}) = 1$
for all $\vct{x} \in \mathcal{X}$; the angular similarity kernel and
the Gaussian kernel both enjoy this property.
For an accuracy parameter $\eps > 0$, suppose that we select
$$
r \geq 2b \eps^{-2} \intdim(\mtx{K}) \log(2n).
$$
The number $b$ is the uniform bound on the feature map $\psi$ defined in~\eqref{eqn:repro-prop},
and the intrinsic dimension is defined in~\eqref{eqn:intdim}.
Theorem~\ref{thm:mtx-sampling} implies that
the empirical approximation $\bar{\mtx{Z}}_r$ satisfies
$$
\frac{\Expect \norm{ \bar{\mtx{Z}}_r - \mtx{K} }}{\norm{\mtx{K}}}
	\leq \eps + \eps^2.
$$
In other words, we achieve a relative error approximation of
the kernel matrix in spectral norm when the number $r$ of random
features is proportional to the number of energetic dimensions
in the range of the matrix $\mtx{K}$.

This analysis is due to~\citeasnoun{LSS+14:Randomized-Nonlinear};
see also~\cite[Sec.~6.5]{Tro15:Introduction-Matrix}.
For learning applications, such as KPCA or KRR,
this result suggests that we need about $\bigO(n\log n)$
random features to obtain optimal generalization guarantees (where $\eps = n^{-1/2}$).
In fact, roughly $\bigO(\sqrt{n} \, \log n)$ random features are sufficient
to achieve optimal learning rates.  This claim depends on involved
arguments from learning theory that are outside the realm of linear algebra.
For example, see \cite{SS15:Optimal-Rates,RR17:Generalization-Properties,UMMA18:Streaming-Kernel,SS19:Kernel-Derivative}
and \cite{Wan19:Simple-Almost}.

\subsubsection{Randomized embeddings and random features}
\label{sec:fastfood}

Random feature approximations are faster than explicit computation
of a kernel matrix.
Even so, it takes a significant amount of effort to extract
$r$ random features and form an empirical approximation $\bar{\mtx{Z}}_r$
of the kernel matrix.  Several groups have proposed using structured
random embeddings (Section~\ref{sec:dimension-reduction}) to accelerate this process;
see \cite{PP13:Fast-Scalable,LSS13:Fastfood-Computing} and \cite{HXGD14:Compact-Random}.

As an example, let us summarize a heuristic method, called \emph{FFT Fastfood}
\cite{LSS13:Fastfood-Computing}, for speeding up the computation of random features
for the complex Gaussian kernel with bandwidth $\sigma^2$.
Consider the matrix $\mtx{X}$ formed from $n$ observations in $\CC^d$:
$$
\mtx{X} = \begin{bmatrix} \vct{x}_1^* \\ \vdots \\ \vct{x}_n^* \end{bmatrix} \in \CC^{n \times d}.
$$
Let $\mtx{\Gamma} \in \CC^{d \times d}$ be a matrix with iid
complex $\textsc{normal}(0, \sigma^{-2})$ entries.  Then we can simultaneously compute
$d$ random features $\vct{z}_1, \dots, \vct{z}_d$ for the Gaussian
kernel by forming a matrix product and applying the exponential map:
$$
\exp\cdot\,(\iunit \mtx{X} \mtx{\Gamma}) = \begin{bmatrix} \vct{z}_1 & \dots & \vct{z}_d \end{bmatrix} \in \CC^{n \times d}.
$$
We have written $\exp\cdot$ for the \emph{entrywise} exponential.
This procedure typically involves $\bigO(nd^2)$ arithmetic. %

The idea behind FFT Fastfood is to accelerate this computation
by replacing the Gaussian matrix with a structured random matrix.
This exchange is motivated by the observed universality properties
of random embeddings.  Consider a random matrix of the form
$$
\mtx{S} = \frac{1}{\sigma} \, \mtx{E} \mtx{\Pi} \mtx{F} \in \CC^{d \times d},
$$
where  $\mtx{E}$ is a random sign flip, $\mtx{\Pi}$ is a random permutation,
and $\mtx{F}$ is the discrete DFT. The FFT algorithm supports efficient
matrix products with $\mtx{S}$.  Therefore,
we can simultaneously extract $d$ random features by computing $\exp\cdot\,(\iunit \mtx{XS}) \in \CC^{n \times d}$.
This procedure uses only $\bigO(nd \log d)$ operations.
To form $r$ random features where $r > d$, we simply repeat
the same process $\lceil r/d \rceil$ times.

Compared with using a Gaussian matrix product,
FFT Fastfood gives a substantial reduction in arithmetic.
Even so, the performance for learning is almost identical
to a direct application of the random feature approximation.

\subsubsection{Random features and streaming matrix approximation}
\label{sec:streaming-kpca}

Suppose that we wish to perform KPCA.
The direct random features approach requires us to form the empirical
approximation $\bar{\mtx{Z}}_r$ of the kernel matrix $\mtx{K}$ and to compute
its rank-$\ell$ truncated eigenvalue decomposition.
It is often the case that the desired number $\ell$ of principal components
is far smaller than the number $r$ of random features we need to
obtain a suitable approximation of the kernel matrix.
In this case, we can combine random features with streaming matrix approximation %
to make economies in storage and computation.

Let $\{ \vct{z}_1, \vct{z}_2, \vct{z}_3, \dots \} \subset \F^n$
be an iid sequence of random features for the kernel matrix $\mtx{K} \in \Sym_n$.
The empirical approximation $\bar{\mtx{Z}}_r$ of the kernel matrix, obtained from the first $r$
random features, follows the recursion
$$
\bar{\mtx{Z}}_0 = \mtx{0}
\quad\text{and}\quad
\bar{\mtx{Z}}_t = (1 - t^{-1}) \bar{\mtx{Z}}_{t-1} + t^{-1} \, \vct{z}_t \vct{z}_t^*
\quad\text{for $t = 1, 2, 3, \dots$.}
$$
This is a psd matrix, generated by a stream of linear updates.  Therefore,
we can track the evolution using a streaming Nystr{\"o}m approximation
(Section~\ref{sec:nystrom}).

Let $\mtx{\Omega} \in \F^{n \times s}$ be a random test matrix,
with $s > \ell$.  By performing rank-one updates, we can efficiently maintain
the sample matrices
$$
\mtx{Y}_t = \bar{\mtx{Z}}_t \mtx{\Omega} \in \F^{n \times s}
\quad\text{for $t = 1, 2, 3, \dots$.}
$$
After collecting a sufficient number $r$ of samples, we can
apply Algorithm~\ref{alg:random-nystrom} to $\mtx{Y}_t$ to
obtain a near-optimal rank-$\ell$ eigenvalue decomposition
of the empirical approximation $\bar{\mtx{Z}}_r$.

It usually suffices to take the sketch size $s$ to be proportional
to the rank $\ell$ of the truncated eigenvalue decomposition.
In this case, the overall approach uses $\bigO(\ell n)$ storage.  We can generate and
process $r$ random features using $\bigO((d+\ell)rn)$ arithmetic, where $d$ is the dimension of $\mathcal{X}$.
The subsequent cost of the Nystr{\"o}m approximation is $\bigO(\ell^2 n)$ operations.
The streaming random features approach has storage and arithmetic costs
roughly $\ell / r$ times those of the direct random features approach.
The streaming method can be combined with dimension reduction techniques
(Section~\ref{sec:fastfood}) for further acceleration.

\begin{remark}[History]
\citeasnoun{GPP16:Streaming-Kernel} proposed using a stream of
random features to perform KPCA; their algorithm tracks the stream
with the (deterministic) frequent directions sketch~\cite{GLPW16:Frequent-Directions}.
We have presented a new variant, based on randomized Nystr{\"o}m approximation,
that is motivated by the work in~\citeasnoun{TYUC17:Fixed-Rank-Approximation}.
\citeasnoun{UMMA18:Streaming-Kernel} have develop a somewhat
different streaming KPCA algorithm based on Oja's method~\cite{Oja82:Simplified-Neuron}.
At present, we lack a full empirical comparison of these alternatives.
\end{remark}

\newcommand{\thedomain}{D}

\section{High-accuracy approximation of kernel matrices}
\label{sec:rankstructured}

In this section, we continue the discussion of kernel matrices that we started in Section \ref{sec:kernel},
but we now consider the high-accuracy regime.
In particular, given a kernel matrix $\mtx{K}$, we seek an approximation $\mtx{K}_{\rm approx}$ for which $\|\mtx{K} - \mtx{K}_{\rm approx}\|$
is small, say of relative accuracy $10^{-3}$ or $10^{-6}$.
This objective was not realistic for the applications discussed in Section \ref{sec:kernel},
but it can be achieved in situations where we have access to fast techniques for evaluating
the matrix-vector product $\vct{x} \mapsto \mtx{K}\vct{x}$ (and also $\vct{x} \mapsto \mtx{K}^{*}\vct{x}$ when $\mtx{K}$
is not self-adjoint).
The algorithms that we describe will build a data sparse approximation
to $\mtx{K}$ by using information in samples such as $\mtx{K}\vct{x}$ and $\mtx{K}^{*}\vct{x}$ for
random vectors $\vct{x}$.
These techniques are particularly well suited to problems that arise in modeling physical
phenomena such as electromagnetic scattering, or the deformation of solid bodies; we describe
how the fast matrix-vector application we need can be realized in Section \ref{sec:cpkernel_application}.

As in Section \ref{sec:kernel}, we say that a matrix $\mtx{K} \in \mathbb{C}^{n\times n}$ is a
\emph{kernel matrix} if its entries are given by a formula such as
\begin{equation}
\label{eq:Aij_kernel}
\mtx{K}(i,j) = k(\vct{x}_{i},\vct{x}_{j}),
\end{equation}
where $\{\vct{x}_{i}\}_{i=1}^{n}$ is a set of points in $\mathbb{R}^{d}$, and where
$k\,\colon\,\mathbb{R}^{d} \times \mathbb{R}^{d} \rightarrow \mathbb{C}$ is a kernel function.
The kernel matrices that we consider tend to have singular values that decay slowly or not at
all, which rules out the possibility that $\mtx{K}_{\rm approx}$ could have low rank.
Instead, we build an approximation
$\mtx{K}_{\rm approx}$ that is tessellated into $\bigO(n)$ blocks in such as way that each
off-diagonal block has low rank. Figure \ref{fig:tree}(b) shows a representative
tessellation pattern. We say that a matrix of this type is a \textit{rank-structured
hierarchical matrix.}

The purpose of determining a rank-structured approximation to an operator for which we
already have fast matrix-vector multiplication techniques available is that the new
representation can be used to rapidly execute a whole range of linear algebraic operations:
matrix inversion, LU factorization, and even full spectral decompositions in certain cases.

This section is structured to provide a high-level description of the core ideas in Sections
\ref{sec:sepvar} -- \ref{sec:cpkernel_application}. Additional details follow in Sections \ref{sec:HODLR} -- \ref{sec:ASKIT}.

\subsection{Separation of variables and low-rank approximation}
\label{sec:sepvar}

The reason that many kernel matrices can be tessellated into blocks that have low numerical
rank is that the function $(\vct{x},\vct{y}) \mapsto k(\vct{x},\vct{y})$ is typically smooth
as long as $\vct{x}$ and $\vct{y}$ are not close.
To illustrate the connection, let us consider a computational domain $\thedomain$ that holds a set
of points $\{\vct{x}_{i}\}_{i=1}^{n} \subset \mathbb{R}^{2}$, as shown in Figure \ref{fig:sources_targets}.
Suppose further that $\thedomain_{\rm s}$ and $\thedomain_{\rm t}$ are two subdomains of $\thedomain$ that are
located a bit apart from each other, as shown in the figure. When the kernel function $k$ is smooth,
we can typically approximate it to high accuracy through an approximate separation of variables of the form
\begin{equation}
\label{eq:kernel_sepvar}
k(\vct{x},\vct{y}) \approx \sum_{p=1}^{P}b_{p}(\vct{x})\,c_{p}(\vct{y}),\qquad\vct{x} \in \thedomain_{\rm t},\ \vct{y} \in \thedomain_{\rm s}.
\end{equation}
Let $I_{\rm s}$ and $I_{\rm t}$ denote two index vectors that identify the points
located in $\thedomain_{\rm s}$ and $\thedomain_{\rm t}$, respectively (so that,
for example, $i\in I_{\rm s}$ if and
only if $\vct{x}_{i} \in \thedomain_{\rm s}$). Then combining the formula (\ref{eq:Aij_kernel}) with
the separation of variables (\ref{eq:kernel_sepvar}), we get
\begin{equation}
\label{eq:A_sepvar}
\mtx{K}(i,j) \approx \sum_{p=1}^{P}b_{p}(\vct{x}_{i})\,c_{p}(\vct{x}_{j}),\qquad i \in I_{\rm t},\ j \in I_{\rm s}.
\end{equation}
Equation (\ref{eq:A_sepvar}) is exactly the low-rank approximation to the block $\mtx{K}(I_{\rm t},I_{\rm s})$
that we seek. To be precise, (\ref{eq:A_sepvar}) can be written as
$$
\mtx{K}(I_{\rm t},I_{\rm s}) \approx \mtx{B}\mtx{C},
$$
where $\mtx{B}$ and $\mtx{C}$ are defined via
$\mtx{B}(i,p) = b_{p}(\vct{x}_{i})$ and
$\mtx{C}(p,j) = c_{p}(\vct{x}_{j})$.

A separation of variables such as (\ref{eq:kernel_sepvar}) is sometimes provided through analytic knowledge
of the kernel function, as illustrated in Example \ref{example:multipole}.
Perhaps more typically, all we know is that such a formula \textit{should} in principle exist, for instance
because we know that the matrix approximates a singular integral operator for which a Calder\'on-Zygmund
decomposition must exist. It is then the task of the randomized algorithm to explicitly build the factors
$\mtx{B}$ and $\mtx{C}$, given a computational tolerance.

\begin{figure}
\setlength{\unitlength}{1mm}
\begin{picture}(123,67)
\put(05,00){\includegraphics[width=65mm]{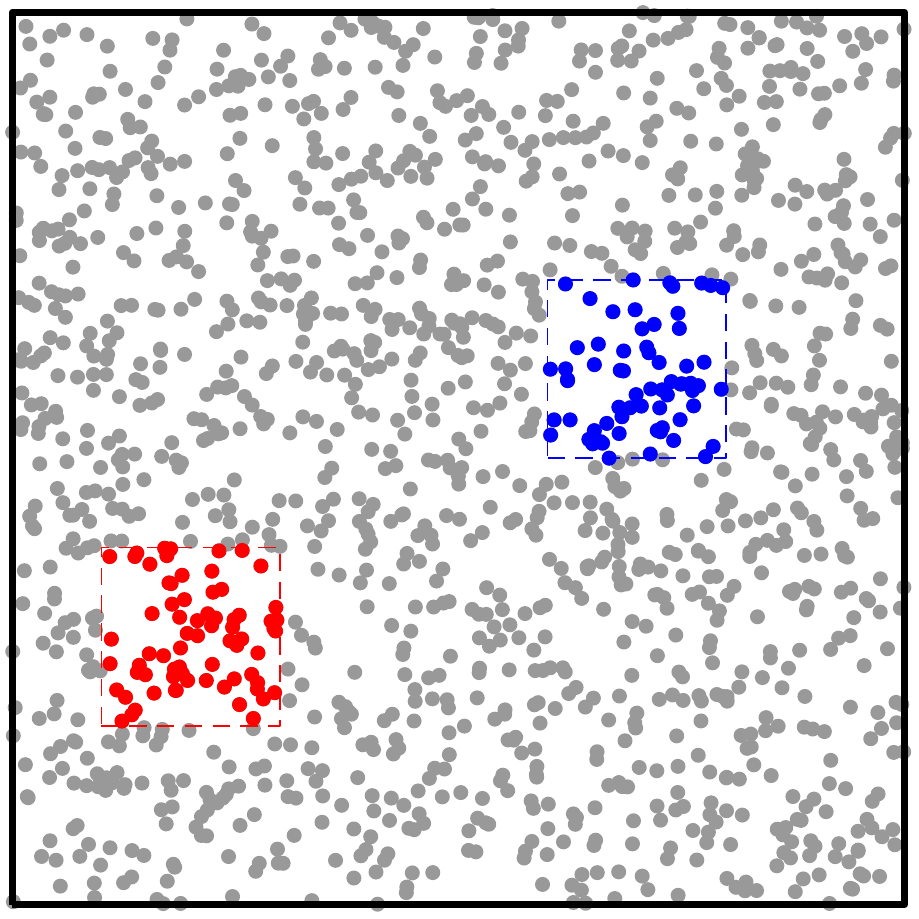}}
\put(00,03){\large$\thedomain$}
\put(26,15){\large\color{red} $\bm\thedomain_{\rm s}$}
\put(57,34){\large\color{blue}$\bm\thedomain_{\rm t}$}
\put(72,34){\begin{minipage}{50mm}
A box $\thedomain$  holding points $\{\vct{x}_{i}\}_{i=1}^{n}$ (the blue, red, and gray dots)
that define a kernel matrix $\mtx{K}(i,j) = k(\vct{x}_{i},\vct{x}_{j})$.
The regions $\thedomain_{\rm s}$ (the red box) and $\thedomain_{\rm t}$ (the blue box) are separated enough
that $k(\vct{x},\vct{y})$ is smooth when $\vct{x} \in \thedomain_{\rm t}$ and $\vct{y} \in \thedomain_{\rm s}$.
In consequence, $\mtx{K}(I_{\rm t},I_{\rm s})$ has low numerical rank,
where $I_{\rm s}$ identifies the red points and $I_{\rm t}$ the blue.
\end{minipage}}
\end{picture}
\caption{The geometry discussed in Section \ref{eq:kernel_sepvar}.}
\label{fig:sources_targets}
\end{figure}

\begin{example}[Laplace kernel in two dimensions]
\label{example:multipole}
A standard example of a kernel matrix in mathematical physics is the matrix $\mtx{K}$ that maps a vector
of electric source strengths $\vct{q} = (q_{j})_{j=1}^{n}$ to a vector of potentials $\vct{u} = (u_{i})_{i=1}^{n}$.
When the set $\{\vct{x}_{i}\}_{i=1}^{n} \subset \mathbb{R}^{2}$ identifies the locations of both the source
and the target points, $\mtx{K}$ takes the form (\ref{eq:Aij_kernel}), for
$$
k(\vct{x},\vct{y}) =
\left\{\begin{aligned}
\log|\vct{x}-\vct{y}|,&\quad\mbox{when}\ \vct{x}\neq \vct{y},\\
0,&\quad\mbox{when}\ \vct{x} = \vct{y}.
\end{aligned}\right.
$$
There is a well-known result from potential theory that provides the required
separation of variables. Expressing $\vct{x}$ and $\vct{y}$ in polar coordinates
with respect to some expansion center $\vct{c}$, so that
$\vct{x} - \vct{c} = r(\cos\theta,\,\sin\theta)$ and
$\vct{y} - \vct{c} = r'(\cos\theta',\,\sin\theta')$,
the separation of variables (known as a \textit{multipole expansion})
\begin{equation}
\label{eq:mpole}
k(\vct{x},\vct{y}) = \log(r) - \sum_{p=1}^{\infty}\frac{1}{p}\left(\frac{r'}{r}\right)^{p}\bigl(\cos(p\theta)\cos(p\theta') + \sin(p\theta)\sin(p\theta')\bigr)
\end{equation}
is valid whenever $r' < r$.
\end{example}

\subsection{Rank-structured matrices and randomized compression}
\label{sec:bestiary}

The observation that the off-diagonal blocks of a kernel matrix often have low
numerical rank underpins many ``fast'' algorithms in computational physics.
In particular, it is the foundation of the Barnes-Hut \cite{1986_barnes_hut}
and fast multipole methods \cite{rokhlin1987} for evaluating all $\bigO(n^{2})$
pairwise interactions between $n$ electrically charged particles in linear or
close-to-linear complexity. These methods were generalized by Hackbusch and co-workers,
who developed the $\mathcal{H}$- and $\mathcal{H}^{2}$-matrix frameworks
\cite{hackbusch,2003_hackbusch,2002_hackbusch_H2}. These explicitly linear algebraic
formulations enable fast algorithms not only for matrix-vector multiplication but
also for matrix inversion, LU factorization, matrix-matrix multiplication and many more.

A fundamental challenge that arises when rank-structured matrix formulations are used
is how to find the data sparse representation of the operator in the first place. The
straightforward approach would be to form the full matrix, and then loop over all the
compressible off diagonal blocks and compress them using, for example, a singular value decomposition.
The cost of such a process is necessarily at least $\bigO(n^{2})$, which is rarely
affordable. Using randomized compression techniques, it turns out to be possible to compress
all the off-diagonal blocks \textit{jointly}, without any need for sampling each block
individually. In this section, we describe two such methods. Both require
the user to supply a fast algorithm for applying the full matrix (and its adjoint) to
vectors; see Section \ref{sec:cpkernel_application}.
The two methods have different computational profiles.
\begin{itemize}
\item[(a)] The technique described in Sections \ref{sec:HODLR} and \ref{sec:HODLRcompression} is
a true ``black-box'' technique that interacts with $\mtx{K}$ only through the matrix-vector
multiplication. It has storage complexity $O(n\log n)$ and it requires $\bigO(\log n)$
applications of $\mtx{K}$ to a random matrix of size $n\times (r+p)$ where $r$ is an upper
bound on the ranks of the off-diagonal blocks, and $p$ is a small over-sampling parameter.
\item[(b)] The technique described in Sections \ref{sec:HBS} and \ref{sec:HBScompression} attains
true linear $O(n)$ complexity, and requires only a single application of $\mtx{K}$ and $\mtx{K}^{*}$
to a random matrix of size $n\times (r+p)$, where $r$ and $p$ are as in (a). Its drawbacks are
that it requires evaluation of $\bigO(rn)$ individual matrix entries, and that it works
only for a smaller class of matrices.
\end{itemize}

\begin{remark}[Scope of Section \ref{sec:rankstructured}]
To keep the presentation as uncluttered by burdensome notation as possible,
in this survey we
restrict attention to two basic ``formats'' for representing a rank-structured hierarchical
matrix. In Sections \ref{sec:HODLR} and \ref{sec:HODLRcompression}, we use the
\textit{hierarchically off-diagonal low rank (HODLR)} format and in Sections
\ref{sec:HBS} and \ref{sec:HBScompression} we use the \textit{hierarchically block separable (HBS)}
format (sometimes referred to as \textit{hierarchically semi separable (HSS)} matrices). The
main limitation of these formats is that they require all off-diagonal blocks of the matrix to have
low numerical rank. This is realistic only when the points $\{\vct{x}_{i}\}_{i=1}^{n}$ are restricted
to a low-dimensional manifold. In practical applications, one sometimes has to leave a larger part of
the matrix uncompressed, to avoid attempting to impose a separation of variables such as
(\ref{eq:kernel_sepvar}) on a kernel $k = k(\vct{x},\vct{y})$ when $\vct{x}$ and $\vct{y}$ are
too close. This is done through enforcing what is called a ``strong admissibility condition''
(in contrast to the ``weak admissibility condition'' of the HODLR and HBS formats), as was done
in the original Barnes-Hut and fast multipole methods.
\end{remark}

\begin{remark}[Alternative compression strategies]
\label{remark:compressioncompetition}
Let us briefly describe what alternatives to randomized compression exist.
The original papers on $\mathcal{H}$-matrices used Taylor approximations to
derive a separation of variables, but this works only when the kernel is
given analytically. It also tends to be quite expensive.
The  \textit{adaptive cross-approximation (ACA)} technique of \cite{2002_rjasanow_ACA} and \cite{2006_bebendorf_ACA}
relies on using ``natural basis'' vectors (see Section \ref{sec:CURandIDdef}) that are
found using semi-heuristic techniques. The method can work very well in practice,
but is not guaranteed to provide an accurate factorization.
When the kernel matrix comes from mathematical physics, specialized techniques
that exploit mathematical properties of the kernel function often perform
well \cite{2005_martinsson_fastdirect}.
For a detailed discussion of compression of rank-structured matrices, we refer
to \cite[Ch.~17]{2019_book}.
\end{remark}

\subsection{Computational environments}
\label{sec:cpkernel_application}

The compression techniques that we describe rely on the user providing
a fast algorithm for applying the operator to be compressed to vectors. Representative
environments where such fast algorithms are available include the following.

\textit{Matrix-matrix multiplication:} Suppose that $\mtx{K} = \mtx{B}\mtx{C}$,
where $\mtx{B}$ and $\mtx{C}$ are matrices that can rapidly be applied to vectors.
Then the randomized compression techniques allow us to compute their product.

\textit{Compression of boundary integral operators:} It is often possible to
reformulate a boundary value problem involving an elliptic partial different
operator, such as the Laplace or the Helmholtz operators, as an equivalent
boundary integral equation (BIE), see \cite[Part III]{2019_book}.
When such a BIE is discretized, the result is a kernel matrix which can
be applied rapidly to vectors using fast summation techniques such as the
fast multipole method \cite{rokhlin1987}. Randomized compression techniques
allow us to build an approximation to the matrix that can be
factorized or inverted, thus enabling direct (as opposed
to iterative) solvers.

\textit{Dirichlet-to-Neumann (DtN) operators:} A singular integral operator of central
importance in engineering, physics, and scientific computing is the DtN operator
which maps given Dirichlet data for an elliptic boundary value problem to the
boundary fluxes of the corresponding solution. The kernel of the DtN operator
is rarely known analytically, but the operator can be applied through a fast solver
for the PDE, such as, e.g., a finite-element discretization combined with a multigrid
solver. The randomized techniques described allow for the DtN operator to
be built and stored explicitly.

\textit{Frontal matrices in sparse direct solvers:} Suppose that $\mtx{S}$ is a large
sparse matrix arising from the discretization of an elliptic PDEs. A common technique
for solving $\mtx{S}\vct{x} = \vct{b}$ is to compute an LU factorization of the
matrix $\mtx{S}$. There are ways to do this that preserve sparsity as far as possible,
but in the course of the factorization procedure, certain dense matrices of increasing
size will need to be factorized. These matrices turn out to be kernel matrices
with kernels that are not known explicitly, but that can be built using the techniques described here
\cite{2013_xia_randomized,2017_ghysels_robust}.

\subsection{Hierarchically off-diagonal low-rank matrices}
\label{sec:HODLR}

In order to illustrate how randomized methods can be used to construct data
sparse representations of rank-structured matrices, we will describe a particularly
simple `format' in this section that is often referred to as the hierarchically
off-diagonal low-rank (HODLR) format.
This is a basic format that works well when the points
are organized on a one- or two-dimensional manifold.

The first step towards defining the HODLR format is to
build a binary  tree on the index vector $I = [1,2,3,\dots,n]$ through a process that is
illustrated in Figure \ref{fig:tree}(a).
With any node $\tau$
in the tree, we associate an index vector $I_{\tau} \subseteq I$. The root of the tree is given the index $\tau=1$,
and we associate it with the full index vector, so that $I_{1} = I$. At the next finer level
of the tree, we split $I_{1}$ into two parts $I_{2}$ and $I_{3}$ so that $I_{1} = I_{2} \cup I_{3}$
forms a disjoint partition. Then continue splitting the index vectors until each remaining
index vector is ``short''. (Exactly what ``short'' means is application-dependent, but one
may think of a short vector as holding a few hundred indices or so.) We let $\ell$ denote
a \textit{level} of the tree, with $\ell=0$ denoting the root, so that level $\ell$ holds
$2^{\ell}$ nodes. We use the terms \textit{parent} and \textit{child} in the obvious way,
and say that a pair of nodes $\{\alpha,\beta\}$ forms a \textit{sibling pair} if they have
the same parent. A \textit{leaf} node is of course a node that has no children.

The binary tree that we defined induces a natural tessellation of the kernel matrix $\mtx{K}$ into
$\bigO(n)$ blocks. Figure \ref{fig:tree}(b) shows the tessellation that follows from the tree
in Figure \ref{fig:tree}(a). Each parent node $\tau$ in the tree gives rise to two
off-diagonal blocks that both have low numerical rank. Letting $\{\alpha,\beta\}$ denote
the children of $\tau$, these two blocks are
\begin{equation}
\label{eq:siblingboxes}
\mtx{K}_{\alpha,\beta} = \mtx{K}(I_{\alpha},I_{\beta})
\qquad\mbox{and}\qquad
\mtx{K}_{\beta,\alpha} = \mtx{K}(I_{\beta},I_{\alpha}).
\end{equation}
For each leaf node $\tau$, we define a corresponding diagonal block as
\begin{equation}
\label{eq:diagbox}
\mtx{D}_{\tau} = \mtx{K}(I_{\tau},I_{\tau}).
\end{equation}
The disjoint partition of $\mtx{K}$ into blocks is now formed by all sibling
pairs, as defined in (\ref{eq:siblingboxes}), together with all diagonal blocks,
as defined by (\ref{eq:diagbox}). When each off-diagonal block in this tessellation
has low rank, we say that $\mtx{K}$ is a \textit{hierarchically off-diagonal low-rank (HODLR)} matrix.

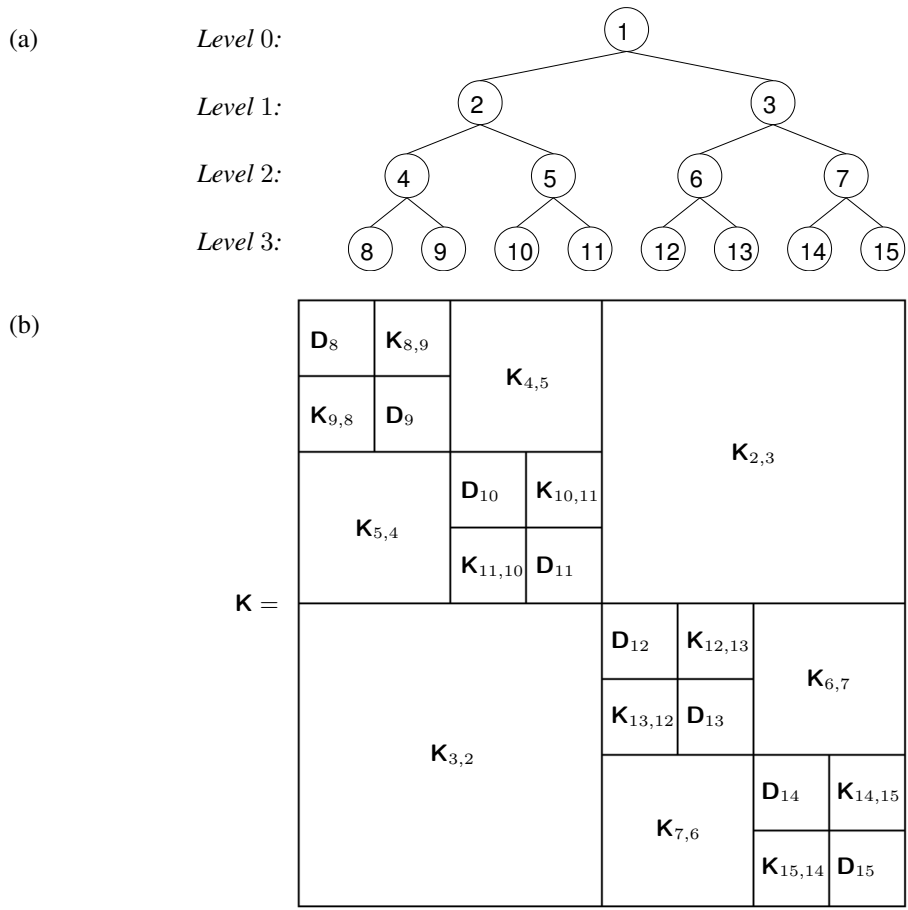
\begin{figure}
\centering
\setlength{\unitlength}{1mm}
\begin{picture}(120,120)
\put(00,115){(a)}
\put(25,85){\begin{picture}(122,44)
\put( 20, 0){\includegraphics[width=75mm]{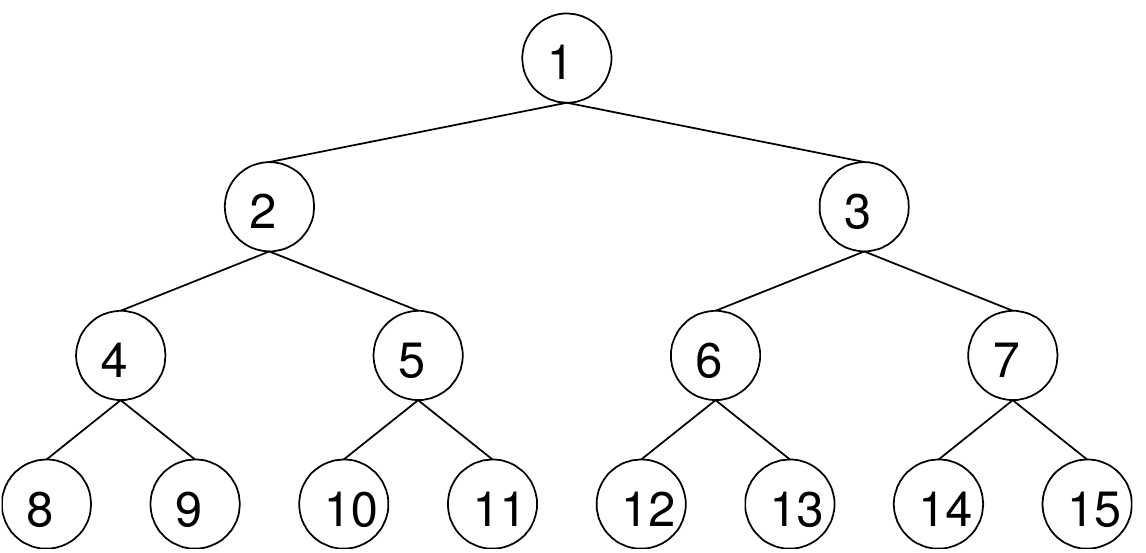}}
\put(  0,30){\textit{Level $0$:}}
\put(  0,21){\textit{Level $1$:}}
\put(  0,12){\textit{Level $2$:}}
\put(  0,03){\textit{Level $3$:}}
\end{picture}}
\put(00,77){(b)}
\put(40,02){\begin{picture}(80,80)
\put(-2,-2){\includegraphics[width=82mm]{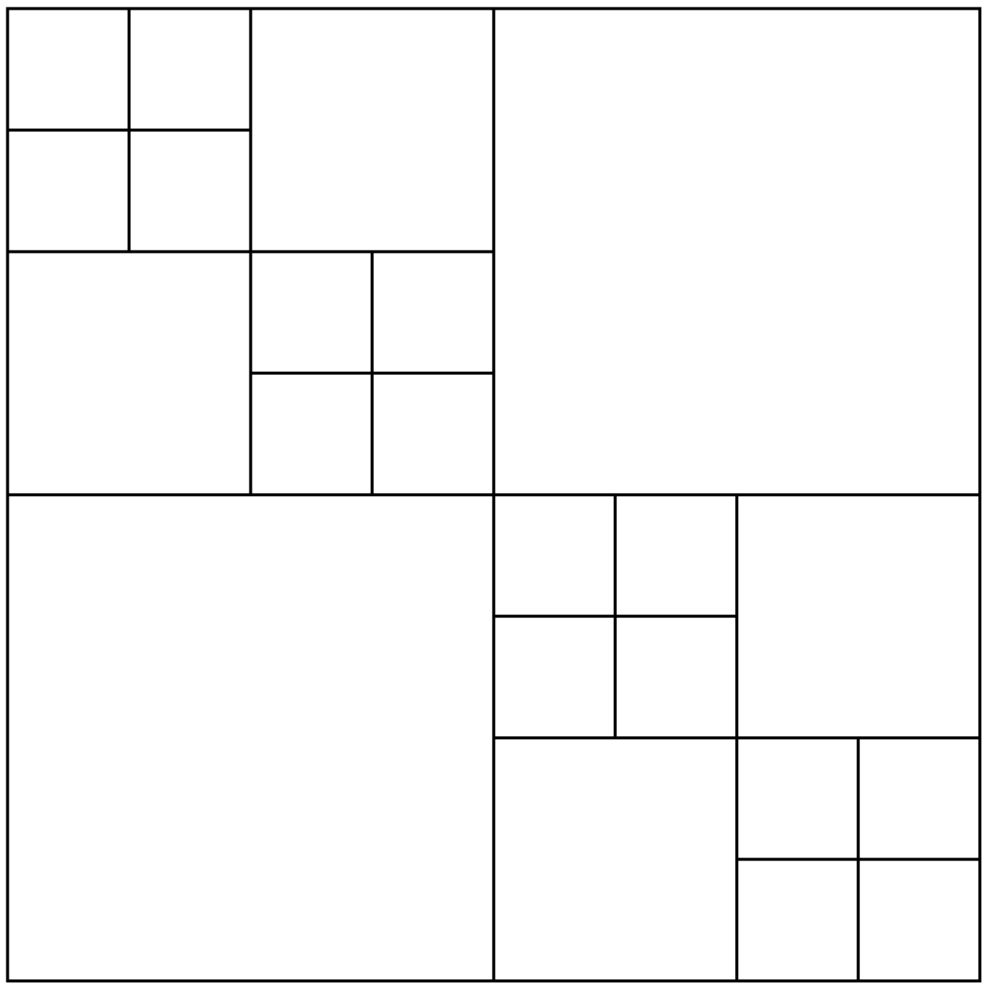}}
\put(00,73){\small$\mtx{D}_{8}$}
\put(10,63){\small$\mtx{D}_{9}$}
\put(20,53){\small$\mtx{D}_{10}$}
\put(30,43){\small$\mtx{D}_{11}$}
\put(40,33){\small$\mtx{D}_{12}$}
\put(50,23){\small$\mtx{D}_{13}$}
\put(60,13){\small$\mtx{D}_{14}$}
\put(70,03){\small$\mtx{D}_{15}$}
\put(00,63){\small$\mtx{K}_{9,8}$}
\put(10,73){\small$\mtx{K}_{8,9}$}
\put(60,03){\small$\mtx{K}_{15,14}$}
\put(70,13){\small$\mtx{K}_{14,15}$}
\put(40,23){\small$\mtx{K}_{13,12}$}
\put(50,33){\small$\mtx{K}_{12,13}$}
\put(20,43){\small$\mtx{K}_{11,10}$}
\put(30,53){\small$\mtx{K}_{10,11}$}
\put(16,18){\small$\mtx{K}_{3,2}$}
\put(56,58){\small$\mtx{K}_{2,3}$}
\put(06,48){\small$\mtx{K}_{5,4}$}
\put(26,68){\small$\mtx{K}_{4,5}$}
\put(46,08){\small$\mtx{K}_{7,6}$}
\put(66,28){\small$\mtx{K}_{6,7}$}
\put(-10,38){\small$\mtx{K} = $}
\end{picture}}
\end{picture}

\caption{(a) A binary tree with three levels; see~Section \ref{sec:HODLR}.
Each node $\tau$ in the tree owns an index vector $I_{\tau}$ that is
a subset of the full index vector $I = 1:n$. For the root node, we set $I_{1} = I$.
Each split in the tree represents a disjoint partition of the corresponding
index vectors, so that, for example, $I_1 = I_2 \cup I_3$ and $I_7 = I_{14} \cup I_{15}$.
(b) A matrix $\mtx{K}$ tessellated according to the tree shown in (a).
}
\label{fig:tree}
\end{figure}

\subsection{Compressing a rank-structured hierarchical matrix through the matrix-vector multiplication only}
\label{sec:HODLRcompression}

In this section, we describe a randomized technique for computing a data sparse representation of a
matrix that is compressible in the ``HODLR'' format that we introduced in Section \ref{sec:HODLR}. This technique
interacts with $\mtx{K}$ only through the application of $\mtx{K}$ and $\mtx{K}^{*}$ to vectors,
and is in this sense a ``black-box'' technique. In particular, we do not need the ability to evaluate
individual entries of $\mtx{K}$. The following theorem summarizes the main result:

\begin{theorem}
Let $\mtx{K}$ be a HODLR matrix associated with a fully populated binary tree on the index vector,
as described in Section \ref{sec:HODLR}. Suppose that the tree has $L$ levels, that each off-diagonal
block has rank at most $k$, and that each leaf node in the tree holds at most $ck$ indices for some fixed number $c$. Then
the diagonal blocks $\mtx{D}_{\tau}$, as well as rank-$k$ factorizations of all sibling interaction
matrices $\mtx{K}_{\alpha,\beta}$ can be computed by a randomized algorithm with cost at most
$$
T_{\rm total} = T_{\rm matvec} \times (4L+c)k + T_{\rm flop} \times \bigO(L^{2}k^{2}n),
$$
where $T_{\rm matvec}$ is the cost of applying either $\mtx{K}$ or $\mtx{K}^{*}$ to a vector, and
$T_{\rm flop}$ is the cost of a floating point operation.
\end{theorem}

The proof of the theorem consists of an explicit algorithm for building all matrices that is often
referred to as the ``peeling algorithm''. It was originally published in \cite{2011_lin_lu_ying},
with later modifications proposed in \cite{2016_martinsson_hudson2}. The proof below is based
on \cite[Sec.~17.4]{2019_book}. We give the proof for the case described in the theorem involving a
matrix whose off-diagonal blocks have exact rank $k$. In practical applications, it is of course more
typical to have the off-diagonal blocks be only of approximate low rank. In this case, the exact same
algorithm can be used, but the number of samples drawn should be increased from $k$ to $k+p$ for some
modest oversampling parameter $p$.

\begin{proof}
The algorithm is a ``top-down'' technique that
compresses the largest blocks first, and then moves on to process one level at a time of successively
smaller blocks. In describing the technique, we assume that the blocks are numbered as shown in Figure
\ref{fig:tree}(a).

In the first step of the algorithm, we build approximations to the
two largest blocks $\mtx{K}_{2,3}$ and $\mtx{K}_{3,2}$, shown in red in Figure \ref{fig:peeling}(a).
To do this, we form a random matrix $\mtx{\Omega}$ of size  $n\times 2k$ and then use the
matrix-vector multiplication to form a sample matrix
$$
\begin{array}{cccccccccccc}
\mtx{Y} &=& \mtx{K}&\mtx{\Omega}.\\
n \times 2k && n \times n & n\times 2k
\end{array}
$$
The key idea of the peeling algorithm is to insert zero blocks in the random matrix $\mtx{\Omega}$ in the
pattern shown in Figure \ref{fig:peeling}(a). The zero blocks permit us to extract ``pure'' samples
from the column spaces of the blocks $\mtx{K}_{2,3}$ and $\mtx{K}_{3,2}$. For instance, the blocks
labeled $\mtx{Y}_{2}$ and $\mtx{Y}_{3}$ in the figure are given by the formulas
$$
\begin{array}{cccccccccccc}
\mtx{Y}_{2} &=& \mtx{K}_{2,3}&\mtx{\Omega}_{3}\\
\tfrac{n}{2} \times k && \tfrac{n}{2} \times \tfrac{n}{2} & \tfrac{n}{2}\times k
\end{array}
\qquad\mbox{and}\qquad
\begin{array}{cccccccccccc}
\mtx{Y}_{3} &=& \mtx{K}_{3,2}&\mtx{\Omega}_{2}\\
\tfrac{n}{2} \times k && \tfrac{n}{2} \times \tfrac{n}{2} & \tfrac{n}{2}\times k
\end{array}
$$
By orthonormalizing the matrices $\mtx{Y}_{2}$ and $\mtx{Y}_{3}$, we obtain ON bases
$\mtx{U}_{2}$ and $\mtx{U}_{3}$ for the off-diagonal blocks $\mtx{K}_{2,3}$ and $\mtx{K}_{3,2}$.
In order to complete the factorization of $\mtx{K}_{2,3}$ and $\mtx{K}_{3,2}$, we will perform
an operation that is the equivalent of ``Stage B'' in Section \ref{sec:rsvd}. To this end,
we form a test matrix by interlacing the matrices $\mtx{U}_{2}$ and $\mtx{U}_{3}$ with zero
blocks, to form the $n\times 2k$ matrix shown in Figure \ref{fig:peeling_trans}. Applying
$\mtx{K}^{*}$ to this test matrix, we get the sample matrices
$$
\mtx{Z}_{2} = \mtx{K}_{3,2}^{*}\mtx{U}_{3},
\qquad\mbox{and}\qquad
\mtx{Z}_{3} = \mtx{K}_{2,3}^{*}\mtx{U}_{2}.
$$
Since $\mtx{U}_{2}$  holds an ON-basis for the column space of $\mtx{K}_{2,3}$, it follows that
$$
\mtx{K}_{2,3} = \mtx{U}_{2}\mtx{U}_{2}^{*}\mtx{K}_{2,3} = \mtx{U}_{2}\mtx{Z}_{3}^{*},
$$
which establishes the low rank factorization of $\mtx{K}_{2,3}$.\footnote{One may if desired
continue the factorization process and form an ``economy size'' SVD of $\mtx{Z}_{3}^{*}$
so that $\mtx{Z}_{3}^{*} = \hat{\mtx{U}}_{2}\mtx{\Sigma}_{2,3}\mtx{V}_{3}^{*}$. This results in the
factorization $\mtx{K}_{2,3} = \bigl(\mtx{U}_{2}\hat{\mtx{U}}_{2}\bigr)\mtx{\Sigma}_{2,3}\mtx{V}_{3}^{*}$,
which is an SVD of $\mtx{K}_{2,3}$. However, for purposes of establishing the theorem, we may leave
$\mtx{Z}_{3}$ alone, and let $\{\mtx{U}_{2},\mtx{Z}_{3}\}$ be our ``compressed'' representation of
$\mtx{K}_{2,3}$.} The block $\mtx{K}_{3,2}$ is of course factorized analogously.

\begin{figure}
\centering
\begin{tabular}{p{40mm}p{88mm}}
(a) In the first step of the algorithm, the sibling interaction matrices
$\mtx{K}_{2,3}$ and $\mtx{K}_{3,2}$ on level 1, shown in red, are compressed.
&
\raisebox{-55mm}{\includegraphics[width=85mm]{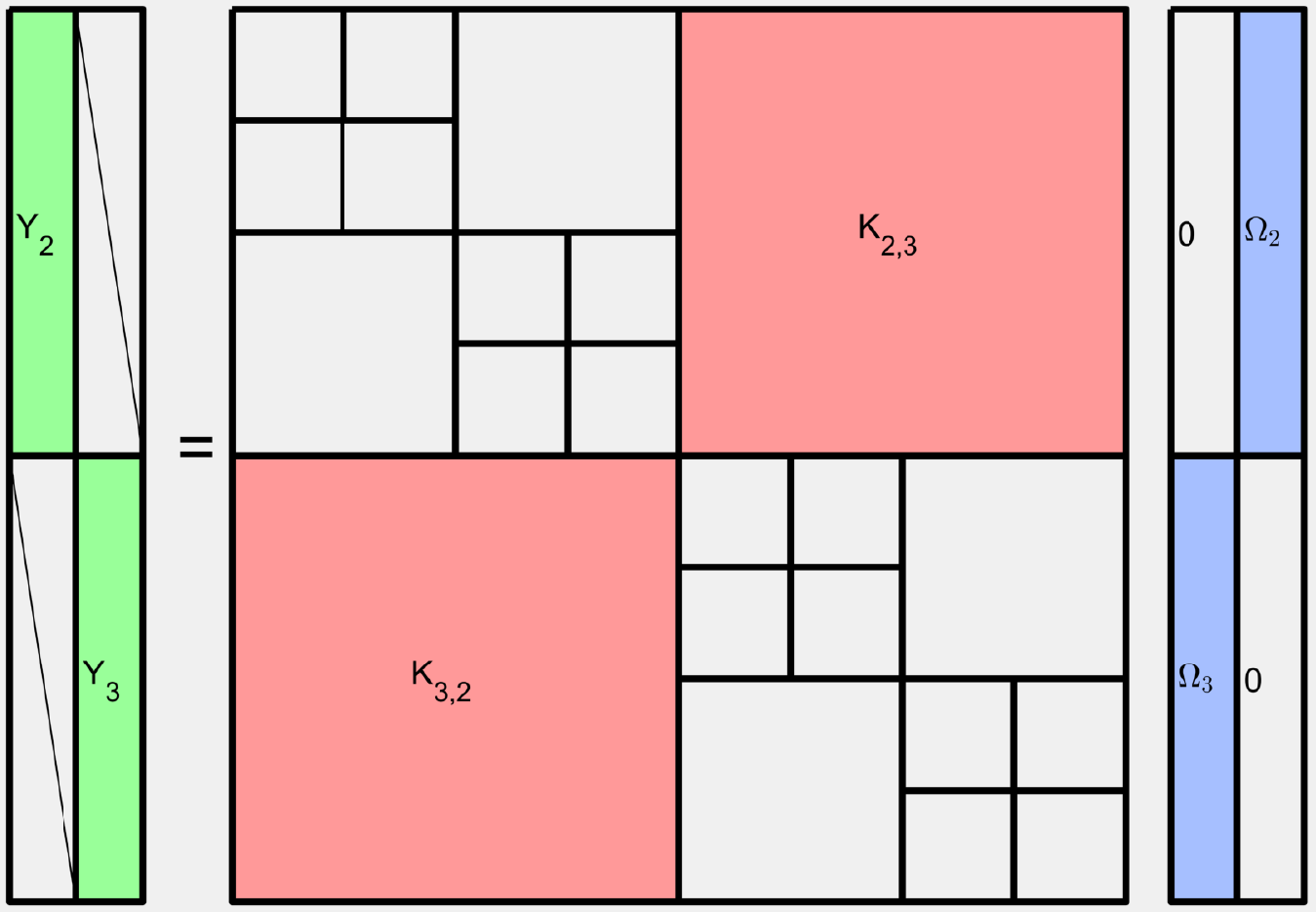}}\\
(b) In the second step, the four sibling interaction matrices on level 2,
shown in red, are compressed. In this step, we exploit that we now possess
factorizations of the gray blocks.
&
\raisebox{-55mm}{\includegraphics[width=85mm]{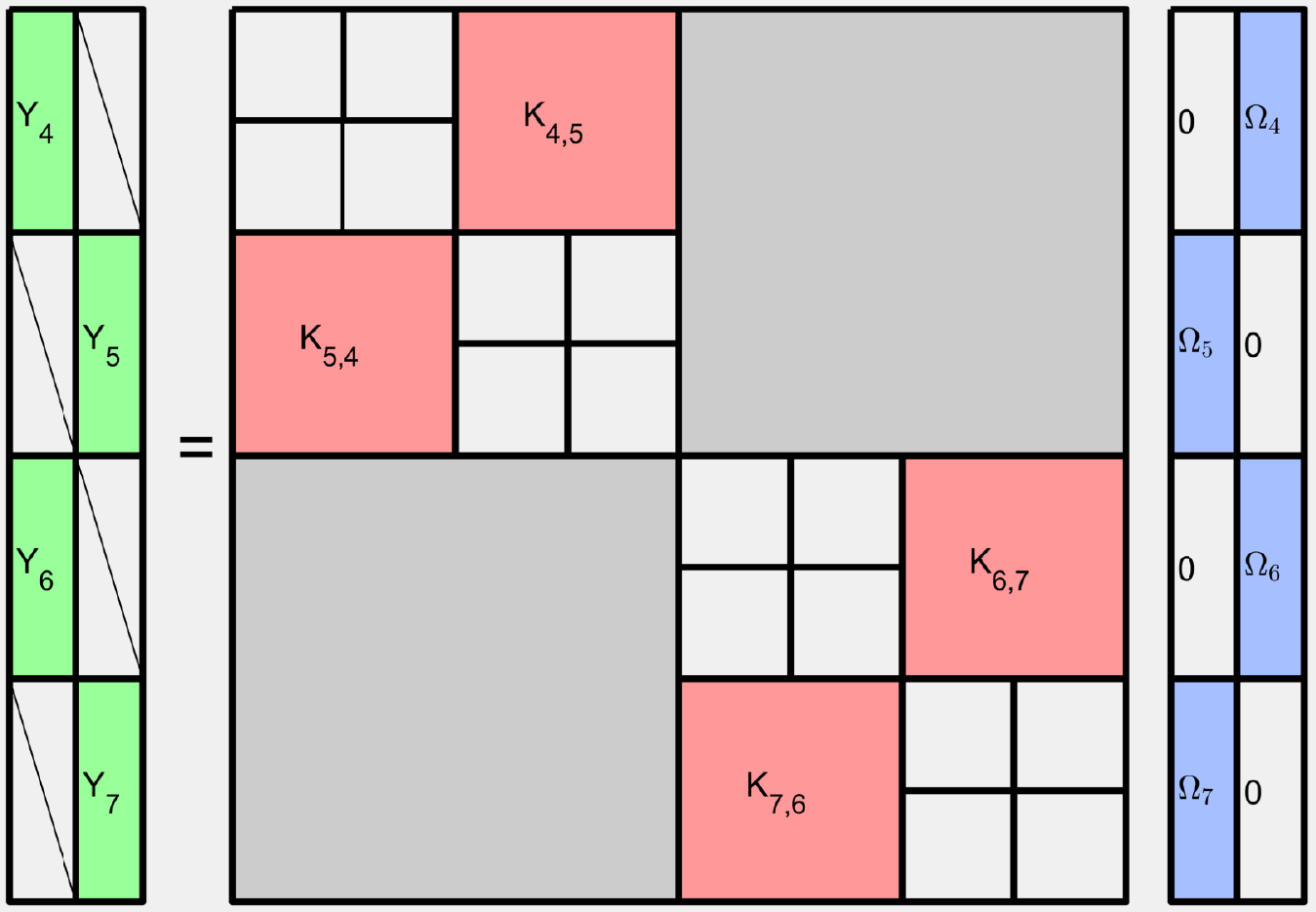}}\\
(c) In the third step, the eight sibling interaction matrices on level 8,
shown in red, are compressed. We again exploit that we possess
factorizations of the gray blocks.
&
\raisebox{-55mm}{\includegraphics[width=85mm]{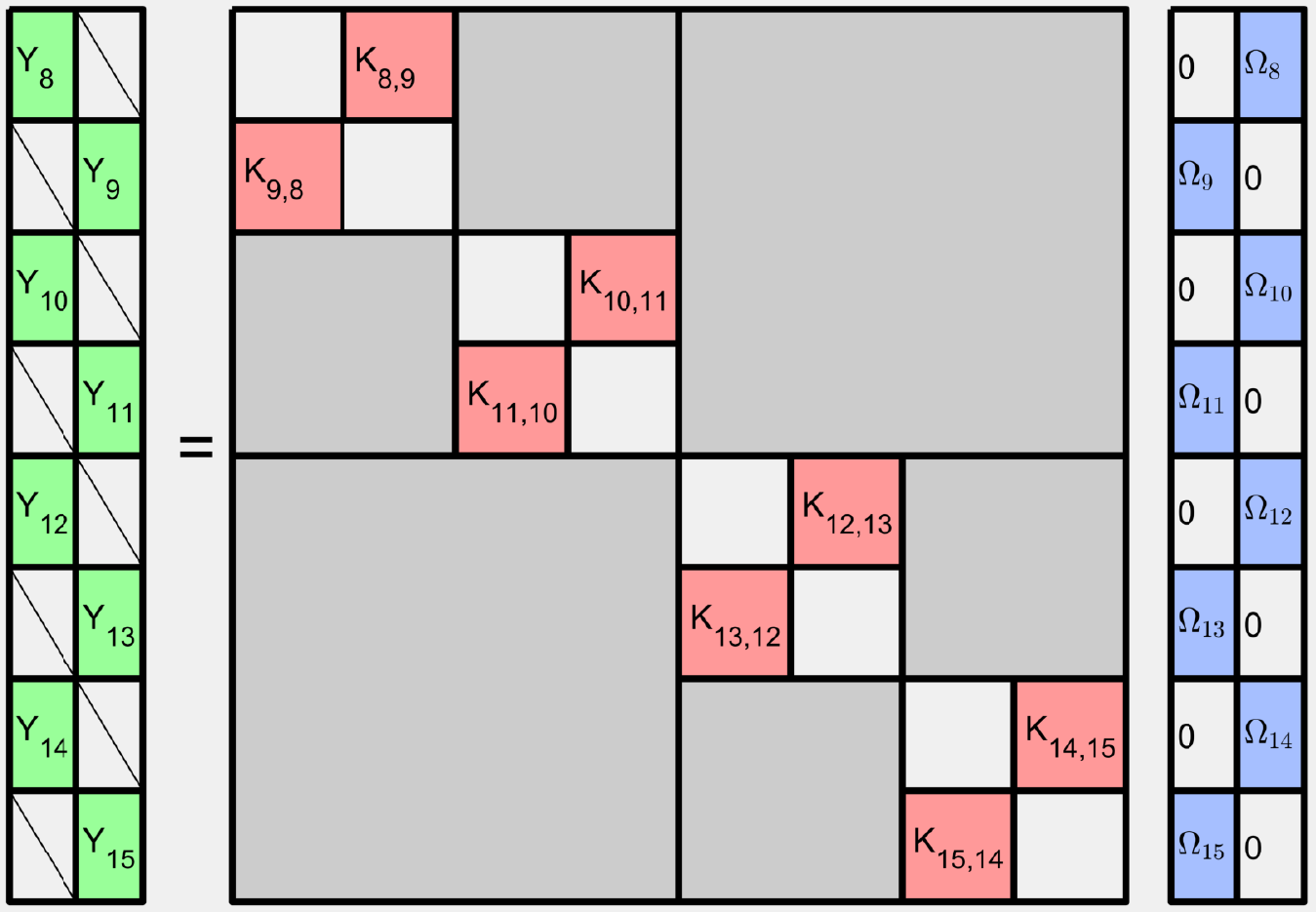}}
\end{tabular}
\caption{The ``peeling algorithm'' for computing a HODLR representation of a matrix
described in Section \ref{sec:HODLRcompression}.}
\label{fig:peeling}
\end{figure}

At the second step of the algorithm, the objective is to build approximations to the sibling matrices
at the next finer level, shown in red in Figure \ref{fig:peeling}(b). We use a random matrix $\mtx{\Omega}$
of the same size, $n\times 2k$, as in the previous step, but now with four zero blocks, as shown
in the figure. Looking at the sample matrix $\mtx{Y}_{4}$, we find that it takes the form
$$
\begin{array}{cccccccccccc}
\mtx{Y}_{4} &=& \mtx{K}_{4,5}&\mtx{\Omega}_{5} &+& \mtx{B}_{4},\\
\tfrac{n}{4} \times k && \tfrac{n}{4} \times \tfrac{n}{4} & \tfrac{n}{4}\times k && \tfrac{n}{4}\times k
\end{array}
$$
where $\mtx{B}_{4}$ represents contributions from interactions between the block $\mtx{K}_{2,3}$ and
the random matrix and $\mtx{\Omega}_{7}$. But now observe that we at this point
are in possession of a low rank approximation to $\mtx{K}_{2,3}$. This allows us to compute $\mtx{B}_{4}$
explicitly, and then obtain a ``pure'' sample from the column space of $\mtx{K}_{4,5}$ by subtracting
this contribution out:
$$
\mtx{Y}_{4} - \mtx{B}_{4} = \mtx{K}_{4,5}\mtx{\Omega}_{5}.
$$
By orthonormalizing the matrix $\mtx{Y}_{4} - \mtx{B}_{4}$, we obtain the ON matrix $\mtx{U}_{4}$ whose
columns span the column space of $\mtx{K}_{4,5}$.
The same procedure of course also allows us to build bases for the column spaces of all matrices at
this level, and we can then extract bases for the row spaces and the cores of the matrices by sampling
$\mtx{K}^{*}$.

At the third step of the algorithm, the object is to compress the blocks marked in red in Figure \ref{fig:peeling}(c).
The random matrix used remains of size $n\times 2k$ but now contains 8 zero blocks, as shown in the figure.
The sample matrix $\mtx{Y}_{8}$ takes the form
$$
\begin{array}{cccccccccccc}
\mtx{Y}_{8} &=& \mtx{K}_{8,9}&\mtx{\Omega_{9}} &+& \mtx{B}_{8}\\
\tfrac{n}{8} \times k && \tfrac{n}{8} \times \tfrac{n}{8} & \tfrac{n}{8}\times k && \tfrac{n}{8}\times k
\end{array}
$$
where $\mtx{B}_{8}$ represents contributions from off-diagonal blocks at the coarser levels, shown in
gray in the figure. Since we already possess data sparse representations of these blocks, we can
subtract their contributions out, just as at the previous level.

Once all levels have been processed, we hold low-rank factorizations of all off-diagonal blocks, and
all that remains is to extract the diagonal matrices $\mtx{D}_{\tau}$ for all leaf nodes $\tau$. Let us
for simplicity assume that all such block are of the same size $m\times m$. We then form a test matrix
by stacking $n/m$ copies of an $m\times m$ identity matrix atop of each other and applying $\mtx{K}$
to this test matrix. We subtract off the contributions from all the off-diagonal blocks using the
representation we have on hand, and are then left with a sample matrix consisting of all
the blocks $\mtx{D}_{\tau}$ stacked on top of each other.
\end{proof}

\begin{figure}
\centering
\includegraphics[width=85mm]{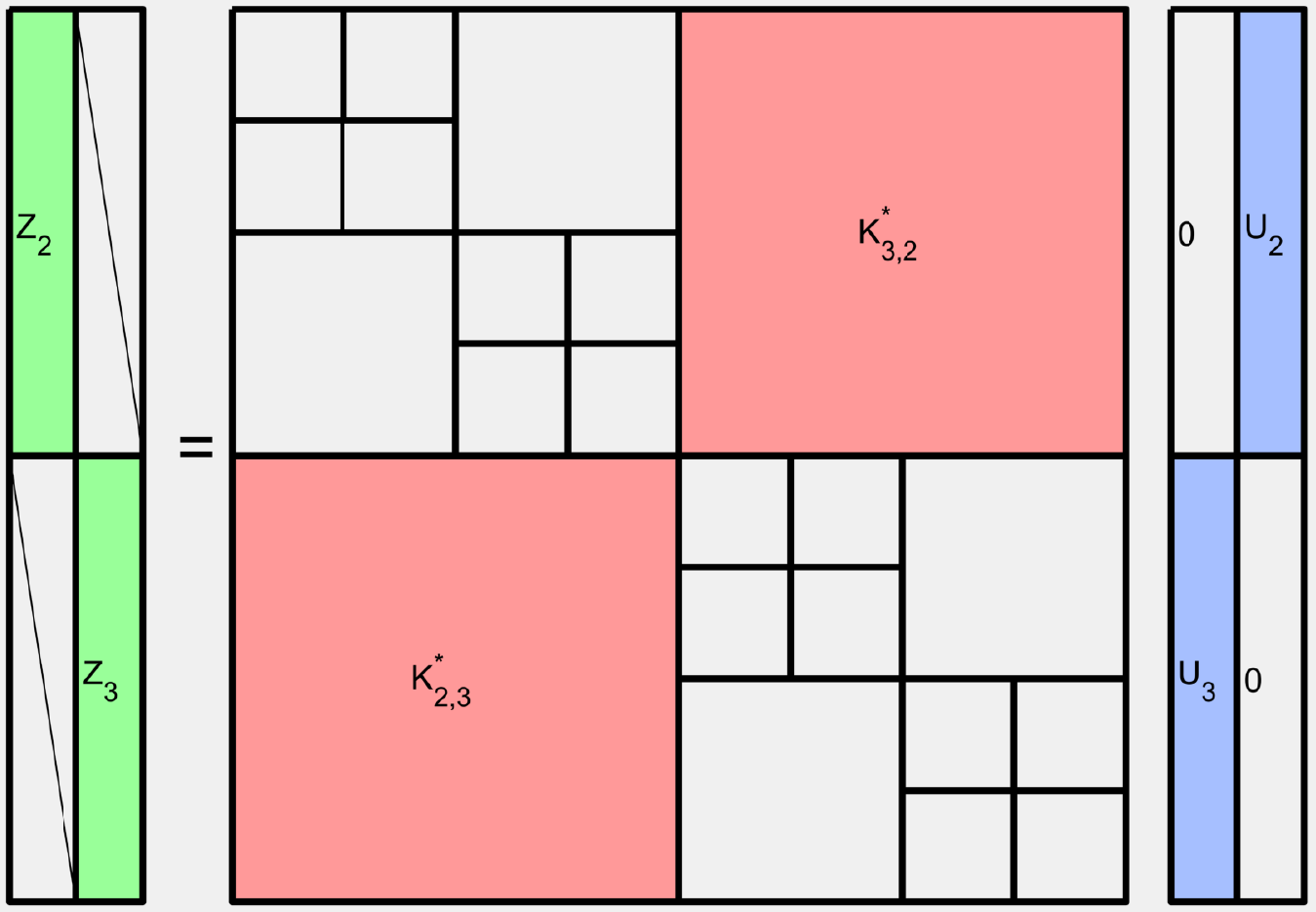}
\caption{The process used to complete the factorization of the blocks $\mtx{K}_{2,3}$
and $\mtx{K}_{3,2}$ once the basis matrices $\mtx{U}_{2}$ and $\mtx{U}_{3}$ have been
constructed. We put $\mtx{U}_{2}$ and $\mtx{U}_{3}$ into the $n\times 2k$ test matrix shown,
and form the sample matrix that holds $\mtx{Z}_{2}$ and $\mtx{Z}_{3}$ by applying $\mtx{K}^{*}$.}
\label{fig:peeling_trans}
\end{figure}

\subsection{A linear complexity data sparse format}
\label{sec:HBS}

The ``HODLR'' data sparse matrix format discussed in Sections \ref{sec:HODLR} and \ref{sec:HODLRcompression}
is simple to use, and is in many applications efficient enough. However, its storage complexity
is at least $\bigO(n\log n)$ as the matrix size $n$ grows. We will next describe
how the logarithmic factor can be eliminated. To keep the presentation concise, we focus on
the asymptotic storage requirements, although similar estimates hold for the asymptotic flop count.

To start, let us first investigate why there are logarithmic factors in the storage requirement in the
HODLR format. Suppose that for a sibling pair $\{\alpha,\,\beta\}$ the corresponding off-diagonal
block $\mtx{K}_{\alpha,\beta}$ is of size $m\times m$, has rank $k$, and is stored in the form
of a factorization
$$
\begin{array}{ccccccccccccccc}
\mtx{K}_{\alpha,\beta} &=& \mtx{U}^{\rm long}_{\alpha}&\tilde{\mtx{K}}_{\alpha,\beta}&\bigl(\mtx{V}_{\beta}^{\rm long}\bigr)^{*}\\
m\times m && m\times k & k\times k& k\times m
\end{array}
$$
where $\mtx{U}_{\alpha}^{\rm long}$ and $\mtx{V}_{\beta}^{\rm long}$ are two matrices whose columns
form bases for the column and row space of $\mtx{K}_{\alpha,\beta}$, respectively.\footnote{The matrices
$\mtx{U}_{\alpha}^{\rm long}$ and $\mtx{V}_{\beta}^{\rm long}$ were denoted
$\mtx{U}_{\alpha}$ and $\mtx{V}_{\beta}$ in Sections \ref{sec:HODLR} and \ref{sec:HODLRcompression}.}
We see that $\mtx{K}_{\alpha,\beta}$ can be stored using about $2mk$ floating point numbers (where we ignore
$\bigO(k^{2})$ terms, since typically $k \ll m$). Now let us consider how many floats are required to store
all sibling interaction matrices on a given level $\ell$ in the tree. There are $2^{\ell}$ such
matrices, and they each have size $m \approx 2^{-\ell}n$, resulting in a total of
$2^{\ell}\times 2\times 2^{-\ell}nk = 2nk$ floats. Since there are $\bigO(\log(n))$ levels in the tree,
it follows that just the task of storing all the basis matrices requires $\bigO(kn\log(n))$ storage.

A standard technique for overcoming the problem of storing all the ``long'' basis matrices is to
assume that the basis matrices on one level can be formed through slight modifications to the basis
matrices on the next finer level. To be precise, we assume that if $\tau$ is a node in the tree
with children $\alpha$ and $\beta$, then there exists a ``short'' basis matrix $\mtx{U}_{\tau}$
of size $2k\times k$ such that
\begin{equation}
\label{eq:nestedbasis}
\begin{array}{ccccccccccccc}
\mtx{U}_{\tau}^{\rm long} &=&
\left[\begin{array}{cc} \mtx{U}_{\alpha}^{\rm long} & \mtx{0} \\ \mtx{0} & \mtx{U}_{\beta}^{\rm long}\end{array}\right]
&
\mtx{U}_{\tau}\\
m\times m && m\times 2k & 2k \times k
\end{array}
\end{equation}
(with of course an analogous statement holding for the ``V'' matrices holding bases for the row spaces).
The point is that if we already have the ``long'' basis matrices for the children available,
then the long basis matrices for the parent can be formed using the information in the small
matrix $\mtx{U}_{\tau}$. By applying this argument recursively, we see
that for any parent node $\tau$, all that needs to be stored explicitly are small
matrices of size $2k\times k$.

To illustrate the idea, suppose that we use this technique for storing basis matrices to the
tree with three levels described in Section \ref{sec:HODLR}. Then for each of the 8 leaf boxes,
we compute basis matrices of size $n/8 \times k$ and store these. At the next coarser level, the
long basis matrix for a box such as $\tau=4$ can be expressed through the formula
$$
\begin{array}{ccccccccccccc}
\mtx{U}_{4}^{\rm long} &=&
\left[\begin{array}{cc} \mtx{U}_{8}^{\rm long} & \mtx{0} \\ \mtx{0} & \mtx{U}_{9}^{\rm long}\end{array}\right]
&
\mtx{U}_{4},\\
(n/4)\times m && (n/4)\times 2k & 2k \times k
\end{array}
$$
so that only the $2k\times k$ matrix $\mtx{U}_{4}$ needs to be stored. At the next coarser level
still, the long basis matrix $\mtx{U}_{2}^{\rm long}$ is expressed as
$$
\begin{array}{ccccccccccccc}
\mtx{U}_{2}^{\rm long} &=&
\left[\begin{array}{cccc}
\mtx{U}_{8}^{\rm long}  & \mtx{0}                 & \mtx{0}                 & \mtx{0} \\
\mtx{0}                 & \mtx{U}_{9}^{\rm long}  & \mtx{0}                 & \mtx{0} \\
\mtx{0}                 & \mtx{0}                 & \mtx{U}_{10}^{\rm long} & \mtx{0} \\
\mtx{0}                 & \mtx{0}                 & \mtx{0}                 & \mtx{U}_{11}^{\rm long}
\end{array}\right]
&
\left[\begin{array}{cc} \mtx{U}_{4} & \mtx{0} \\ \mtx{0} & \mtx{U}_{5}\end{array}\right]
&
\mtx{U}_{2}.\\
(n/2)\times m && (n/2)\times 4k & 4k\times 2k & 2k \times k
\end{array}
$$

We say that a matrix that can be represented using nested basis matrices in the manner
described is a \textit{hierarchically block separable (HBS)} matrix. (This format is
very closely related to the \textit{hierarchically semi separable (HSS)} matrix format
described in \cite{2010_gu_xia_HSS,2005_gu_HSS}.) When the HBS format is used,
it is natural to assume that each leaf index vector holds $\bigO(k)$ indices,
which means that there are overall $\bigO(n/k)$ nodes in the tree. Since we only need $\bigO(k^2)$
storage per node, we see that the overall storage requirements improve from $\bigO(kn\log(n))$ to
$\bigO(kn)$.

There is a price to be paid for using ``nested'' basis matrices, and it is that the rank $k$ required
to reach a requested precision $\varepsilon$ typically is higher when nested basis matrices are used,
in comparison to the HODLR format. To be precise, one can show that for the relation (\ref{eq:nestedbasis})
to hold, it is necessary and sufficient that for a node $\alpha$, the columns of $\mtx{U}_{\alpha}^{\rm long}$
must span the column space of the matrix $\mtx{K}(I_{\alpha},I_{\alpha}^{\rm c})$
(where $I_{\alpha}^{\rm c} = I \backslash I_{\alpha}$). In contrast, in the HODLR
format, they only need to span the columns of $\mtx{K}(I_{\alpha},I_{\beta})$, where $\beta$ is the sibling of
$\alpha$. This is easier to do since $I_{\beta}$ is a subset of the index vector $I_{\alpha}^{\rm c}$.

\subsection{A linear complexity randomized compression technique}
\label{sec:HBScompression}

We saw in Section \ref{sec:HBS} that the so called ``HBS'' data sparse format allows us to store
a matrix using $\bigO(kn)$ floating point numbers, without any factor of $\log(n)$. But is it also
possible to compute such a representation in optimal linear complexity? The answer is yes, and a
number of deterministic techniques that work in specific environments have been proposed
\cite[Ch.~17]{2019_book}.
A linear complexity randomized technique was proposed in \cite{2011_martinsson_randomhudson}.
This method is based on analyzing samples of the matrix obtained through the application of $\mtx{K}$
and $\mtx{K}^{*}$ to Gaussian random vectors, but is not quite a black-box method, as it
also requires the ability to evaluate $\bigO(kn)$ individual elements of $\mtx{K}$ itself. This
is not always possible, of course, but when it is, the technique is lightning fast in
practice. To be precise, \cite{2011_martinsson_randomhudson} establishes the following:

\begin{theorem}
\label{thm:randHBS}
Let $\mtx{K}$ be an $n\times n$ matrix that is compressible in the HBS format described in Section
\ref{sec:HBS}, for some rank $k$. Let $T_{\rm apply}$ denote the time it takes to evaluate
the two products
\begin{align}
\label{eq:randomhudsonsample1}
\mtx{Y} =&\ \mtx{K}\mtx{G},\\
\label{eq:randomhudsonsample2}
\mtx{Z} =&\ \mtx{K}^{*}\mtx{G},
\end{align}
where $\mtx{G}$ is an $n\times k$ matrix drawn from a Gaussian distribution.
Then a full HBS representation of $\mtx{K}$ can be computed at cost bounded by
$$
T_{\rm total} = T_{\rm apply} + T_{\rm entry} \times \bigO(nk) + T_{\rm flop} \times \bigO(nk^{2}),
$$
where $T_{\rm flop}$ is the cost of a floating point operation, and where $T_{\rm entry}$ is the
time it takes to evaluate an individual entry of $\mtx{K}$
\end{theorem}

When the off-diagonal blocks are only approximately of rank $k$, we do some over sampling as usual,
and replace the number $k$ in the theorem by $k+p$ for some small number $p$.
This of course results in an approximate HBS representation of $\mtx{K}$.

The proof of Theorem \ref{thm:randHBS} is an algorithm that explicitly builds the data sparse
representation of the matrix within the specified time budget. The algorithm consists of a pass
through all nodes in the hierarchical tree, going from smaller boxes to larger; ``bottom-up'', as
opposed to the ``top-down'' method for HODLR matrices in Section \ref{sec:HODLRcompression}.
We provide an outline of the proof below. This outline is written to convey the main ideas without
introducing cumbersome notation; for full details, see the original article
\cite{2011_martinsson_randomhudson} or \cite[Sec.~17.4]{2019_book}.

\begin{proof}
To describe the algorithm that establishes Theorem \ref{thm:randHBS}, we will walk through how it applies
to a matrix tessellated in accordance with the
hierarchical tree in Figure \ref{fig:tree}(a). We start by showing how we can use the information in the
matrix $\mtx{Y}$ defined by (\ref{eq:randomhudsonsample1})
to build the basis matrix $\mtx{U}_{8}$. As we saw in Section \ref{sec:HBS}, we need the columns of $\mtx{U}_{8}$
to span the columns in the submatrix $\mtx{K}(I_{8},I_{8}^{\rm c})$. We achieve this by constructing the
sample matrix
$$
\mtx{Y}_{8} :=
\mtx{K}(I_{8},I_{8}^{\rm c})\mtx{G}(I_{8}^{\rm c},:).
$$
Since $\mtx{G}(I_{8}^{\rm c},:)$ is a Gaussian random matrix of the appropriate size, the columns of $\mtx{Y}_{8}$
will form an approximate basis for the column space of $\mtx{K}(I_{8},I_{8}^{\rm c})$.
But how do we form $\mtx{Y}_{8}$ from the sample matrix $\mtx{Y}$  defined by (\ref{eq:randomhudsonsample1})?
The method we use is illustrated in Figure \ref{fig:hbs_randomized}(a), where the block $\mtx{K}(I_{8},I_{8}^{\rm c})$
consists of the gray block inside the red rectangle. We see that in order to isolate the product between
$\mtx{K}(I_{8},I_{8}^{\rm c})$  and $\mtx{G}(I_{8}^{\rm c},:)$, all we need to do is to subtract the contribution
from the diagonal block $\mtx{K}(I_{8},I_{8})$ (marked in white in the figure). In other words,
\begin{equation}
\label{eq:Y8}
\mtx{Y}_{8} =
\mtx{K}(:,I_{8})\mtx{G} - \mtx{K}(I_{8},I_{8})\mtx{G}(I_{8},:) =
\mtx{Y}(I_{8},:) - \mtx{K}(I_{8},I_{8})\mtx{G}(I_{8},:).
\end{equation}
The formula (\ref{eq:Y8}) can be evaluated inexpensively since the block $\mtx{K}(I_{8},I_{8})$ is small, and can
be explicitly formed since we assume that we have access to individual entries of the matrix $\mtx{K}$.
Once $\mtx{Y}_{8}$ has been formed, we compute its row interpolatory decomposition to form a basis matrix $\mtx{U}_{8}$
and an index vector $\tilde{I}_{8}^{\rm row}$ such that
$$
\mtx{K}(I_{8},I_{8}^{\rm c}) = \mtx{U}_{8}\,\mtx{K}(\tilde{I}_{8}^{\rm row},I_{8}^{\rm c}).
$$
In an entirely analogous way, we form a sample matrix $\mtx{Y}_{9}$ whose columns span the block
$\mtx{K}(I_{9},I_{9}^{\rm c})$ through the formula
$$
\mtx{Y}_{9} =
\mtx{Y}(I_{9},:) - \mtx{K}(I_{9},I_{9})\mtx{G}(I_{9},:),
$$
cf.~Figure \ref{fig:hbs_randomized}(b), and then compute the row ID $\{\mtx{U}_{9},\tilde{I}_{9}^{\rm row}\}$ of $\mtx{Y}_{9}$.

\begin{figure}
\centering
\includegraphics[width=\textwidth]{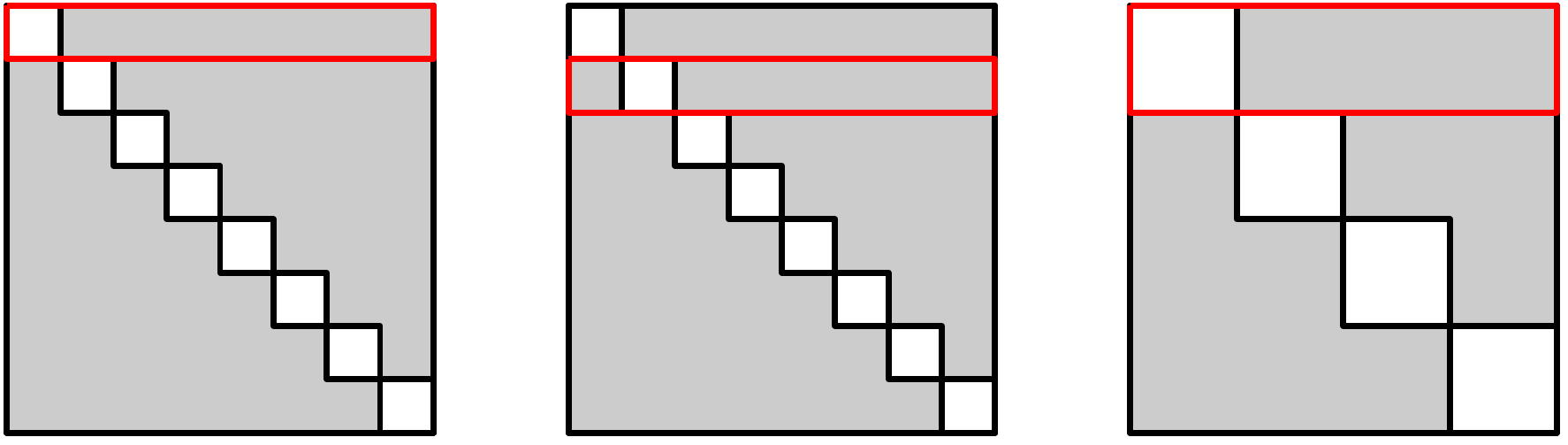}

(a) \hspace{35mm} (b) \hspace{35mm} (c)
\caption{The algorithm that establishes Theorem \ref{thm:randHBS} is illustrated using
a matrix $\mtx{K}$ tessellated in accordance with the tree shown in Figure \ref{fig:tree}(a).
(a) The block $\mtx{K}(I_{8},\colon)$ is marked with a red rectangle. Within the red
rectangle, the shaded area represents the block $\mtx{K}(I_{8},I_{8}^{\rm c})$.
(b) The block $\mtx{K}(I_{9},\colon)$ is marked with a red rectangle, with $\mtx{K}(I_{9},I_{9}^{\rm c})$ shaded.
(c) The block $\mtx{K}(I_{4},\colon)$ is marked with a red rectangle, with $\mtx{K}(I_{4},I_{4}^{\rm c})$ shaded.
}
\label{fig:hbs_randomized}
\end{figure}

The basis matrices $\mtx{V}_{\tau}$ that span the row spaces of the offdiagonal blocks for any leaf $\tau$
are built through the same procedure, but starting with the sample matrix $\mtx{Z}$ defined by
(\ref{eq:randomhudsonsample2}). For instance, we form the sample matrix
\begin{equation}
\label{eq:Z8}
\mtx{Z}_{8} =
\mtx{K}(I_{8},:)^{*}\mtx{G} - \mtx{K}(I_{8},I_{8})^{*}\mtx{G}(I_{8},:) =
\mtx{Z}(I_{8},:) - \mtx{K}(I_{8},I_{8})^{*}\mtx{G}(I_{8},:),
\end{equation}
and then compute the ID of $\mtx{Z}_{8}$ to build a basis matrix $\mtx{V}_{8}$ and an index vector
$\tilde{I}_{8}^{\rm col}$ such that
$$
\mtx{K}(I_{8}^{\rm c},I_{8}) = \mtx{K}(I_{8}^{\rm c},\tilde{I}_{8}^{\rm col})\mtx{V}_{8}^{*}.
$$

The choice to use the interpolatory decomposition in factorizing the off-diagonal blocks
is essential to making the linear complexity compression scheme work. In particular, observe
that once we have formed the row and column IDs of all the leaf boxes, we automatically obtain
rank-$k$ factorizations of all the corresponding sibling interaction matrices. For instance,
if $\{\alpha,\beta\}$ is a sibling pair consisting of two leaf nodes, then the $m\times m$
sibling interaction matrix $\mtx{K}(I_{\alpha},I_{\beta})$ admits the factorization
\begin{equation}
\label{eq:randHBSsibs}
\begin{array}{cccccccccccccc}
\mtx{K}(I_{\alpha},I_{\beta}) &=& \mtx{U}_{\alpha} & \mtx{K}(\tilde{I}_{\alpha}^{\rm row},\tilde{I}_{\beta}^{\rm col}) & \mtx{V}_{\beta}^{*}.\\
m\times m && m\times k & k\times k & k\times m
\end{array}
\end{equation}
In order to evaluate (\ref{eq:randHBSsibs}), we merely need to form the
matrix $\mtx{K}(\tilde{I}_{\alpha}^{\rm row},\tilde{I}_{\beta}^{\rm col})$.

Once the leaf nodes have all been processed, we proceed to the next coarser level.
Consider for instance the node $\tau=4$ with children $\alpha=8$ and $\beta=9$.
Our task is in principle to build the long basis
matrix $\mtx{U}_{4}^{\rm long}$ that spans the columns of $\mtx{K}(I_{4},I_{4}^{\rm c})$,
which is the shaded matrix inside the red rectangle in Figure \ref{fig:hbs_randomized}(c).
(We say ``in principle'' since it will not actually be explicitly formed.)
To this end, we define the sample matrix
$$
\mtx{Y}_{4} := \mtx{K}(I_{4},I_{4}^{\rm c})\mtx{G}(I_{4}^{\rm c},:).
$$
The idea is now to repeat the same technique that we used for the leaf boxes, and write
\begin{equation}
\label{eq:Y4}
\mtx{Y}_{4} = \mtx{Y}(I_{4},:) - \mtx{K}(I_{4},I_{4})\mtx{G}(I_{4},:).
\end{equation}
It turns out that we can evaluate (\ref{eq:Y4}) very efficiently:
The block $\mtx{K}(I_{4},I_{4})$ is made up of the two diagonal blocks
$\mtx{K}(I_{8},I_{8})$ and $\mtx{K}(I_{9},I_{9})$, and the two off-diagonal blocks
$\mtx{K}(I_{8},I_{9})$ and $\mtx{K}(I_{9},I_{8})$. Now observe that at this point
in the execution of the algorithm, we have compressed representations of all
these blocks available. Using this information, we
form a short sample matrix $\tilde{\mtx{Y}}_{4}$ of size $2k\times k$ that we can
compress to build the basis matrix $\mtx{U}_{4}$ and the associated index vector
$\tilde{I}_{4}^{\rm row}$ in an ID of $\tilde{\mtx{Y}}_{4}$.
\end{proof}

\subsection{Butterfly matrices}
\label{sec:butterfly}

An interesting class of rank-structured matrices arises from a generalization of the discrete
Fourier transform.
These so called ``butterfly matrices'' provide a data sparse format
for many matrices that appear in the analysis of neural networks, wave propagation problems, and signal processing
\cite{dao2019learning,2007_oneil_thesis,2009_candes_demanet_ying_butterfly}.
This format involves an additional complication in comparison to the HODLR and HBS
formats in that it requires $\bigO(\log n)$ different tessellations of the matrix, as illustrated
for a simple case in Figure \ref{fig:butterfly}. It can be demonstrated that when all of the resulting
submatrices are of numerically low rank, the matrix as a whole can be written (approximately) as a product
of $\bigO(\log n)$
sparse matrices, in a manner analogous to the butterfly representation of an FFT
\cite[Sec.~10.4]{1995_briggs_DFT}.

\begin{figure}
\includegraphics[width=\textwidth]{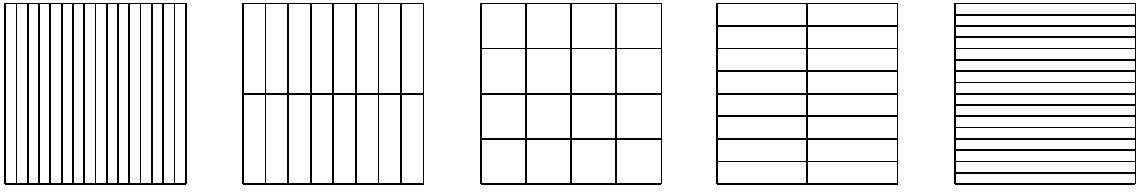}
\caption{Illustration of the \emph{butterfly} rank-structured matrix format described in Section \ref{sec:butterfly}.
The figures show the blocks that must all be of numerically low rank for a butterfly matrix arising from a binary
tree with 16 leaf nodes.}
\label{fig:butterfly}
\end{figure}

Randomization has proven to be a powerful tool for finding butterfly representations of matrices. The techniques
involved are more complex than the methods described in Sections \ref{sec:HODLR} -- \ref{sec:HBS}
due to the multiplicative nature of the representation, and typically involve iterative refinement
rather than direct approximation
\cite{2015_ying_butterfly,2017_interpolative_butterfly,2018_ying_multidimensional_butterfly}.
Similar techniques played an essential role in a recent ground breaking paper
\cite{2017_michielssen_butterfly_IE} that exploits butterfly representations to
directly solve linear systems arising from the modeling of scattering problems in
the high frequency regime.

\subsection{Applications of rank-structured matrices in data analysis}
\label{sec:ASKIT}

The machinery for working with rank-structured hierarchical matrices that we have described
can be used also for the kernel matrices discussed in Section \ref{sec:kernel}
that arise in machine learning and computational statistics.
For instance, \citeasnoun{2015_oneil_gaussian_processes}
demonstrate that data sparse formats of this type can be very effective for simulating
Gaussian processes in low dimensional spaces.

When the underlying dimension grows, the techniques that we have described so far
become uncompetitive. Even $d=4$ would be considered a stretch. Fortunately, significant
progress has recently been made towards extending the essential ideas to higher dimensions.
For instance, \citeasnoun{2015_biros_ASKIT} describe a technique that is designed to
uncover intrinsic lower dimensional structures that are often present in sets of
points $\{\vct{x}_{i}\}_{i=1}^{n}$ that ostensibly live in higher dimensional spaces.
The idea is to use randomized algorithms both for organizing the points into
a hierarchical tree (that induces the tessellation of the matrix)
and for computing low rank approximations to the resulting admissible blocks.
The authors report promising numerical results for a wide selection of
kernel matrices.

In this context, it is as we saw in Section \ref{sec:kernel} rarely possible
to execute a matrix-vector multiplication, or even to evaluate more than a tiny
fraction of the entries of the matrix. This means on the one hand that sampling
must form an integral part of the compression strategy, and on the other that
firm performance guarantees are typically not available. The saving grace is that when
a kernel matrix is used for learning and data analysis, a rough approximation is
often sufficient.

\begin{remark}[Geometry oblivious methods]
A curious observation is that techniques developed for kernel matrices appear to also
be applicable for certain symmetric positive definite (pd) matrices that are not explicitly
presented as kernel matrices. This is a consequence of the well known fact that any pd matrix
$\mtx{K}$ admits a factorization
\begin{equation}
\label{eq:pd_K_fact}
\mtx{K} = \mtx{G}^{*}\mtx{G}
\end{equation}
for a so called ``Gramian matrix'' $\mtx{G}$. (If the eigenvalue decomposition of $\mtx{K}$
takes the form $\mtx{K} = \mtx{U}\mtx{\Lambda}\mtx{U}^{*}$, then a matrix of the form
$\mtx{G} = \mtx{V}\mtx{\Lambda}^{1/2}\mtx{U}^{*}$ is a Gramian if and only if $\mtx{V}$ is unitary.)
The factorization (\ref{eq:pd_K_fact}) says that the entries of $\mtx{K}$ are formed by
the inner products between the columns of $\mtx{G}$,
\begin{equation}
\label{eq:gofmm_kernel}
\mtx{K}(i,j) = \ip{\vct{g}_{i}}{\vct{g}_{j}} = k(\vct{g}_{i},\vct{g}_{i}),
\end{equation}
where $\vct{g}_{i}$ is the $i$'th column of $\mtx{G}$ (the ``Gram vector'') and
where the $k$ is the inner product kernel of Example \ref{ex:innerprodkernel}.
At this point, it becomes plausible that the techniques of \cite{2015_biros_ASKIT}
for kernel matrices associated with points in high dimensional spaces may apply
to certain pd matrices. The key to make this work is the observation that it is
not necessary to explicitly form the Gram factors $\mtx{G}$. All that is needed in
order to organize the points $\{\vct{g}_{i}\}_{i=1}^{n}$ are relative distances
and angles between the points, and we can evaluate these
from the matrix entries of $\mtx{K}$, via the formula
$$
\|\vct{g}_{i} - \vct{g}_{j}\|^{2} =
\|\vct{g}_{i}\|^{2} - 2\mbox{Re}\,\ip{\vct{g}_{i}}{\vct{g}_{j}} + \|\vct{g}_{j}\|^{2} =
\mtx{K}(i,i) - 2\mbox{Re}\,\mtx{K}(i,j) + \mtx{K}(j,j).
$$
The resulting technique was presented in  \cite{2017_biros_GOFMM} as a ``geometry oblivious FMM (GOFMM)'',
along with numerical evidence of its usefulness for important classes of matrices.
\end{remark}
\bibliographystyle{actaagsm}

\bibliography{acta_pgm_bib,acta_jat_bib}

\end{document}